\documentclass[12pt]{amsart}
\usepackage{amssymb, eucal, amsfonts, amsmath, xypic,latexsym}

\def\gr{\mathrm{gr}}

\def\g{{\mathfrak g}}

\def\h{{\mathfrak h}}

\def\sl{{\mathfrak sl}}
\def\so{{\mathfrak so}}

\def\t{{\mathfrak t}}

\def\V{{\mathcal V}}

\def\E{{\mathcal E}}
\def\z{{\mathfrak z}}
\def\m{{\mathfrak m}}
\def\l{{\mathfrak l}}
\def\n{{\mathfrak n}}

\def\p{{\mathfrak p}}

\def\Z{{\mathbb Z}}

\def\O{{\mathcal O}}
\def\End{\mathop{\fam0 End}}

\def\Lie{\mathop{\fam0 Lie}}
\def\Ad{\mathrm{Ad}\,}

\def\sl{\mathop{\mathfrak{sl}}\nolimits}
\def\so{\mathop{\mathfrak{so}}\nolimits}

\def\ad{\mathrm{ad\,}}

\theoremstyle{plain}

\newtheorem{prop}{Proposition}[section]

\theoremstyle{definition}

\theoremstyle{remark}

\newtheorem{rem}{Remark}[section]

\def\subtitle#1. {{\medskip\bf#1\par\nobreak\smallskip}}
\def\proclaim#1. {\medbreak\bgroup\noindent\bf#1. \it}

\def\endproclaim{\egroup
\ifdim\lastskip<\medskipamount\removelastskip\medskip\fi}
\newcount
=0
\def\citedef#1 {\advance\citation by1
  \expandafter\edef\csname#1\endcsname{{\the\citation}}
  \checkendcitedef}
\def\checkendcitedef#1{\ifx#1\endcitedef\else\citedef#1\fi}
\def\cite#1{\csname#1\endcsname}
\citedef  BV0 BV BK Bo BG  BGK Bry C Chen CS CM dG Dix Du1 Du2  GG Ge Gi GRU
Hu Hu1 Ja Ja1 Jo1 JoT Jo2 KL LT LS Lo  Lo1  Lo2  Lo10 Mc Mo P02 P07 P07'
P10 P11 P2 PT Sk U Vo Y
\endcitedef
\newtoks\nextauth
\newif\iffirstauth
\def\checkendauth#1{\ifx\endauth#1
        \iffirstauth\the\nextauth
        \else{} and \the\nextauth\fi,
    \else\iffirstauth\the\nextauth\firstauthfalse
        \else, \the\nextauth\fi
        \expandafter\auth\expandafter#1\fi}
\def\auth#1 #2 {\nextauth={#1 #2}\checkendauth}
\newif\ifinbook
\newif\ifbookref
\def\nextref#1 {\bookreffalse\inbookfalse
    \bibitem[\cite{#1}]{}
    \firstauthtrue
    \ignorespaces}
\def\paper#1{{\it#1,}}
\def\In#1{\inbooktrue In #1,}
\def\book#1{\bookreftrue{\it#1,}}
\def\journal#1{#1\ifinbook,\fi}
\def\bookseries#1{#1,}
\def\Vol#1{\ifbookref Vol. #1,\else\ifinbook Vol. #1,\else{\bf#1}\fi\fi
    \space\ignorespaces}
\def\nombre#1{no. #1}
\def\publisher#1{#1,}
\def\Year#1{\ifbookref #1.\else\ifinbook #1,\else(#1)\fi\fi
    \space\ignorespaces}
\def\Pages#1{\ifinbook pp. #1.\else #1.\fi}
\begin{document}
\title{Multiplicity-free primitive ideals associated with rigid nilpotent orbits}
\author{Alexander Premet}
\thanks{\nonumber{\it Mathematics Subject Classification} (2000 {\it revision}).
Primary 17B35. Secondary 17B25.}
\address{School of Mathematics, University of Manchester, Oxford Road,
M13 9PL, UK} \email{Alexander.Premet@manchester.ac.uk} \maketitle

\begin{center}
{\it Dedicated to Evgenii Borisovich Dynkin on the accasion of his 90th birthday}
\end{center}

\bigskip

\begin{abstract}
\noindent Let $G$ be a simple algebraic group over $\mathbb{C}$. Let $e$ be a nilpotent element in $\g=\Lie(G)$ and denote by 
$U(\g,e)$ the finite $W$-algebra associated with the pair $(\g,e)$. It is known that the component group
$\Gamma$ of the centraliser of $e$ in $G$ acts on the set $\mathcal{E}$ of all one-dimensional representations of $U(\g,e)$. In this paper we prove that the fixed point set $\mathcal{E}^\Gamma$ is 
non-empty. As a corollary, all finite $W$-algebras associated with $\g$ admit one-dimensional representations. In the case of rigid nilpotent elements in exceptional Lie algebras we find irreducible highest weight $\g$-modules whose annihilators in
$U(\g)$ come from one-dimensional representations of $U(\g,e)$ via Skryabin's equivalence.
As a consequence, we show that for any nilpotent
orbit $\O$ in $\g$ there exists a multiplicity-free (and hence completely prime) primitive ideal of $U(\g)$ whose associated variety
equals the Zariski closure of $\O$ in $\g$.
\end{abstract}
\maketitle

\bigskip

\section{\bf Introduction}
Denote by $G$ a simple algebraic group of adjoint type over
$\mathbb C$ with Lie algebra $\g=\Lie(G)$ and let  $\mathcal{X}$ be the set of all primitive ideals of the universal enveloping algebra $U(\g)$. We shall occasionally identify $\g$ and $\g^*$ by means of an $(\Ad G)$-invariant non-degenerate symmetric bilinear form of $\g$. 
Given $x\in \g$ we denote by $G_x$ the centraliser of $x$ in $G$ and set $\g_x:=\Lie(G_x)$.

It is well known that for any finitely generated $S(\g^*)$-module $R$ there exist
prime ideals $\mathfrak{q}_1,\ldots,\mathfrak{q}_n$ containing ${\rm Ann}_{S(\g^*)}\,R$ and a chain
$0=R_0\subset R_1\subset\ldots
\subset R_n=R$ of $S(\g^*)$-modules such that $R_i/R_{i-1}\cong S(\g^*)/{\mathfrak q}_i$ for $1\le i\le n$.
Let $\p_1,\ldots,\p_l$ be the minimal elements in the set $\{\mathfrak{q}_1,\ldots,\mathfrak{q}_n\}$. The zero sets $\V(\p_i)$ of the $\p_i$'s in $\g$ are the irreducible components of the support ${\rm Supp}(R)$ of $R$. If $\p$ is one of the $\p_i$'s then we define $m(\p):=\{1\le i\le n\,|\,\,\mathfrak{q}_i=\p\}$ and we call $m(\p)$ the {\it multiplicity} of $\V(\p)$ in ${\rm Supp}(R)$. The formal linear combination $\sum_{i=1}^lm(\p_i)[\p_i]$ is sometimes referred to as the {\it associated cycle} of $R$ and denoted ${\rm AC}(R)$.

For any $I\in\mathcal{X}$ we can apply the above construction to the $S(\g^*)$-module $R=S(\g^*)/{\rm gr}(I)$ where $\gr(I)$ is the corresponding graded ideal in
$\gr(U(\g))=S(\g)\cong S(\g^*)$.
The support of $S(\g^*)/\gr(I)$ in $\g$ is called the {\it associated variety} of $I$ and denoted ${\rm VA}(I)$. We refer to [\cite{Ja1}, Section~9] for more detail on supports and associated cycles. By Joseph's theorem, ${\rm VA}(I)$ is the closure of a single nilpotent orbit $\O$ in $\g$. In particular, it is always irreducible.  Therefore, in our situation the set $\{\p_1,\ldots,\p_l\}$ is  the singleton containing $J:=\sqrt{\gr(I)}$ and we have that ${\rm AC}\big(S(\g^*)/\gr(I)\big)=m(J)[J]$. The positive integer
$m(J)$ is often referred to as the {\it multiplicity} of $\O$ in $U(\g)/I$ and denoted ${\rm mult}_\O(U(\g)/I)$.  
 
Given a nilpotent orbit $\O$ in $\g$ we denote by $\mathcal{X}_{\O}$ the set of all $I\in\mathcal{X}$ with ${\rm VA}(I)=\overline{\O}$. We call $I\in\mathcal{X}_\O$ {\it multiplicity free} if ${\rm mult}_\O(U(\g)/I)=1$ and we say that $I$ is {\it completely prime} if $U(\g)/I$ is a domain. 
In type
$\sf A$ all completely prime primitive ideals of $U(\g)$ are classified by M{\oe}glin in [\cite{Mo}], but for 
simple Lie algebras $\g\not\cong\sl_n$ of rank $\ge 3$ the problem is wide open, in general. Classification of completely prime primitive ideals of $U(\g)$ is an important problem of Lie Theory which finds applications in the theory of unitary representations of complex Lie groups; see [\cite{JoT}] and [\cite{Vo}] for motivation and detail. This subject has a very long and rich history and many partial results can be found in the literature. In particular, it is known that any multiplicity-free primitive ideal is completely prime and that the converse may fail outside type $\sf A$ for simple Lie algebras of rank $\ge 3$.
A description of multiplicity-free primitive ideals in classical Lie algebras  was very recently obtained in [\cite{PT}], but even in this case many interesting problems remain open.

If $\O=\{0\}$ then the set $\mathcal{X}_\O$ consists of all primitive ideals of finite codimension in $U(\g)$. 
So, if $I\in\mathcal{X}_{\{0\}}$ then $U(\g)/I$ is a full matrix algebra by Schur's lemma. Since ${\rm Mat}_n(\mathbb{C})$ is a domain if and only if $n=1$, the only completely prime ideal in $\mathcal{X}_{\{0\}}$ is the augmentation ideal of $U(\g)$.
From now on we assume that $\O$ is a nonzero nilpotent orbit in $\g$. Let $\{e,h, f\}$ be an $\sl_2$-triple in $\g$ with $e\in\O$ and let $U(\g,e)=({\rm End}_\g\, Q_e)^{\rm op}$ be the finite $W$-algebra associated with $(\g,e)$. This algebra will be discussed in more detail in Section~2 and here we just mention that $Q_e$ is a generalised Gelfand--Graev module associated
with $\{e,h, f\}$.  
We have started with the case of the zero orbit because the general case is rather similar albeit technically much more involved. Since it is hard to deal with $U(\g)$ directly in the general case, it is sometimes convenient to use $U(\g,e)$ as a stepping stone. As various objects of interest related to the primitive quotients $U(\g)/I$ with $I\in\mathcal{X}_\O$ turn out to be
{\it finite dimensional} $U(\g,e)$-modules, this in a sense brings us back to the case where $\O=\{0\}$, the difference being that the role of $U(\g)$ is played by $U(\g,e)$ now.

If $V$ is a finite dimensional irreducible
$U(\g,e)$-module, then Skryabin's theorem in conjunction with [\cite{P07}] implies $Q_e\otimes_{U(\g,\,e)}V$ is an irreducible $\g$-module
and its annihilator $I_V$ in $U(\g)$ lies in $\mathcal{X}_\O$. Conversely, any primitive ideal in $\mathcal{X}_\O$ has this form for some finite dimensional
irreducible $U(\g,e)$-module $V$; see
[\cite{Lo}, \cite{P10}, \cite{Gi}]. The ideal $I_V$ depends only on the image of $V$ in the set ${\rm Irr}\,U(\g,e)$ of all isoclasses of finite dimensional irreducible $U(\g,e)$-modules. We write $[V]$ for the class of $V$ in ${\rm Irr}\,U(\g,e)$.

By the Dynkin--Kostant theory, the algebraic
group $C(e):=G_e\cap G_f$  is reductive and its finite quotient
$\Gamma:=C(e)/C(e)^\circ$ identifies with the component group of
the centraliser $G_e$. From the Gan--Ginzburg realization of $U(\g,e)$ it follows that $C(e)$ acts
on $U(\g,e)$ by algebra automorphisms. 
Since the connected component
$C(e)^\circ$ preserves any two-sided ideal of $U(\g,e)$, by [\cite{P07}],  we have a natural action of $\Gamma$ on ${\rm Irr}\,U(\g,e)$. For $V$ as above, we let $\Gamma_V$ denote the stabiliser of $[V]$ in $\Gamma$. Confirming the author's conjecture, Losev proved in [\cite{Lo1}] that
$I_{V'}=I_V$ if an only if $[V']=\, ^{\gamma}[V]$ for some $\gamma\in\Gamma$. In particular, $\dim V=\dim V'$. 
This result implies that
${\rm mult}_{\O}(U(\g)/I_V)\,=\,[\Gamma:\Gamma_V]\cdot (\dim V)^2$. As a consequence,
a primitive ideal $I_V$ is multiplicity free if and only if $\dim V=1$ and $\Gamma_V=\Gamma$.

By Goldie theory, the prime Noetherian ring $U(\g)/I_V$ admits a ring of fractions
isomorphic to ${\rm Mat}_n(\mathcal{D})$ for some division algebra $\mathcal{D}$ and a positive integer $n$ which  coincides with the Goldie rank of $U(\g)/I_V$.
Moreover, it is well known that $I_V$ is completely prime if and only if $n=1$. Since $n$ divides $\dim V$ by the main result of [\cite{P11}], any primitive ideal $I_V\in\mathcal{X}_\O$ for which $\dim V=1$ is completely prime.
This brings our attention to the set $\mathcal{E}$ of all one-dimensional representations of $U(\g,e)$ and its subset $\mathcal{E}^\Gamma$ consisting of all $C(e)$-stable such representations. 
Since
$\mathcal{E}$ identifies with the maximal spectrum of the largest commutative quotient $U(\g,e)^{\rm ab}$ of $U(\g,e)$, it follows that $\mathcal{E}$ is an affine variety and $\mathcal{E}^\Gamma$ is a Zariski closed subset of $\mathcal{E}$.

In [\cite{P07}], the author
conjectured that $\mathcal{E}\ne\varnothing$ for any nilpotent element of $\g$.
This conjecture attracted a considerable amount of
interest in the past few years. In [\cite{Lo}], Losev proved that $\mathcal{E}^\Gamma\ne \varnothing$
for $\g$ classical, but realised later
that this also followed from earlier work of R.~Brylinski
[\cite{Bry}]. Both Brylinski and Losev rely on the classical work of
Kraft and Procesi on nilpotent orbit closures in orthogonal and symplectic Lie algebras. In [\cite{P10}], the
author computed the dimension of $\E$ under the assumption that the conjecture holds for all proper Levi subalgebras of $\g$  and reduced proving the conjecture to the case of {\it rigid} nilpotent elements, i.e. those nilpotent elements of $\g$ that cannot be
obtained by Lusztig--Spaltenstein induction from  proper Levi
subalgebras of $\g$. It turned out that if the conjecture is true, then $\mathcal{E}$ is a finite
set if and only if $e$ is rigid.
Now if $e$ is rigid and $\g$ is classical then $\g_e=[\g_e,\g_e]$
by [\cite{Y}], whilst if $\g$ is exceptional then 
either $\g_e=[\g_e,\g_e]$ or $\g_e=\mathbb{C}e\oplus[\g_e,\g_e]$ and the second case occurs just for one rigid orbit in types ${\sf G_2}$, ${\sf F_4}$, ${\sf E_7}$ and for three rigid orbits in type ${\sf E_8}$; see [\cite{dG}].
On the other hand,
in [\cite{PT}] it is proved that for any simple Lie algebra $\g$ the equality $\g_e=[\g_e,\g_e]$
implies that $\mathcal{E}$ is either empty or contains one element. Putting all this together we see that for any rigid nilpotent element in a classical Lie algebra the set
$\mathcal{E}=\mathcal{E}^\Gamma$ is a singleton, whilst for $\g$ exceptional the inequality $|\mathcal{E}|\ge 2$ may occur only for six rigid nilpotent orbits. 
Finally, if $\g$ is classical and $e$ is rigid and special in the sense of Lusztig, then it follows from [\cite{Mc}] that one of the Duflo realisations of the multiplicity-free primitive ideal in $\mathcal{X}_\O$ is obtained by using
the Arthur--Barbasch--Vogan recipe; see [\cite{BV}].

The case of rigid nilpotent orbits in exceptional Lie algebras was
investigated by Goodwin--R{\"o}hrle--Ubly by computational methods.
In [\cite{GRU}], it was shown with the help of a specially designed
GAP programme that the conjecture holds for all 
rigid  orbits in types ${\sf G_2}$, ${\sf F_4}$, ${\sf E_6}$ and ${\sf
E_7}$ and for a handful of small rigid orbits in type ${\sf E_8}$. After [\cite{GRU}] was submitted Ubly improved the programme and was able to verify the conjecture for all rigid orbits 
excepting the
three largest ones in type ${\sf E_8}$; see
[\cite{U}]. Incidentally, one of the orbits he excluded
was used as an example by
Losev who showed in [\cite{Lo2}] that the corresponding finite
$W$-algebra admits a one-dimensional representation and found a Duflo relaisation of the corresponding primitive ideal in $\mathcal{X}_\O$. 

At this point we
should mention that when (and if) the programme used in [\cite{GRU},
\cite{U}] terminates, it outputs all one-dimensional representations of
$U(\g,e)$ along with some numerical data, but provides no
information on Duflo realisations of the corresponding primitive ideals in $\mathcal{X}_\O$. Curiously, it turned out that in all cases considered in [\cite{GRU}]
the algebras $U(\g,e)$ possessed either $1$ or $2$ one-dimensional
representations. 
More precisely, the case of two representations
occurs for nilpotent orbits labelled $\widetilde{\sf A_1}$,
${\sf\widetilde{A}_2}+{\sf A_1}$, $({\sf A_3+ A_1})'$ and ${\sf
A_3+ A_1}$ in Lie algebras of type ${\sf G_2}$, ${\sf F_4}$,
${\sf E_7}$ and ${\sf E_8}$, respectively. At the end of his thesis,
Ubly predicts that the finite $W$-algebras of type ${\sf E_8}$
corresponding to the undecided orbits with Dynkin labels ${\sf A_5+A_1}$
and ${\sf D_5(a_1)+A_2}$ should afford at least $2$
one-dimensional representations. Not surprisingly, the six Dynkin labels above correspond to the rigid orbits $\O$  with the property that  $\g_e=\mathbb{C}e\oplus[\g_e,\g_e]$ for any $e\in\O$. We conjecture
that these are precisely 
the rigid nilpotent orbits $\O$ for which the Zariski closure $\overline{\O}$ is not normal in $\g$; see Remark~5.2 for more detail. 
 
The main result of this paper is the following:

\medskip

\noindent {\bf Theorem~A.} {\it Let $e$ be any nilpotent element
in a finite dimensional simple Lie algebra $\g$ over $\mathbb C$. Let $\mathcal{E}$ be the set of all one-dimensional representations of $U(\g,e)$ and denote by $\mathcal{E}^\Gamma$ the subset of $\E$ consisting of all elements fixed by the action of $\Gamma$. Then the following hold:
\begin{itemize}
\item[(1)] $\mathcal{E}^\Gamma\ne\varnothing$.

\smallskip

\item[(2)] If $e$ is rigid then $\mathcal{E}=\mathcal{E}^\Gamma$.

\smallskip 

\item[(3)] If $e$ is rigid then $\mathcal{E}$ has one element if and only if 
$\g_e=[\g_e,\g_e]$.
\end{itemize}}

\medskip

We also show, for $e$ rigid, that the one-dimensional representations of $U(\g,e)$ are separated by the action
of the image of a Casimir element in $U(\g,e)$.
As we explained earlier, Theorem~A implies that each set $\mathcal{X}_\O$ contains a multiplicity-free primitive ideal.

Let $T$ be a maximal torus of $G$ and $\t=\Lie(T)$. Let $\Phi$ be the root system of $\g$ with respect to $T$ and let $\Pi$ be a basis of simple roots in $\Phi$.  All multiplicity-free primitive ideals $I$ constructed in this paper are given in their Duflo realisations, i.e. presented in the form 
$I=I(\lambda):={\rm Ann}_{U(\g)}\,L(\lambda)$ for some irreducible highest weight $\g$-modules $L(\lambda)$ with $\lambda\in\t^*$.
It is well known that if 
$\langle\lambda,\alpha^\vee\rangle\in\Z$ for all $\alpha\in\Pi$ then ${\rm VA}(I)$ is the closure of a {\it special} nilpotent orbit in $\g$. One also knows that to any $\sl_2$-triple $\{e,h,f\}$ in $\g$ with $e$ special there corresponds an $\sl_2$-triple $\{e^\vee,h^\vee,f^\vee\}$ in the Langlands dual Lie algebra $\g^\vee$ with $h^\vee\in\t^*$.  As Barbasch--Vogan observed in [\cite{BV}], for $e$  special and rigid there is a unique choice of $h^\vee$ such that
$\langle\frac{1}{2}h^\vee,\alpha^\vee\rangle\in\{0,1\}$ for all $\alpha\in\Pi$. Furthermore, in this case we have that ${\rm VA}(I(\frac{1}{2}h^\vee-\rho))=\overline{\O}$ where $\rho$ is the half-sum of the positive roots with respect to $\Pi$ and $\O$ is the nilpotent orbit containing $e$; see [\cite{BV}, Proposition~5.10]. Our computations in conjunction with earlier results obtained in [\cite{Mc}] and [\cite{Lo2}] imply the following:
\medskip

\noindent {\bf Theorem~B.} {\it Let $\O$ be a nilpotent orbit in a finite dimensional simple Lie algebra $\g$ over $\mathbb{C}$. Suppose further that $e\in\O$ is special and rigid and let $h^\vee\in\t^*$ be as above. Then 
$I(\frac{1}{2}h^\vee-\rho)$ is the only multiplicity-free primitive ideal in $\mathcal{X}_{\O}$.}

\medskip

In the forthcoming paper [\cite{P2}] Theorem~A will be applied for proving Humphreys' conjecture
[\cite{Hu}, p.~110] on representations of minimal dimension for reduced enveloping algebras associated with Lie algebras of reductive groups over 
fields of characteristic $p>0$. Thanks to Theorem~A and [\cite{P11}, Theorem~1.4] this conjecture is now confirmed for $p\gg 0$ but, of course, one wants  an explicit  bound on $p$ here.
We shall also
determine the number of one-dimensional representations of $U(\g,e)$ for all rigid nilpotent elements in $\g$.

\smallskip

\noindent{\bf Acknowledgements.} Part of this work was done during my stay at RIMS (Kyoto)
and MSRI (Berkeley) in the first half of 2013. I would like to thank both institutions for warm hospitality,  excellent working conditions and support.  I would like to express my gratitude to Yu Chen, Meinolf Geck, Thierry Levasseur, Ivan Losev and Monty McGovern for very helpful discussions and email correspondence on
the subject of this paper. Yu Chen has kindly provided me with a proof of Proposition~3.1 which plays an important role in establishing Theorem~B.
I am also thankful to the referee for some helpful suggestions. 

\section{\bf Preliminaries and recollections}
\subsection{}\label{2.1}
Denote by $G$ a simple, simply connected algebraic group over
$\mathbb C$ and let $\g=\Lie(G)$. Given $x\in\g$ set $G_x:=\{g\in
G\,|\,\,(\Ad g)\,x=x\}$ and $\g_x=\{y\in\g\,|\,\,[x,y]=0\}$, so that $\g_x=\Lie(G_x)$. Let $\{e,h,f\}$ be a non-trivial
$\mathfrak{sl}_2$-triple in the Lie algebra $\g$ and denote by
$(\,\cdot\,,\,\cdot\,)$ a non-degenerate $G$-invariant bilinear form
on $\g$.  Let $\chi\in\g^*$ be such that $\chi(x)=(e,x)$ for all
$x\in\g$. It is well known that $e$ is a $G$-unstable vector of $\g$ (in the sense of geometric invariant theory) and admits an optimal
cocharacter $\tau\colon {\mathbb C}^\times\rightarrow G$ such that $({\rm d}\tau)(1)=h$. The adjoint action of $\tau(\mathbb{C}^\times)$ gives $\g$ a graded Lie algebra structure, $\g\,=\,\bigoplus_{i\in\Z}\,\g(i)$, and it follows from $\mathfrak{sl}_2$-theory that $e\in\g(2)$ and $\g_e\subset \bigoplus_{i\ge 0}\,\g(i)$. From this it follows that the linear function $\chi$ induces a {\it  non-degenerate}
skew-symmetric bilinear form on the graded subspace $\g(-1)$ given by $\langle x,y\rangle=(e,[x,y])$ for all $x,y\in\g(-1)$. Also, $\g_e=\,\bigoplus_{i\ge 0}\,\g_e(i)$ where $\g_e(i):=\g_e\cap\g(i)$.

Let $U(\g,e)$ denote the enveloping algebra of the
Slodowy slice $e+\g_f$ to the orbit $\O(e):=(\Ad G)e$; see [\cite{P02},
\cite{GG}].
Recall that
$U(\g,e)=({\End}_{\g}\,Q_e)^{\text{op}}$, where $Q_e$ is the
generalised Gelfand--Graev $\g$-module associated with the triple
$\{e,h,f\}$. The $\g$-module $Q_e$ is induced from a one-dimensional module
${\mathbb C}_\chi$ over the nilpotent subalgebra 
$\m=\m(\ell):=\ell\bigoplus_{i\le-2}\,\g(i)$ of $\g$, where
$\ell$ 
is a  maximal totally isotropic subspace of 
$\g(-1)$ with respect to the skew-symmetric form defined above. It follows from $\mathfrak{sl}_2$-theory that $\dim\m=d(e)$ where
$d(e):=\frac{1}{2}\dim \O(e)$. By construction, the Lie subalgebra $\m$ is $(\ad
h)$-stable, all eigenvalues of $\ad h$ on $\m$ are negative, and
$\chi$ vanishes on $[\m,\m]$. The action of $\m$ on ${\mathbb
C}_\chi={\mathbb C} 1_\chi$ is given by $x(1_\chi)=\chi(x)1_\chi$
for all $x\in\m$. 
Let $\m_\chi$ be the subspace of $U(\g)$ spanned by all $x-\chi(x)$ with $x\in\m$. The left ideal $U(\g)\m_\chi$ is stable under the adjoint action of $\m$ on $U(\g)$ and so $\ad \m$ acts on the quotient space $U(\g)/U(\g)\m_\chi$. The fixed point space $\big(U(\g)/U(\g)\m_\chi\big)^{\ad\m}$ inherits from $U(\g)$ an algebra structure and it is well known that
$$U(\g,e)\,\cong\big(U(\g)/U(\g)\m_\chi\big)^{\ad\m}$$ as algebras.
By some historical reasons, $U(\g,e)$ is referred to as the {\it
finite $W$-algebra} associated with the pair $(\g, e)$. According to [\cite{GG}], it is independent of the choice the subspace $\ell\subset\g(-1)$.

Skryabin proved in [\cite{Sk}] that for any irreducible $U(\g,e)$-module $V$ the $\g$-module
$\widetilde{V}:=Q_e\otimes_{U(\g,e)}V$ is irreducible and belongs to the full subcategory $\mathcal{C}_\chi$ of $U(\g)$\,--\,${\sf mod}$ consisting of all modules acted upon locally nilpotently by all $x-\chi(x)$ with $x\in\m$. He also showed that for any $E\in{\rm Ob}\,\mathcal{C}_\chi$ the subspace $${\rm Wh}_\chi(E)=\{v\in E\,|\,\,x\cdot v=\chi(x)v\ \mbox{ for all }\, x\in\m\}$$ is a nonzero
$U(\g,e)$-module and $E\cong Q_e\otimes_{U(\g,e)}{\rm Wh}_\chi(E)$ as $\g$-modules. 
Finally, Skryabin proved in {\it loc.\,cit.} that 
$Q_e$ is free as a right $U(\g,e)$-module.
For $V$ as above we denote by $I_V$ the annihilator of $\widetilde{V}$ in $U(\g)$. It follows from Skryabin's results that $I_V$ is a primitive ideal of $U(\g)$.

\subsection{}\label{2.2}
Let $\mathcal X$ be the set of all primitive ideals of $U(\g)$. 
By a well-known result of Joseph, for any $I\in\mathcal X$ the zero locus ${\rm VA}(I)$ of ${\rm gr}(I)\subset S(\g)$ coincides with the Zariski closure of a nilpotent
$G$-orbit in $\g$ (here we identify $\g$ with $\g^*$ by means of our $G$-invariant bilinear form). We call ${\rm VA}(I)$ the {\it associated variety} of $I$ and denote by $\mathcal{X}_{\O}$ the set of all $I\in\mathcal{X}$ with ${\rm VA}(I)=\overline{\O}$.
It is known that
$I_V\in\mathcal{X}_{\O(e)}$ for any finite dimensional irreducible $U(\g,e)$-module $V$ and any $I\in \mathcal{X}_{\O(e)}$ can be presented in this form for at least one such $V$; see 
[\cite{P07}], [\cite{P07'}], [\cite{Lo}], [\cite{Gi}], [\cite{P10}]. 

By the Dynkin--Kostant theory, the reductive group
$C(e):=G_e\cap G_f$ is a Levi subgroup of $G_e$ with $\Lie(C(e))=\g_e(0)$ and the factor-group $\Gamma(e):=C(e)/C(e)^\circ$ identifies with the component group $G_e/G_e^\circ$. The construction of $U(\g,e)$ given in [\cite{GG}] implies that $C(e)$ acts on $U(\g,e)$ as algebra automorphisms, whilst [\cite{P07}, Lemma~2.4] yields that any two-sided ideal of $U(\g,e)$ is stable under
action of the connected group $C(e)^\circ$. It follows that $\Gamma=\Gamma(e)$ acts naturally on the set ${\rm Irr}\,U(\g,e)$ of all isoclasses of finite dimensional irreducible $U(\g,e)$-modules. Given a finite dimensional irreducible $U(\g,e)$-module $M$ we write $[M]$ for the class of $M$ in ${\rm Irr}\,U(\g,e)$. Confirming the author's conjecture, Losev proved in [\cite{Lo1}] that $I_V=I_{V'}$ if and only if $[V']=\,{^\gamma\!\,}[V]$ for some $\gamma\in\Gamma$. In particular, $\dim V'=\dim V$.

If $I\in\mathcal{X}_\O$ then $J:=\sqrt{{\rm gr}(I)}$ coincides with the ideal defining the affine variety $\overline{\O}$. Since
$A:=S(\g)/{\rm gr}(I)$ is a finitely generated $S(\g)$-module and
$J$ is the only minimal prime ideal of $S(\g)$ containing the annihilator of $A$ in $S(\g)$, there exist prime
ideals $\mathfrak{p}_1,\ldots, \mathfrak{p}_l$ of $S(\g)$ and a
finite chain $\{0\}=A_0\subset A_1\subset \cdots\subset A_l=A$ of
$S(\g)$-submodules of $A$ such that $\p_i\supseteq J$ and $A_i/A_{i-1}\cong
S(\g)/\mathfrak{p}_i$ for all $1\le i\le l$. The {\it multiplicity} of
$\O$ in $U(\g)/I$, denoted ${\rm mult}_{\O}\,(U(\g)/I)$, is
defined as
$${\rm mult}_{\O}\,(U(\g)/I):={\rm Card}\,\{1\le i\le l\,|\,\, \mathfrak{p}_i=J\}.$$
It is well known that this number is independent of the choices
made; see [\cite{Ja1}, 9.6], for  example. 
We say that $I\in\mathcal{X}_{\O}$ is {\it multiplicity free} if ${\rm
mult}_{\O}(U(\g)/I)=1$.

If $I=I_V\in\mathcal{X}_{\O(e)}$ then Losev's results proved in [\cite{Lo}, \cite{Lo1}] imply that
$$
{\rm mult}_{\O(e)}\,(U(\g)/I)\,=\,[\Gamma:\Gamma_V]\cdot (\dim V)^2
$$ where $\Gamma_V$ stands for the stabiliser in $\Gamma$ of the isoclass of $V$. It follows that $I_V\in\mathcal{X}_{\O(e)}$ is multiplicity free if and only if $\dim V=1$ and $\Gamma_V=\Gamma$.

Let $\mathcal A$ be a prime Noetherian ring. An element of
$\mathcal{A}$ is called \emph{regular} if it is not a zero divisor
in $\mathcal{A}$. The multiplicative set $S$ of all regular elements
of $\mathcal{A}$ satisfies the left and right Ore conditions, hence
can be used to form a classical ring of fractions
$\mathcal{Q}(\mathcal{A})=S^{-1}\mathcal{A}$; see [\cite{Dix}, 3.6],
for example. By Goldie's theory, the ring $\mathcal{Q}(\mathcal{A})$
is isomorphic to ${\rm Mat}_n(\mathcal{D})$ for some
$n\in\mathbb{N}$ and some skew-field $\mathcal{D}$. We write $n={\rm
rk}(\mathcal{A})$ and call $n$ the \emph{Goldie rank} of
$\mathcal{A}$. The division ring $\mathcal{D}$ is called the {\it
Goldie field} of $\mathcal{A}$. It is well known that ${\rm
rk}(\mathcal{A})=1$ if and only if $\mathcal{A}$ is a domain. More
generally, it follows from the Feith--Utumi theorem that the Goldie
rank of $\mathcal A$ coincides with the maximum value of the nilpotency
classes of nilpotent elements of $\mathcal A$. For any primitive
ideal $I$ of $U(\g)$ the quotient $U(\g)/I$ is a prime Noetherian
ring. A two-sided ideal $I$ of $U(\g)$ is said to be {\it completely
prime} if $U(\g)/I$ is a domain. To prove that each set
$\mathcal{X}_\O$ contains such ideals is a classical open problem of
the theory of primitive ideals; see [\cite{JoT}, 3.5] for a related
discussion. 

In [\cite{Lo}], Losev shows that ${\rm rk}(U(\g)/I_V)\le \dim V$
for every finite dimensional irreducible $U(\g,e)$-module $V$.
Generalising this result the author proved in [\cite{P11}] that the
Goldie rank of $U(\g)/I_V$ always divides $\dim V$ and the equality
holds whenever the Goldie field of $U(\g)/I_V$ is isomorphic to
the skew-field of fractions of a Weyl algebra. By
a result of Joseph, the latter is always the case for Lie algebras of type $\sf A$. However, if ${\rm rk}\,\g\ge
3$ and $\g$ is not of type $\sf A$, then it can happen that $\dim
V>{\rm rk}(U(\g)/I_V)=1$; see [\cite{P11}, Remark~4.3].
The above discussion shows that any multiplicity-free
primitive ideal is completely prime, but outside type $\sf A$ there exist completely prime primitive ideals which are not multiplicity free. The Goldie fields of the primitive quotients associated with such ideals could be very interesting from the viewpoint of the theory of division algebras because they cannot be isomorphic to Weyl skew fields; see [\cite{P11}, Theorem~B].
\subsection{}\label{2.3}
Let $U(\g,e)^{\rm ab}$ denote the quotient of
$U(\g,e)$ by its two-sided ideal generated by all commutators. This is a finitely generated commutative algebra over $\mathbb C$ and the largest commutative quotient of $U(\g,e)$. It follows from from [\cite{P07}, Remark~2.1 and Lemma~2.4] that
the connected group $C(e)^\circ$ acts trivially on
$U(\g,e)^{\rm ab}$. As a consequence, the component group $\Gamma=C(e)/C(e)^\circ$ acts on $U(\g,e)^{\rm ab}$
by algebra automorphisms. 
Let $I_\Gamma$ denote the ideal of $U(\g,e)^{\rm ab}$ generated by all elements $u-u^\gamma$, where $u\in U(\g,e)^{\rm ab}$ and $\gamma\in\Gamma$, and set $U(\g,e)_\Gamma^{\rm ab}\,:=\,U(\g,e)^{\rm ab}/I_\Gamma$. 

We let $\mathcal{E}$ denote the maximal spectrum of $U(\g,e)^{\rm ab}$, an affine variety over $\mathbb C$. The finite group $\Gamma$ acts on $\mathcal{E}$
and the fixed point space $\mathcal{E}^\Gamma$ 
coincides with the zero locus of $I_\Gamma$ in $\mathcal{E}$, so that $\mathcal{E}^\Gamma={\rm Specm}\,U(\g,e)^{\rm ab}_\Gamma$; see [\cite{PT}, 5.1]. 
The above discussion shows that $\mathcal{E}^\Gamma$ parametrises multiplicity-free primitive ideals in $\mathcal{X}_{\O(e)}$. So in order to show that $\mathcal{X}_{\O(e)}$ contains such ideals we just need to check that $\mathcal{E}^\Gamma\ne\varnothing$. For $\g$ classical, this follows from results of R.~Brylinski
[\cite{Bry}] (reproved in [\cite{Lo}]) which, in turn, relies on earlier work of Kraft--Procesi on nilpotent orbit closures in the orthogonal and symplectic Lie algebras.

Recall that a {\it sheet} of $\g$ is one of the irreducible components of the quasi-affine varieties
$\g^{(k)}=\{x\in\g\,|\,\,\dim\g_x=k\}$ where $k\in\Z^{>0}$. Each sheet $\mathcal{S}\subset\g$ is $(\Ad G)$-stable, locally closed and contains a unique nilpotent $G$-orbit. By a result of Katsylo, each orbit set $\mathcal{S}/G$ has a
natural structure of an affine variety. The dimension of $\mathcal{S}/G$ is referred to as the {\it rank} of $\mathcal S$ and denoted ${\rm rk}(\mathcal{S})$. It is well known that ${\rm rk}(\mathcal{S})=0$ if and only if $\mathcal{S}=(\Ad G) x$ where $x\in\g$ is a nilpotent element which cannot be obtained by Lusztig--Spaltenstein induction from a proper Levi subalgebra of $\g$. Such elements are called {\it rigid} and they are our main object of study in this paper. Non-rigid nilpotent elements of $\g$ are called {\it induced}.  The standard reference on this topics is
[\cite{CM}].

Let $\mathcal{S}_1,\ldots,\mathcal{S}_t$ be all sheets of $\g$ containing $e$. Let $r_i={\rm rk}(\mathcal{S}_i)$ and define $r(e):=\max_{1\le i\le t}r_i$. In [\cite{P10}], the author proved that for $e$ induced $\E$ has at least $t$ irreducible components and $\dim\mathcal{E}=r(e)$, whilst for $e$ rigid the set $\E$ is finite  (and possibly empty  in type
${\sf E_8}$). The arguments in {\it loc.\,cit.} rely on reduction modulo $p$ and results of Goodwin--R\"ohrle--Ubly [\cite{GRU}] of one-dimensional representations of $U(\g,e)$ for $\g$ exceptional which were obtained by computational methods. It is also proved in [\cite{P10}] that
$U(\g,e)^{\rm ab}=U(\g,e)^{\rm ab}_\Gamma$
is a polynomial algebra when $\g=\sl_N$.

Very recently, a further study of $\E$ and $\E^\Gamma$ was undertaken in [\cite{PT}]. In particular, it is proved there that for any nilpotent element $e$ in $\g={\so}_N$ or $\g=\mathfrak{sp}_N$ the algebra
$U(\g,e)^{\rm ab}_\Gamma$ is polynomial 
and its Krull dimension is given by a simple formula
involving the partition of $N$ associated with $e$.
Furthermore, for $\g$ classical
the variety $\E$ is irreducible if and only if $e$ lies is a single sheet of $\g$ and when that happens the algebra $U(\g,e)^{\rm ab}$ is also polynomial. 
The nilpotent elements which lie in a single sheet of
$\g$ are described explicitly in combinatorial terms; see [\cite{PT}, Theorem~1]. One can also find in {\it loc.\,cit.} some strong results toward describing
the algebra $U(\g,e)^{\rm ab}_\Gamma$ for $\g$ exceptional. In particular, it is established that
$\E^\Gamma\ne\varnothing$ for any induced nilpotent element of $\g$ and $\dim\E^\Gamma$ is determined in all such cases.

Thanks to [\cite{GRU}] and [\cite{PT}], in order to show that $\mathcal{E}^\Gamma$ is non-empty for any finite $W$-algebra $U(\g,e)$ one just needs to sort out a handful of rigid orbits in exceptional Lie algebras. However, one is also interested in the  multiplicity-free primitive ideals of $U(\g)$ corresponding to the points of $\E^\Gamma$ via Skryabin's equivalence (as discussed at the end of Subsection~\ref{2.1}). By Duflo's theorem [\cite{Du1}], any $I\in\mathcal{X}$ has the form
$I={\rm Ann}_{U(\g)}\,L(\lambda)$ for some irreducible highest weight $\g$-module $L(\lambda)$. Such a presentation (which is generally not unique) will be referred to as a {\it Duflo realisation} of $I$.
Ideally, one would like to find at least one Duflo realisation for any  multiplicity-free primitive ideal $I=I_V$ in $\mathcal{X}_{\O(e)}$. In this paper and its continuation [\cite{P2}] we solve this problem for all rigid nilpotent elements in exceptional Lie algebras. For classical Lie algebras, the problem is open already in the case where $e$ is rigid and non-special in the sense of Lusztig.     
\subsection{}\label{2.4}
The finite $W$-algebra $U(\g,e)$ has two
natural filtrations, the {\it Kazhdan filtration} $\{{\sf
K}_i\,U(\g,e)\,|\,i\ge 0\}$ and the {\it loop filtration} $\{{\sf
L}_i\,U(\g,e)\,|\,i\ge 0\}$, which we are now going to describe. Assigning to any $x\in\g(i)$ degree $i+2$ one obtains a $\Z$-filtration on the enveloping algebra $U(\g)$
which then induces that on the the subquotient
$U(\g,e)\,=\,(U(\g)/U(\g)/\m_\chi)^{\ad\m}$. 
The filtration of $U(\g,e)$ thus obtained is called the  Kazhdan filtration of $U(\g,e)$. According to [\cite{P02}, Theorem~4.6], this filtration is 
non-negative and the corresponding graded algebra
$\gr_{\sf K}\,U(\g,e)$ is
isomorphic to the algebra of regular functions on the slice $e+\g_f$ endowed with its Slodowy grading. 
Since $\g_e$ and $\g_f$ are dual to each other 
with respect to $(\,\cdot\,,\,\cdot\,)$,
the latter algebra is often identified with the symmetric algebra $S(\g_e)$.

Let $z_1,\ldots,z_s$ be a basis of the subspace $\ell$ of $\m$. We extend it up to a Witt basis $\{z_1,\ldots,z_s,z_1',\ldots,z_s'\}$ of $\g(-1)$ relative to the skew-symmetric form $\langle\,\cdot\,,\,\cdot\,\rangle$ and denote by $\mathcal{A}$  the associative $\mathbb{C}$-algebra generated by  $z_i$ and $z_i'$ with $1\le i\le s$ subject to the usual relations $[z_i,z_j]=[z_i',z_j']=0$ and $[z_i',z_j]=\delta_{ij}$, for all $1\le 1,j\le s$.
Clearly, $\mathcal{A}$ is isomorphic to the 
$s$-th Weyl algebra over $\mathbb{C}$ (it may happen, of course, that $s=0$). The adjoint action of $\tau(\mathbb{C}^\times)$ on the parabolic subalgebra $\p_+:=\bigoplus_{i\ge 0}\,\g(i)$ of $\g$
extends to the action of $\tau(\mathbb{C}^\times)$ on the algebra $U(\p_+)\otimes\mathcal{A}$ such that $$\tau(c)(u\otimes a)\,=\big((\Ad \tau(c))\,u\big)\otimes a\quad\ \mbox{for all }\ c\in\mathbb{C}^\times,\,u\in U(\p_+),\,a\in\mathcal{A}.$$ This gives $U(\p_+)\otimes\mathcal{A}$ a $\Z^{\geqslant 0}$-graded algebra structure. It is worth mentioning that the zero part of $U(\p_+)\otimes\mathcal{A}$ coincides with
$U(\g(0))\otimes\mathcal{A}$. By [\cite{P07}, Proposition~2.2], the algebra $U(\g,e)$ embeds into $U(\p_+)\otimes\mathcal{A}$ and thus
acquires a $\Z^{\geqslant 0}$-filtration which is called the loop filtration of $U(\g,e)$. It is proved in
[\cite{BGK}, Theorem~3.8] that
the corresponding graded algebra 
$\gr_{\sf L}\,U(\g,e)$ is isomorphic to the universal
enveloping algebra $U(\g_e)$ regarded with its $\Z_{\ge 0}$-grading
induced by the action of $\ad\ h$ (see also [\cite{P07}, Proposition~2.1]). 

Let $x_1,\ldots, x_r$ be a basis of $\g_e$ such that $x_i\in\g(n_i)$ for some $n_i\in\Z^{\geqslant 0}$.
According to [\cite{P07}, 2.3], there exists a (non-unique) injective linear map $\Theta\colon \g_e\rightarrow U(\g,e)$ such that the monomials
$\Theta(x_1)^{k_1}\cdots\Theta(x_r)^{k_r}$ with
$k_i\in\Z^{\geqslant 0}$ forms a $\mathbb{C}$-basis of $U(\g,e)$. The elements
${\rm gr}_{\sf K}\,\Theta(x_i)$ and ${\rm gr}_{\sf L}\,\Theta(x_i)$ with $1\le i\le r$ have degree
$n_i+2$ and $n_i$, respectively, and generate the graded algebras
$S(\g_e)\cong\, {\rm gr}_{\sf K}(U(\g,e))$ and 
$U(\g_e)\cong \,{\rm gr}_{\sf L}(U(\g,e))$.
Furthermore,  $[\Theta(x),\Theta(y)]=\Theta([x,y])$ for all $x\in\g_e(0)$ and all $y\in\g_e$. From this it is immediate that the subalgebra of $U(\g,e)$ generated by $\Theta(\g_e(0))$ is isomorphic to the enveloping algebra $U(\g_e(0))$ and $U(\g,e)$ is free as a left $U(\g_e(0))$-module.

According to [\cite{Y}], if $e$ is a rigid nilpotent element in a classical Lie algebra, then $\g_e=[\g_e,\g_e]$. Thanks to [\cite{Bry}] and [\cite{Lo1}] this result implies that for $e$ rigid and $\g$ classical the finite $W$-algebra $U(\g,e)$ admits a unique one-dimensional representation, so that $\E=\E^\Gamma$ is a single point; see [\cite{PT}, Proposition~11]. If $e$
is rigid and $\g$ is exceptional, then it is proved in [\cite{dG}] by computational methods that either $\g_e=[\g_e,\g_e]$ or $\g_e=\mathbb{C}e\oplus[\g_e,\g_e]$.
Moreover, the inequality $\g_e\ne[\g_e,\g_e]$ occurs just for one rigid orbit in types ${\sf G_2}$, ${\sf F_4}, {\sf E_7}$ and for three rigid orbits in type ${\sf E_8}$ two of which have dimension $202$ (the largest dimension for the rigid orbits in type ${\sf E_8}$).
\begin{prop}
Suppose $e$ is a rigid nilpotent element in an exceptional Lie algebra $\g$ and assume further that $\E\ne \varnothing$. Then the following hold:
\begin{itemize}
\item[(i)\,]
$\E=\E^\Gamma$.

\smallskip

\item[(ii)\,] If $\g_e=[\g_e,\g_e]$ then $U(\g,e)$ affords a unique one-dimensional representation.

\smallskip
 
\item[(iii)\,] If $\g_e=\mathbb{C}e\oplus[\g_e,\g_e]$ then any 
one-dimensional representation of $U(\g,e)$ is uniquely determined by the action of the image of a Casimir element of $U(\g)$ in $U(\g,e)$. 
\end{itemize}
\end{prop}
\begin{proof}
If $\g_e=[\g_e,\g_e]$ then [\cite{PT}, Proposition~11(i)] implies that $U(\g,e)$ affords a
unique one-dimensional representation. Hence $\E$ is a single point, implying $\E^\Gamma=\E$. 
Now suppose that $\g_e=\mathbb{C}e\oplus [\g_e,\g_e]$ and let $a_1,\ldots,a_n$ and $b_1,\ldots,b_n$ be dual bases of $\g$ with respect to $(\,\cdot\,,\,\cdot\,)$ such that $a_i\in\g(n_i)$ and $b_i\in\g(-n_i)$ for some $n_i\in\Z$, where $1\le i\le n$. We may (and will) assume further that $a_1,\ldots,a_m$ is a basis of $\g(2)$ and
$b_i=a_{m+i}$ for $1\le i\le m$.  Since $e\in\g(e)$, we have that $e= \sum_{i=1}^m(e,b_i)a_i$. 

Let
$C=\sum_{i=1}^na_ib_i$, a Casimir element of  $U(\g)$, which we regard as a central element of $U(\g,e)\,=\,({\rm End}_\g\,Q_e)^{\rm op}$. The above remark shows that $$C(1_\chi)=
\big(2e+\textstyle{\sum}_{i=1}^sp_iz_i'+C_0+h\big)\otimes 1_\chi$$ for some some
$C_0\in U(\g(0))$ of Kazhdan degree $4$, some $h\in \g(0)$, and $p_1,\ldots,p_s\in\g(1)$. As the embedding $U(\g,e)\hookrightarrow U(\p_+)\otimes\mathcal{A}$ in [\cite{P07}, 2.4] is described explicitly, it is now straightforward to check that
${\rm gr}_{\sf L}(C)=2e$. Let $c$ denote the image of $C$ in $U(\g,e)^{\rm ab}$. Since ${\rm gr}_{\sf L}\,U(\g,e)\cong U(\g_e)$ and $\g_e=\mathbb{C}e\oplus [\g_e,\g_e]$, straightforward induction on $k$ shows that for $x\in\g_e(k)$ the image of $\Theta(x)$ in $U(\g,e)^{\rm ab}$ lies in the subalgebra generated by $c$ (when $k=0$ this is clear because $\Theta(\g_e(0))\cong \g_e(0)$ is a Lie subalgebra of $U(\g,e)$ under the commutator product). As a consequence, $U(\g,e)^{\rm ab}\,=\,\mathbb{C}[c]$.
Since the action of $\Gamma$ on $U(\g,e)^{\rm ab}$ is induced by the adjoint action of $C(e)\subset G$ on $U(\g)$ we have that $\gamma(c)=c$ for all $\gamma\in\Gamma$. This yields $\E=\E^\Gamma$ completing the proof.
\end{proof}
\subsection{}\label{2.5}
Let $T$ be a maximal torus of $G$, $\t=\Lie(T)$ and $W=N_G(T)/T$, the Weyl group of $\g$. We shall always assume in what follows that $T$ contains  $\tau(\mathbb{C}^\times)$. Let $\Phi$ be the root system of $\g$ with respect to $T$. We fix a positive system $\Phi^+$ in $\Phi$ and let $\Pi=\{\alpha_1,\ldots,\alpha_\ell\}$ be the basis of $\Phi$ contained in $\Phi^+$. We denote by $\{\varpi_1,\ldots,\varpi_\ell\}$ the corresponding system of fundamental weights in $\t^*$.
Let $\t_e$ be a maximal toral subalgebra of $\g_e$ contained in
$\g_e(0)$ and denote by $\l$ the centraliser of
$\t_e$ in $\g$. We shall assume without loss of generality that $\t_e$ is contained in $\t$ and
$\l=\Lie(L)$ is a standard Levi subalgebra of $\g$ associated with a subset $\Pi_0=\{\alpha_i\,|\,i\in I\}$ of $\Pi$ (here $L$ is the standard Levi subgroup of $G$ associated with $\Pi_0$). Let $\Phi_0$ be the root system of $\l$ with respect to $T$. Then $\Pi_0$ is the basis of simple roots of $\Phi_0$ contained in the positive system $\Phi_0^+:=\Phi_0\cap \Phi^+$ of $\Phi_0$. 
In all cases considered in the next three sections we indicate our choice of $\Pi_0$ by blackening the corresponding nodes on the Dynkin diagram of $\Pi$ and we call the resulting picture a {\it pinning} for $e$.
It should be stressed at this point that for the majority of nilpotent orbits the choice of $\Pi_0$ is not unique and quite often our computations simplify when $e$ is pinned down properly.
For that reason, our choice of pinning can differ from those used in other references on nilpotent orbits. 
The maximality of $\t_e$ implies that $\t_e=\z(\l)$, the centre of $\l$, and the $\sl_2$-triple $\{e,h,f\}$ 
is distinguished in $[\l,\l]$ in the sense of the Bala--Carter theory. Since in this paper we are interested in rigid orbits (which are never distinguished) we shall assume from now on that $\l$ is a proper Levi subalgebra of $\g$ and $\t_e=\z(\l)\ne 0$. 

Let $\g=\n_{-}\oplus\l\oplus\n_+$ be the parabolic decomposition of $\g$ such that the set of $T$-roots of $\n_+$ and $\n_-$ equals $\Phi^+\setminus \Phi_0$
and $-(\Phi^+\setminus \Phi_0)$, respectively. 
Set $\Phi^+(0):=\{\gamma\in\Phi^+\,|\,\,\gamma(h)=0\}$ and $\Phi^+(1):=\{\gamma\in\Phi^+\,|\,\,\gamma(h)=1\}$. The elements of $\Phi^+(0)$
and $\Phi^+(1)$ will be referred to as $0$-{\it roots} and $1$-{\it roots}, respectively. One of the interesting features of the theory of finite $W$-algebras is the appearance of a shift by the weight
$$\rho_e:=\textstyle{\frac{1}{2}\sum}_{\gamma\in\Phi^+(0)
\sqcup\Phi^+(1)}\,\gamma\in\t^*$$ which plays the  role of the half-sum of the positive roots $\rho=\frac{1}{2}\sum_{\gamma\in\Phi^+}\gamma$, but depends of $e$. We write $\overline{\rho}_e$ for the restriction of $\rho_e$ to $\t_e$. 

It will be crucial for us to determine
$\overline{\rho}_e$ for all rigid nilpotent elements in exceptional Lie algebras. In order to do that we list
the $0$-roots and the $1$-roots in all cases of interest. Since it is easy to miss some of these roots (especially in type ${\sf E_8}$) we always keep in mind that their total number is given by the formula
$$|\Phi^+(0)|+|\Phi^+(1)|=\textstyle{\frac{1}{2}}
(\dim\g_e- {\rm rk}\,\g).$$
Indeed, from $\sl_2$-theory we know that $\dim\g_e=\dim\g(0)+\dim\g(1)$ and the subspaces $\n_\pm\cap\g(1)$ and $\n_\pm\cap\g(-1)$ have the same dimension,  whilst the Bala--Carter theory yields $\g(1)\cap \l=0$. Combining this with the fact that $\n^-\cong
(\n^+)^*$ as $(\ad \l)$-modules yields $\frac{1}{2}\dim\g(1)=|\Phi^+(1)|$ whereas the equality $\frac{1}{2}(\dim\g(0)-{\rm rk}\,\g)=
|\Phi^+(0)|$ is obvious.

For Lie algebras of type $\sf E$ we shall identify the root system $\Phi\subset\t^*$ with its dual root system $\Phi^\vee\subset \t$ by using the 
$W$-invariant
scalar product $(\,\cdot\,|\,\cdot\,)$ on the  $\mathbb{R}$-span of $\Pi$ such that $(\gamma|\gamma)=2$ for all $\gamma\in\Phi$. This will also enable us to identify $\g$ with its Langlands dual Lie algebra $\g^\vee$. Then we may assume that the set of fundamental weights $\{\varpi_i\,|\,\,
i\not\in I\}$ forms a basis for
$\t_e$. Very often we have to express $\varpi_i$'s in terms of simple roots; for that we use [\cite{Bo}, Planches~I--IX] as a standard reference. 

Throughout the paper we use Bourbaki's numbering of simple roots in $\Pi$. In type ${\sf E}_\ell$ this simply means that the roots in $\Pi$ forming  the subgraph of type ${\sf A}_{\ell -1}$ are labelled $\alpha_1,\alpha_3,\ldots,\alpha_\ell$ (from left to right) and the root that sticks out of that subgraph is always called $\alpha_2$. 
In type ${\sf F_4}$ the roots in $\Pi$ are numbered from left to right and it is assumed that $\alpha_1$ and $\alpha_2$ are long roots whilst $\alpha_3$
and $\alpha_4$ are short. In type ${\sf G_2}$ it is assumed that $\alpha_1$ is a short root.
\subsection{}\label{2.6} In order to explain the appearance of $\overline{\rho}_e$ in this theory we have to say a few words about the category $\sf O$ for finite $W$-algebras first introduced by Brundan--Goodwin--Kleshchev in [\cite{BGK}]. Although the map 
$\Theta\colon \g_e\hookrightarrow U(\g,e)$ is generally not unique, there is a very natural way to define an embedding $\theta\colon U(\g_e(0))\hookrightarrow U(\g,e)$. It is given by an explicit formula in [\cite{P07}, Lemma~2.3]. From now on we shall assume that the restriction of $\Theta$ to $\t_e\subset\g_e(0)$ coincides with 
$\theta_{\vert\t_e}$ and we choose an embedding $\Theta_\l\colon \l_e\hookrightarrow U(\l,e)$ with the same property.

Since $e\in\l$ we that $\g_e=\g_e^-\oplus\l_e\oplus\g_e^+$ where $\g_e^\pm=\g_e\cap\n^\pm$. 
There exists a unique element $h_0\in\t_e$ for which
$\alpha(h_0)=0$ for any $\alpha\in\Pi_0$ and $\alpha(h_0)=1$ for any $\alpha\in\Pi\setminus\Pi_0$.
By construction, the eigenvalues of $\ad h_0$ on $\n_+$ are positive integers. Moreover, the adjoint action of $\theta(h_0)$ on $U(\g,e)$ gives the latter a $\Z$-graded algebra structure: $U(\g,e)=\bigoplus_{i\in\Z}\,U(\g,e)_i$. We set $U(\g,e)_{\ge 0}\,:=\,\bigoplus_{i\ge 0}\,U(\g,e)_i$ and denote by  $U(\g,e)_{\ge 0}^+$ the intersection of $U(\g,e)_{\ge 0}$ with
the left ideal of $U(\g,e)$ generated by all $\Theta(x)$ with
$x\in\g_e^+$. Then $U(\g,e)_{\ge 0}^+$ is a two-sided ideal of $U(\g,e)_{\ge 0}$ and it follows from 
[\cite{BGK}, Theorem~4.3] and [\cite{Lo10}, Remark~5.4] that 
there is an algebra isomorphism 
$\Psi\colon\,U(\g,e)_{\ge 0}/U(\g,e)^+_{\ge 0\,}\stackrel{\sim}{\longrightarrow}\, U(\l,e)$ such that 
$$(\Psi\circ\Theta)(x)=\Theta_\l(x)+(\rho-\rho_e)(x)\quad\ \ (\forall\,x\in\t_e)$$
(one should also keep in mind here that in the notation of [\cite{Lo2}, p.~4864] the weight $\delta'-\delta-\rho$ vanishes on $\t_e$).
For any finite dimensional irreducible 
$U(\l,e)$-module $M_0$ we  now can define an induced $U(\g,e)$-module  
$${\rm Ind}_{\mathcal W}(M_0):=\,U(\g,e)\otimes_{U(\g,e)_{\ge 0}}\,M_0,$$
where $U(\g,e)_{\ge 0}$ acts on $M_0$ via $\Psi$. This is an analogue of a Verma module in our more general setting; it has a nice PBW basis  and contains a unique maximal submodule which will be denoted by ${\rm Max}_{\mathcal W}(M_0)$. The $U(\g,e)$-module  
$L_{\mathcal W}(M_0)\,:=\,{\rm Ind}_{\mathcal W}(M_0)/
{\rm Max}_{\mathcal W}(M_0)$ is irreducible and contains an isomorphic copy of $M_0$. 

In [\cite{Lo2}, Corollary~5.1.2], Losev shows that $L_{\mathcal W}(M_0)$ is finite dimensional if and only if the annihilator in $U(\g)$ of the $\g$-module
$Q_e\otimes_{U(\g,e)} L_{\mathcal W}(M_0)$ lies in $\mathcal{X}_{\O(e)}$. Although, in general, this  condition is difficult to verify, Losev's criterion can be applied effectively to check 
the equality 
$\dim L_{\mathcal W}(M_0)=1$  
for a fairly large class of nilpotent orbits; see [\cite{Lo2}, 5.2]. Luckily for us, this class includes all rigid orbits of $\g$.

Let $T_e$ be the maximal torus of $C:=C(e)^\circ$ with $\Lie(T_e)=\t_e$. When $e$ is rigid, the group $C$ is semisimple.
If $\dim L_{\mathcal W}(M_0)<\infty$, the action of $\Theta(\g_e(0))$ on $L_{\mathcal W}(M_0)$ extends to a rational action of the simply connected cover of $C$.
As $C$ is semisimple, the Weyl group $W_e:=N_{C}(T_e)/T_e$ has no non-zero fixed points on $\t_e$, whilst the preceding remark implies that it acts on the set of all $\t_e$-weights of  $L_{\mathcal W}(M_0)$. Analysing this $W_e$-action, Losev proved, for $e$ rigid, that
$L_{\mathcal W}(M_0)$ is one-dimensional if and only if $\dim M_0=1$ and $\Theta(\t_e)$ acts on $M_0$
via $(\rho_e-\rho)_{\vert \t_e}$; see [\cite{Lo2}, Theorem~5.2.1]. 

Given $\lambda\in\t^*$ we write $L_0(\lambda)$ for the irreducible $\l$-module of highest weight $\lambda$ and set $I_0(\lambda):=\,{\rm Ann}_{U(\l)}L_0(\lambda)$. We shall occasionally denote by $Q_{\l,e}$ the analogue of $Q_e$ for $\l$ and use the fact that $U(\l,e)=\,({\rm End}_{\l}\,Q_{\l,e})^{\rm op}$. As a corollary of the above-mentioned result, Losev proved in [\cite{Lo2}, 5.3] that for any rigid nilpotent element $e\in\g$ a primitive ideal $I(\lambda)={\rm Ann}_{U(\g)}\,L(\lambda)$ with $\lambda\in\t^*$ has the form $I(\lambda)=I_M$ for some one-dimensional $U(\g,e)$-module $M$ if and only if $\lambda$ satisfies the following four conditions:
\begin{itemize}
\item[(A)\,] ${\rm VA}(I_0(\lambda))=\overline{\O}_0$ where $\O_0$ is the adjoint $L$-orbit of $e$;

\smallskip

\item[(B)\,] $I(\lambda)\in\mathcal{X}_{\O(e)}$ or equivalently, modulo (A), $\dim {\rm VA}(I(\lambda))\le \dim \O(e)$; 

\smallskip

\item[(C)\,] $(\lambda+\rho)_{\vert \t_e}=\overline{\rho}_e$;

\smallskip

\item[(D)\,] $I_0(\lambda)\,=\,{\rm Ann}_{U(\l)}\big(Q_{\l,e}\otimes_{U(\l,e)}M_0\big)$ for some one-dimensional $U(\l,e)$-module $M_0$.
\end{itemize}
We should mention that Condition~(D) holds automatically when $e$ has {\it standard Levi type} i.e. is regular in $\l$. In that case, the algebra $U(\g,e)$ is commutative and hence $\dim M_0=1$ for any irreducible $U(\l,e)$-module $M_0$. Furthermore, if $e$ has standard Levi type, then Condition~(A) can be rephrased by saying that $\langle\lambda+\rho,\alpha_i^\vee\rangle\not\in\Z^{>0}$ for all $i\in I$. 
The rigid orbits in exceptional Lie algebras were first classified by Borho in type ${\sf F_4}$ and by Elashvili in type ${\sf E}$. In particular, it is known that all of them except two in type ${\sf E_8}$ have standard Levi type.
For the reader's convenience all rigid orbits in exceptional Lie algebras (along with some relevant data) are listed in the table at the end of the paper. 

It is quite easy to find $\lambda\in\t^*$ satisfying two or even three of the above conditions, but getting all four of them at the same time can be very tricky and in some difficult cases we have to rely on guessing and intuition. 
\section{\bf The case of
rigid nilpotent orbits in Lie algebras of type ${\sf E_8}$.}
\subsection{Type $({\sf E_8, A_1})$}\label{3.1}
It is well known that $e$ is a special nilpotent element and
$e^\vee$ is subregular in the Langlands dual Lie algebra $\g^\vee$.
Moreover, it can be assumed that
$$h^\vee=\textstyle{{2\atop{}}{2\atop{}}{0\atop 2}
{2\atop{}}{2\atop{}}{2\atop{}}{2\atop{}}}.$$ Keeping both the shape
of $h^\vee$ and Losev's condition (A) in mind we choose the pinning
for $e$ as follows:
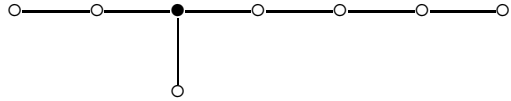
\[\xymatrix{*{{\circ}}\ar@{-}[r]& *{\circ} \ar@{-}[r] & *{\bullet} \ar@{-}[r]\ar@{-}[d] & *{\circ} \ar@{-}[r]
 & *{\circ} \ar@{-}[r]&*{\circ}\ar@{-}[r]&*{\circ}\\ && *{\circ}&&&}\]
Then we may assume that the optimal cocharacter for $e$ with $({\rm d}\tau)(1)=h$ equals
$$\tau=\textstyle{{0\atop{}}{(-1)\atop{}}{2\atop (-1)}
{(-1)\atop{}}{0\atop{}}{0\atop{}}{0\atop{}}}.$$ Since
$\dim\,\g_e=190$ the total number of positive $0$-roots and
$1$-roots is $(190-8)/2=91$. The roots are listed in the table
below.

\begin{table}[htb]
\label{data1}
\begin{tabular}{|c|c|}
\hline$0$-roots  & $1$-roots
\\ \hline
$\scriptstyle{{1\atop{}}{0\atop{}}{0\atop0}{0\atop{}}{0\atop{}}{0\atop{}}{0\atop{}}}$,
$\scriptstyle{{0\atop{}}{0\atop{}}{0\atop0}{0\atop{}}{1\atop{}}{0\atop{}}{0\atop{}}}$,
$\scriptstyle{{0\atop{}}{0\atop{}}{0\atop0}{0\atop{}}{0\atop{}}{1\atop{}}{0\atop{}}}$,
$\scriptstyle{{0\atop{}}{0\atop{}}{0\atop0}{0\atop{}}{0\atop{}}{0\atop{}}{1\atop{}}}$,
$\scriptstyle{{0\atop{}}{0\atop{}}{0\atop0}{0\atop{}}{1\atop{}}{1\atop{}}{0\atop{}}}$,\,
&$\scriptstyle{{0\atop{}}{1\atop{}}{1\atop0}{0\atop{}}{0\atop{}}{0\atop{}}{0\atop{}}}$,
$\scriptstyle{{0\atop{}}{0\atop{}}{1\atop0}{1\atop{}}{0\atop{}}{0\atop{}}{0\atop{}}}$,
$\scriptstyle{{0\atop{}}{0\atop{}}{1\atop 1}{0\atop{}}{0\atop{}}{0\atop{}}{0\atop{}}}$,\\
$\scriptstyle{{0\atop{}}{1\atop{}}{1\atop0}{1\atop{}}{0\atop{}}{0\atop{}}{0\atop{}}}$,
$\scriptstyle{{0\atop{}}{1\atop{}}{1\atop0}{1\atop{}}{1\atop{}}{0\atop{}}{0\atop{}}}$,
$\scriptstyle{{1\atop{}}{1\atop{}}{1\atop0}{1\atop{}}{0\atop{}}{0\atop{}}{0\atop{}}}$,
$\scriptstyle{{1\atop{}}{1\atop{}}{1\atop0}{1\atop{}}{1\atop{}}{0\atop{}}{0\atop{}}}$,
$\scriptstyle{{0\atop{}}{1\atop{}}{1\atop0}{1\atop{}}{1\atop{}}{1\atop{}}{0\atop{}}}$,&
$\scriptstyle{{1\atop{}}{1\atop{}}{1\atop0}{0\atop{}}{0\atop{}}{0\atop{}}{0\atop{}}}$,
$\scriptstyle{{0\atop{}}{0\atop{}}{1\atop0}{1\atop{}}{1\atop{}}{0\atop{}}{0\atop{}}}$,
$\scriptstyle{{0\atop{}}{0\atop{}}{1\atop0} {1\atop{}}{1\atop{}}{1\atop{}}{0\atop{}}}$,\\
$\scriptstyle{{0\atop{}}{0\atop{}}{0\atop0}{0\atop{}}{0\atop{}}{1\atop{}}{1\atop{}}}$,
$\scriptstyle{{0\atop{}}{0\atop{}}{0\atop0}{0\atop{}}{1\atop{}}{1\atop{}}{1\atop{}}}$,
$\scriptstyle{{1\atop{}}{1\atop{}}{1\atop0}{1\atop{}}{1\atop{}}{1\atop{}}{0\atop{}}}$,
$\scriptstyle{{0\atop{}}{1\atop{}}{1\atop0}{1\atop{}}{1\atop{}}{1\atop{}}{1\atop{}}}$,
$\scriptstyle{{1\atop{}}{1\atop{}}{1\atop0}{1\atop{}}{1\atop{}}{1\atop{}}{1\atop{}}}$,&
$\scriptstyle{{0\atop{}}{0\atop{}}{1\atop0}{1\atop{}}{1\atop{}}{1\atop{}}{1\atop{}}}$,
$\scriptstyle{{0\atop{}}{1\atop{}}{2\atop1}{1\atop{}}{0\atop{}}{0\atop{}}{0\atop{}}}$,
$\scriptstyle{{1\atop{}}{1\atop{}}{2\atop 1}{1\atop{}}{0\atop{}}{0\atop{}}{0\atop{}}}$,\\
$\scriptstyle{{0\atop{}}{0\atop{}}{1\atop1}{1\atop{}}{0\atop{}}{0\atop{}}{0\atop{}}}$,
$\scriptstyle{{0\atop{}}{0\atop{}}{1\atop1}{1\atop{}}{1\atop{}}{1\atop{}}{0\atop{}}}$,
$\scriptstyle{{0\atop{}}{0\atop{}}{1\atop1}{1\atop{}}{1\atop{}}{1\atop{}}{1\atop{}}}$,
$\scriptstyle{{0\atop{}}{1\atop{}}{1\atop1}{0\atop{}}{0\atop{}}{0\atop{}}{0\atop{}}}$,
$\scriptstyle{{1\atop{}}{1\atop{}}{1\atop1}{0\atop{}}{0\atop{}}{0\atop{}}{0\atop{}}}$,&
$\scriptstyle{{0\atop{}}{1\atop{}}{2\atop1}{1\atop{}}{1\atop{}}{0\atop{}}{0\atop{}}}$,
$\scriptstyle{{1\atop{}}{1\atop{}}{2\atop1}{1\atop{}}{1\atop{}}{0\atop{}}{0\atop{}}}$,
$\scriptstyle{{0\atop{}}{1\atop{}}{2\atop 1}{1\atop{}}{1\atop{}}{1\atop{}}{0\atop{}}}$,\\
$\scriptstyle{{0\atop{}}{1\atop{}}{2\atop1}{2\atop{}}{1\atop{}}{0\atop{}}{0\atop{}}}$,
$\scriptstyle{{0\atop{}}{1\atop{}}{2\atop1}{2\atop{}}{1\atop{}}{1\atop{}}{0\atop{}}}$,
$\scriptstyle{{0\atop{}}{1\atop{}}{2\atop1}{2\atop{}}{1\atop{}}{1\atop{}}{1\atop{}}}$,
$\scriptstyle{{0\atop{}}{1\atop{}}{2\atop1}{2\atop{}}{2\atop{}}{1\atop{}}{1\atop{}}}$,
$\scriptstyle{{0\atop{}}{1\atop{}}{2\atop1}{2\atop{}}{2\atop{}}{2\atop{}}{1\atop{}}}$,&
$\scriptstyle{{1\atop{}}{1\atop{}}{2\atop1}{1\atop{}}{1\atop{}}{1\atop{}}{0\atop{}}}$,
$\scriptstyle{{0\atop{}}{1\atop{}}{2\atop1}{1\atop{}}{1\atop{}}{1\atop{}}{1\atop{}}}$,
$\scriptstyle{{1\atop{}}{1\atop{}}{2\atop 1}{1\atop{}}{1\atop{}}{1\atop{}}{1\atop{}}}$,\\
$\scriptstyle{{1\atop{}}{1\atop{}}{2\atop1}{2\atop{}}{1\atop{}}{0\atop{}}{0\atop{}}}$,
$\scriptstyle{{1\atop{}}{1\atop{}}{2\atop1}{2\atop{}}{1\atop{}}{1\atop{}}{0\atop{}}}$,
$\scriptstyle{{1\atop{}}{1\atop{}}{2\atop1}{2\atop{}}{1\atop{}}{1\atop{}}{1\atop{}}}$,
$\scriptstyle{{1\atop{}}{1\atop{}}{2\atop1}{2\atop{}}{2\atop{}}{1\atop{}}{1\atop{}}}$,
$\scriptstyle{{1\atop{}}{1\atop{}}{2\atop1}{2\atop{}}{2\atop{}}{2\atop{}}{1\atop{}}}$,&
$\scriptstyle{{1\atop{}}{2\atop{}}{3\atop1}{2\atop{}}{1\atop{}}{0\atop{}}{0\atop{}}}$,
$\scriptstyle{{1\atop{}}{2\atop{}}{3\atop1}{2\atop{}}{1\atop{}}{1\atop{}}{0\atop{}}}$,
$\scriptstyle{{1\atop{}}{2\atop{}}{3\atop 1}{2\atop{}}{1\atop{}}{1\atop{}}{1\atop{}}}$,\\
$\scriptstyle{{1\atop{}}{2\atop{}}{2\atop1}{1\atop{}}{0\atop{}}{0\atop{}}{0\atop{}}}$,
$\scriptstyle{{1\atop{}}{2\atop{}}{2\atop1}{1\atop{}}{1\atop{}}{0\atop{}}{0\atop{}}}$,
$\scriptstyle{{1\atop{}}{2\atop{}}{2\atop1}{1\atop{}}{1\atop{}}{1\atop{}}{0\atop{}}}$,
$\scriptstyle{{1\atop{}}{2\atop{}}{2\atop1}{1\atop{}}{1\atop{}}{1\atop{}}{1\atop{}}}$,
$\scriptstyle{{0\atop{}}{1\atop{}}{2\atop1}{2\atop{}}{2\atop{}}{1\atop{}}{0\atop{}}}$,&
$\scriptstyle{{1\atop{}}{2\atop{}}{3\atop1}{2\atop{}}{2\atop{}}{1\atop{}}{0\atop{}}}$,
$\scriptstyle{{1\atop{}}{2\atop{}}{3\atop1}{2\atop{}}{2\atop{}}{2\atop{}}{1\atop{}}}$,
$\scriptstyle{{1\atop{}}{2\atop{}}{4\atop 2}{3\atop{}}{2\atop{}}{1\atop{}}{0\atop{}}}$,\\
$\scriptstyle{{1\atop{}}{1\atop{}}{2\atop1}{2\atop{}}{2\atop{}}{1\atop{}}{0\atop{}}}$,
$\scriptstyle{{1\atop{}}{2\atop{}}{3\atop2}{2\atop{}}{1\atop{}}{0\atop{}}{0\atop{}}}$,
$\scriptstyle{{1\atop{}}{2\atop{}}{3\atop2}{2\atop{}}{1\atop{}}{1\atop{}}{0\atop{}}}$,
$\scriptstyle{{1\atop{}}{2\atop{}}{3\atop2}{2\atop{}}{1\atop{}}{1\atop{}}{1\atop{}}}$,
$\scriptstyle{{1\atop{}}{2\atop{}}{3\atop2}{2\atop{}}{2\atop{}}{1\atop{}}{0\atop{}}}$,&
$\scriptstyle{{1\atop{}}{2\atop{}}{4\atop2}{3\atop{}}{2\atop{}}{1\atop{}}{1\atop{}}}$,
$\scriptstyle{{1\atop{}}{2\atop{}}{4\atop2}{3\atop{}}{2\atop{}}{2\atop{}}{1\atop{}}}$,
$\scriptstyle{{1\atop{}}{2\atop{}}{4\atop 2}{3\atop{}}{3\atop{}}{2\atop{}}{1\atop{}}}$,\\
$\scriptstyle{{1\atop{}}{2\atop{}}{3\atop2}{2\atop{}}{2\atop{}}{1\atop{}}{1\atop{}}}$,
$\scriptstyle{{1\atop{}}{2\atop{}}{3\atop2}{2\atop{}}{2\atop{}}{2\atop{}}{1\atop{}}}$,
$\scriptstyle{{1\atop{}}{2\atop{}}{3\atop1}{3\atop{}}{2\atop{}}{1\atop{}}{0\atop{}}}$,
$\scriptstyle{{1\atop{}}{2\atop{}}{3\atop1}{3\atop{}}{2\atop{}}{1\atop{}}{1\atop{}}}$,
$\scriptstyle{{0\atop{}}{0\atop{}}{1\atop1}{1\atop{}}{1\atop{}}{0\atop{}}{0\atop{}}}$&
$\scriptstyle{{1\atop{}}{3\atop{}}{5\atop2}{4\atop{}}{3\atop{}}{2\atop{}}{1\atop{}}}$,
$\scriptstyle{{2\atop{}}{3\atop{}}{5\atop2}{4\atop{}}{3\atop{}}{2\atop{}}{1\atop{}}}$,
$\scriptstyle{{2\atop{}}{4\atop{}}{6\atop 3}{4\atop{}}{3\atop{}}{2\atop{}}{1\atop{}}}$,\\
$\scriptstyle{{1\atop{}}{2\atop{}}{3\atop1}{3\atop{}}{2\atop{}}{2\atop{}}{1\atop{}}}$,
$\scriptstyle{{1\atop{}}{2\atop{}}{3\atop1}{3\atop{}}{3\atop{}}{2\atop{}}{1\atop{}}}$,
$\scriptstyle{{1\atop{}}{2\atop{}}{4\atop2}{4\atop{}}{3\atop{}}{2\atop{}}{1\atop{}}}$,
$\scriptstyle{{1\atop{}}{3\atop{}}{4\atop2}{3\atop{}}{2\atop{}}{1\atop{}}{0\atop{}}}$,
$\scriptstyle{{2\atop{}}{3\atop{}}{4\atop2}{3\atop{}}{2\atop{}}{1\atop{}}{0\atop{}}}$,&
$\scriptstyle{{1\atop{}}{2\atop{}}{3\atop 1}{2\atop{}}{2\atop{}}{1\atop{}}{1\atop{}}}$.\\
$\scriptstyle{{1\atop{}}{3\atop{}}{4\atop2}{3\atop{}}{2\atop{}}{1\atop{}}{1\atop{}}}$,
$\scriptstyle{{2\atop{}}{3\atop{}}{4\atop2}{3\atop{}}{2\atop{}}{1\atop{}}{1\atop{}}}$,
$\scriptstyle{{1\atop{}}{3\atop{}}{4\atop2}{3\atop{}}{2\atop{}}{2\atop{}}{1\atop{}}}$,
$\scriptstyle{{2\atop{}}{3\atop{}}{4\atop2}{3\atop{}}{2\atop{}}{2\atop{}}{1\atop{}}}$,
$\scriptstyle{{1\atop{}}{3\atop{}}{4\atop2}{3\atop{}}{3\atop{}}{2\atop{}}{1\atop{}}}$,&\\
$\scriptstyle{{2\atop{}}{3\atop{}}{4\atop2}{3\atop{}}{3\atop{}}{2\atop{}}{1\atop{}}}$,
$\scriptstyle{{1\atop{}}{3\atop{}}{5\atop3}{4\atop{}}{3\atop{}}{2\atop{}}{1\atop{}}}$,
$\scriptstyle{{2\atop{}}{3\atop{}}{5\atop3}{4\atop{}}{3\atop{}}{2\atop{}}{1\atop{}}}$,
$\scriptstyle{{2\atop{}}{4\atop{}}{5\atop2}{4\atop{}}{3\atop{}}{2\atop{}}{1\atop{}}}$,
$\scriptstyle{{2\atop{}}{4\atop{}}{6\atop3}{5\atop{}}{3\atop{}}{2\atop{}}{1\atop{}}}$,&\\
$\scriptstyle{{2\atop{}}{4\atop{}}{6\atop3}{5\atop{}}{4\atop{}}{2\atop{}}{1\atop{}}}$,
$\scriptstyle{{2\atop{}}{4\atop{}}{6\atop3}{5\atop{}}{4\atop{}}{3\atop{}}{1\atop{}}}$,
$\scriptstyle{{2\atop{}}{4\atop{}}{6\atop3}{5\atop{}}{4\atop{}}{3\atop{}}{2\atop{}}}$.
&\\
\hline
\end{tabular}
\end{table}

By a routine computation we see that the
$(1,2,3,5,6,7,8)$-contributions of the $0$-roots and $1$-roots are
$(52,76,102,124,96,66,34)$ and $(20,30,40,48,36,24,12)$,
respectively. So the total contribution of all roots is
$(36,53,71,86,66,45,23)$. In view of [\cite{Bo}, Planche~VII]
Losev's condition~(C) for
$\lambda+\rho=
\textstyle{\sum}_{i=1}^8\,a_i\varpi_i$ reads
\begin{eqnarray*}
4a_1+5a_2+7a_3+10a_4+8a_5+6a_6+4a_7+2a_8&=&36\\
5a_1+8a_2+10a_3+15a_4+12a_5+9a_6+6a_7+3a_8&=&53\\
7a_1+10a_2+14a_3+20a_4+16a_5+12a_6+8a_7+4a_8&=&71\\
8a_1+12a_2+16a_3+24a_4+20a_5+15a_6+10a_7+5a_8&=&86\\
6a_1+9a_2+12a_3+18a_4+15a_5+12a_6+8a_7+4a_8&=&66\\
4a_1+6a_2+8a_3+12a_4+10a_5+8a_6+6a_7+3a_8&=&45\\
2a_1+3a_2+4a_3+6a_4+5a_5+4a_6+3a_7+2a_8&=&23.
\end{eqnarray*}
Setting $a_4=0$ and $a_i=1$ for $i\ne 4$ we obtain
a partucularly nice solution to this system of linear equations, which leads to
$$\lambda+\rho=\rho-\varpi_4=\,\textstyle{\frac{1}{2}}h^\vee.$$ As
$\langle\lambda,\alpha_4^\vee\rangle=0$,
Condition (A) holds for $\lambda+\rho$. Since $e$ is special and
$\lambda+\rho=\frac{1}{2}h^\vee$ applying [\cite{BV},
Proposition~5.10] shows that this weight also satisfies Condition~(B). Condition~(D) is vacuous in the present case as $e$ has
standard Levi type. Applying [\cite{Lo2}, 5.3] we conclude that
$I(\lambda)=I(-\varpi_4)$ is the only multiplicity-free primitive
ideal in $\mathcal{X}_{\O(e)}$ (one should keep in mind that in the
present case $[\g_e,\g_e]=\g_e$). Since $\O(e)=\O_{\rm min}$, we
thus recover the well-known Joseph ideal and determine its Duflo
realisation by applying methods of the theory of finite $W$-algebras.

\subsection{Type $({\sf E_8, 2A_1})$}\label{3.2} In the present case
$e$ is a special nilpotent element, $e^\vee\in\g^\vee $ has Dynkin
label ${\sf E_8(a_2)}$, and
$$
h^\vee=\textstyle{{2\atop{}}{2\atop{}}{0\atop 2}
{2\atop{}}{0\atop{}}{2\atop{}}{2\atop{}}}.
$$
Keeping in mind Condition~(A) we choose our pinning for $e$
as follows:
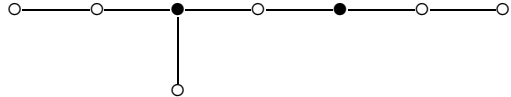
\[\xymatrix{*{\circ}\ar@{-}[r]& *{\circ} \ar@{-}[r] & *{\bullet} \ar@{-}[r]\ar@{-}[d] & *{\circ} \ar@{-}[r]
 & *{\bullet} \ar@{-}[r]&*{\circ}\ar@{-}[r]&*{\circ}\\ && *{\circ}&&&}\]
Then we may assume that the optimal cocharacter for $e$ with $({\rm d}\tau)(1)=h$ has the
form
$$\tau=\textstyle{{0\atop{}}{(-1)\atop{}}{2\atop (-1)}
{(-2)\atop{}}{2\atop{}}{(-1)\atop{}}{0\atop{}}}.$$ Since
$\dim\,\g_e=156$ the total number of positive $0$-roots and
$1$-roots is $(156-8)/2=74$. The roots are listed in the table below.

\begin{table}[htb]
\label{data2}
\begin{tabular}{|c|c|}
\hline$0$-roots  & $1$-roots
\\ \hline
$\scriptstyle{{1\atop{}}{0\atop{}}{0\atop0}{0\atop{}}{0\atop{}}{0\atop{}}{0\atop{}}}$,
$\scriptstyle{{0\atop{}}{0\atop{}}{0\atop0}{0\atop{}}{0\atop{}}{0\atop{}}{1\atop{}}}$,
$\scriptstyle{{0\atop{}}{1\atop{}}{1\atop1}{0\atop{}}{0\atop{}}{0\atop{}}{0\atop{}}}$,
$\scriptstyle{{0\atop{}}{0\atop{}}{1\atop0}{1\atop{}}{0\atop{}}{0\atop{}}{0\atop{}}}$,\,
&$\scriptstyle{{1\atop{}}{1\atop{}}{1\atop0}{0\atop{}}{0\atop{}}{0\atop{}}{0\atop{}}}$,
$\scriptstyle{{0\atop{}}{1\atop{}}{1\atop0}{0\atop{}}{0\atop{}}{0\atop{}}{0\atop{}}}$,
$\scriptstyle{{0\atop{}}{0\atop{}}{1\atop1}{0\atop{}}{0\atop{}}{0\atop{}}{0\atop{}}}$,
$\scriptstyle{{0\atop{}}{1\atop{}}{1\atop0}{1\atop{}}{1\atop{}}{0\atop{}}{0\atop{}}}$,\\
$\scriptstyle{{1\atop{}}{1\atop{}}{1\atop0}{1\atop{}}{1\atop{}}{1\atop{}}{0\atop{}}}$,
$\scriptstyle{{0\atop{}}{1\atop{}}{2\atop1}{1\atop{}}{0\atop{}}{0\atop{}}{0\atop{}}}$,
$\scriptstyle{{1\atop{}}{1\atop{}}{2\atop1}{1\atop{}}{0\atop{}}{0\atop{}}{0\atop{}}}$,
$\scriptstyle{{0\atop{}}{1\atop{}}{1\atop0}{1\atop{}}{1\atop{}}{1\atop{}}{0\atop{}}}$,
&$\scriptstyle{{1\atop{}}{1\atop{}}{1\atop0}{1\atop{}}{1\atop{}}{0\atop{}}{0\atop{}}}$,
$\scriptstyle{{1\atop{}}{2\atop{}}{2\atop1}{1\atop{}}{1\atop{}}{0\atop{}}{0\atop{}}}$,
$\scriptstyle{{0\atop{}}{0\atop{}}{0\atop0}{0\atop{}}{1\atop{}}{1\atop{}}{0\atop{}}}$,
$\scriptstyle{{0\atop{}}{0\atop{}}{0\atop0}{0\atop{}}{1\atop{}}{1\atop{}}{1\atop{}}}$,\\
$\scriptstyle{{0\atop{}}{1\atop{}}{1\atop0}{1\atop{}}{1\atop{}}{1\atop{}}{1\atop{}}}$,
$\scriptstyle{{1\atop{}}{1\atop{}}{1\atop0}{1\atop{}}{1\atop{}}{1\atop{}}{1\atop{}}}$,
$\scriptstyle{{1\atop{}}{2\atop{}}{2\atop1}{1\atop{}}{1\atop{}}{1\atop{}}{0\atop{}}}$,
$\scriptstyle{{1\atop{}}{2\atop{}}{2\atop1}{1\atop{}}{1\atop{}}{1\atop{}}{1\atop{}}}$,
&$\scriptstyle{{0\atop{}}{0\atop{}}{1\atop0}{1\atop{}}{1\atop{}}{1\atop{}}{0\atop{}}}$,
$\scriptstyle{{0\atop{}}{0\atop{}}{1\atop0}{1\atop{}}{1\atop{}}{1\atop{}}{1\atop{}}}$,
$\scriptstyle{{0\atop{}}{1\atop{}}{2\atop1}{1\atop{}}{1\atop{}}{1\atop{}}{0\atop{}}}$,
$\scriptstyle{{0\atop{}}{1\atop{}}{2\atop1}{1\atop{}}{1\atop{}}{1\atop{}}{1\atop{}}}$,\\
$\scriptstyle{{1\atop{}}{1\atop{}}{1\atop1}{0\atop{}}{0\atop{}}{0\atop{}}{0\atop{}}}$,
$\scriptstyle{{0\atop{}}{0\atop{}}{0\atop0}{1\atop{}}{1\atop{}}{0\atop{}}{0\atop{}}}$,
$\scriptstyle{{0\atop{}}{1\atop{}}{1\atop1}{1\atop{}}{1\atop{}}{0\atop{}}{0\atop{}}}$,
$\scriptstyle{{1\atop{}}{1\atop{}}{1\atop1}{1\atop{}}{1\atop{}}{0\atop{}}{0\atop{}}}$,
&$\scriptstyle{{1\atop{}}{1\atop{}}{2\atop1}{1\atop{}}{1\atop{}}{1\atop{}}{0\atop{}}}$,
$\scriptstyle{{1\atop{}}{1\atop{}}{2\atop1}{1\atop{}}{1\atop{}}{1\atop{}}{1\atop{}}}$,
$\scriptstyle{{0\atop{}}{0\atop{}}{1\atop1}{1\atop{}}{1\atop{}}{0\atop{}}{0\atop{}}}$,
$\scriptstyle{{0\atop{}}{1\atop{}}{2\atop1}{2\atop{}}{2\atop{}}{1\atop{}}{0\atop{}}}$,\\
$\scriptstyle{{0\atop{}}{0\atop{}}{1\atop1}{1\atop{}}{1\atop{}}{1\atop{}}{0\atop{}}}$,
$\scriptstyle{{1\atop{}}{2\atop{}}{3\atop1}{2\atop{}}{1\atop{}}{1\atop{}}{0\atop{}}}$,
$\scriptstyle{{0\atop{}}{0\atop{}}{1\atop1}{1\atop{}}{1\atop{}}{1\atop{}}{1\atop{}}}$,
$\scriptstyle{{1\atop{}}{2\atop{}}{3\atop1}{2\atop{}}{1\atop{}}{1\atop{}}{1\atop{}}}$,
&$\scriptstyle{{1\atop{}}{2\atop{}}{3\atop2}{2\atop{}}{2\atop{}}{1\atop{}}{1\atop{}}}$,
$\scriptstyle{{1\atop{}}{2\atop{}}{3\atop1}{2\atop{}}{2\atop{}}{2\atop{}}{1\atop{}}}$,
$\scriptstyle{{1\atop{}}{2\atop{}}{3\atop1}{3\atop{}}{3\atop{}}{2\atop{}}{1\atop{}}}$,
$\scriptstyle{{1\atop{}}{3\atop{}}{4\atop2}{3\atop{}}{3\atop{}}{2\atop{}}{1\atop{}}}$,\\
$\scriptstyle{{0\atop{}}{1\atop{}}{2\atop1}{2\atop{}}{1\atop{}}{0\atop{}}{0\atop{}}}$,
$\scriptstyle{{1\atop{}}{1\atop{}}{2\atop1}{2\atop{}}{1\atop{}}{0\atop{}}{0\atop{}}}$,
$\scriptstyle{{1\atop{}}{2\atop{}}{3\atop2}{2\atop{}}{1\atop{}}{0\atop{}}{0\atop{}}}$,
$\scriptstyle{{1\atop{}}{2\atop{}}{2\atop1}{2\atop{}}{2\atop{}}{1\atop{}}{0\atop{}}}$,
&$\scriptstyle{{2\atop{}}{3\atop{}}{4\atop2}{3\atop{}}{3\atop{}}{2\atop{}}{1\atop{}}}$,
$\scriptstyle{{1\atop{}}{2\atop{}}{4\atop2}{3\atop{}}{2\atop{}}{1\atop{}}{0\atop{}}}$,
$\scriptstyle{{1\atop{}}{2\atop{}}{4\atop2}{3\atop{}}{2\atop{}}{1\atop{}}{1\atop{}}}$,
$\scriptstyle{{1\atop{}}{3\atop{}}{5\atop2}{4\atop{}}{3\atop{}}{2\atop{}}{1\atop{}}}$,\\
$\scriptstyle{{1\atop{}}{2\atop{}}{2\atop1}{2\atop{}}{2\atop{}}{1\atop{}}{1\atop{}}}$,
$\scriptstyle{{0\atop{}}{1\atop{}}{2\atop1}{2\atop{}}{2\atop{}}{2\atop{}}{1\atop{}}}$,
$\scriptstyle{{1\atop{}}{1\atop{}}{2\atop1}{2\atop{}}{2\atop{}}{2\atop{}}{1\atop{}}}$,
$\scriptstyle{{1\atop{}}{2\atop{}}{3\atop2}{2\atop{}}{2\atop{}}{2\atop{}}{1\atop{}}}$,
&$\scriptstyle{{2\atop{}}{3\atop{}}{5\atop2}{4\atop{}}{3\atop{}}{2\atop{}}{1\atop{}}}$,
$\scriptstyle{{2\atop{}}{4\atop{}}{6\atop3}{4\atop{}}{3\atop{}}{2\atop{}}{1\atop{}}}$,
$\scriptstyle{{2\atop{}}{4\atop{}}{6\atop3}{5\atop{}}{4\atop{}}{2\atop{}}{1\atop{}}}$,
$\scriptstyle{{1\atop{}}{2\atop{}}{3\atop1}{2\atop{}}{1\atop{}}{0\atop{}}{0\atop{}}}$,\\
$\scriptstyle{{1\atop{}}{2\atop{}}{3\atop1}{3\atop{}}{2\atop{}}{1\atop{}}{0\atop{}}}$,
$\scriptstyle{{1\atop{}}{2\atop{}}{3\atop1}{3\atop{}}{2\atop{}}{1\atop{}}{1\atop{}}}$,
$\scriptstyle{{1\atop{}}{3\atop{}}{4\atop2}{3\atop{}}{2\atop{}}{1\atop{}}{0\atop{}}}$,
$\scriptstyle{{1\atop{}}{3\atop{}}{4\atop2}{3\atop{}}{2\atop{}}{1\atop{}}{1\atop{}}}$,&
$\scriptstyle{{0\atop{}}{1\atop{}}{2\atop1}{2\atop{}}{2\atop{}}{1\atop{}}{1\atop{}}}$,
$\scriptstyle{{1\atop{}}{1\atop{}}{2\atop1}{2\atop{}}{2\atop{}}{1\atop{}}{0\atop{}}}$,
$\scriptstyle{{1\atop{}}{2\atop{}}{3\atop2}{2\atop{}}{2\atop{}}{1\atop{}}{0\atop{}}}$,
$\scriptstyle{{1\atop{}}{1\atop{}}{2\atop1}{2\atop{}}{2\atop{}}{1\atop{}}{1\atop{}}}$.\\
$\scriptstyle{{2\atop{}}{3\atop{}}{4\atop2}{3\atop{}}{2\atop{}}{1\atop{}}{0\atop{}}}$,
$\scriptstyle{{2\atop{}}{3\atop{}}{4\atop2}{3\atop{}}{2\atop{}}{1\atop{}}{1\atop{}}}$,
$\scriptstyle{{2\atop{}}{4\atop{}}{5\atop2}{4\atop{}}{3\atop{}}{2\atop{}}{1\atop{}}}$,
$\scriptstyle{{1\atop{}}{2\atop{}}{4\atop2}{4\atop{}}{3\atop{}}{2\atop{}}{1\atop{}}}$,&\\
$\scriptstyle{{1\atop{}}{3\atop{}}{5\atop3}{4\atop{}}{3\atop{}}{2\atop{}}{1\atop{}}}$,
$\scriptstyle{{2\atop{}}{3\atop{}}{5\atop3}{4\atop{}}{3\atop{}}{2\atop{}}{1\atop{}}}$,
$\scriptstyle{{2\atop{}}{4\atop{}}{6\atop3}{5\atop{}}{4\atop{}}{3\atop{}}{1\atop{}}}$,
$\scriptstyle{{2\atop{}}{4\atop{}}{6\atop3}{5\atop{}}{4\atop{}}{3\atop{}}{2\atop{}}}$,&\\
$\scriptstyle{{1\atop{}}{2\atop{}}{4\atop2}{3\atop{}}{2\atop{}}{2\atop{}}{1\atop{}}}$,
$\scriptstyle{{1\atop{}}{2\atop{}}{3\atop2}{3\atop{}}{3\atop{}}{2\atop{}}{1\atop{}}}$.&\\
\hline
\end{tabular}
\end{table}
Direct computation shows that the $(1,2,3,5,7,8)$-contributions of
the $0$-roots and $1$-roots are $(36,52,70,84,42,22)$ and
$(24,36,48,58,32,16)$, respectively. So the total contribution of
all roots is $(30,44,59,71,37,19)$. In view of [\cite{Bo},
Planche~VII] Condition~(C) for 
$\lambda+\rho=\sum_{i=1}^8a_i\varpi_i$
reads
\begin{eqnarray*}
4a_1+5a_2+7a_3+10a_4+8a_5+6a_6+4a_7+2a_8&=&30\\
5a_1+8a_2+10a_3+15a_4+12a_5+9a_6+6a_7+3a_8&=&44\\
7a_1+10a_2+14a_3+20a_4+16a_5+12a_6+8a_7+4a_8&=&59\\
8a_1+12a_2+16a_3+24a_4+20a_5+15a_6+10a_7+5a_8&=&71\\
4a_1+6a_2+8a_3+12a_4+10a_5+8a_6+6a_7+3a_8&=&37\\
2a_1+3a_2+4a_3+6a_4+5a_5+4a_6+3a_7+2a_8&=&19.
\end{eqnarray*} Setting $a_i=0$ for $i=4,6$ and $a_i=1$ for $i\ne 4,6$ we obtain a nice solution
to this system of linear equations which leads to
the weight $$\lambda+\rho=\,\rho-\varpi_4-
\varpi_6.$$
As
$\langle\lambda,\alpha_4^\vee\rangle=\langle\lambda,\alpha_6^\vee\rangle=0$,
it also satisfies Condition (A). Since in the present case
$e$ is special and $e^\vee\in\g^\vee $ has Dynkin label ${\sf
E_8(a_2)}$, we see that $\lambda+\rho=\frac{1}{2}h^\vee$. So
[\cite{BV}, Proposition~5.10] shows that this weight also satisfies
Condition (B). Condition (D) is obvious because $e$ has standard
Levi type. Applying [\cite{Lo2}, 5.3] and [\cite{PT},
Proposition~11] we conclude that $I(\lambda)=I(-\varpi_4-\varpi_6)$
is the only multiplicity-free primitive ideal in
$\mathcal{X}_{\O(e)}$ (one should also keep in mind here that
$[\g_e,\g_e]=\g_e$ by [\cite{dG}]).
\subsection{Type $({\sf E_8, 3A_1})$}\label{3.3}
In this case our pinning for $e$ is
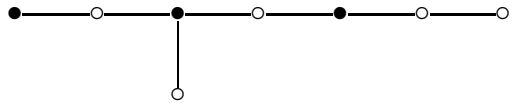
\[\xymatrix{*{\bullet}\ar@{-}[r]& *{\circ} \ar@{-}[r] & *{\bullet} \ar@{-}[r]\ar@{-}[d] & *{\circ} \ar@{-}[r]
 & *{\bullet} \ar@{-}[r]&*{\circ}\ar@{-}[r]&*{\circ}\\ && *{\circ}&&&}\]
and the optimal cocharacter for $e$ is given by
$$\tau=\textstyle{{2\atop{}}{(-2)\atop{}}{2\atop (-1)}
{(-2)\atop{}}{2\atop{}}{(-1)\atop{}}{0\atop{}}}.$$ Since
$\dim\,\g_e=136$, the total number of positive $0$-roots and
$1$-roots is $(136-8)/2=64$. The roots are given below. 

\begin{table}[htb]
\label{data2}
\begin{tabular}{|c|c|}
\hline$0$-roots  & $1$-roots
\\ \hline
$\scriptstyle{{1\atop{}}{1\atop{}}{0\atop0}{0\atop{}}{0\atop{}}{0\atop{}}{0\atop{}}}$,
$\scriptstyle{{0\atop{}}{1\atop{}}{1\atop0}{0\atop{}}{0\atop{}}{0\atop{}}{0\atop{}}}$,
$\scriptstyle{{0\atop{}}{0\atop{}}{1\atop0}{1\atop{}}{0\atop{}}{0\atop{}}{0\atop{}}}$,
$\scriptstyle{{0\atop{}}{0\atop{}}{0\atop0}{1\atop{}}{1\atop{}}{0\atop{}}{0\atop{}}}$,\,
&$\scriptstyle{{0\atop{}}{0\atop{}}{1\atop1}{0\atop{}}{0\atop{}}{0\atop{}}{0\atop{}}}$,
$\scriptstyle{{0\atop{}}{0\atop{}}{0\atop0}{0\atop{}}{1\atop{}}{1\atop{}}{0\atop{}}}$,
$\scriptstyle{{0\atop{}}{0\atop{}}{0\atop0}{0\atop{}}{1\atop{}}{1\atop{}}{1\atop{}}}$,
$\scriptstyle{{1\atop{}}{1\atop{}}{1\atop1}{0\atop{}}{0\atop{}}{0\atop{}}{0\atop{}}}$,\\
$\scriptstyle{{0\atop{}}{0\atop{}}{0\atop0}{0\atop{}}{0\atop{}}{0\atop{}}{1\atop{}}}$,
$\scriptstyle{{1\atop{}}{1\atop{}}{1\atop0}{1\atop{}}{0\atop{}}{0\atop{}}{0\atop{}}}$,
$\scriptstyle{{0\atop{}}{1\atop{}}{1\atop0}{1\atop{}}{1\atop{}}{0\atop{}}{0\atop{}}}$,
$\scriptstyle{{0\atop{}}{0\atop{}}{1\atop1}{1\atop{}}{1\atop{}}{1\atop{}}{0\atop{}}}$,
&$\scriptstyle{{0\atop{}}{0\atop{}}{1\atop1}{1\atop{}}{1\atop{}}{0\atop{}}{0\atop{}}}$,
$\scriptstyle{{0\atop{}}{0\atop{}}{1\atop0}{1\atop{}}{1\atop{}}{1\atop{}}{0\atop{}}}$,
$\scriptstyle{{0\atop{}}{0\atop{}}{1\atop0}{1\atop{}}{1\atop{}}{1\atop{}}{1\atop{}}}$,
$\scriptstyle{{1\atop{}}{1\atop{}}{1\atop1}{1\atop{}}{1\atop{}}{0\atop{}}{0\atop{}}}$,\\
$\scriptstyle{{0\atop{}}{0\atop{}}{1\atop1}{1\atop{}}{1\atop{}}{1\atop{}}{1\atop{}}}$,
$\scriptstyle{{1\atop{}}{1\atop{}}{1\atop1}{1\atop{}}{1\atop{}}{1\atop{}}{0\atop{}}}$,
$\scriptstyle{{1\atop{}}{1\atop{}}{1\atop1}{1\atop{}}{1\atop{}}{1\atop{}}{1\atop{}}}$,
$\scriptstyle{{1\atop{}}{2\atop{}}{2\atop1}{1\atop{}}{1\atop{}}{1\atop{}}{0\atop{}}}$,
&$\scriptstyle{{1\atop{}}{1\atop{}}{1\atop0}{1\atop{}}{1\atop{}}{1\atop{}}{0\atop{}}}$,
$\scriptstyle{{1\atop{}}{1\atop{}}{1\atop0}{1\atop{}}{1\atop{}}{1\atop{}}{1\atop{}}}$,
$\scriptstyle{{1\atop{}}{1\atop{}}{2\atop1}{1\atop{}}{0\atop{}}{0\atop{}}{0\atop{}}}$,
$\scriptstyle{{0\atop{}}{1\atop{}}{2\atop1}{1\atop{}}{1\atop{}}{0\atop{}}{0\atop{}}}$,\\
$\scriptstyle{{1\atop{}}{2\atop{}}{2\atop1}{1\atop{}}{1\atop{}}{1\atop{}}{1\atop{}}}$,
$\scriptstyle{{1\atop{}}{1\atop{}}{2\atop1}{2\atop{}}{1\atop{}}{1\atop{}}{0\atop{}}}$,
$\scriptstyle{{1\atop{}}{1\atop{}}{2\atop1}{2\atop{}}{1\atop{}}{1\atop{}}{1\atop{}}}$,
$\scriptstyle{{0\atop{}}{1\atop{}}{2\atop1}{2\atop{}}{2\atop{}}{1\atop{}}{0\atop{}}}$,
&$\scriptstyle{{1\atop{}}{2\atop{}}{2\atop1}{1\atop{}}{1\atop{}}{0\atop{}}{0\atop{}}}$,
$\scriptstyle{{1\atop{}}{1\atop{}}{2\atop1}{2\atop{}}{1\atop{}}{0\atop{}}{0\atop{}}}$,
$\scriptstyle{{1\atop{}}{1\atop{}}{2\atop1}{2\atop{}}{2\atop{}}{2\atop{}}{1\atop{}}}$,
$\scriptstyle{{1\atop{}}{2\atop{}}{3\atop1}{2\atop{}}{1\atop{}}{0\atop{}}{0\atop{}}}$,\\
$\scriptstyle{{0\atop{}}{1\atop{}}{2\atop1}{2\atop{}}{2\atop{}}{1\atop{}}{1\atop{}}}$,
$\scriptstyle{{1\atop{}}{2\atop{}}{2\atop1}{2\atop{}}{2\atop{}}{1\atop{}}{0\atop{}}}$,
$\scriptstyle{{1\atop{}}{2\atop{}}{2\atop1}{2\atop{}}{2\atop{}}{1\atop{}}{1\atop{}}}$,
$\scriptstyle{{1\atop{}}{2\atop{}}{3\atop1}{3\atop{}}{2\atop{}}{1\atop{}}{0\atop{}}}$,
&$\scriptstyle{{1\atop{}}{2\atop{}}{3\atop1}{2\atop{}}{2\atop{}}{2\atop{}}{1\atop{}}}$,
$\scriptstyle{{1\atop{}}{2\atop{}}{3\atop2}{2\atop{}}{2\atop{}}{1\atop{}}{1\atop{}}}$,
$\scriptstyle{{1\atop{}}{2\atop{}}{3\atop2}{2\atop{}}{2\atop{}}{1\atop{}}{0\atop{}}}$,
$\scriptstyle{{1\atop{}}{2\atop{}}{3\atop1}{3\atop{}}{3\atop{}}{2\atop{}}{1\atop{}}}$,\\
$\scriptstyle{{1\atop{}}{2\atop{}}{3\atop1}{3\atop{}}{2\atop{}}{1\atop{}}{1\atop{}}}$,
$\scriptstyle{{1\atop{}}{2\atop{}}{3\atop2}{2\atop{}}{1\atop{}}{0\atop{}}{0\atop{}}}$,
$\scriptstyle{{1\atop{}}{2\atop{}}{3\atop2}{3\atop{}}{3\atop{}}{2\atop{}}{1\atop{}}}$,
$\scriptstyle{{1\atop{}}{2\atop{}}{4\atop2}{3\atop{}}{2\atop{}}{2\atop{}}{1\atop{}}}$,
&$\scriptstyle{{1\atop{}}{2\atop{}}{4\atop2}{3\atop{}}{2\atop{}}{1\atop{}}{0\atop{}}}$,
$\scriptstyle{{1\atop{}}{2\atop{}}{4\atop2}{3\atop{}}{2\atop{}}{1\atop{}}{1\atop{}}}$,
$\scriptstyle{{2\atop{}}{3\atop{}}{4\atop2}{3\atop{}}{2\atop{}}{1\atop{}}{0\atop{}}}$,
$\scriptstyle{{2\atop{}}{3\atop{}}{4\atop2}{3\atop{}}{2\atop{}}{1\atop{}}{1\atop{}}}$,\\
$\scriptstyle{{1\atop{}}{3\atop{}}{4\atop2}{3\atop{}}{3\atop{}}{2\atop{}}{1\atop{}}}$,
$\scriptstyle{{2\atop{}}{3\atop{}}{4\atop2}{3\atop{}}{2\atop{}}{2\atop{}}{1\atop{}}}$,
$\scriptstyle{{2\atop{}}{3\atop{}}{4\atop2}{4\atop{}}{3\atop{}}{2\atop{}}{1\atop{}}}$,
$\scriptstyle{{1\atop{}}{3\atop{}}{5\atop2}{4\atop{}}{3\atop{}}{2\atop{}}{1\atop{}}}$,
&$\scriptstyle{{2\atop{}}{3\atop{}}{5\atop3}{4\atop{}}{3\atop{}}{2\atop{}}{1\atop{}}}$,
$\scriptstyle{{2\atop{}}{4\atop{}}{6\atop3}{4\atop{}}{3\atop{}}{2\atop{}}{1\atop{}}}$,
$\scriptstyle{{2\atop{}}{4\atop{}}{6\atop3}{5\atop{}}{4\atop{}}{2\atop{}}{1\atop{}}}$.\\
$\scriptstyle{{1\atop{}}{2\atop{}}{3\atop2}{2\atop{}}{2\atop{}}{2\atop{}}{1\atop{}}}$,
$\scriptstyle{{2\atop{}}{4\atop{}}{5\atop2}{4\atop{}}{3\atop{}}{2\atop{}}{1\atop{}}}$,
$\scriptstyle{{2\atop{}}{4\atop{}}{6\atop3}{5\atop{}}{4\atop{}}{3\atop{}}{1\atop{}}}$,
$\scriptstyle{{2\atop{}}{4\atop{}}{6\atop3}{5\atop{}}{4\atop{}}{3\atop{}}{2\atop{}}}$,&
\\
$\scriptstyle{{0\atop{}}{1\atop{}}{2\atop1}{1\atop{}}{1\atop{}}{1\atop{}}{0\atop{}}}$,
$\scriptstyle{{0\atop{}}{1\atop{}}{2\atop1}{1\atop{}}{1\atop{}}{1\atop{}}{1\atop{}}}$,
$\scriptstyle{{1\atop{}}{2\atop{}}{3\atop1}{2\atop{}}{1\atop{}}{1\atop{}}{0\atop{}}}$,
$\scriptstyle{{1\atop{}}{2\atop{}}{3\atop1}{2\atop{}}{1\atop{}}{1\atop{}}{1\atop{}}}$,&\\
$\scriptstyle{{1\atop{}}{2\atop{}}{4\atop2}{4\atop{}}{3\atop{}}{2\atop{}}{1\atop{}}}$.&\\
\hline
\end{tabular}
\end{table}
It is straightforward to see that the $(2,3,5,7,8)$-contributions of
the $0$-roots and $1$-roots are $(44,62,76,42,22)$ and
$(33,41,49,24,12)$, respectively. So the total contribution of all
roots is
$\big(\frac{77}{2},\frac{103}{2},\frac{125}{2},\frac{66}{2},\frac{34}{2}\big)$.
Since in the present case the orbit $\O(e)$ is non-special, it would
be impossible for us to find $\lambda+\rho=\sum_{i=1}^8\,a_i\varpi_i$ with
$a_i\in\Z$ satisfying Conditions (A), (B), (C), (D). All things
considered, it seems reasonable to seek $\lambda+\rho=\sum_{i=1}^8a_i\varpi_i$ with 
$a_i\in\frac{1}{2}\Z$. (We have already mentioned at the end of Section~2
that guessing plays an important role in this investigation.)
Setting $\widetilde{a}_i:=2a_i$ we rewrite Condition~(C) for
$\lambda+\rho$ as follows:
\begin{eqnarray*}
5\widetilde{a}_1+8\widetilde{a}_2+10\widetilde{a}_3+15\widetilde{a}_4+12\widetilde{a}_5+
9\widetilde{a}_6+6\widetilde{a}_7+3\widetilde{a}_8&=&77\\
7\widetilde{a}_1+10\widetilde{a}_2+14\widetilde{a}_3+20\widetilde{a}_4+
16\widetilde{a}_5+12\widetilde{a}_6+8\widetilde{a}_7+4\widetilde{a}_8&=&103\\
8\widetilde{a}_1+12\widetilde{a}_2+16\widetilde{a}_3+24\widetilde{a}_4+
20\widetilde{a}_5+15\widetilde{a}_6+10\widetilde{a}_7+5\widetilde{a}_8&=&125\\
4\widetilde{a}_1+6\widetilde{a}_2+8\widetilde{a}_3+12\widetilde{a}_4+
10\widetilde{a}_5+8\widetilde{a}_6+6\widetilde{a}_7+3\widetilde{a}_8&=&66\\
2\widetilde{a}_1+3\widetilde{a}_2+4\widetilde{a}_3+6\widetilde{a}_4+
5\widetilde{a}_5+4\widetilde{a}_6+3\widetilde{a}_7+2\widetilde{a}_8&=&34.
\end{eqnarray*}
A fairly nice solution to this system of linear
equations is given by setting $\widetilde{a}_i=1$ for $1\le i\le 6$ and
$\widetilde{a}_i=2$ for $i=7,8$. Now we need to check (with fingers
crossed) that it satisfies the rest of
Losev's conditions.

First we observe that the integral root system
$\Phi_{\lambda+\rho}=\Phi_\lambda$ of our weight
$$\lambda+\rho=\textstyle{\frac{1}{2}}(\rho+\varpi_7+\varpi_8)$$ contains the positive roots
$$
\xymatrix{{\beta_1}\ar@{-}[r]& {\beta_3} \ar@{-}[r] & {\beta_4}
\ar@{-}[r]\ar@{-}[d] & {\beta_5} \ar@{-}[r]
 & {\beta_6}\ar@{-}[r] &{\beta_7} &{\beta_8}\\ && {\beta_2}&&&}
$$
where
\begin{eqnarray*}
\beta_1\,=\,{\scriptstyle{0\atop{}}{0\atop{}}{0\atop0}{0\atop{}}{0\atop{}}{0\atop{}}{1\atop{}}},\
\
\beta_2\,=\,{\scriptstyle{0\atop{}}{1\atop{}}{1\atop0}{0\atop{}}{0\atop{}}{0\atop{}}{0\atop{}}},\
\
\beta_3\,=\,{\scriptstyle{0\atop{}}{0\atop{}}{0\atop0}{0\atop{}}{0\atop{}}{1\atop{}}{0\atop{}}},\
\
\beta_4\,=\,{\scriptstyle{0\atop{}}{0\atop{}}{0\atop0}{1\atop{}}{1\atop{}}{0\atop{}}{0\atop{}}},\\
\beta_5\,=\,{\scriptstyle{0\atop{}}{0\atop{}}{1\atop1}{0\atop{}}{0\atop{}}{0\atop{}}{0\atop{}}},\
\
\beta_6\,=\,{\scriptstyle{1\atop{}}{1\atop{}}{0\atop0}{0\atop{}}{0\atop{}}{0\atop{}}{0\atop{}}},\
\
\beta_7\,=\,{\scriptstyle{0\atop{}}{0\atop{}}{1\atop0}{1\atop{}}{0\atop{}}{0\atop{}}{0\atop{}}},\
\
\beta_8\,=\,{\scriptstyle{0\atop{}}{1\atop{}}{1\atop1}{1\atop{}}{0\atop{}}{0\atop{}}{0\atop{}}},
\end{eqnarray*}
and hence has type ${\sf E_7+A_1}$ (by the maximality of that root
subsystem). Moreover, the roots $\{\beta_i\,|\,\,1\le i\le 8\}$ form
the basis of $\Phi_{\lambda}$ contained in $\Phi^+$. It if
straightforward to see that $\lambda+\rho$ is strongly dominant (and
integral) on $\Phi_\lambda^+=\Phi^+\cap\Phi_\lambda$. Since
$$|\Phi^+|-|\Phi_\lambda^+|=120-64=56=\textstyle{\frac{1}{2}}(248-136)=\frac{1}{2}\dim\,\O(e)$$
we can apply [\cite{Jo1}, Corollary~3.5] to deduce that $\dim\,{\rm
VA}(I(\lambda))=\dim\,\O(e)$. Since
$\langle\lambda+\rho,\alpha_i^\vee\rangle=\frac{1}{2}$ for
$i\in\{1,4,6\}$, our choice of pinning shows that $\lambda+\rho$
satisfies Condition~(A). Then $e\in{\rm VA}(I(\lambda))$ and the
above yields ${\rm VA}(I(\lambda))=\overline{\O(e)}$. Therefore,
$\lambda+\rho$ satisfies Condition~(B) as well. Condition~(D) is
vacuous as $e$ has standard Levi type. Since in the present case
$\g_e=[\g_e,\g_e]$ by [\cite{dG}] we can apply [\cite{Lo2}, 5.3] and
[\cite{PT}, Proposition~11] to conclude that
$I(-\frac{1}{2}(\varpi_1+\varpi_2+\varpi_3+\varpi_4+\varpi_5+\varpi_6))$
is the only multiplicity-free primitive ideal in
$\mathcal{X}_{\O(e)}$.

\subsection{Type $({\sf E_8, 4A_1})$}\label{3.4}
In this case our pinning for $e$ is
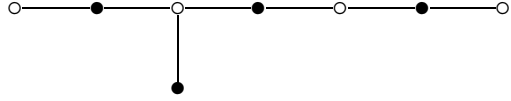
\[\xymatrix{*{\circ}\ar@{-}[r]& *{\bullet} \ar@{-}[r] & *{\circ} \ar@{-}[r]\ar@{-}[d] & *{\bullet} \ar@{-}[r]
 & *{\circ} \ar@{-}[r]&*{\bullet}\ar@{-}[r]&*{\circ}\\ && *{\bullet}&&&}\]
Then we can choose the optimal cocharacter for $e$ as follows:
$$\tau=\textstyle{{(-1)\atop{}}{2\atop{}}{(-3)\atop 2}
{2\atop{}}{(-2)\atop{}}{2\atop{}}{(-1)\atop{}}}.$$ As in the present
case $\dim\,\g_e=120$, the total number of positive $0$-roots and
$1$-roots is $(120-8)/2=56$. The roots are listed below.

\begin{table}[htb]
\label{data2}
\begin{tabular}{|c|c|}
\hline$0$-roots  & $1$-roots
\\ \hline
$\scriptstyle{{1\atop{}}{1\atop{}}{1\atop1}{0\atop{}}{0\atop{}}{0\atop{}}{0\atop{}}}$,
$\scriptstyle{{1\atop{}}{1\atop{}}{1\atop0}{1\atop{}}{0\atop{}}{0\atop{}}{0\atop{}}}$,
$\scriptstyle{{0\atop{}}{0\atop{}}{0\atop0}{1\atop{}}{1\atop{}}{0\atop{}}{0\atop{}}}$,
$\scriptstyle{{0\atop{}}{0\atop{}}{0\atop0}{0\atop{}}{1\atop{}}{1\atop{}}{0\atop{}}}$,\,
&$\scriptstyle{{1\atop{}}{1\atop{}}{0\atop0}{0\atop{}}{0\atop{}}{0\atop{}}{0\atop{}}}$,
$\scriptstyle{{0\atop{}}{1\atop{}}{1\atop1}{0\atop{}}{0\atop{}}{0\atop{}}{0\atop{}}}$,
$\scriptstyle{{0\atop{}}{0\atop{}}{1\atop1}{1\atop{}}{0\atop{}}{0\atop{}}{0\atop{}}}$,
$\scriptstyle{{0\atop{}}{0\atop{}}{0\atop0}{0\atop{}}{0\atop{}}{1\atop{}}{1\atop{}}}$,\\
$\scriptstyle{{0\atop{}}{1\atop{}}{1\atop0}{1\atop{}}{1\atop{}}{1\atop{}}{1\atop{}}}$,
$\scriptstyle{{0\atop{}}{0\atop{}}{1\atop1}{1\atop{}}{1\atop{}}{1\atop{}}{1\atop{}}}$,
$\scriptstyle{{1\atop{}}{1\atop{}}{1\atop1}{1\atop{}}{1\atop{}}{0\atop{}}{0\atop{}}}$,
$\scriptstyle{{1\atop{}}{1\atop{}}{1\atop0}{1\atop{}}{1\atop{}}{1\atop{}}{0\atop{}}}$,
&$\scriptstyle{{0\atop{}}{0\atop{}}{0\atop0}{1\atop{}}{1\atop{}}{1\atop{}}{1\atop{}}}$,
$\scriptstyle{{0\atop{}}{1\atop{}}{1\atop1}{1\atop{}}{1\atop{}}{0\atop{}}{0\atop{}}}$,
$\scriptstyle{{0\atop{}}{1\atop{}}{1\atop0}{1\atop{}}{1\atop{}}{1\atop{}}{0\atop{}}}$,
$\scriptstyle{{1\atop{}}{1\atop{}}{1\atop1}{1\atop{}}{1\atop{}}{1\atop{}}{1\atop{}}}$,\\
$\scriptstyle{{0\atop{}}{1\atop{}}{2\atop1}{1\atop{}}{1\atop{}}{1\atop{}}{0\atop{}}}$,
$\scriptstyle{{1\atop{}}{1\atop{}}{2\atop1}{2\atop{}}{1\atop{}}{0\atop{}}{0\atop{}}}$,
$\scriptstyle{{1\atop{}}{1\atop{}}{2\atop1}{2\atop{}}{1\atop{}}{1\atop{}}{1\atop{}}}$,
$\scriptstyle{{1\atop{}}{2\atop{}}{2\atop1}{1\atop{}}{1\atop{}}{1\atop{}}{1\atop{}}}$,
&$\scriptstyle{{1\atop{}}{2\atop{}}{2\atop1}{2\atop{}}{1\atop{}}{0\atop{}}{0\atop{}}}$,
$\scriptstyle{{1\atop{}}{1\atop{}}{2\atop1}{2\atop{}}{1\atop{}}{1\atop{}}{0\atop{}}}$,
$\scriptstyle{{0\atop{}}{1\atop{}}{2\atop1}{2\atop{}}{1\atop{}}{1\atop{}}{1\atop{}}}$,
$\scriptstyle{{1\atop{}}{2\atop{}}{2\atop1}{2\atop{}}{2\atop{}}{1\atop{}}{0\atop{}}}$,\\
$\scriptstyle{{1\atop{}}{2\atop{}}{3\atop2}{2\atop{}}{1\atop{}}{0\atop{}}{0\atop{}}}$,
$\scriptstyle{{1\atop{}}{2\atop{}}{3\atop1}{2\atop{}}{1\atop{}}{1\atop{}}{0\atop{}}}$,
$\scriptstyle{{1\atop{}}{2\atop{}}{3\atop2}{2\atop{}}{2\atop{}}{1\atop{}}{0\atop{}}}$,
$\scriptstyle{{1\atop{}}{2\atop{}}{3\atop1}{3\atop{}}{2\atop{}}{1\atop{}}{0\atop{}}}$,
&$\scriptstyle{{0\atop{}}{1\atop{}}{2\atop1}{2\atop{}}{2\atop{}}{2\atop{}}{1\atop{}}}$,
$\scriptstyle{{1\atop{}}{2\atop{}}{3\atop2}{2\atop{}}{2\atop{}}{2\atop{}}{1\atop{}}}$,
$\scriptstyle{{1\atop{}}{2\atop{}}{3\atop1}{3\atop{}}{2\atop{}}{2\atop{}}{1\atop{}}}$,
$\scriptstyle{{1\atop{}}{2\atop{}}{3\atop2}{3\atop{}}{2\atop{}}{1\atop{}}{1\atop{}}}$,\\
$\scriptstyle{{1\atop{}}{3\atop{}}{4\atop2}{3\atop{}}{2\atop{}}{1\atop{}}{1\atop{}}}$,
$\scriptstyle{{1\atop{}}{2\atop{}}{4\atop2}{3\atop{}}{2\atop{}}{2\atop{}}{1\atop{}}}$,
$\scriptstyle{{1\atop{}}{3\atop{}}{4\atop2}{3\atop{}}{3\atop{}}{2\atop{}}{1\atop{}}}$,
$\scriptstyle{{1\atop{}}{2\atop{}}{4\atop2}{4\atop{}}{3\atop{}}{2\atop{}}{1\atop{}}}$,
&$\scriptstyle{{1\atop{}}{2\atop{}}{3\atop2}{3\atop{}}{3\atop{}}{2\atop{}}{1\atop{}}}$,
$\scriptstyle{{1\atop{}}{3\atop{}}{5\atop3}{4\atop{}}{3\atop{}}{2\atop{}}{1\atop{}}}$,
$\scriptstyle{{1\atop{}}{3\atop{}}{4\atop2}{3\atop{}}{2\atop{}}{1\atop{}}{0\atop{}}}$,
$\scriptstyle{{0\atop{}}{1\atop{}}{1\atop0}{1\atop{}}{0\atop{}}{0\atop{}}{0\atop{}}}$,\\
$\scriptstyle{{2\atop{}}{3\atop{}}{5\atop3}{4\atop{}}{3\atop{}}{2\atop{}}{1\atop{}}}$,
$\scriptstyle{{2\atop{}}{4\atop{}}{5\atop2}{4\atop{}}{3\atop{}}{2\atop{}}{1\atop{}}}$,
$\scriptstyle{{0\atop{}}{1\atop{}}{2\atop1}{2\atop{}}{2\atop{}}{1\atop{}}{0\atop{}}}$,
$\scriptstyle{{2\atop{}}{3\atop{}}{4\atop2}{3\atop{}}{2\atop{}}{1\atop{}}{0\atop{}}}$,
&$\scriptstyle{{2\atop{}}{3\atop{}}{4\atop2}{3\atop{}}{2\atop{}}{2\atop{}}{1\atop{}}}$,
$\scriptstyle{{2\atop{}}{3\atop{}}{4\atop2}{4\atop{}}{3\atop{}}{2\atop{}}{1\atop{}}}$,
$\scriptstyle{{0\atop{}}{0\atop{}}{1\atop1}{1\atop{}}{1\atop{}}{1\atop{}}{0\atop{}}}$,
$\scriptstyle{{1\atop{}}{2\atop{}}{2\atop1}{1\atop{}}{1\atop{}}{1\atop{}}{0\atop{}}}$,\\
$\scriptstyle{{2\atop{}}{4\atop{}}{6\atop3}{5\atop{}}{4\atop{}}{3\atop{}}{2\atop{}}}$,
$\scriptstyle{{1\atop{}}{1\atop{}}{2\atop1}{2\atop{}}{2\atop{}}{2\atop{}}{1\atop{}}}$,
$\scriptstyle{{1\atop{}}{2\atop{}}{2\atop1}{2\atop{}}{2\atop{}}{1\atop{}}{1\atop{}}}$,
$\scriptstyle{{0\atop{}}{1\atop{}}{2\atop1}{1\atop{}}{0\atop{}}{0\atop{}}{0\atop{}}}$.
&$\scriptstyle{{2\atop{}}{4\atop{}}{6\atop3}{5\atop{}}{3\atop{}}{2\atop{}}{1\atop{}}}$,
$\scriptstyle{{2\atop{}}{4\atop{}}{6\atop3}{5\atop{}}{4\atop{}}{3\atop{}}{1\atop{}}}$,
$\scriptstyle{{1\atop{}}{2\atop{}}{2\atop1}{1\atop{}}{0\atop{}}{0\atop{}}{0\atop{}}}$,
$\scriptstyle{{1\atop{}}{2\atop{}}{3\atop2}{2\atop{}}{1\atop{}}{1\atop{}}{1\atop{}}}$.
\\
\hline
\end{tabular}
\end{table}
Computations show that the $(1,4,6,8)$-contributions of the
$0$-roots and $1$-roots are $(24,70,44,14)$ and $(22,65,40,15)$,
respectively, and the total contribution of all roots is
$\big(\frac{46}{2},\frac{135}{2},\frac{84}{2},\frac{29}{2}\big)$.
Since in the present case the orbit $\O(e)$ is non-special, it seems
reasonable to seek $\lambda+\rho=\sum_{i=1}^8a_i\varpi_i$ with all $a_i$ half-integers (we
have to guess again at this point). Setting $\widetilde{a}_i:=2a_i$
we rewrite Condition~(C) for $\lambda+\rho$ as follows:
\begin{eqnarray*}
4\widetilde{a}_1+5\widetilde{a}_2+7\widetilde{a}_3+10\widetilde{a}_4+8\widetilde{a}_5+
6\widetilde{a}_6+4\widetilde{a}_7+2\widetilde{a}_8&=&46\\
10\widetilde{a}_1+15\widetilde{a}_2+20\widetilde{a}_3+30\widetilde{a}_4+
24\widetilde{a}_5+18\widetilde{a}_6+12\widetilde{a}_7+6\widetilde{a}_8&=&135\\
6\widetilde{a}_1+9\widetilde{a}_2+12\widetilde{a}_3+18\widetilde{a}_4+
15\widetilde{a}_5+12\widetilde{a}_6+8\widetilde{a}_7+4\widetilde{a}_8&=&84\\
2\widetilde{a}_1+3\widetilde{a}_2+4\widetilde{a}_3+6\widetilde{a}_4+
5\widetilde{a}_5+4\widetilde{a}_6+3\widetilde{a}_7+2\widetilde{a}_8&=&29.
\end{eqnarray*}
A really nice solution to this system of linear
equations is given by setting $\widetilde{a}_i=1$ for $1\le i\le
8$. Of course, we still need to check that it
satisfies Conditions (A), (B) and (D).

We note that the integral root system
$\Phi_{\lambda+\rho}=\Phi_\lambda$ of our weight
$\lambda+\rho=\textstyle{\frac{1}{2}}\rho$ contains the positive
roots
$$
\xymatrix{{\beta_1}\ar@{-}[r]& {\beta_2} \ar@{-}[r] & {\beta_3}
\ar@{-}[r] & {\beta_4} \ar@{-}[r]
 & {\beta_5} \ar@{-}[r]&{\beta_6}\ar@{-}[r]\ar@{-}[d] &{\beta_7}\\ &&&&& {\beta_8}}
$$
where
\begin{eqnarray*}
\beta_1\,=\,{\scriptstyle{0\atop{}}{1\atop{}}{1\atop1}{1\atop{}}{0\atop{}}{0\atop{}}{0\atop{}}},\
\
\beta_2\,=\,{\scriptstyle{0\atop{}}{0\atop{}}{0\atop0}{0\atop{}}{1\atop{}}{1\atop{}}{0\atop{}}},\
\
\beta_3\,=\,{\scriptstyle{0\atop{}}{0\atop{}}{1\atop0}{1\atop{}}{0\atop{}}{0\atop{}}{0\atop{}}},\
\
\beta_4\,=\,{\scriptstyle{1\atop{}}{1\atop{}}{0\atop0}{0\atop{}}{0\atop{}}{0\atop{}}{0\atop{}}},\\
\beta_5\,=\,{\scriptstyle{0\atop{}}{0\atop{}}{1\atop1}{0\atop{}}{0\atop{}}{0\atop{}}{0\atop{}}},\
\
\beta_6\,=\,{\scriptstyle{0\atop{}}{0\atop{}}{0\atop0}{1\atop{}}{1\atop{}}{0\atop{}}{0\atop{}}},\
\
\beta_7\,=\,{\scriptstyle{0\atop{}}{0\atop{}}{0\atop0}{0\atop{}}{0\atop{}}{1\atop{}}{1\atop{}}},\
\
\beta_8\,=\,{\scriptstyle{0\atop{}}{1\atop{}}{1\atop0}{0\atop{}}{0\atop{}}{0\atop{}}{0\atop{}}},
\end{eqnarray*}
and hence has type ${\sf D_8}$ (by the maximality of that root
subsystem). Moreover, the roots $\{\beta_i\,|\,\,1\le i\le 8\}$ form
the basis of $\Phi_{\lambda}$ contained in $\Phi^+$. It if
straightforward to see that $\lambda+\rho$ is strongly dominant on
$\Phi^+\cap\Phi_\lambda$ and does not take values in $\Z^{>0}$ on
$\{\alpha_2,\alpha_3,\alpha_5,\alpha_7\}$. Since
$$|\Phi^+|-|\Phi_\lambda^+|=120-56=64=\textstyle{\frac{1}{2}}(248-120)=\frac{1}{2}\dim\,\O(e)$$
we can now argue as in the previous case to deduce that
$\lambda+\rho$ satisfies Conditions (A), (B), (C), (D). Since in the
present case $\g_e=[\g_e,\g_e]$ by [\cite{dG}] we conclude that
$I(-\frac{1}{2}\rho)$ is the only multiplicity-free primitive ideal
in $\mathcal{X}_{\O(e)}$.
\subsection{Type $({\sf E_8, A_2+A_1})$}\label{3.5}
Here $e$ is a special nilpotent element, $e^\vee$ has type ${\sf
E_8(a_4)}$, and
$$h^\vee=\textstyle{{2\atop{}}{0\atop{}}{2\atop 0}
{0\atop{}}{2\atop{}}{0\atop{}}{2\atop{}}}.$$ Numerical experiments indicate that $e$ is quite immune
to the choice of pinning and the following one is probably the best
possible:
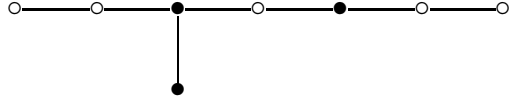
\[\xymatrix{*{\circ}\ar@{-}[r]& *{\circ} \ar@{-}[r] & *{\bullet} \ar@{-}[r]\ar@{-}[d] & *{\circ} \ar@{-}[r]
 & *{\bullet} \ar@{-}[r]&*{\circ}\ar@{-}[r]&*{\circ}\\ && *{\bullet}&&&}\]
Here the optimal cocharacter has the form
$$\tau=\textstyle{{0\atop{}}{(-2)\atop{}}{2\atop 2}
{(-3)\atop{}}{2\atop{}}{(-1)\atop{}}{0\atop{}}}.$$ Since
$\dim\,\g_e=112$ the total number of positive $0$-roots and
$1$-roots is $(112-8)/2=52$. The roots are listed in the table
below.
\begin{table}[htb]
\label{data2}
\begin{tabular}{|c|c|}
\hline$0$-roots  & $1$-roots
\\ \hline
$\scriptstyle{{1\atop{}}{0\atop{}}{0\atop0}{0\atop{}}{0\atop{}}{0\atop{}}{0\atop{}}}$,
$\scriptstyle{{1\atop{}}{1\atop{}}{1\atop0}{0\atop{}}{0\atop{}}{0\atop{}}{0\atop{}}}$,
$\scriptstyle{{0\atop{}}{0\atop{}}{0\atop0}{0\atop{}}{0\atop{}}{0\atop{}}{1\atop{}}}$,
$\scriptstyle{{0\atop{}}{1\atop{}}{1\atop0}{0\atop{}}{0\atop{}}{0\atop{}}{0\atop{}}}$,\,
&$\scriptstyle{{0\atop{}}{0\atop{}}{0\atop0}{0\atop{}}{1\atop{}}{1\atop{}}{0\atop{}}}$,
$\scriptstyle{{0\atop{}}{0\atop{}}{0\atop0}{0\atop{}}{1\atop{}}{1\atop{}}{1\atop{}}}$,
$\scriptstyle{{0\atop{}}{0\atop{}}{1\atop1}{1\atop{}}{0\atop{}}{0\atop{}}{0\atop{}}}$,
$\scriptstyle{{0\atop{}}{1\atop{}}{1\atop1}{1\atop{}}{1\atop{}}{0\atop{}}{0\atop{}}}$,\\
$\scriptstyle{{0\atop{}}{0\atop{}}{1\atop0}{1\atop{}}{1\atop{}}{1\atop{}}{0\atop{}}}$,
$\scriptstyle{{0\atop{}}{0\atop{}}{1\atop0}{1\atop{}}{1\atop{}}{1\atop{}}{1\atop{}}}$,
$\scriptstyle{{0\atop{}}{1\atop{}}{1\atop1}{1\atop{}}{1\atop{}}{1\atop{}}{0\atop{}}}$,
$\scriptstyle{{1\atop{}}{1\atop{}}{1\atop1}{1\atop{}}{1\atop{}}{1\atop{}}{0\atop{}}}$,
&$\scriptstyle{{1\atop{}}{1\atop{}}{1\atop1}{1\atop{}}{1\atop{}}{0\atop{}}{0\atop{}}}$,
$\scriptstyle{{0\atop{}}{0\atop{}}{1\atop0}{1\atop{}}{1\atop{}}{0\atop{}}{0\atop{}}}$,
$\scriptstyle{{0\atop{}}{1\atop{}}{2\atop1}{1\atop{}}{0\atop{}}{0\atop{}}{0\atop{}}}$,
$\scriptstyle{{1\atop{}}{1\atop{}}{2\atop1}{1\atop{}}{0\atop{}}{0\atop{}}{0\atop{}}}$,\\
$\scriptstyle{{0\atop{}}{1\atop{}}{1\atop1}{1\atop{}}{1\atop{}}{1\atop{}}{1\atop{}}}$,
$\scriptstyle{{1\atop{}}{1\atop{}}{1\atop1}{1\atop{}}{1\atop{}}{1\atop{}}{1\atop{}}}$,
$\scriptstyle{{0\atop{}}{1\atop{}}{2\atop1}{2\atop{}}{1\atop{}}{0\atop{}}{0\atop{}}}$,
$\scriptstyle{{1\atop{}}{1\atop{}}{2\atop1}{2\atop{}}{1\atop{}}{0\atop{}}{0\atop{}}}$,
&$\scriptstyle{{1\atop{}}{2\atop{}}{2\atop1}{1\atop{}}{1\atop{}}{0\atop{}}{0\atop{}}}$,
$\scriptstyle{{0\atop{}}{1\atop{}}{2\atop1}{2\atop{}}{2\atop{}}{1\atop{}}{0\atop{}}}$,
$\scriptstyle{{1\atop{}}{1\atop{}}{2\atop1}{2\atop{}}{2\atop{}}{1\atop{}}{0\atop{}}}$,
$\scriptstyle{{0\atop{}}{1\atop{}}{2\atop1}{2\atop{}}{2\atop{}}{1\atop{}}{1\atop{}}}$,\\
$\scriptstyle{{1\atop{}}{2\atop{}}{3\atop1}{2\atop{}}{1\atop{}}{0\atop{}}{0\atop{}}}$,
$\scriptstyle{{1\atop{}}{2\atop{}}{2\atop1}{1\atop{}}{1\atop{}}{1\atop{}}{0\atop{}}}$,
$\scriptstyle{{1\atop{}}{2\atop{}}{2\atop1}{1\atop{}}{1\atop{}}{1\atop{}}{1\atop{}}}$,
$\scriptstyle{{0\atop{}}{1\atop{}}{2\atop1}{2\atop{}}{2\atop{}}{2\atop{}}{1\atop{}}}$,
&$\scriptstyle{{1\atop{}}{1\atop{}}{2\atop1}{2\atop{}}{2\atop{}}{1\atop{}}{1\atop{}}}$,
$\scriptstyle{{1\atop{}}{2\atop{}}{3\atop1}{2\atop{}}{2\atop{}}{1\atop{}}{1\atop{}}}$,
$\scriptstyle{{1\atop{}}{2\atop{}}{3\atop1}{2\atop{}}{1\atop{}}{1\atop{}}{0\atop{}}}$,
$\scriptstyle{{1\atop{}}{2\atop{}}{3\atop2}{2\atop{}}{1\atop{}}{1\atop{}}{0\atop{}}}$,\\
$\scriptstyle{{1\atop{}}{1\atop{}}{2\atop1}{2\atop{}}{2\atop{}}{2\atop{}}{1\atop{}}}$,
$\scriptstyle{{1\atop{}}{2\atop{}}{3\atop1}{2\atop{}}{2\atop{}}{2\atop{}}{1\atop{}}}$,
$\scriptstyle{{1\atop{}}{2\atop{}}{3\atop2}{3\atop{}}{2\atop{}}{1\atop{}}{0\atop{}}}$,
$\scriptstyle{{1\atop{}}{2\atop{}}{3\atop2}{2\atop{}}{2\atop{}}{1\atop{}}{1\atop{}}}$,
&$\scriptstyle{{1\atop{}}{2\atop{}}{3\atop2}{2\atop{}}{1\atop{}}{1\atop{}}{1\atop{}}}$,
$\scriptstyle{{1\atop{}}{2\atop{}}{3\atop2}{3\atop{}}{3\atop{}}{2\atop{}}{1\atop{}}}$,
$\scriptstyle{{1\atop{}}{3\atop{}}{4\atop2}{3\atop{}}{3\atop{}}{2\atop{}}{1\atop{}}}$,
$\scriptstyle{{2\atop{}}{3\atop{}}{4\atop2}{3\atop{}}{3\atop{}}{2\atop{}}{1\atop{}}}$,\\
$\scriptstyle{{1\atop{}}{3\atop{}}{4\atop1}{3\atop{}}{2\atop{}}{1\atop{}}{0\atop{}}}$,
$\scriptstyle{{1\atop{}}{3\atop{}}{4\atop2}{3\atop{}}{2\atop{}}{1\atop{}}{1\atop{}}}$,
$\scriptstyle{{1\atop{}}{3\atop{}}{5\atop2}{4\atop{}}{3\atop{}}{2\atop{}}{1\atop{}}}$,
$\scriptstyle{{2\atop{}}{3\atop{}}{5\atop2}{4\atop{}}{3\atop{}}{2\atop{}}{1\atop{}}}$,
&$\scriptstyle{{2\atop{}}{4\atop{}}{6\atop3}{5\atop{}}{4\atop{}}{2\atop{}}{1\atop{}}}$,
$\scriptstyle{{1\atop{}}{2\atop{}}{4\atop2}{3\atop{}}{2\atop{}}{2\atop{}}{1\atop{}}}$.
\\
$\scriptstyle{{2\atop{}}{3\atop{}}{4\atop2}{3\atop{}}{2\atop{}}{1\atop{}}{0\atop{}}}$,
$\scriptstyle{{2\atop{}}{3\atop{}}{4\atop2}{3\atop{}}{2\atop{}}{1\atop{}}{1\atop{}}}$,
$\scriptstyle{{2\atop{}}{4\atop{}}{5\atop3}{4\atop{}}{3\atop{}}{2\atop{}}{1\atop{}}}$,
$\scriptstyle{{1\atop{}}{2\atop{}}{4\atop2}{4\atop{}}{3\atop{}}{2\atop{}}{1\atop{}}}$,
&\\
$\scriptstyle{{2\atop{}}{4\atop{}}{6\atop3}{5\atop{}}{4\atop{}}{3\atop{}}{1\atop{}}}$,
$\scriptstyle{{2\atop{}}{4\atop{}}{6\atop3}{5\atop{}}{4\atop{}}{3\atop{}}{2\atop{}}}$.
&\\
\hline
\end{tabular}
\end{table}
It is straightforward to check that the $(1,3,5,7,8)$-contributions
of the $0$-roots and $1$-roots are $(28,54,64,34,18)$ and
$(16,32,40,20,10)$, respectively. The total contribution of all
roots is $(22,43,52,27,14)$. In view of [\cite{Bo}, Planche~VII]
Losev's condition~(C) for $\lambda+\rho=\sum_{i=1}^8a_i\varpi_i$ reads
\begin{eqnarray*}
4a_1+5a_2+7a_3+10a_4+8a_5+6a_6+4a_7+2a_8&=&22\\
7a_1+10a_2+14a_3+20a_4+16a_5+12a_6+8a_7+4a_8&=&43\\
8a_1+12a_2+16a_3+24a_4+20a_5+15a_6+10a_7+5a_8&=&52\\
4a_1+6a_2+8a_3+12a_4+10a_5+8a_6+6a_7+3a_8&=&27\\
2a_1+3a_2+4a_3+6a_4+5a_5+4a_6+3a_7+2a_8&=&14.
\end{eqnarray*}
A reasonably nice solution to this system of linear equations is
provided by
$$\lambda+\rho=\textstyle{{1\atop{}}{1\atop{}}{2\atop (-1)}
{(-1)\atop{}}{1\atop{}}{1\atop{}}{1\atop{}}}=s_6s_2s_4\big(\textstyle{{1\atop{}}{0\atop{}}{1\atop
0}
{0\atop{}}{1\atop{}}{0\atop{}}{1\atop{}}}\big)=s_6s_2s_4({\scriptstyle\frac{1}{2}}h^\vee).$$
Due to our choice of pinning $\lambda +\rho$ satisfies Conditions (A)
and (C), but it is not at all clear that it satisfies Condition (B).
In order to justify this we first note that ${\rm
VA}\big(I(\frac{1}{2}h^\vee)\big)=\overline{\O(e)}$ thanks to
[\cite{BV}, Proposition~5.10]. So
Condition (C) would certainly hold for $\lambda+\rho=s_6s_2s_4(\frac{1}{2}h^\vee)$
if we manage to show that the primitive ideals $I(\frac{1}{2}h^\vee-\rho)$
and $I(s_6s_2s_4\centerdot(\frac{1}{2}h^\vee-\rho))$ coincide (as usual, $w\centerdot\mu=w(\mu+\rho)-\rho$ indicates the dot action.)

Here and in what follows we are going to apply [\cite{Ja},
Corollar~10.10]; so let us adopt the notation used there. We write
$P^{+}(\Phi)$ for the set of of all $\mu\in\t^*$ such that
$\langle\mu+\rho,\alpha^\vee\rangle$ is a non-negative integer for
all $\alpha\in\Pi$. Given $\mu\in\t^*$ we write $\Pi_\mu^0$ for the
set of all $\alpha\in\Pi$  such that
$\langle\mu+\rho,\alpha^\vee\rangle=0$. For a primitive ideal $I$ of
$U(\g)$ we denote by $d(U(\g)/I)$ the Gelfand--Kirillov dimension of
the primitive quotient $U(\g)/I$.

Let $W=\langle
s_\alpha\,|\,\alpha\in \Phi\rangle$ be the Weyl group of $\g$ and
set $w:=s_6s_2s_4s_2s_3s_5s_7$ (we always use Bourbaki's numbering).
Given $x\in W$ we write $\tau(x)$ for the $\tau$-invariant of $x$,
the set of all $\alpha\in\Pi$ for which $x(\alpha)\in -\Phi^+$. Let
$\mathbf{a}\colon\,W\rightarrow \Z^{\geqslant 0}$ be Lusztig's
$\mathbf{a}$-function; see [\cite{Chen}, p.~808] for more detail. It plays an important role in the representation theory of Iwahori--Hecke
algebras, but is, in general, quite difficult to compute. In [\cite{Ge}],
Meinolf Geck describes an algorithm, named {\sf PyCox}, which enables one
among other things to determine some values of the
$\mathbf{a}$-function. As he has kindly informed the author, running the
algorithm on $w$ produced the output $\mathbf{a}(w)=4$.

One can also obtain this result by analysing Table~0 below (we refer to Subsection~\ref{3.12} for the unexplained notation related to Kazhdan--Lusztig cells).
The first and the last columns are straightforward to fill in, whilst the first line in the middle column is true by the definition of the relation ``$\,-\,$''.
Then the second line in the middle column is true because $s_2s_4s_6s_2s_3s_5s_7$,  $s_4s_6s_2s_3s_5s_7$ are in $\mathcal{D}_L(s_3,s_4)$ and  $s_3s_2s_4s_6s_2s_3s_5s_7={^*}(s_2s_4s_6s_2s_3s_5s_7)$, $s_6s_2s_3s_5s_7={^*}(s_4s_6s_2s_3s_5s_7)$.
Then the third line in the middle column is true because $s_3s_2s_4s_6s_2s_3s_5s_7$,  $s_6s_2s_3s_3s_5s_7$ are in $\mathcal{D}_L(s_5,s_6)$ and
$s_5s_3s_2s_4s_6s_2s_3s_5s_7={^*}(s_3s_2s_4s_6s_2s_3s_5s_7)$,
$s_2s_3s_5s_7={^*}(s_6s_2s_3s_5s_7)$.

\smallskip

\begin{center}
\begin{tabular}{p{2.5 truecm} | p{4truecm} @{--\!\!--} p{4truecm} | p{2.5truecm}}
\noalign {\hrule height1.2pt}
\hfil $\mathcal{L}(x)$ & \hfil $x$ & \hfil $y$ & \hfil $\mathcal{L}(y)$ \\ \noalign {\hrule height1.2pt}
\hfil $\{s_2,s_4,s_6 \}$ & \hfil $s_2s_4s_6s_2s_3s_5s_7$ & \hfil $s_4s_6s_2s_3s_5s_7$ & \hfil $\{ s_4,s_6 \}$ \\ \hline
\hfil $\{ s_2,s_3,s_6 \}$ & \hfil $s_3s_2s_4s_6s_2s_3s_5s_7$ & \hfil $s_6s_2s_3s_5s_7$ & \hfil $\{ s_2,s_3,s_6 \}$ \\ \hline
\hfil $\{ s_2,s_3,s_5 \}$ & \hfil $s_5s_3s_2s_4s_6s_2s_3s_5s_7$ & \hfil $s_2s_3s_5s_7$ & \hfil $\{ s_2,s_3,s_5,s_7 \}$ \\
\noalign {\hrule height1.2pt}
\end{tabular}
\end{center}
\centerline{\small Table 0}

\smallskip

\noindent
Combining this with the information given in the first and the last columns  we get
\begin{eqnarray*}
s_2s_3s_5s_7& \leqslant_L & s_5s_3s_2s_4s_6s_2s_3s_5s_7
\leqslant_L s_3s_2s_4s_6s_2s_3s_5s_7\leqslant_L s_2s_4s_6s_2s_3s_5s_7
\\
&\leqslant_L& s_4s_6s_2s_3s_5s_7
\leqslant_L
s_6s_2s_3s_5s_7\leqslant_L s_2s_3s_5s_7.
\end{eqnarray*}
This yields $w\sim_L s_2s_3s_5s_7$. In particular, $\mathbf{a}(w)=4$.

The element $w$ enters the picture because
$$\lambda+\rho=\textstyle{{1\atop{}}{1\atop{}}{2\atop (-1)}
{(-1)\atop{}}{1\atop{}}{1\atop{}}{1\atop{}}}=s_6s_2s_4\big(\textstyle{{1\atop{}}{0\atop{}}{1\atop
0}
{0\atop{}}{1\atop{}}{0\atop{}}{1\atop{}}}\big)=s_6s_2s_4s_2s_3s_5s_7\big(\textstyle{{1\atop{}}{0\atop{}}{1\atop
0} {0\atop{}}{1\atop{}}{0\atop{}}{1\atop{}}}
\big)=w({\scriptstyle\frac{1}{2}}h^\vee)$$ and the roots
$w(\alpha_2)$, $w(\alpha_3)$, $w(\alpha_5)$, $w(\alpha_7)$ are
negative with respect to $\Pi$. Since
$\Pi_{-\rho+h^\vee/2}^0=\{\alpha_2,\alpha_3,\alpha_5,\alpha_7\}$ we
have the inclusion $\Pi_{-\rho+h^\vee/2}^0\subseteq \tau(w)$. In view
of [\cite{Ja}, Corollar~10.10(a)] this means that
$$d\big(U(\g)/I(s_6s_2s_4\centerdot(\textstyle{\frac{1}{2}}h^\vee-\rho))\big)\,=\,
d\big(U(\g)/I(w\centerdot(\textstyle{\frac{1}{2}}h^\vee-\rho))\big)\,=
\,d\big(U(\g)/I(w\centerdot\mu)\big)$$
for any regular $\mu\in P^{+}(\Phi)$.

According to a result of Yu Chen, the set $W^{(4)}=\{x\in
W\,|\,\,\mathbf{a}(x)=4\}$ is a single double cell of $W$ and, moreover, its
dual double cell $w_0W^{(4)}$ is associated with the nilpotent orbit
labelled ${\sf A_2+A_1}$; see [\cite{Chen}, 6.1]. As $\mathbf{a}(w)=4$ we
deduce that
$$d(U(\g)/I(w\centerdot \mu))=\dim\,\overline{\O(e)}$$ (the duality
between double cells in $W$ and the related data for the
corresponding nilpotent orbits in $\g$ can be found in [\cite{Chen},
Table~I]). But then $$\dim{\rm VA}(I(\lambda))=\dim{\rm
VA}(I(\textstyle{\frac{1}{2}}h^\vee-\rho))=\dim\overline{\O(e)}$$ thanks to [\cite{BV},
Proposition~5.10]. As $e\in {\rm VA}(I(\lambda))$ by our earlier
remarks, we conclude that ${\rm VA}(I(\lambda))=\overline{\O(e)}$.
So Condition~(B) holds for $\lambda+\rho$.

Since Condition~(D) is vacuous in the present case, applying
[\cite{Lo2}, 5.3] and [\cite{PT}, Proposition~11] we are now able to
deduce that $I(\lambda)$ is the only multiplicity-free primitive
ideal in $\mathcal{X}_{\O(e)}$ (one should also keep in mind here
that $[\g_e,\g_e]=\g_e$ by [\cite{dG}]). Since
${\scriptstyle\frac{1}{2}}h^\vee$ is a dominant weight,
$I(\frac{1}{2}h^\vee-\rho)$ is the unique maximal element in the partially
ordered set $\{I(x\centerdot(\frac{1}{2}h^\vee-\rho))\,|\,\,x\in W\}$. As a
consequence, $I(\lambda)\subseteq I(\frac{1}{2}h^\vee-\rho)$. Since both
primitive ideals have the same associated variety, the equality must
hold by [\cite{BK}, Corollar~3.6], i.e.
$I(\lambda)=I(\frac{1}{2}h^\vee-\rho)$.
\subsection{Type $({\sf E_8, A_2+2A_1})$}\label{3.6}
In this case $e$ is a special nilpotent element and $e^\vee$ has
type ${\sf E_8(b_4)}$, so that
$$h^\vee=\textstyle{{2\atop{}}{0\atop{}}{2\atop 0}
{0\atop{}}{0\atop{}}{2\atop{}}{2\atop{}}}.$$ Keeping the shape of
$h^\vee$ in mind we choose the pinning for $e$ as follows:
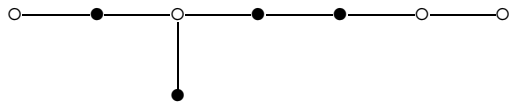
\[\xymatrix{*{\circ}\ar@{-}[r]& *{\bullet} \ar@{-}[r] & *{\circ} \ar@{-}[r]\ar@{-}[d] & *{\bullet} \ar@{-}[r]
 & *{\bullet} \ar@{-}[r]&*{\circ}\ar@{-}[r]&*{\circ}\\ && *{\bullet}&&&}\]
Then we may assume that the optimal cocharacter for $e$ has the form
$$\tau=\textstyle{{(-1)\atop{}}{2\atop{}}{(-4)\atop 2}
{2\atop{}}{2\atop{}}{(-2)\atop{}}{0\atop{}}}.$$ Since
$\dim\,\g_e=102$ the total number of $0$-roots and $1$-roots is
$(102-8)/2=47$. The roots are given in the table below.
\begin{table}[htb]
\label{data2}
\begin{tabular}{|c|c|}
\hline$0$-roots  & $1$-roots
\\ \hline
$\scriptstyle{{0\atop{}}{0\atop{}}{0\atop0}{0\atop{}}{0\atop{}}{0\atop{}}{1\atop{}}}$,
$\scriptstyle{{0\atop{}}{1\atop{}}{1\atop1}{0\atop{}}{0\atop{}}{0\atop{}}{0\atop{}}}$,
$\scriptstyle{{0\atop{}}{1\atop{}}{1\atop0}{1\atop{}}{0\atop{}}{0\atop{}}{0\atop{}}}$,
$\scriptstyle{{0\atop{}}{0\atop{}}{1\atop1}{1\atop{}}{0\atop{}}{0\atop{}}{0\atop{}}}$,\,
&$\scriptstyle{{1\atop{}}{1\atop{}}{0\atop0}{0\atop{}}{0\atop{}}{0\atop{}}{0\atop{}}}$,
$\scriptstyle{{1\atop{}}{1\atop{}}{1\atop1}{1\atop{}}{0\atop{}}{0\atop{}}{0\atop{}}}$,
$\scriptstyle{{1\atop{}}{1\atop{}}{1\atop0}{1\atop{}}{1\atop{}}{0\atop{}}{0\atop{}}}$,
$\scriptstyle{{1\atop{}}{1\atop{}}{1\atop1}{1\atop{}}{1\atop{}}{1\atop{}}{0\atop{}}}$,\\
$\scriptstyle{{0\atop{}}{0\atop{}}{1\atop0}{1\atop{}}{1\atop{}}{0\atop{}}{0\atop{}}}$,
$\scriptstyle{{0\atop{}}{0\atop{}}{0\atop0}{0\atop{}}{1\atop{}}{1\atop{}}{0\atop{}}}$,
$\scriptstyle{{0\atop{}}{0\atop{}}{0\atop0}{0\atop{}}{1\atop{}}{1\atop{}}{1\atop{}}}$,
$\scriptstyle{{0\atop{}}{0\atop{}}{1\atop1}{1\atop{}}{1\atop{}}{1\atop{}}{0\atop{}}}$,
&$\scriptstyle{{1\atop{}}{1\atop{}}{1\atop1}{1\atop{}}{1\atop{}}{1\atop{}}{1\atop{}}}$,
$\scriptstyle{{1\atop{}}{1\atop{}}{2\atop1}{2\atop{}}{1\atop{}}{0\atop{}}{0\atop{}}}$,
$\scriptstyle{{1\atop{}}{2\atop{}}{2\atop1}{1\atop{}}{1\atop{}}{0\atop{}}{0\atop{}}}$,
$\scriptstyle{{1\atop{}}{2\atop{}}{2\atop1}{2\atop{}}{1\atop{}}{1\atop{}}{0\atop{}}}$,\\
$\scriptstyle{{0\atop{}}{0\atop{}}{1\atop1}{1\atop{}}{1\atop{}}{1\atop{}}{1\atop{}}}$,
$\scriptstyle{{0\atop{}}{1\atop{}}{1\atop0}{1\atop{}}{1\atop{}}{1\atop{}}{0\atop{}}}$,
$\scriptstyle{{0\atop{}}{1\atop{}}{1\atop0}{1\atop{}}{1\atop{}}{1\atop{}}{1\atop{}}}$,
$\scriptstyle{{0\atop{}}{1\atop{}}{2\atop1}{1\atop{}}{1\atop{}}{0\atop{}}{0\atop{}}}$,
&$\scriptstyle{{1\atop{}}{2\atop{}}{2\atop1}{2\atop{}}{1\atop{}}{1\atop{}}{1\atop{}}}$,
$\scriptstyle{{1\atop{}}{1\atop{}}{2\atop1}{2\atop{}}{2\atop{}}{1\atop{}}{0\atop{}}}$,
$\scriptstyle{{1\atop{}}{1\atop{}}{2\atop1}{2\atop{}}{2\atop{}}{1\atop{}}{1\atop{}}}$,
$\scriptstyle{{1\atop{}}{2\atop{}}{2\atop1}{2\atop{}}{2\atop{}}{2\atop{}}{1\atop{}}}$,\\
$\scriptstyle{{0\atop{}}{1\atop{}}{2\atop1}{2\atop{}}{1\atop{}}{1\atop{}}{0\atop{}}}$,
$\scriptstyle{{0\atop{}}{1\atop{}}{2\atop1}{2\atop{}}{1\atop{}}{1\atop{}}{1\atop{}}}$,
$\scriptstyle{{0\atop{}}{1\atop{}}{2\atop1}{2\atop{}}{2\atop{}}{2\atop{}}{1\atop{}}}$,
$\scriptstyle{{2\atop{}}{3\atop{}}{4\atop2}{3\atop{}}{2\atop{}}{1\atop{}}{0\atop{}}}$,
&$\scriptstyle{{1\atop{}}{2\atop{}}{3\atop2}{2\atop{}}{1\atop{}}{0\atop{}}{0\atop{}}}$,
$\scriptstyle{{1\atop{}}{2\atop{}}{3\atop2}{2\atop{}}{2\atop{}}{1\atop{}}{0\atop{}}}$,
$\scriptstyle{{1\atop{}}{2\atop{}}{3\atop2}{2\atop{}}{2\atop{}}{1\atop{}}{1\atop{}}}$,
$\scriptstyle{{1\atop{}}{2\atop{}}{3\atop1}{3\atop{}}{2\atop{}}{1\atop{}}{1\atop{}}}$,\\
$\scriptstyle{{2\atop{}}{3\atop{}}{4\atop2}{3\atop{}}{2\atop{}}{1\atop{}}{1\atop{}}}$,
$\scriptstyle{{2\atop{}}{3\atop{}}{4\atop2}{3\atop{}}{3\atop{}}{2\atop{}}{1\atop{}}}$,
$\scriptstyle{{2\atop{}}{3\atop{}}{5\atop3}{4\atop{}}{3\atop{}}{2\atop{}}{1\atop{}}}$,
$\scriptstyle{{2\atop{}}{4\atop{}}{5\atop2}{4\atop{}}{3\atop{}}{2\atop{}}{1\atop{}}}$,
&$\scriptstyle{{1\atop{}}{2\atop{}}{3\atop1}{3\atop{}}{3\atop{}}{2\atop{}}{1\atop{}}}$,
$\scriptstyle{{1\atop{}}{2\atop{}}{3\atop1}{3\atop{}}{2\atop{}}{1\atop{}}{0\atop{}}}$,
$\scriptstyle{{1\atop{}}{2\atop{}}{4\atop2}{4\atop{}}{3\atop{}}{2\atop{}}{1\atop{}}}$,
$\scriptstyle{{1\atop{}}{3\atop{}}{4\atop2}{3\atop{}}{2\atop{}}{1\atop{}}{0\atop{}}}$,\\
$\scriptstyle{{2\atop{}}{4\atop{}}{6\atop3}{5\atop{}}{3\atop{}}{2\atop{}}{1\atop{}}}$,
$\scriptstyle{{2\atop{}}{4\atop{}}{6\atop3}{5\atop{}}{4\atop{}}{3\atop{}}{1\atop{}}}$,
$\scriptstyle{{2\atop{}}{4\atop{}}{6\atop3}{5\atop{}}{4\atop{}}{3\atop{}}{2\atop{}}}$.
&$\scriptstyle{{1\atop{}}{3\atop{}}{4\atop2}{3\atop{}}{2\atop{}}{1\atop{}}{1\atop{}}}$,
$\scriptstyle{{1\atop{}}{3\atop{}}{4\atop2}{3\atop{}}{3\atop{}}{2\atop{}}{1\atop{}}}$,
$\scriptstyle{{1\atop{}}{2\atop{}}{3\atop2}{3\atop{}}{2\atop{}}{2\atop{}}{1\atop{}}}$,
$\scriptstyle{{1\atop{}}{3\atop{}}{5\atop3}{4\atop{}}{3\atop{}}{2\atop{}}{1\atop{}}}$.\\
\hline
\end{tabular}
\end{table}

Direct verification shows that the  $(1,4,7,8)$-contributions of the
$0$-roots and $1$-roots are $(16,56,26,14)$ and $(24,60,24,12)$,
respectively. So the total contribution of all roots is
$(20,58,25,13)$. In view of [\cite{Bo}, Planche~VII] Condition~(C)
for $\lambda+\rho=\sum_{i=1}^8a_i\varpi_i$ reads
\begin{eqnarray*}
4a_1+5a_2+7a_3+10a_4+8a_5+6a_6+4a_7+2a_8&=&20\\
10a_1+15a_2+20a_3+30a_4+24a_5+18a_6+12a_7+6a_8&=&58\\
4a_1+6a_2+8a_3+12a_4+10a_5+8a_6+6a_7+3a_8&=&25\\
2a_1+3a_2+4a_3+6a_4+5a_5+4a_6+3a_7+2a_8&=&13.
\end{eqnarray*}
A fairly nice solution to this system of linear equations is given
by
$$\lambda+\rho=\textstyle{{1\atop{}}{0\atop{}}{1\atop 0}
{0\atop{}}{0\atop{}}{1\atop{}}{1\atop{}}}={\scriptstyle\frac{1}{2}}h^\vee.$$
Due to our choice of pinning it satisfies Losev's conditions (A) and
(C). Condition~(D) is vacuous as $e$ has standard Levi type and
Condition~(B) holds thanks to [\cite{BV}, Proposition~5.10]. In view
of [\cite{Lo2}, 5.3], [\cite{PT}, Proposition~11] and [\cite{dG}] we
now may conclude that $I(\lambda)$ is the unique multiplicity-free
primitive ideal in $\mathcal{X}_{\O(e)}$.
\subsection{Type $({\sf E_8, A_2+3A_1})$}\label{3.7}
Here our pinning for $e$ is
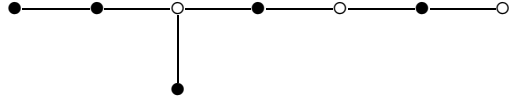
\[\xymatrix{*{\bullet}\ar@{-}[r]& *{\bullet} \ar@{-}[r] & *{\circ} \ar@{-}[r]\ar@{-}[d] & *{\bullet} \ar@{-}[r]
 & *{\circ} \ar@{-}[r]&*{\bullet}\ar@{-}[r]&*{\circ}\\ && *{\bullet}&&&}\]
and the optimal cocharacter is as follows:
$$\tau=\textstyle{{2\atop{}}{2\atop{}}{(-4)\atop 2}
{2\atop{}}{(-2)\atop{}}{2\atop{}}{(-1)\atop{}}}.$$ Since
$\dim\,\g_e=94$, the total number of positive $0$-roots and
$1$-roots is $(94-8)/2=43$. The roots are given in the table below.
\begin{table}[htb]
\label{data2}
\begin{tabular}{|c|c|}
\hline$0$-roots  & $1$-roots
\\ \hline
$\scriptstyle{{0\atop{}}{0\atop{}}{0\atop0}{1\atop{}}{1\atop{}}{0\atop{}}{0\atop{}}}$,
$\scriptstyle{{0\atop{}}{0\atop{}}{0\atop0}{0\atop{}}{1\atop{}}{1\atop{}}{0\atop{}}}$,
$\scriptstyle{{1\atop{}}{1\atop{}}{1\atop0}{0\atop{}}{0\atop{}}{0\atop{}}{0\atop{}}}$,
$\scriptstyle{{0\atop{}}{1\atop{}}{1\atop0}{1\atop{}}{0\atop{}}{0\atop{}}{0\atop{}}}$,\,
&$\scriptstyle{{0\atop{}}{0\atop{}}{0\atop0}{0\atop{}}{0\atop{}}{1\atop{}}{1\atop{}}}$,
$\scriptstyle{{0\atop{}}{0\atop{}}{0\atop0}{1\atop{}}{1\atop{}}{1\atop{}}{1\atop{}}}$,
$\scriptstyle{{1\atop{}}{1\atop{}}{1\atop0}{1\atop{}}{1\atop{}}{1\atop{}}{1\atop{}}}$,
$\scriptstyle{{0\atop{}}{1\atop{}}{1\atop1}{1\atop{}}{1\atop{}}{1\atop{}}{1\atop{}}}$,\\
$\scriptstyle{{0\atop{}}{1\atop{}}{1\atop1}{0\atop{}}{0\atop{}}{0\atop{}}{0\atop{}}}$,
$\scriptstyle{{0\atop{}}{0\atop{}}{1\atop1}{1\atop{}}{0\atop{}}{0\atop{}}{0\atop{}}}$,
$\scriptstyle{{0\atop{}}{1\atop{}}{1\atop0}{1\atop{}}{1\atop{}}{1\atop{}}{0\atop{}}}$,
$\scriptstyle{{1\atop{}}{1\atop{}}{1\atop0}{1\atop{}}{1\atop{}}{0\atop{}}{0\atop{}}}$,
&$\scriptstyle{{1\atop{}}{1\atop{}}{2\atop1}{2\atop{}}{1\atop{}}{1\atop{}}{1\atop{}}}$,
$\scriptstyle{{1\atop{}}{2\atop{}}{2\atop1}{1\atop{}}{1\atop{}}{1\atop{}}{1\atop{}}}$,
$\scriptstyle{{1\atop{}}{2\atop{}}{2\atop1}{2\atop{}}{2\atop{}}{1\atop{}}{1\atop{}}}$,
$\scriptstyle{{1\atop{}}{1\atop{}}{2\atop1}{2\atop{}}{2\atop{}}{2\atop{}}{1\atop{}}}$,\\
$\scriptstyle{{0\atop{}}{1\atop{}}{1\atop1}{1\atop{}}{1\atop{}}{0\atop{}}{0\atop{}}}$,
$\scriptstyle{{0\atop{}}{0\atop{}}{1\atop1}{1\atop{}}{1\atop{}}{1\atop{}}{0\atop{}}}$,
$\scriptstyle{{1\atop{}}{1\atop{}}{2\atop1}{1\atop{}}{0\atop{}}{0\atop{}}{0\atop{}}}$,
$\scriptstyle{{1\atop{}}{2\atop{}}{2\atop1}{1\atop{}}{1\atop{}}{0\atop{}}{0\atop{}}}$,
&$\scriptstyle{{1\atop{}}{2\atop{}}{3\atop2}{2\atop{}}{1\atop{}}{1\atop{}}{1\atop{}}}$,
$\scriptstyle{{1\atop{}}{2\atop{}}{3\atop2}{2\atop{}}{2\atop{}}{2\atop{}}{1\atop{}}}$,
$\scriptstyle{{1\atop{}}{2\atop{}}{3\atop1}{3\atop{}}{2\atop{}}{2\atop{}}{1\atop{}}}$,
$\scriptstyle{{1\atop{}}{2\atop{}}{3\atop2}{3\atop{}}{2\atop{}}{1\atop{}}{1\atop{}}}$,\\
$\scriptstyle{{1\atop{}}{1\atop{}}{2\atop1}{2\atop{}}{1\atop{}}{0\atop{}}{0\atop{}}}$,
$\scriptstyle{{0\atop{}}{1\atop{}}{2\atop1}{2\atop{}}{1\atop{}}{1\atop{}}{0\atop{}}}$,
$\scriptstyle{{1\atop{}}{1\atop{}}{2\atop1}{2\atop{}}{2\atop{}}{1\atop{}}{0\atop{}}}$,
$\scriptstyle{{1\atop{}}{2\atop{}}{3\atop2}{2\atop{}}{1\atop{}}{0\atop{}}{0\atop{}}}$,
&$\scriptstyle{{1\atop{}}{2\atop{}}{3\atop2}{3\atop{}}{3\atop{}}{2\atop{}}{1\atop{}}}$,
$\scriptstyle{{2\atop{}}{3\atop{}}{4\atop2}{3\atop{}}{2\atop{}}{1\atop{}}{1\atop{}}}$,
$\scriptstyle{{1\atop{}}{3\atop{}}{4\atop2}{3\atop{}}{2\atop{}}{2\atop{}}{1\atop{}}}$,
$\scriptstyle{{2\atop{}}{4\atop{}}{5\atop2}{4\atop{}}{3\atop{}}{2\atop{}}{1\atop{}}}$,\\
$\scriptstyle{{1\atop{}}{2\atop{}}{3\atop1}{2\atop{}}{1\atop{}}{1\atop{}}{0\atop{}}}$,
$\scriptstyle{{1\atop{}}{2\atop{}}{3\atop2}{2\atop{}}{2\atop{}}{1\atop{}}{0\atop{}}}$,
$\scriptstyle{{1\atop{}}{2\atop{}}{3\atop1}{3\atop{}}{2\atop{}}{1\atop{}}{0\atop{}}}$,
$\scriptstyle{{1\atop{}}{3\atop{}}{4\atop2}{3\atop{}}{2\atop{}}{1\atop{}}{0\atop{}}}$,
&$\scriptstyle{{2\atop{}}{3\atop{}}{5\atop3}{4\atop{}}{3\atop{}}{2\atop{}}{1\atop{}}}$,
$\scriptstyle{{1\atop{}}{3\atop{}}{4\atop2}{4\atop{}}{3\atop{}}{2\atop{}}{1\atop{}}}$,
$\scriptstyle{{2\atop{}}{3\atop{}}{4\atop2}{3\atop{}}{3\atop{}}{2\atop{}}{1\atop{}}}$,
$\scriptstyle{{2\atop{}}{4\atop{}}{6\atop3}{5\atop{}}{4\atop{}}{3\atop{}}{1\atop{}}}$,\\
$\scriptstyle{{2\atop{}}{4\atop{}}{6\atop3}{5\atop{}}{4\atop{}}{3\atop{}}{2\atop{}}}$,
$\scriptstyle{{1\atop{}}{1\atop{}}{2\atop1}{1\atop{}}{1\atop{}}{1\atop{}}{0\atop{}}}$.
&$\scriptstyle{{2\atop{}}{4\atop{}}{6\atop3}{5\atop{}}{3\atop{}}{2\atop{}}{1\atop{}}}$.\\
\hline
\end{tabular}
\end{table}

The $(4,6,8)$-contributions of the $0$-roots and $1$-roots are
$(42,24,2)$ and $(63,42,21)$, respectively, and the total
contribution of all roots is
$\big(\frac{105}{2},\frac{66}{2},\frac{23}{2}\big)$. In the present
case the orbit $\O(e)$ is non-special, and computations with orbit
sizes (which we omit) indicate that one should seek $\lambda+\rho$
such that all $a_i$ are half-integers. Setting
$\widetilde{a}_i:=2a_i$ we can rewrite Condition~(C) as follows:
\begin{eqnarray*}
10\widetilde{a}_1+15\widetilde{a}_2+20\widetilde{a}_3+30\widetilde{a}_4+
24\widetilde{a}_5+18\widetilde{a}_6+12\widetilde{a}_7+6\widetilde{a}_8&=&105\\
6\widetilde{a}_1+9\widetilde{a}_2+12\widetilde{a}_3+18\widetilde{a}_4+
15\widetilde{a}_5+12\widetilde{a}_6+8\widetilde{a}_7+4\widetilde{a}_8&=&66\\
2\widetilde{a}_1+3\widetilde{a}_2+4\widetilde{a}_3+6\widetilde{a}_4+
5\widetilde{a}_5+4\widetilde{a}_6+3\widetilde{a}_7+2\widetilde{a}_8&=&23.
\end{eqnarray*}
A fairly nice solution to this system of linear equations is given
by setting $\widetilde{a_1}=-2$ and $\widetilde{a}_i=1$ for $2\le
i\le 8$. This leads to the weight
$\lambda+\rho=\frac{1}{2}\rho-\frac{3}{2}\varpi_1$ (one can easily
spot even nicer solutions to that system, but unfortunately they do
not lead to weights satisfying Losev's conditions (A) and (B)).

Note that our choice of $\lambda+\rho$ does satisfy Conditions~(A)
whilst Condition~(D) is vacuous as $e$ once again has standard Levi
type. In order to check that $\lambda+\rho$ satisfies Condition~(B)
we first observe that its integral root system
$\Phi_\lambda=\Phi_{\lambda+\rho}$ contains the positive roots
$$
\xymatrix{{\beta_1}\ar@{-}[r]& {\beta_2} \ar@{-}[r] & {\beta_3}
\ar@{-}[r] & {\beta_4} \ar@{-}[r]
 & {\beta_5} \ar@{-}[r]&{\beta_6}\ar@{-}[r]\ar@{-}[d] &{\beta_7}\\ &&&&& {\beta_8}}
$$
where
\begin{eqnarray*}
\beta_1\,=\,{\scriptstyle{0\atop{}}{0\atop{}}{1\atop0}{1\atop{}}{0\atop{}}{0\atop{}}{0\atop{}}},\
\
\beta_2\,=\,{\scriptstyle{0\atop{}}{0\atop{}}{0\atop0}{0\atop{}}{1\atop{}}{1\atop{}}{0\atop{}}},\
\
\beta_3\,=\,{\scriptstyle{0\atop{}}{1\atop{}}{1\atop1}{1\atop{}}{0\atop{}}{0\atop{}}{0\atop{}}},\
\
\beta_4\,=\,{\scriptstyle{1\atop{}}{0\atop{}}{0\atop0}{0\atop{}}{0\atop{}}{0\atop{}}{0\atop{}}},\\
\beta_5\,=\,{\scriptstyle{0\atop{}}{1\atop{}}{1\atop0}{0\atop{}}{0\atop{}}{0\atop{}}{0\atop{}}},\
\
\beta_6\,=\,{\scriptstyle{0\atop{}}{0\atop{}}{0\atop0}{1\atop{}}{1\atop{}}{0\atop{}}{0\atop{}}},\
\
\beta_7\,=\,{\scriptstyle{0\atop{}}{0\atop{}}{0\atop0}{0\atop{}}{0\atop{}}{1\atop{}}{1\atop{}}},\
\
\beta_8\,=\,{\scriptstyle{0\atop{}}{1\atop{}}{1\atop0}{0\atop{}}{0\atop{}}{0\atop{}}{0\atop{}}},
\end{eqnarray*}
and hence has type ${\sf D_8}$ by the maximality of that root
subsystem. Furthermore, $\Pi_\lambda=\{\beta_i\,|\,\,1\le i\le 8\}$
is the basis of the root system $\Phi_{\lambda}$ contained in
$\Phi^+$. Let $\{\varpi'_i\,|\,\,1\le i\le 8\}\subset
P(\Phi)_{\mathbb Q}:=P(\Phi)\otimes_\Z{\mathbb Q}$ denote the
corresponding system of fundamental weights, so that
$\langle\varpi_i',\beta_j^\vee\rangle=\delta_{ij}$ for all $1\le i,j\le
8$. Then
$$\lambda+\rho\,=-\varpi'_4+2\varpi'_3+\textstyle{\sum}_{i\ne,3,4}\,\varpi'_i\,=\,
s_{\beta_4}s_{\beta_5}\big(\textstyle{\sum}_{i\ne
5}\,\varpi'_i\big).$$

Following [\cite{Ja}, Kapitel~2] we put $\Lambda:=\lambda+P(\Phi)$,
a subset of $P(\Phi)_{\mathbb Q}$, and denote by $\Lambda^+$ the set
of all $\nu\in\Lambda$ such that $\langle
\nu+\rho,\beta^\vee\rangle\ge 0$ for all $\beta\in\Pi_\lambda$.
Given a weight $\nu\in \Lambda$ we write $\Pi_\nu^0$ for the set of
all $\beta\in\Pi_\lambda$ such that
$\langle\nu+\rho,\beta^\vee\rangle=0$. For an element $x$ in the
integral Weyl group $W_\lambda=W(\Phi_\lambda)$ of $\lambda$ we let
$\tau_\Lambda(x)$ be the set of all $\beta\in\Pi_\lambda$ for which
$x(\beta)\in -\Phi^+$. In the present case we have that
$\lambda=s_{\beta_4}s_{\beta_5}\centerdot\mu$ where
$\mu=s_{\beta_5}s_{\beta_4}(\lambda+\rho)-\rho\in\Lambda^+$.
Besides, $\Pi_{\mu}^0=\{\beta_5\}\subset
\tau_\Lambda(s_{\beta_4}s_{\beta_5})$. In this situation [\cite{Ja},
Corollar~10.10(a)] yields that
$$d\big(U(\g)/I(\lambda)\big)=d\big(U(\g)/I(s_{\beta_4}s_{\beta_5}\centerdot\mu)\big)=
d\big(U(\g)/I(s_{\beta_4}s_{\beta_5}\centerdot\nu)\big)$$ for any
regular $\nu\in\Lambda^+$. Since all roots in $\Phi_\lambda$ have
the same length it follows from [\cite{Ja}, Corollar~10.10(c)] that
$$d\big(U(\g)/I(s_{\beta_4}s_{\beta_5}\centerdot\nu)\big)=
d\big(U(\g)/I(s_{\beta_4}\centerdot\nu)\big)=d\big(U(\g)/I(s_{\beta}\centerdot\nu)\big)$$
for any $\beta\in\Pi_\lambda$.

We let $\g(\lambda)$ be the Lie subalgebra of maximal rank in $\g$
with root system $\Phi_\lambda$ (since all roots in $\Phi$ have the
same length we may identify $\g$ with its Langlands dual Lie algebra
$\g^\vee$). Applying [\cite{Lo2}, Proposition~5.3.2] we now deduce
that
$$d\big(U(\g)/I(\lambda)\big)=d\big(U(\g)/I(s_\beta\centerdot\nu)\big)=
\dim\g-\dim\g(\lambda)+\dim\O_\lambda(w_0s_\beta)$$ where $w_0$ is
the longest element of $W_\lambda$ and $\O_\lambda(w_0s_\beta)$ is
the special nilpotent orbit in $\g(\lambda)$ attached to the double
cell of $W_\lambda$ containing $w_0s_\beta$. Since the letter is the
minimal nonzero nilpotent orbit of $\g(\lambda)$ and $\g(\lambda)$
has type ${\sf D_8}$ we finally obtain $$\dim {\rm
VA}(I(\lambda))=248-120+26=154=248-94=\dim \O(e).$$ As
$\lambda+\rho$ satisfies Condition~A we have that $e\in {\rm
VA}(I(\lambda))$. Then ${\rm VA}(I(\lambda))=\overline{\O(e)}$ and
we see that $\lambda+\rho$ satisfies all four Losev's conditions.
Since in the present case $\g_e=[\g_e,\g_e]$ by [\cite{dG}] we
conclude that $I(-\frac{1}{2}\rho-\frac{3}{2}\varpi_1)$ is the only
multiplicity-free primitive ideal in $\mathcal{X}_{\O(e)}$.
\subsection{Type $({\sf E_8, 2A_2+A_1})$}\label{3.8}
In this case we choose the following pinning for $e$:
\[\xymatrix{*{\bullet}\ar@{-}[r]& *{\bullet} \ar@{-}[r] & *{\circ} \ar@{-}[r]\ar@{-}[d] & *{\bullet} \ar@{-}[r]
 & *{\bullet} \ar@{-}[r]&*{\circ}\ar@{-}[r]&*{\circ}\\ && *{\bullet}&&&}\]
Then we can take
$$\tau=\textstyle{{2\atop{}}{2\atop{}}{(-5)\atop 2}
{2\atop{}}{2\atop{}}{(-2)\atop{}}{0\atop{}}}$$ as an optimal
cocharacter. Since $\dim\,\g_e=86$, the total number of positive
$0$-roots and $1$-roots is $(86-8)/2=39$. The roots are displayed below.

\begin{table}[htb]
\label{data2}
\begin{tabular}{|c|c|}
\hline$0$-roots  & $1$-roots
\\ \hline
$\scriptstyle{{0\atop{}}{0\atop{}}{0\atop0}{0\atop{}}{1\atop{}}{1\atop{}}{0\atop{}}}$,
$\scriptstyle{{0\atop{}}{0\atop{}}{0\atop0}{0\atop{}}{1\atop{}}{1\atop{}}{1\atop{}}}$,
$\scriptstyle{{0\atop{}}{0\atop{}}{0\atop0}{0\atop{}}{0\atop{}}{0\atop{}}{1\atop{}}}$,
$\scriptstyle{{1\atop{}}{1\atop{}}{2\atop1}{1\atop{}}{1\atop{}}{0\atop{}}{0\atop{}}}$,\,
&$\scriptstyle{{1\atop{}}{1\atop{}}{1\atop1}{0\atop{}}{0\atop{}}{0\atop{}}{0\atop{}}}$,
$\scriptstyle{{0\atop{}}{1\atop{}}{1\atop1}{1\atop{}}{0\atop{}}{0\atop{}}{0\atop{}}}$,
$\scriptstyle{{0\atop{}}{0\atop{}}{1\atop1}{1\atop{}}{1\atop{}}{0\atop{}}{0\atop{}}}$,
$\scriptstyle{{1\atop{}}{1\atop{}}{1\atop0}{1\atop{}}{0\atop{}}{0\atop{}}{0\atop{}}}$,\\
$\scriptstyle{{1\atop{}}{2\atop{}}{2\atop1}{1\atop{}}{0\atop{}}{0\atop{}}{0\atop{}}}$,
$\scriptstyle{{0\atop{}}{1\atop{}}{2\atop1}{2\atop{}}{1\atop{}}{0\atop{}}{0\atop{}}}$,
$\scriptstyle{{1\atop{}}{1\atop{}}{2\atop1}{2\atop{}}{1\atop{}}{1\atop{}}{0\atop{}}}$,
$\scriptstyle{{1\atop{}}{1\atop{}}{2\atop1}{2\atop{}}{1\atop{}}{1\atop{}}{1\atop{}}}$,
&$\scriptstyle{{0\atop{}}{1\atop{}}{1\atop0}{1\atop{}}{1\atop{}}{0\atop{}}{0\atop{}}}$,
$\scriptstyle{{1\atop{}}{1\atop{}}{1\atop0}{1\atop{}}{1\atop{}}{1\atop{}}{0\atop{}}}$,
$\scriptstyle{{1\atop{}}{1\atop{}}{1\atop0}{1\atop{}}{1\atop{}}{1\atop{}}{1\atop{}}}$,
$\scriptstyle{{0\atop{}}{1\atop{}}{1\atop1}{1\atop{}}{1\atop{}}{1\atop{}}{0\atop{}}}$,\\
$\scriptstyle{{1\atop{}}{2\atop{}}{2\atop1}{1\atop{}}{1\atop{}}{1\atop{}}{0\atop{}}}$,
$\scriptstyle{{1\atop{}}{2\atop{}}{2\atop1}{1\atop{}}{1\atop{}}{1\atop{}}{1\atop{}}}$,
$\scriptstyle{{0\atop{}}{1\atop{}}{2\atop1}{2\atop{}}{2\atop{}}{1\atop{}}{0\atop{}}}$,
$\scriptstyle{{0\atop{}}{1\atop{}}{2\atop1}{2\atop{}}{2\atop{}}{1\atop{}}{1\atop{}}}$,
&$\scriptstyle{{0\atop{}}{1\atop{}}{1\atop1}{1\atop{}}{1\atop{}}{1\atop{}}{1\atop{}}}$,
$\scriptstyle{{1\atop{}}{2\atop{}}{3\atop2}{2\atop{}}{2\atop{}}{1\atop{}}{0\atop{}}}$,
$\scriptstyle{{1\atop{}}{2\atop{}}{3\atop2}{2\atop{}}{2\atop{}}{1\atop{}}{1\atop{}}}$,
$\scriptstyle{{1\atop{}}{2\atop{}}{3\atop1}{3\atop{}}{2\atop{}}{1\atop{}}{0\atop{}}}$,\\
$\scriptstyle{{1\atop{}}{1\atop{}}{2\atop1}{2\atop{}}{2\atop{}}{2\atop{}}{1\atop{}}}$,
$\scriptstyle{{1\atop{}}{3\atop{}}{4\atop2}{3\atop{}}{2\atop{}}{1\atop{}}{0\atop{}}}$,
$\scriptstyle{{1\atop{}}{3\atop{}}{4\atop2}{3\atop{}}{2\atop{}}{1\atop{}}{1\atop{}}}$,
$\scriptstyle{{2\atop{}}{3\atop{}}{4\atop2}{3\atop{}}{2\atop{}}{2\atop{}}{1\atop{}}}$,
&$\scriptstyle{{1\atop{}}{2\atop{}}{3\atop1}{3\atop{}}{2\atop{}}{1\atop{}}{1\atop{}}}$,
$\scriptstyle{{1\atop{}}{2\atop{}}{3\atop2}{3\atop{}}{2\atop{}}{2\atop{}}{1\atop{}}}$,
$\scriptstyle{{1\atop{}}{2\atop{}}{3\atop1}{3\atop{}}{3\atop{}}{2\atop{}}{1\atop{}}}$,
$\scriptstyle{{2\atop{}}{3\atop{}}{5\atop3}{4\atop{}}{3\atop{}}{2\atop{}}{1\atop{}}}$,\\
$\scriptstyle{{1\atop{}}{3\atop{}}{4\atop2}{3\atop{}}{3\atop{}}{2\atop{}}{1\atop{}}}$,
$\scriptstyle{{2\atop{}}{4\atop{}}{6\atop3}{4\atop{}}{3\atop{}}{2\atop{}}{1\atop{}}}$,
$\scriptstyle{{2\atop{}}{4\atop{}}{6\atop3}{5\atop{}}{4\atop{}}{3\atop{}}{1\atop{}}}$,
$\scriptstyle{{2\atop{}}{4\atop{}}{6\atop3}{5\atop{}}{4\atop{}}{3\atop{}}{2\atop{}}}$,
&$\scriptstyle{{2\atop{}}{4\atop{}}{5\atop2}{4\atop{}}{3\atop{}}{2\atop{}}{1\atop{}}}$,
$\scriptstyle{{1\atop{}}{2\atop{}}{3\atop2}{2\atop{}}{1\atop{}}{0\atop{}}{0\atop{}}}$.\\
$\scriptstyle{{1\atop{}}{2\atop{}}{4\atop2}{4\atop{}}{3\atop{}}{2\atop{}}{1\atop{}}}$.
&\\
\hline
\end{tabular}
\end{table}
The $(4,7,8)$-contributions of the $0$-roots and $1$-roots are
$(58,26,14)$ and $(40,16,8)$, respectively, and the total
contribution of all roots is $(49,21,11)$. In the present case the
orbit $\O(e)$ is non-special and comparing the size of $\O(e)$ with
the size of a root subsystem of type ${\sf E_6+A_2}$ in $\Phi$
indicates that one should seek
$\lambda+\rho=\sum_{i=1}^8a_i\varpi_i$ such that $a_i\in
\frac{1}{3}\Z$ for all $1\le i\le 8$. Setting
$\widetilde{a}_i:=3a_i$ we can rewrite Condition~(C) for
$\lambda+\rho$ as follows:
\begin{eqnarray*}
10\widetilde{a}_1+15\widetilde{a}_2+20\widetilde{a}_3+30\widetilde{a}_4+
24\widetilde{a}_5+18\widetilde{a}_6+12\widetilde{a}_7+6\widetilde{a}_8&=&147\\
4\widetilde{a}_1+6\widetilde{a}_2+8\widetilde{a}_3+12\widetilde{a}_4+
10\widetilde{a}_5+8\widetilde{a}_6+6\widetilde{a}_7+3\widetilde{a}_8&=&63\\
2\widetilde{a}_1+3\widetilde{a}_2+4\widetilde{a}_3+6\widetilde{a}_4+
5\widetilde{a}_5+4\widetilde{a}_6+3\widetilde{a}_7+2\widetilde{a}_8&=&33.
\end{eqnarray*}
A very nice solution to this system of linear
equations is given by setting $\widetilde{a}_i=1$ for $1\le i\le 7$ and
$\widetilde{a}_8=3$. Due to our choice of pinning this solution
satisfies Condition~(A). Since Condition~(D) is vacuous in the
present case we just need to check that it satisfies Condition~(B).

In order to do so we observe that the integral root system
$\Phi_{\lambda+\rho}$ of our weight
$\lambda+\rho=\textstyle{\frac{1}{3}}(\rho+2\varpi_8)$ contains the
positive roots
$$
\xymatrix{{\beta_1}\ar@{-}[r]& {\beta_3} \ar@{-}[r] & {\beta_4}
\ar@{-}[r]\ar@{-}[d] & {\beta_5} \ar@{-}[r]
 & {\beta_6} &{\beta_7}\ar@{-}[r] &{\beta_8}\\ && {\beta_2}&&&}
$$
where
\begin{eqnarray*}
\beta_1\,=\,{\scriptstyle{0\atop{}}{0\atop{}}{1\atop1}{1\atop{}}{0\atop{}}{0\atop{}}{0\atop{}}},\
\
\beta_2\,=\,{\scriptstyle{0\atop{}}{0\atop{}}{0\atop0}{0\atop{}}{0\atop{}}{0\atop{}}{1\atop{}}},\
\
\beta_3\,=\,{\scriptstyle{1\atop{}}{1\atop{}}{1\atop0}{0\atop{}}{0\atop{}}{0\atop{}}{0\atop{}}},\
\
\beta_4\,=\,{\scriptstyle{0\atop{}}{0\atop{}}{0\atop0}{1\atop{}}{1\atop{}}{1\atop{}}{0\atop{}}},\\
\beta_5\,=\,{\scriptstyle{0\atop{}}{1\atop{}}{1\atop1}{0\atop{}}{0\atop{}}{0\atop{}}{0\atop{}}},\
\
\beta_6\,=\,{\scriptstyle{0\atop{}}{0\atop{}}{1\atop0}{1\atop{}}{1\atop{}}{0\atop{}}{0\atop{}}},\
\
\beta_7\,=\,{\scriptstyle{0\atop{}}{0\atop{}}{1\atop0}{1\atop{}}{1\atop{}}{0\atop{}}{0\atop{}}},\
\
\beta_8\,=\,{\scriptstyle{1\atop{}}{1\atop{}}{1\atop1}{1\atop{}}{1\atop{}}{0\atop{}}{0\atop{}}},
\end{eqnarray*}
and hence has type ${\sf E_6+A_2}$ (by the maximality of that root
subsystem). Furthermore, the roots $\{\beta_i\,|\,\,1\le i\le 8\}$
form the basis of $\Phi_{\lambda}$ contained in $\Phi^+$ and it is
straightforward to see that $\lambda+\rho$ is strongly dominant on
$\Phi^+\cap\Phi_\lambda$. Since
$$|\Phi^+|-|\Phi_\lambda^+|=120-36-3=81=\textstyle{\frac{1}{2}}(248-86)=\frac{1}{2}\dim\,\O(e),$$
applying [\cite{Jo1}, Corollary~3.5] yields $\dim\,{\rm
VA}(I(\lambda))=\dim\,\O(e)$. Since $\lambda+\rho$ satisfies
Condition~(A) we now deduce that $e\in{\rm VA}(I(\lambda))$.
Consequently, ${\rm VA}(I(\lambda))=\overline{\O(e)}$ i.e.
Condition~(B) holds for $\lambda+\rho$.  Since in the present case
$\g_e=[\g_e,\g_e]$ [\cite{dG}, we conclude, thanks to [\cite{Lo2},
5.3] and [\cite{PT}, Proposition~11], that
$I(-\frac{2}{3}\rho+\frac{2}{3}\varpi_8)$ is the unique
multiplicity-free primitive ideal in $\mathcal{X}_{\O(e)}$.
\subsection{Type $({\sf E_8, A_3+A_1})$}\label{3.9}
Here $e$ is non-special and our pinning is
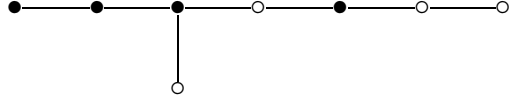
\[\xymatrix{*{\bullet}\ar@{-}[r]& *{\bullet} \ar@{-}[r] & *{\bullet} \ar@{-}[r]\ar@{-}[d] & *{\circ} \ar@{-}[r]
 & *{\bullet} \ar@{-}[r]&*{\circ}\ar@{-}[r]&*{\circ}\\ && *{\circ}&&&}\]
Then the optimal cocharacter for $e$ can be chosen as follows:
$$\tau=\textstyle{{2\atop{}}{2\atop{}}{2\atop (-3)}
{(-4)\atop{}}{2\atop{}}{(-1)\atop{}}{0\atop{}}}.$$

Serious complications that we have encountered in this case are mainly due to the
fact that $\g_e={\mathbb C}e\oplus [\g_e,\g_e]$; see [\cite{dG}].
In this situation one may expect $U(\g,e)$ to possess at least $2$
one-dimensional representation. In fact, [\cite{GRU}] says that this
is indeed the case and there are exactly two of them.
Our task in this subsection will be to determine the corresponding primitive ideals of $U(\g)$
in their Duflo realisations.
Since $\dim\,\g_e=84$, the total number of positive $0$-roots and
$1$-roots is $(84-8)/2=38$. The roots are given below.
\begin{table}[htb]
\label{data2}
\begin{tabular}{|c|c|}
\hline$0$-roots  & $1$-roots
\\ \hline
$\scriptstyle{{0\atop{}}{1\atop{}}{1\atop0}{1\atop{}}{0\atop{}}{0\atop{}}{0\atop{}}}$,
$\scriptstyle{{0\atop{}}{0\atop{}}{1\atop0}{1\atop{}}{1\atop{}}{0\atop{}}{0\atop{}}}$,
$\scriptstyle{{0\atop{}}{0\atop{}}{0\atop0}{0\atop{}}{0\atop{}}{0\atop{}}{1\atop{}}}$,
$\scriptstyle{{1\atop{}}{1\atop{}}{1\atop1}{1\atop{}}{1\atop{}}{1\atop{}}{0\atop{}}}$,\,
&$\scriptstyle{{0\atop{}}{1\atop{}}{1\atop1}{0\atop{}}{0\atop{}}{0\atop{}}{0\atop{}}}$,
$\scriptstyle{{0\atop{}}{0\atop{}}{0\atop0}{0\atop{}}{1\atop{}}{1\atop{}}{0\atop{}}}$,
$\scriptstyle{{0\atop{}}{0\atop{}}{0\atop0}{0\atop{}}{1\atop{}}{1\atop{}}{1\atop{}}}$,
$\scriptstyle{{1\atop{}}{1\atop{}}{1\atop1}{1\atop{}}{1\atop{}}{0\atop{}}{0\atop{}}}$,\\
$\scriptstyle{{1\atop{}}{1\atop{}}{1\atop1}{1\atop{}}{1\atop{}}{1\atop{}}{1\atop{}}}$,
$\scriptstyle{{0\atop{}}{1\atop{}}{2\atop1}{1\atop{}}{1\atop{}}{1\atop{}}{0\atop{}}}$,
$\scriptstyle{{0\atop{}}{1\atop{}}{2\atop1}{1\atop{}}{1\atop{}}{1\atop{}}{1\atop{}}}$,
$\scriptstyle{{1\atop{}}{2\atop{}}{2\atop1}{2\atop{}}{1\atop{}}{1\atop{}}{1\atop{}}}$,
&$\scriptstyle{{0\atop{}}{1\atop{}}{2\atop1}{1\atop{}}{1\atop{}}{0\atop{}}{0\atop{}}}$,
$\scriptstyle{{1\atop{}}{1\atop{}}{2\atop1}{1\atop{}}{0\atop{}}{0\atop{}}{0\atop{}}}$,
$\scriptstyle{{0\atop{}}{1\atop{}}{1\atop0}{1\atop{}}{1\atop{}}{1\atop{}}{0\atop{}}}$,
$\scriptstyle{{0\atop{}}{1\atop{}}{1\atop0}{1\atop{}}{1\atop{}}{1\atop{}}{1\atop{}}}$,\\
$\scriptstyle{{1\atop{}}{2\atop{}}{2\atop1}{2\atop{}}{1\atop{}}{1\atop{}}{0\atop{}}}$,
$\scriptstyle{{1\atop{}}{1\atop{}}{2\atop1}{2\atop{}}{2\atop{}}{1\atop{}}{0\atop{}}}$,
$\scriptstyle{{1\atop{}}{1\atop{}}{2\atop1}{2\atop{}}{2\atop{}}{1\atop{}}{1\atop{}}}$,
$\scriptstyle{{1\atop{}}{2\atop{}}{3\atop2}{2\atop{}}{1\atop{}}{0\atop{}}{0\atop{}}}$,
&$\scriptstyle{{1\atop{}}{2\atop{}}{2\atop1}{2\atop{}}{1\atop{}}{0\atop{}}{0\atop{}}}$,
$\scriptstyle{{1\atop{}}{2\atop{}}{2\atop1}{2\atop{}}{2\atop{}}{2\atop{}}{1\atop{}}}$,
$\scriptstyle{{1\atop{}}{2\atop{}}{3\atop2}{2\atop{}}{2\atop{}}{1\atop{}}{0\atop{}}}$,
$\scriptstyle{{1\atop{}}{2\atop{}}{3\atop2}{3\atop{}}{2\atop{}}{1\atop{}}{1\atop{}}}$,\\
$\scriptstyle{{1\atop{}}{2\atop{}}{3\atop2}{2\atop{}}{2\atop{}}{2\atop{}}{1\atop{}}}$,
$\scriptstyle{{1\atop{}}{2\atop{}}{3\atop1}{3\atop{}}{2\atop{}}{1\atop{}}{0\atop{}}}$,
$\scriptstyle{{1\atop{}}{2\atop{}}{3\atop1}{3\atop{}}{2\atop{}}{1\atop{}}{1\atop{}}}$,
$\scriptstyle{{1\atop{}}{3\atop{}}{4\atop2}{3\atop{}}{2\atop{}}{2\atop{}}{1\atop{}}}$,
&$\scriptstyle{{1\atop{}}{2\atop{}}{3\atop1}{3\atop{}}{3\atop{}}{2\atop{}}{1\atop{}}}$,
$\scriptstyle{{1\atop{}}{3\atop{}}{4\atop2}{3\atop{}}{2\atop{}}{1\atop{}}{0\atop{}}}$,
$\scriptstyle{{1\atop{}}{3\atop{}}{4\atop2}{3\atop{}}{2\atop{}}{1\atop{}}{1\atop{}}}$,
$\scriptstyle{{2\atop{}}{4\atop{}}{5\atop3}{4\atop{}}{3\atop{}}{2\atop{}}{1\atop{}}}$,\\
$\scriptstyle{{1\atop{}}{2\atop{}}{4\atop2}{3\atop{}}{3\atop{}}{2\atop{}}{1\atop{}}}$,
$\scriptstyle{{1\atop{}}{3\atop{}}{5\atop2}{4\atop{}}{3\atop{}}{2\atop{}}{1\atop{}}}$,
$\scriptstyle{{2\atop{}}{4\atop{}}{6\atop3}{5\atop{}}{4\atop{}}{3\atop{}}{1\atop{}}}$,
$\scriptstyle{{2\atop{}}{4\atop{}}{6\atop3}{5\atop{}}{4\atop{}}{3\atop{}}{2\atop{}}}$,
&$\scriptstyle{{2\atop{}}{4\atop{}}{6\atop3}{5\atop{}}{4\atop{}}{2\atop{}}{1\atop{}}}$.\\
$\scriptstyle{{2\atop{}}{3\atop{}}{4\atop2}{4\atop{}}{3\atop{}}{2\atop{}}{1\atop{}}}$.&
\\
\hline
\end{tabular}
\end{table}
The $(2,5,7,8)$-contributions of the $0$-roots and $1$-roots are
$(28,48,26,14)$ and $(21,31,16,6)$, respectively, and the total
contribution of all roots is
$\big(\frac{49}{2},\frac{79}{2},\frac{42}{2}, \frac{22}{2}\big)$. It
seems reasonable to seek $\lambda+\rho=\sum_{i=1}^8a_i\varpi_i$ for which all $a_i$ are half-integers. Numerology of orbit sizes indicates that the integral root system of $\lambda$
should have type ${\sf E_7+A_1}$ and the nilpotent orbit in the Lie
algebra $\g(\lambda)$ (as defined in the previous subsection) should
have type ${\sf 2A_1+0}$. Indeed, the nilpotent orbit of type ${\sf 2A_1}$
in a Lie algebra of type ${\sf E_7}$ is special of dimension $52$
and
$$\dim\O(e)=248-84=164=112+52+0=(248-133-3)+52+0,
$$ which is clearly in line with [\cite{Lo2}, Proposition~5.3.2].
Since all other options a priori available to us turned out to be
dead ends, this is actually our only hope.

Note that due to our choice of pinning Condition~(A) can be restated
by saying that
$\langle\lambda+\rho,\gamma^\vee\rangle\not\in\Z^{>0}$ for any
$\gamma\in\Phi^+$ which is a linear combination of $\alpha_i$ with
$i\in\{1,3,4,6\}$.

Setting $\widetilde{a}_i:=2a_i$ we can rewrite Condition~(C) as
follows:
\begin{eqnarray*}
5\widetilde{a}_1+8\widetilde{a}_2+10\widetilde{a}_3+15\widetilde{a}_4+
12\widetilde{a}_5+9\widetilde{a}_6+6\widetilde{a}_7+3\widetilde{a}_8&=&49\\
8\widetilde{a}_1+12\widetilde{a}_2+16\widetilde{a}_3+24\widetilde{a}_4+
20\widetilde{a}_5+15\widetilde{a}_6+10\widetilde{a}_7+5\widetilde{a}_8&=&79\\
4\widetilde{a}_1+6\widetilde{a}_2+8\widetilde{a}_3+12\widetilde{a}_4+
10\widetilde{a}_5+8\widetilde{a}_6+6\widetilde{a}_7+3\widetilde{a}_8&=&42\\
2\widetilde{a}_1+3\widetilde{a}_2+4\widetilde{a}_3+6\widetilde{a}_4+
5\widetilde{a}_5+4\widetilde{a}_6+3\widetilde{a}_7+2\widetilde{a}_8&=&22.
\end{eqnarray*}
A not entirely random solution to this system of linear equations is
given by putting $\widetilde{a_i}=-1$ for $i\in \{1,4\}$,
$\widetilde{a}_i=1$ for $i\in\{\,6,7\}$, $\widetilde{a}_i=2$ for
$i\in\{5,8\}$, $\widetilde{a_3}=0$ and $\widetilde{a}_2=3$. This
leads to the weight
$$\lambda+\rho=\textstyle{\frac{1}{2}}(-\varpi_1+3\varpi_2-\varpi_4+
2\varpi_5+\varpi_6+\varpi_7+2\varpi_8)$$ which satisfies
Conditions~(A), whilst Condition~(D) is again vacuous. It is
straightforward to see that $\Phi_\lambda$ contains the positive
roots
$$
\xymatrix{{\beta_1}\ar@{-}[r]& {\beta_3} \ar@{-}[r] & {\beta_4}
\ar@{-}[r]\ar@{-}[d] & {\beta_5} \ar@{-}[r]
 & {\beta_6}\ar@{-}[r] &{\beta_7} &{\beta_8}\\ && {\beta_2}&&&}
$$
where
\begin{eqnarray*}
\beta_1\,=\,{\scriptstyle{0\atop{}}{0\atop{}}{0\atop0}{0\atop{}}{0\atop{}}{0\atop{}}{1\atop{}}},\
\
\beta_2\,=\,{\scriptstyle{1\atop{}}{1\atop{}}{1\atop0}{0\atop{}}{0\atop{}}{0\atop{}}{0\atop{}}},\
\
\beta_3\,=\,{\scriptstyle{0\atop{}}{0\atop{}}{0\atop0}{0\atop{}}{1\atop{}}{1\atop{}}{0\atop{}}},\
\
\beta_4\,=\,{\scriptstyle{0\atop{}}{0\atop{}}{0\atop0}{1\atop{}}{0\atop{}}{0\atop{}}{0\atop{}}},\\
\beta_5\,=\,{\scriptstyle{0\atop{}}{0\atop{}}{1\atop1}{0\atop{}}{0\atop{}}{0\atop{}}{0\atop{}}},\
\
\beta_6\,=\,{\scriptstyle{0\atop{}}{1\atop{}}{0\atop0}{0\atop{}}{0\atop{}}{0\atop{}}{0\atop{}}},\
\
\beta_7\,=\,{\scriptstyle{0\atop{}}{0\atop{}}{1\atop0}{1\atop{}}{1\atop{}}{0\atop{}}{0\atop{}}},\
\
\beta_8\,=\,{\scriptstyle{1\atop{}}{1\atop{}}{1\atop1}{1\atop{}}{1\atop{}}{0\atop{}}{0\atop{}}},
\end{eqnarray*}
and hence has type ${\sf E_7+A_1}$ (by the maximality of that root
subsystem). Moreover, the roots $\{\beta_i\,|\,\,1\le i\le 8\}$ form
the basis of $\Phi_{\lambda}$ contained in $\Phi^+$. Let
$\{\varpi'_i\,|\,\,1\le i\le 8\}$ denote the corresponding system of
fundamental weights, so that
$\langle\varpi_i',\beta_j^\vee\rangle=\delta_{ij}$ for all $1\le i,j\le
8$. Then
$$\lambda+\rho\,=-\varpi'_2+2\varpi'_8+\textstyle{\sum}_{i\ne,2,6,8}\,\varpi'_i\,=\,
s_{\beta_2}s_{\beta_4}s_{\beta_6}\big(2\varpi_8+\textstyle{\sum}_{i\ne
4,6,8}\,\varpi'_i\big).$$ So
$\lambda=s_{\beta_2}s_{\beta_4}s_{\beta_6}\centerdot\mu$ where
$\mu=s_{\beta_6}s_{\beta_4}s_{\beta_2}(\lambda+\rho)-\rho\in\Lambda^+$.
Also, $$\Pi_{\mu}^0=\{\beta_4,\beta_6\}\subset
\tau_\Lambda(s_{\beta_2}s_{\beta_4}s_{\beta_6}).$$ Applying
[\cite{Ja}, Corollar~10.10(a)] we now get
$$d\big(U(\g)/I(\lambda)\big)=d\big(U(\g)/I(s_{\beta_2}s_{\beta_4}s_{\beta_6}\centerdot\mu)\big)=
d\big(U(\g)/I(s_{\beta_2}s_{\beta_4}s_{\beta_6}\centerdot\nu)\big)$$
for any regular $\nu\in\Lambda^+$, whilst [\cite{Ja},
Corollar~10.10(c)] yields
$$d\big(U(\g)/I(s_{\beta_2}s_{\beta_4}s_{\beta_6}\centerdot\nu)\big)=
d\big(U(\g)/I(s_{\beta_6}s_{\beta_4}s_{\beta_2}\centerdot\nu)\big)=
d\big(U(\g)/I(s_{\beta_6}s_{\beta_4}\centerdot\nu)\big).$$ Since $\g(\lambda)$ is a Lie algebra of type
of type ${\sf E_7+A_1}$, applying [\cite{Lo2},
Proposition~5.3.2] now entails that
$$d\big(U(\g)/I(\lambda)\big)=d\big(U(\g)/I(s_{\beta_6}s_{\beta_4}\centerdot\nu)\big)=
112+\dim\O_\lambda(w_0s_{\beta_6}s_{\beta_4})$$
where $w_0$ is the longest element of $W_\lambda$ and
$\O_\lambda(w_0s_{\beta_6}s_{\beta_4})$ is the special nilpotent
orbit in $\g(\lambda)$ attached to the double cell of $W_\lambda$
containing $w_0s_{\beta_6}s_{\beta_4}$.

At this point we can invoke the description of double cells in the
Weyl group $W({\sf E_7})$ obtained by Yu Chen and Jian-Yi Shi in
[\cite{CS}]. According to Table~4.2.1 in {\it loc.\,cit.} the set
$W^{(2)}=\{w\in W\,|\,\,\mathbf{a}(w)=2\}$ is a single double cell of $W$
whose dual double cell $w_0W^{(2)}$ is associated with the nilpotent
orbit of type ${\sf 2A_1}$ (here $w_0$ stands for the longest
element of $W({\sf E_7}$). Since the roots $\beta_4$ and $\beta_6$
are orthogonal to each other, we have that
$\mathbf{a}(s_{\beta_6}s_{\beta_4})=2$. From this it is immediate that the
orbit $\O_\lambda(w_0s_{\beta_6}s_{\beta_4})$ has type ${\sf
2A_1+0}$ in $\g(\lambda)$. In conjunction with our remarks earlier
in this subsection this shows that $\dim {\rm VA}(I(\lambda))=\dim
\O(e).$ Since $\lambda+\rho$ satisfies Condition~A we have that
$e\in {\rm VA}(I(\lambda))$. Hence ${\rm
VA}(I(\lambda))=\overline{\O(e)}$ and we finally deduce that
$\lambda+\rho$ satisfies all four Losev's conditions. As a result,
$I(-\frac{3}{2}\varpi_1+\frac{1}{2}\varpi_2-\varpi_3-\frac{3}{2}\varpi_4-\frac{1}{2}\varpi_6-\frac{1}{2}\varpi_7)$
is a multiplicity-free primitive ideal in $\mathcal{X}_{\O(e)}$.

The second weight, $\lambda'+\rho$, leading to a 
one-dimensional
representations of $U(\g,e)$ was eventually found (after many
unsuccessful attempts) when we imposed that
\begin{itemize}
\item[(a)\ ]
$\Phi_{\lambda'}=\Phi_\lambda$;

\smallskip

\item[(b)\ ]
$\langle\lambda'+\rho,\beta_8^\vee\rangle=1$.
\end{itemize}
There was no ``scientific'' reason to impose these conditions, but
the central characters of one-dimensional representations of finite
$W$-algebras do tend to be rather small and all root subsystems of
type ${\sf E_7+A_1}$ in $\Phi$ are conjugate under the action of the
Weyl group. Condition~(C) together with (b) leads to the following system of linear equations:
\begin{eqnarray*}
\widetilde{b}_1+\widetilde{b}_2+\widetilde{b}_3+\widetilde{b}_4+
\widetilde{b}_5+\widetilde{b}_6&=&2\\
5\widetilde{b}_1+8\widetilde{b}_2+10\widetilde{b}_3+15\widetilde{b}_4+
12\widetilde{b}_5+9\widetilde{b}_6+6\widetilde{b}_7+3\widetilde{b}_8&=&49\\
8\widetilde{b}_1+12\widetilde{b}_2+16\widetilde{b}_3+24\widetilde{b}_4+
20\widetilde{b}_5+15\widetilde{b}_6+10\widetilde{b}_7+5\widetilde{b}_8&=&79\\
4\widetilde{b}_1+6\widetilde{b}_2+8\widetilde{b}_3+12\widetilde{b}_4+
10\widetilde{b}_5+8\widetilde{b}_6+6\widetilde{b}_7+3\widetilde{b}_8&=&42\\
2\widetilde{b}_1+3\widetilde{b}_2+4\widetilde{b}_3+6\widetilde{b}_4+
5\widetilde{b}_5+4\widetilde{b}_6+3\widetilde{b}_7+2\widetilde{b}_8&=&22.
\end{eqnarray*}
Here $\lambda'+\rho=\sum_{i=1}^8 b_i\varpi_i$, where all $b_i$ are
half-integers, and $\widetilde{b}_i=2b_i$ for $1\le i\le 8$. Setting
$\widetilde{b_i}=1$ for $i\in \{4,6,7\}$, $\widetilde{b}_i=2$ for
$i\in\{\,5,8\}$, $\widetilde{b}_1=-3$, $\widetilde{b_2}=3$ and
$\widetilde{b}_3=-2$ provides us with a rather ordinary looking
solution which leads to the weight
$$\lambda'+\rho=\textstyle{\frac{1}{2}}(-3\varpi_1+3\varpi_2-2\varpi_3+\varpi_4+
2\varpi_5+\varpi_6+\varpi_7+2\varpi_8).$$ This weight does satisfy
Conditions~(A) and computations show that
$$\lambda'+\rho\,=-2\varpi'_2-\varpi'_6+2\varpi'_7+\textstyle{\sum}_{i\ne,2,6,7}\,\varpi'_i\,=\,
s_{\beta_6}s_{\beta_2}s_{\beta_4}s_{\beta_3}s_{\beta_5}\big(\textstyle{\sum}_{i\ne
3,5}\,\varpi'_i\big).$$ Then
$\lambda'=s_{\beta_6}s_{\beta_2}s_{\beta_4}s_{\beta_3}s_{\beta_5}\centerdot\mu'$
where $\mu'=
s_{\beta_5}s_{\beta_3}s_{\beta_4}s_{\beta_2}s_{\beta_6}(\lambda'+\rho)-\rho\in\Lambda^+$,
and
$$\Pi_{\mu'}^0=\{\beta_3,\beta_5\}\subseteq
\tau_\Lambda(s_{\beta_6}s_{\beta_2}s_{\beta_4}s_{\beta_3}s_{\beta_5}).$$
Applying [\cite{Ja}, Corollar~10.10(a)] once again we get
$$d\big(U(\g)/I(\lambda')\big)=d\big(U(\g)/I(s_{\beta_6}s_{\beta_2}s_{\beta_4}s_{\beta_3}s_{\beta_5}\centerdot\mu')\big)=
d\big(U(\g)/I(s_{\beta_6}s_{\beta_2}s_{\beta_4}s_{\beta_3}s_{\beta_5}\centerdot\nu)\big)$$
for any regular $\nu\in\Lambda^+$. Then [\cite{Ja},
Corollar~10.10(c)] yields
\begin{eqnarray*}
d\big(U(\g)/I(\lambda')\big)&=&
d\big(U(\g)/I(s_{\beta_2}s_{\beta_4}s_{\beta_3}\underline{s_{\beta_6}s_{\beta_5}}\centerdot\nu)\big)=\,
d\big(U(\g)/I(s_{\beta_2}s_{\beta_4}s_{\beta_3}s_{\beta_6}\centerdot\nu)\big)\\
&=&d\big(U(\g)/I(s_{\beta_6}s_{\beta_2}\underline{s_{\beta_4}s_{\beta_3}}\centerdot\nu)\big)=\,
d\big(U(\g)/I(s_{\beta_6}\underline{s_{\beta_2}s_{\beta_4}}\centerdot\nu)\big)\\
&=&d\big(U(\g)/I(s_{\beta_6}s_{\beta_2}\centerdot\nu)\big)=\,\dim
\O(e)
\end{eqnarray*}
because the equality
$\mathbf{a}(s_{\beta_6}s_{\beta_2})=2$ holds in the Weyl
group $W(\Phi_\lambda)$ (here we again rely on [\cite{CS},
Table~4.2.1]). Arguing as in the previous case we are now able to
conclude that
$I(-\frac{5}{2}\varpi_1+\frac{1}{2}\varpi_2-2\varpi_3-\frac{1}{2}\varpi_4-\frac{1}{2}\varpi_6-\frac{1}{2}\varpi_7)$
is a multiplicity-free primitive ideal in $\mathcal{X}_{\O(e)}$.

If $I(\lambda)=I(\lambda')$ then $\lambda'+\rho=w(\lambda+\rho)$ for
some $w\in W$. Since $\Phi_\lambda=\Phi_{\lambda'}$ the element $w$
must preserve $\Phi_\lambda$. But then $w\in W(\Phi_\lambda)$ as the
root system of type ${\sf E_7+A_1}$ has no outer automorphisms. The
latter, however, is impossible since $\lambda+\rho$ and
$\lambda'+\rho$ take different positive values at $\beta_8^\vee$ (which generates
the component of type ${\sf A_1}$ in $\Phi_\lambda$). Therefore, the
primitive ideals $I(\lambda)$ and $I(\lambda')$ are distinct and
hence so are the respective one-dimensional representations of
$U(\g,e)$.
\subsection{Type $({\sf E_8, 2A_2+2A_1})$}\label{3.10}
Here our pinning for $e$ is
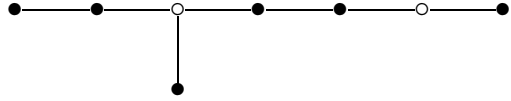
\[\xymatrix{*{\bullet}\ar@{-}[r]& *{\bullet} \ar@{-}[r] & *{\circ} \ar@{-}[r]\ar@{-}[d] & *{\bullet} \ar@{-}[r]
 & *{\bullet} \ar@{-}[r]&*{\circ}\ar@{-}[r]&*{\bullet}\\ && *{\bullet}&&&}\]
and we can take
$$\tau=\textstyle{{2\atop{}}{2\atop{}}{(-5)\atop 2}
{2\atop{}}{2\atop{}}{(-3)\atop{}}{2\atop{}}}$$ as an optimal
cocharacter. Since $\dim\,\g_e=80$, the total number of positive
$0$-roots and $1$-roots is $(80-8)/2=36$. The roots are given below.
\begin{table}[htb]
\label{data2}
\begin{tabular}{|c|c|}
\hline$0$-roots  & $1$-roots
\\ \hline
$\scriptstyle{{1\atop{}}{1\atop{}}{1\atop0}{1\atop{}}{1\atop{}}{1\atop{}}{0\atop{}}}$,
$\scriptstyle{{0\atop{}}{1\atop{}}{1\atop1}{1\atop{}}{1\atop{}}{1\atop{}}{0\atop{}}}$,
$\scriptstyle{{0\atop{}}{0\atop{}}{1\atop1}{1\atop{}}{1\atop{}}{1\atop{}}{1\atop{}}}$,
$\scriptstyle{{1\atop{}}{2\atop{}}{2\atop1}{1\atop{}}{0\atop{}}{0\atop{}}{0\atop{}}}$,\,
&$\scriptstyle{{1\atop{}}{1\atop{}}{1\atop1}{0\atop{}}{0\atop{}}{0\atop{}}{0\atop{}}}$,
$\scriptstyle{{1\atop{}}{1\atop{}}{1\atop0}{1\atop{}}{0\atop{}}{0\atop{}}{0\atop{}}}$,
$\scriptstyle{{0\atop{}}{1\atop{}}{1\atop1}{1\atop{}}{0\atop{}}{0\atop{}}{0\atop{}}}$,
$\scriptstyle{{1\atop{}}{1\atop{}}{1\atop0}{1\atop{}}{0\atop{}}{0\atop{}}{0\atop{}}}$,\\
$\scriptstyle{{0\atop{}}{1\atop{}}{2\atop1}{2\atop{}}{1\atop{}}{0\atop{}}{0\atop{}}}$,
$\scriptstyle{{1\atop{}}{1\atop{}}{2\atop1}{1\atop{}}{1\atop{}}{0\atop{}}{0\atop{}}}$,
$\scriptstyle{{1\atop{}}{1\atop{}}{2\atop1}{2\atop{}}{2\atop{}}{2\atop{}}{1\atop{}}}$,
$\scriptstyle{{1\atop{}}{2\atop{}}{3\atop2}{2\atop{}}{1\atop{}}{1\atop{}}{1\atop{}}}$,
&$\scriptstyle{{0\atop{}}{0\atop{}}{0\atop0}{0\atop{}}{1\atop{}}{1\atop{}}{1\atop{}}}$,
$\scriptstyle{{0\atop{}}{0\atop{}}{1\atop1}{1\atop{}}{1\atop{}}{0\atop{}}{0\atop{}}}$,
$\scriptstyle{{0\atop{}}{1\atop{}}{1\atop0}{1\atop{}}{1\atop{}}{0\atop{}}{0\atop{}}}$,
$\scriptstyle{{1\atop{}}{1\atop{}}{2\atop1}{2\atop{}}{1\atop{}}{1\atop{}}{1\atop{}}}$,\\
$\scriptstyle{{1\atop{}}{2\atop{}}{3\atop1}{2\atop{}}{2\atop{}}{1\atop{}}{1\atop{}}}$,
$\scriptstyle{{1\atop{}}{2\atop{}}{3\atop2}{2\atop{}}{2\atop{}}{1\atop{}}{0\atop{}}}$,
$\scriptstyle{{1\atop{}}{2\atop{}}{3\atop1}{3\atop{}}{2\atop{}}{1\atop{}}{0\atop{}}}$,
$\scriptstyle{{1\atop{}}{2\atop{}}{4\atop2}{4\atop{}}{3\atop{}}{2\atop{}}{1\atop{}}}$,
&$\scriptstyle{{1\atop{}}{2\atop{}}{2\atop1}{1\atop{}}{1\atop{}}{1\atop{}}{1\atop{}}}$,
$\scriptstyle{{1\atop{}}{2\atop{}}{2\atop1}{2\atop{}}{1\atop{}}{1\atop{}}{0\atop{}}}$,
$\scriptstyle{{0\atop{}}{1\atop{}}{2\atop1}{2\atop{}}{2\atop{}}{1\atop{}}{1\atop{}}}$,
$\scriptstyle{{1\atop{}}{1\atop{}}{2\atop1}{3\atop{}}{2\atop{}}{1\atop{}}{1\atop{}}}$,\\
$\scriptstyle{{1\atop{}}{3\atop{}}{4\atop2}{3\atop{}}{3\atop{}}{2\atop{}}{1\atop{}}}$,
$\scriptstyle{{2\atop{}}{3\atop{}}{4\atop2}{3\atop{}}{2\atop{}}{2\atop{}}{1\atop{}}}$,
$\scriptstyle{{2\atop{}}{4\atop{}}{6\atop3}{5\atop{}}{3\atop{}}{2\atop{}}{1\atop{}}}$,
$\scriptstyle{{0\atop{}}{1\atop{}}{1\atop0}{1\atop{}}{1\atop{}}{1\atop{}}{1\atop{}}}$.
&$\scriptstyle{{1\atop{}}{2\atop{}}{3\atop2}{2\atop{}}{1\atop{}}{0\atop{}}{0\atop{}}}$,
$\scriptstyle{{1\atop{}}{2\atop{}}{3\atop2}{3\atop{}}{2\atop{}}{2\atop{}}{1\atop{}}}$,
$\scriptstyle{{1\atop{}}{2\atop{}}{3\atop1}{3\atop{}}{3\atop{}}{2\atop{}}{1\atop{}}}$,
$\scriptstyle{{2\atop{}}{3\atop{}}{4\atop2}{3\atop{}}{2\atop{}}{1\atop{}}{0\atop{}}}$,\\
&$\scriptstyle{{1\atop{}}{3\atop{}}{4\atop2}{3\atop{}}{2\atop{}}{1\atop{}}{1\atop{}}}$,
$\scriptstyle{{2\atop{}}{4\atop{}}{5\atop2}{4\atop{}}{3\atop{}}{2\atop{}}{1\atop{}}}$,
$\scriptstyle{{2\atop{}}{3\atop{}}{5\atop3}{4\atop{}}{3\atop{}}{2\atop{}}{1\atop{}}}$,
$\scriptstyle{{2\atop{}}{4\atop{}}{6\atop3}{5\atop{}}{4\atop{}}{3\atop{}}{2\atop{}}}$.\\
\hline
\end{tabular}
\end{table}
The $(4,7)$-contributions of the $0$-roots and $1$-roots are
$(42,18)$ and $(48,20)$, respectively, and the total contribution of
all roots is $(45,19)$. The orbit $\O(e)$ is non-special in the
present case and comparing its size with the size of a root
subsystem of type ${\sf A_8}$ in $\Phi$ indicates that one should
again seek $\lambda+\rho=\sum_{i=1}^8a_i\varpi_i$ such that $a_i\in
\frac{1}{3}\Z$ for all $1\le i\le 8$. Setting
$\widetilde{a}_i:=3a_i$ we can rewrite Condition~(C) as follows:
\begin{eqnarray*}
10\widetilde{a}_1+15\widetilde{a}_2+20\widetilde{a}_3+30\widetilde{a}_4+
24\widetilde{a}_5+18\widetilde{a}_6+12\widetilde{a}_7+6\widetilde{a}_8&=&135\\
4\widetilde{a}_1+6\widetilde{a}_2+8\widetilde{a}_3+12\widetilde{a}_4+
10\widetilde{a}_5+8\widetilde{a}_6+6\widetilde{a}_7+3\widetilde{a}_8&=&57.
\end{eqnarray*}
Already at this point we can see that as our orbits get larger
Condition~(C) gets weaker and weaker. But as a counterweight
Condition~(A) gets more and more restrictive.

An exceptionally nice solution to our system  is given by
$\widetilde{a}_i=1$ for all $1\le i\le 8$. Due to the choice of
pinning it does not violate Condition~(A). Since Condition~(D) is
again vacuous it remains to show  that it satisfies Condition~(B).
To do that we first observe that the integral root system of our
weight $\lambda+\rho=\textstyle{\frac{1}{3}}\rho$ contains the
positive roots
$$
\xymatrix{{\beta_1}\ar@{-}[r]& {\beta_2} \ar@{-}[r] & {\beta_3}
\ar@{-}[r] & {\beta_4} \ar@{-}[r]
 & {\beta_5}\ar@{-}[r] &{\beta_6}\ar@{-}[r] &{\beta_7}\ar@{-}[r]& {\beta_8}}
$$
where
\begin{eqnarray*}
\beta_1\,=\,{\scriptstyle{0\atop{}}{1\atop{}}{1\atop0}{1\atop{}}{0\atop{}}{0\atop{}}{0\atop{}}},\
\
\beta_2\,=\,{\scriptstyle{0\atop{}}{0\atop{}}{0\atop0}{0\atop{}}{1\atop{}}{1\atop{}}{1\atop{}}},\
\
\beta_3\,=\,{\scriptstyle{0\atop{}}{0\atop{}}{1\atop1}{1\atop{}}{0\atop{}}{0\atop{}}{0\atop{}}},\
\
\beta_4\,=\,{\scriptstyle{1\atop{}}{1\atop{}}{1\atop0}{0\atop{}}{0\atop{}}{0\atop{}}{0\atop{}}},\\
\beta_5\,=\,{\scriptstyle{0\atop{}}{0\atop{}}{0\atop0}{1\atop{}}{1\atop{}}{1\atop{}}{0\atop{}}},\
\
\beta_6\,=\,{\scriptstyle{0\atop{}}{1\atop{}}{1\atop1}{0\atop{}}{0\atop{}}{0\atop{}}{0\atop{}}},\
\
\beta_7\,=\,{\scriptstyle{0\atop{}}{0\atop{}}{1\atop0}{1\atop{}}{1\atop{}}{0\atop{}}{0\atop{}}},\
\
\beta_8\,=\,{\scriptstyle{1\atop{}}{1\atop{}}{1\atop1}{1\atop{}}{1\atop{}}{0\atop{}}{0\atop{}}},
\end{eqnarray*}
and hence has type ${\sf A_8}$ (by the maximality of that root
subsystem). Furthermore, the roots $\{\beta_i\,|\,\,1\le i\le 8\}$
form the basis of $\Phi_{\lambda}$ contained in $\Phi^+$ and
$\lambda+\rho$ is strongly dominant on $\Phi^+\cap\Phi_\lambda$.
Since
$$|\Phi^+|-|\Phi_\lambda^+|=120-36=84=\textstyle{\frac{1}{2}}(248-80)=\frac{1}{2}\dim\,\O(e),$$
applying [\cite{Jo1}, Corollary~3.5] we obtain $\dim\,{\rm
VA}(I(\lambda))=\dim\,\O(e)$. Since $\lambda+\rho$ satisfies
Condition~(A) we get $e\in{\rm VA}(I(\lambda))$. So ${\rm
VA}(I(\lambda))=\overline{\O(e)}$, that is Condition~(B) holds for
$\lambda+\rho$.  Since in the present case $\g_e=[\g_e,\g_e]$ by
[\cite{dG}], combining [\cite{Lo2}, 5.3] and [\cite{PT},
Proposition~11] yields that $I(-\frac{2}{3}\rho)$ is the unique
multiplicity-free primitive ideal in $\mathcal{X}_{\O(e)}$.
\subsection{Type $({\sf E_8, A_3+2A_1})$}\label{3.11}
We choose the following pinning for $e$:
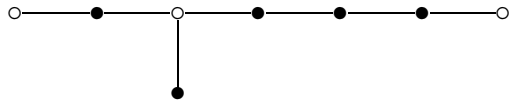
\[\xymatrix{*{\circ}\ar@{-}[r]& *{\bullet} \ar@{-}[r] & *{\circ} \ar@{-}[r]\ar@{-}[d] & *{\bullet} \ar@{-}[r]
 & *{\bullet} \ar@{-}[r]&*{\bullet}\ar@{-}[r]&*{\circ}\\ && *{\bullet}&&&}\]
and we take
$$\tau=\textstyle{{(-1)\atop{}}{2\atop{}}{(-5)\atop 2}
{2\atop{}}{2\atop{}}{2\atop{}}{(-3)\atop{}}}$$ as an optimal
cocharacter. Since $\dim\,\g_e=76$, the total number of positive
$0$-roots and $1$-roots is $(76-8)/2=34$. The roots are given below.

\begin{table}[htb]
\label{data2}
\begin{tabular}{|c|c|}
\hline$0$-roots  & $1$-roots
\\ \hline
$\scriptstyle{{1\atop{}}{1\atop{}}{1\atop1}{1\atop{}}{0\atop{}}{0\atop{}}{0\atop{}}}$,
$\scriptstyle{{1\atop{}}{1\atop{}}{1\atop0}{1\atop{}}{1\atop{}}{0\atop{}}{0\atop{}}}$,
$\scriptstyle{{0\atop{}}{1\atop{}}{1\atop0}{1\atop{}}{1\atop{}}{1\atop{}}{1\atop{}}}$,
$\scriptstyle{{0\atop{}}{0\atop{}}{1\atop1}{1\atop{}}{1\atop{}}{1\atop{}}{1\atop{}}}$,\,
&$\scriptstyle{{1\atop{}}{1\atop{}}{0\atop0}{0\atop{}}{0\atop{}}{0\atop{}}{0\atop{}}}$,
$\scriptstyle{{0\atop{}}{1\atop{}}{1\atop1}{1\atop{}}{0\atop{}}{0\atop{}}{0\atop{}}}$,
$\scriptstyle{{0\atop{}}{1\atop{}}{1\atop0}{1\atop{}}{1\atop{}}{0\atop{}}{0\atop{}}}$,
$\scriptstyle{{0\atop{}}{0\atop{}}{1\atop0}{1\atop{}}{1\atop{}}{0\atop{}}{0\atop{}}}$,\\
$\scriptstyle{{1\atop{}}{2\atop{}}{2\atop1}{2\atop{}}{1\atop{}}{1\atop{}}{1\atop{}}}$,
$\scriptstyle{{1\atop{}}{1\atop{}}{2\atop1}{2\atop{}}{2\atop{}}{1\atop{}}{1\atop{}}}$,
$\scriptstyle{{1\atop{}}{2\atop{}}{3\atop1}{2\atop{}}{2\atop{}}{1\atop{}}{0\atop{}}}$,
$\scriptstyle{{1\atop{}}{2\atop{}}{3\atop2}{2\atop{}}{1\atop{}}{1\atop{}}{0\atop{}}}$,
&$\scriptstyle{{0\atop{}}{0\atop{}}{1\atop0}{1\atop{}}{1\atop{}}{1\atop{}}{0\atop{}}}$,
$\scriptstyle{{0\atop{}}{0\atop{}}{0\atop0}{0\atop{}}{1\atop{}}{1\atop{}}{1\atop{}}}$,
$\scriptstyle{{1\atop{}}{1\atop{}}{1\atop1}{1\atop{}}{1\atop{}}{1\atop{}}{1\atop{}}}$,
$\scriptstyle{{1\atop{}}{2\atop{}}{2\atop1}{1\atop{}}{1\atop{}}{1\atop{}}{0\atop{}}}$,\\
$\scriptstyle{{2\atop{}}{3\atop{}}{4\atop2}{3\atop{}}{2\atop{}}{1\atop{}}{0\atop{}}}$,
$\scriptstyle{{2\atop{}}{3\atop{}}{5\atop3}{4\atop{}}{3\atop{}}{2\atop{}}{1\atop{}}}$,
$\scriptstyle{{2\atop{}}{4\atop{}}{5\atop2}{4\atop{}}{3\atop{}}{2\atop{}}{1\atop{}}}$,
$\scriptstyle{{1\atop{}}{2\atop{}}{4\atop2}{3\atop{}}{3\atop{}}{2\atop{}}{1\atop{}}}$,
&$\scriptstyle{{1\atop{}}{1\atop{}}{2\atop1}{2\atop{}}{1\atop{}}{1\atop{}}{0\atop{}}}$,
$\scriptstyle{{0\atop{}}{1\atop{}}{2\atop1}{2\atop{}}{2\atop{}}{1\atop{}}{1\atop{}}}$,
$\scriptstyle{{1\atop{}}{2\atop{}}{3\atop2}{3\atop{}}{2\atop{}}{1\atop{}}{1\atop{}}}$,
$\scriptstyle{{2\atop{}}{3\atop{}}{4\atop2}{3\atop{}}{3\atop{}}{2\atop{}}{1\atop{}}}$,\\
$\scriptstyle{{1\atop{}}{3\atop{}}{4\atop2}{3\atop{}}{2\atop{}}{2\atop{}}{1\atop{}}}$,
$\scriptstyle{{2\atop{}}{4\atop{}}{6\atop3}{5\atop{}}{4\atop{}}{3\atop{}}{2\atop{}}}$,
$\scriptstyle{{0\atop{}}{1\atop{}}{2\atop1}{2\atop{}}{1\atop{}}{0\atop{}}{0\atop{}}}$,
$\scriptstyle{{0\atop{}}{1\atop{}}{2\atop1}{1\atop{}}{1\atop{}}{1\atop{}}{0\atop{}}}$.
&$\scriptstyle{{1\atop{}}{3\atop{}}{5\atop3}{4\atop{}}{3\atop{}}{2\atop{}}{1\atop{}}}$,
$\scriptstyle{{1\atop{}}{3\atop{}}{4\atop2}{3\atop{}}{2\atop{}}{1\atop{}}{0\atop{}}}$,
$\scriptstyle{{2\atop{}}{4\atop{}}{6\atop3}{5\atop{}}{4\atop{}}{2\atop{}}{1\atop{}}}$,
$\scriptstyle{{1\atop{}}{2\atop{}}{2\atop1}{2\atop{}}{1\atop{}}{0\atop{}}{0\atop{}}}$,\\
&$\scriptstyle{{1\atop{}}{2\atop{}}{3\atop1}{3\atop{}}{2\atop{}}{2\atop{}}{1\atop{}}}$,
$\scriptstyle{{1\atop{}}{2\atop{}}{3\atop2}{2\atop{}}{2\atop{}}{2\atop{}}{1\atop{}}}$.
\\
\hline
\end{tabular}
\end{table}

The $(1,4,8)$-contributions of the $0$-roots and $1$-roots are $(16,
46,10)$ and $(14,41,9)$, respectively, and the total contribution of
all roots is $\big(\frac{30}{2},\frac{87}{2},\frac{19}{2}\big)$. The
orbit $\O(e)$ is non-special and looking through the options
available in the present case one comes to the conclusion that one
should seek $\lambda+\rho$ with $\Phi_\lambda$ of type ${\sf D_8}$.
In view of [\cite{Lo2}, Proposition~5.3.2] the special nilpotent
orbit in the Lie algebra $\g(\lambda)$ then must have dimension
$248-76-128= 54.$ We show below that such an orbit does exist.

Let $V$ be a $16$-dimensional vector space over $\mathbb C$ and let
$\Psi$ be a non-degenerate symmetric bilinear form on $V$. We shall
identify $\g(\lambda)$ with the Lie subalgebra of $\mathfrak{gl}(V)$
consisting of all linear transformations $A$ such
$$\Psi\big(A(v),v'\big)+\Psi\big(v,A(v')\big)=0\qquad\quad (\forall\, v,v'\in
V).$$ Then to each nilpotent element $X\in\g(\lambda)$ we can
associate a partition ${\bf p}={\bf p}(X)$ of $16$ whose parts are
the Jordan block sizes of the linear transformation $X\in
\mathfrak{gl}(V)$ arranged in the descending order. Let ${\bf
r}={\bf r}(X)=(r_1\ge\cdots\ge r_k)$ be the partition transpose
(conjugate) to ${\bf p}$. Then it is well known that
\begin{equation}\label{so}
\dim \g(\lambda)_{X}=\,\textstyle{\frac{1}{2}}\big(\dim
\mathfrak{gl}(V)_{X}-n_{\mathrm{odd}}(X)\big)=\textstyle{\frac{1}{2}}\big(\big(\textstyle\sum_{i=1}^k
r_i^2\big)-n_{\rm odd}(X)\big)
\end{equation}
where $n_{\rm odd}(X)$ is the number of odd parts of $\mathbf{p}$;
see [\cite{Ja1}, 3.1] for instance. Furthermore, it is known that
there exists a unique up to conjugacy in ${\rm SO}(\Psi)$ special
nilpotent element $X\in\g(\lambda)$ such that ${\bf
p}(X)=(2^{4},1^8)$. Then ${\bf r}(X)=(12,4)$, $n_{\rm odd}(X)=8$ and
(\ref{so}) yields $\dim \g(\lambda)_X=\frac{1}{2}(144+16-8)=76$. The
nilpotent orbit in $\g(\lambda)$ containing $X$ has dimension
$120-76=54$ which is what we need.

The above discussion indicates that we should seek
$\lambda+\rho=\sum_{i=1}^8a_i\varpi_i$ such that all $a_i$ are
half-integers. Setting $\widetilde{a}_i=2a_i$ for $1\le i\le 8$ we
can rewrite Condition~(C) as follows:
\begin{eqnarray*}
4\widetilde{a}_1+5\widetilde{a}_2+7\widetilde{a}_3+10\widetilde{a}_4+
8\widetilde{a}_5+6\widetilde{a}_6+4\widetilde{a}_7+2\widetilde{a}_8&=&30\\
10\widetilde{a}_1+15\widetilde{a}_2+20\widetilde{a}_3+30\widetilde{a}_4+
24\widetilde{a}_5+18\widetilde{a}_6+12\widetilde{a}_7+6\widetilde{a}_8&=&87\\
2\widetilde{a}_1+3\widetilde{a}_2+4\widetilde{a}_3+6\widetilde{a}_4+
5\widetilde{a}_5+4\widetilde{a}_6+3\widetilde{a}_7+2\widetilde{a}_8&=&19.
\end{eqnarray*}
Since in the present case there is no hope to find a correct
solution by a lucky guess, we shall try a more roundabout approach.
Namely, we are going to start by selecting a root subsystem of type
${\sf D_8}$ in $\Phi$ to which we shall assign the role of
$\Phi_\lambda$. Next we shall rewrite our system of
linear equations in terms of the fundamental weights of that
subsystem. Then a correct solution will eventually unveil itself.

We consider the following linearly independent positive roots in
$\Phi$:
\begin{eqnarray*}
\beta_1\,=\,{\scriptstyle{0\atop{}}{1\atop{}}{1\atop0}{0\atop{}}{0\atop{}}{0\atop{}}{0\atop{}}},\
\
\beta_2\,=\,{\scriptstyle{0\atop{}}{0\atop{}}{0\atop0}{1\atop{}}{1\atop{}}{1\atop{}}{0\atop{}}},\
\
\beta_3\,=\,{\scriptstyle{0\atop{}}{0\atop{}}{1\atop1}{0\atop{}}{0\atop{}}{0\atop{}}{0\atop{}}},\
\
\beta_4\,=\,{\scriptstyle{1\atop{}}{1\atop{}}{0\atop0}{0\atop{}}{0\atop{}}{0\atop{}}{0\atop{}}},\\
\beta_5\,=\,{\scriptstyle{0\atop{}}{0\atop{}}{1\atop0}{1\atop{}}{0\atop{}}{0\atop{}}{0\atop{}}},\
\
\beta_6\,=\,{\scriptstyle{0\atop{}}{0\atop{}}{0\atop0}{0\atop{}}{1\atop{}}{0\atop{}}{0\atop{}}},\
\
\beta_7\,=\,{\scriptstyle{0\atop{}}{1\atop{}}{1\atop1}{1\atop{}}{0\atop{}}{0\atop{}}{0\atop{}}},\
\
\beta_8\,=\,{\scriptstyle{0\atop{}}{0\atop{}}{0\atop0}{0\atop{}}{0\atop{}}{1\atop{}}{1\atop{}}}.
\end{eqnarray*}
It is straightforward to see that the standard recipe for assigning
graphs to root systems leads to the Dynkin diagram of type ${\sf
D_8}$:
$$
\xymatrix{{\beta_1}\ar@{-}[r]& {\beta_2} \ar@{-}[r] & {\beta_3}
\ar@{-}[r] & {\beta_4} \ar@{-}[r]
 & {\beta_5} \ar@{-}[r]&{\beta_6}\ar@{-}[r]\ar@{-}[d] &{\beta_7}\\ &&&&& {\beta_8}}
$$
It follows that
$\Phi_0:=\Phi\cap\big(\bigoplus_{i=1}^8\Z\beta_i\big)$ is a root
subsystem of type ${\sf D_8}$ in $\Phi$ (by the maximality of that
root subsystem) and the set $\Pi_0:=\{\beta_i\,|\,\,1\le i\le 8\}$
is the basis of simple roots in $\Phi_0$ contained in $\Phi^+$. Let
$\{\varpi'_i\,|\,\,1\le i\le 8\}\subset P(\Phi)_{\mathbb Q}$ be the
corresponding system of fundamental weights, so that $\langle
\varpi_i',\beta_i^\vee\rangle=\delta_{ij}$ for $1\le i,j\le 8$.

Looking at the explicit expressions for the $\beta_i$'s it is easy
to observe that
$$\varpi_1=\varpi'_4,\quad\ \varpi_8=\varpi'_8, \quad\
\varpi_4=\varpi_1'+\varpi_3'+\varpi_5'+\varpi_7'.$$ As $\Pi_0$ spans
$P(\Phi)_{\mathbb Q}$ we have that
$\lambda+\rho=\sum_{i=1}^8b_i\varpi_i'$ for some $b_i\in\mathbb Q$,
but since we fashion $\Phi_0$ as the integral root system of
$\lambda$ it must be that $b_i\in\Z$ for all $i$. Using [\cite{Bo},
Planche~IV] we find explicit expressions of $\varpi'_4$,
$\varpi_1'+\varpi_3'+\varpi_5'+\varpi_7'$ and $\varpi_8'$ as linear
combinations of roots in $\Pi_0$. This enables us to rewrite
Condition~(C) as follows:
\begin{eqnarray*}
b_1+2b_2+3b_3+4b_4+4b_5+4b_6+2b_7+2b_8&=&15\\
{\textstyle \frac{7}{2}}b_1+6b_2+{\textstyle\frac{17}{2}}b_3+10b_4+
{\textstyle \frac{23}{2}}b_5+12b_6+{\textstyle\frac{13}{2}}b_7+6b_8&=&{\textstyle\frac{87}{2}}\\
{\textstyle \frac{1}{2}}b_1+b_2+{\textstyle \frac{3}{2}}b_3+2b_4+
{\textstyle \frac{5}{2}}b_5+3b_6+{\textstyle
\frac{3}{2}}b_7+2b_8&=&{\textstyle \frac{19}{2}}.
\end{eqnarray*}
An almost perfect solution to this system is given
by setting $b_3=b_6=0$ and $b_i=1$ for $i\in\{1,2,4,5,7,8\}$ and
leads to the weight $\mu+\rho=\sum_{i\ne 3,6}\varpi_i'$. Although it
does not satisfy Condition~(A), this can be easily amended by
replacing it with the weight
$$\lambda+\rho=s_{\beta_2}(\mu+\rho)=\,2\varpi_1'-\varpi_2'+\varpi_3'+
\varpi_4'+\varpi_5'+\varpi_7'+\varpi_8'.$$ This provides us with
another solution because $\langle\varpi_i',\beta_2^\vee\rangle=0$
for $i\in\{1,3,4,5,7,8\}$.

Let $\g_0$ be the regular Lie subalgebra of $\g$ with root system
$\Phi_0$ and let $I_0(\lambda)$ be the annihilator in $U(\g_0)$ of
the irreducible $\g_0$-module of highest weight $\lambda$. Since
$\g_0$ has type $\sf D$ we can compute the associated variety ${\rm
VA}(I_0(\lambda))$ by using the Barbasch--Vogan algorithm; see
[\cite{BV0}]. In fact, we found it more convenient to use the
modified version of that algorithm described in [\cite{BG}, 5.3].

In Step~1 our aim is to expand $\lambda+\rho$ via the orthonormal
basis $\{\varepsilon_1,\varepsilon_2,\ldots,\varepsilon_8\}$ of
$P(\Phi_0)_{\mathbb R}$ by using the formulae for fundamental
weights from [\cite{Bo}, Planche~IV]. After doing so we get
$\lambda+\rho=
5\varepsilon_1+3\varepsilon_2+4\varepsilon_3+3\varepsilon_4+2\varepsilon_5
+\varepsilon_6+\varepsilon_7+0\varepsilon_8$.

In Step~2, we are supposed to apply the Robinson--Shensted algorithm
on the sequence of numbers
$(5,3,4,3,2,1,1,\widehat{0},0,-1,-2,-3,-4,-3,-5)$ obtained from
Step~1, where $\widehat{0}>0$ by convention. Recall that the
Robinson--Shensted algorithm inputs a tuple of real numbers, $\bf
n$, and outputs a Young tableau ${\rm  RS}({\bf n})$. The reader
familiar with the combinatorial procedure (sometimes referred to as
{\it bumping}) upon
which the algorithm is based will have no problem to
check that  the partition ${\bf q}$ associated with the
tableau ${\rm RS}(\lambda+\rho)$ equals $(2^4,1^8)$ (the $i$-th part
of ${\bf q}$ is the number of {\it rows} of  ${\rm
RS}(\lambda+\rho)$; see [\cite{BG}, 4.1]).

Next we are supposed to write the parts of ${\bf q}$ in the
ascending order, $q_1\le\cdots\le q_k$, and create the list of
numbers $(r_i)$ where $r_i=q_i-i+1$. We then have to separate the
even and odd parts of that list to obtain two sublists,
$(2s_1,\ldots,2s_l)$ and $(2t_1+1,\ldots, 2t_m+1)$, consisting of
the even and odd parts of $(r_i)$ respectively. After doing so we
obtain $(2,4,6,8,10,12)$ and $(1,3,5,7,11,13)$.

In Step~3, we are first required to form two lists, $(s_i)$ and
$(t_i)$, and concatenate them. In our case this will produce the
list $(0,1,1,2,2,3,3,4,5,5,6,6)$. Next we are to split this into two
sublists, $(s_i')$ and $(t_i')$, consisting of the terms with odd
and even indices, respectively. We then get $(s'_i)=(0,1,2,3,5,6)$
and $(t_i')=(1,2,3,4,5,6)$. Next we are required to form two new
lists, $(2s_i'+1)$ and $(2t_i')$, and concatenate them. In our case
this will produce the list $(r_i')=(1,2,3,4,5,6,7,8,10,11,12,13)$.
Finally, we put $q_i'=r'_i-i+1$ and rearrange the parts $q_i'$ in
the descending order. The resulting partition will label the open
orbit in ${\rm VA}(I_0(\lambda))$. Since in our case
$(q_i')=(1,1,1,1,1,1,1,1,2,2,2,2)$ we deduce that the open orbit of
the associated variety of  $I_0(\lambda)$ contains a nilpotent
element $X$ with ${\bf p}(X)=(2^4,1^8)$.

After expressing the $\varpi_i'$ in terms of
$\{\varpi_1,\ldots,\varpi_8\}$ one finds out that
$$\lambda+\rho={\textstyle\frac{1}{2}}\varpi_1-{\textstyle\frac{1}{2}}\varpi_2
+{\textstyle\frac{1}{2}}\varpi_3+{\textstyle\frac{3}{2}}\varpi_4-
{\textstyle\frac{1}{2}}\varpi_5-{\textstyle\frac{1}{2}}
\varpi_7+{\textstyle\frac{3}{2}}\varpi_8$$ (one can also see
directly that this weight satisfies all required linear equations).
Since $\Phi_0\subseteq\Phi_\lambda\subsetneq \Phi$, the maximality
of the root subsystem $\Phi_0$ implies that $\Phi_\lambda=\Phi_0$.
But then $\g(\lambda)=\g_0$ and combining [\cite{Lo2},
Proposition~5.3.2] with the above discussion shows that $\dim\,{\rm
VA}(I\lambda))=\dim\overline{\O(e)}$.

Due to our choice of pinning $\lambda+\rho$ satisfies Condition~(A)
which implies that ${\rm VA}(I(\lambda))=\overline{\O(e)}$. We
conclude that $\lambda+\rho$ satisfies all Losev's conditions. As
$\g_e=[\g_e,\g_e]$ by [\cite{dG}] we see that
$I({-\textstyle\frac{1}{2}}\varpi_1-{\textstyle\frac{3}{2}}\varpi_2
-{\textstyle\frac{1}{2}}\varpi_3+{\textstyle\frac{1}{2}}\varpi_4-
{\textstyle\frac{3}{2}}\varpi_5-\varpi_6-{\textstyle\frac{3}{2}}
\varpi_7+{\textstyle\frac{1}{2}}\varpi_8)$ is the only
multiplicity-free primitive ideal in ${\mathcal X}_{\O(e)}$.
\subsection{Type $({\sf E_8, D_4(a_1)+A_1})$}\label{3.12}
In this case our pinning for $e$ is
\[\xymatrix{*{\circ}\ar@{-}[r]& *{\bullet} \ar@{-}[r] & *{\bullet} \ar@{-}[r]\ar@{-}[d] & *{\bullet} \ar@{-}[r]
 & *{\circ} \ar@{-}[r]&*{\bullet}\ar@{-}[r]&*{\circ}\\ && *{\bullet}&&&}\]
and following [\cite{LT}] we assume that $e=e_0+e_{\alpha_7}$ where
$$e_0=e_{\alpha_2}+e_{\alpha_3}+e_{\alpha_5}+e_{\alpha_2+\alpha_4}+e_{\alpha_4+\alpha_5}.$$
Then, as in [\cite{LT}], we may choose
$$\tau=\textstyle{{(-4)\atop{}}{2\atop{}}{0\atop 2}
{2\atop{}}{(-5)\atop{}}{2\atop{}}{(-1)\atop{}}}$$ as an optimal
cocharacter for $e$.  Let $\l$ be the standard Levi subalgebra of
$\g$ generated by $\Lie(T)$ and simple root vectors $e_{\pm
\alpha_i}$ with $i\in\{2,3,4,5\}$. Note that the derived subalgebra
$\l'=[\l,\l]$ has type ${\sf D_4}$ and $e_0$ is a subregular
nilpotent element of $\l'$.

Since $\dim\,\g_e=72$, the total number of positive $0$-roots and
$1$-roots is $(72-8)/2=32$. The roots are given below.

\begin{table}[htb]
\label{data2}
\begin{tabular}{|c|c|}
\hline$0$-roots  & $1$-roots
\\ \hline
$\scriptstyle{{0\atop{}}{0\atop{}}{1\atop0}{0\atop{}}{0\atop{}}{0\atop{}}{0\atop{}}}$,
$\scriptstyle{{1\atop{}}{1\atop{}}{1\atop1}{0\atop{}}{0\atop{}}{0\atop{}}{0\atop{}}}$,
$\scriptstyle{{1\atop{}}{1\atop{}}{1\atop0}{1\atop{}}{0\atop{}}{0\atop{}}{0\atop{}}}$,
$\scriptstyle{{0\atop{}}{1\atop{}}{1\atop0}{1\atop{}}{1\atop{}}{1\atop{}}{1\atop{}}}$,\,
&$\scriptstyle{{0\atop{}}{0\atop{}}{0\atop0}{0\atop{}}{0\atop{}}{1\atop{}}{1\atop{}}}$,
$\scriptstyle{{0\atop{}}{1\atop{}}{1\atop0}{1\atop{}}{1\atop{}}{1\atop{}}{0\atop{}}}$,
$\scriptstyle{{0\atop{}}{0\atop{}}{1\atop1}{1\atop{}}{1\atop{}}{1\atop{}}{0\atop{}}}$,
$\scriptstyle{{1\atop{}}{2\atop{}}{2\atop1}{1\atop{}}{1\atop{}}{1\atop{}}{0\atop{}}}$,\\
$\scriptstyle{{0\atop{}}{0\atop{}}{1\atop1}{1\atop{}}{1\atop{}}{1\atop{}}{1\atop{}}}$,
$\scriptstyle{{1\atop{}}{2\atop{}}{2\atop1}{1\atop{}}{1\atop{}}{1\atop{}}{1\atop{}}}$,
$\scriptstyle{{1\atop{}}{1\atop{}}{2\atop1}{2\atop{}}{1\atop{}}{1\atop{}}{1\atop{}}}$,
$\scriptstyle{{1\atop{}}{2\atop{}}{3\atop1}{3\atop{}}{2\atop{}}{1\atop{}}{0\atop{}}}$,
&$\scriptstyle{{1\atop{}}{1\atop{}}{2\atop1}{2\atop{}}{1\atop{}}{1\atop{}}{0\atop{}}}$,
$\scriptstyle{{1\atop{}}{2\atop{}}{2\atop1}{2\atop{}}{1\atop{}}{0\atop{}}{0\atop{}}}$,
$\scriptstyle{{0\atop{}}{1\atop{}}{1\atop1}{1\atop{}}{1\atop{}}{0\atop{}}{0\atop{}}}$,
$\scriptstyle{{0\atop{}}{1\atop{}}{2\atop1}{1\atop{}}{1\atop{}}{0\atop{}}{0\atop{}}}$,\\
$\scriptstyle{{1\atop{}}{2\atop{}}{3\atop2}{2\atop{}}{2\atop{}}{1\atop{}}{0\atop{}}}$,
$\scriptstyle{{1\atop{}}{3\atop{}}{4\atop2}{3\atop{}}{3\atop{}}{2\atop{}}{1\atop{}}}$,
$\scriptstyle{{1\atop{}}{2\atop{}}{4\atop2}{4\atop{}}{3\atop{}}{2\atop{}}{1\atop{}}}$,
$\scriptstyle{{2\atop{}}{3\atop{}}{4\atop2}{3\atop{}}{2\atop{}}{1\atop{}}{0\atop{}}}$,
&$\scriptstyle{{0\atop{}}{1\atop{}}{2\atop1}{2\atop{}}{2\atop{}}{2\atop{}}{1\atop{}}}$,
$\scriptstyle{{1\atop{}}{2\atop{}}{3\atop1}{3\atop{}}{2\atop{}}{2\atop{}}{1\atop{}}}$,
$\scriptstyle{{1\atop{}}{2\atop{}}{3\atop2}{2\atop{}}{2\atop{}}{2\atop{}}{1\atop{}}}$,
$\scriptstyle{{1\atop{}}{2\atop{}}{3\atop2}{3\atop{}}{2\atop{}}{1\atop{}}{1\atop{}}}$,\\
$\scriptstyle{{2\atop{}}{4\atop{}}{5\atop2}{4\atop{}}{3\atop{}}{2\atop{}}{1\atop{}}}$,
$\scriptstyle{{2\atop{}}{3\atop{}}{5\atop3}{4\atop{}}{3\atop{}}{2\atop{}}{1\atop{}}}$,
$\scriptstyle{{2\atop{}}{4\atop{}}{6\atop3}{5\atop{}}{4\atop{}}{3\atop{}}{2\atop{}}}$,
$\scriptstyle{{0\atop{}}{1\atop{}}{2\atop1}{2\atop{}}{2\atop{}}{1\atop{}}{0\atop{}}}$.
&$\scriptstyle{{2\atop{}}{3\atop{}}{4\atop2}{3\atop{}}{2\atop{}}{2\atop{}}{1\atop{}}}$,
$\scriptstyle{{2\atop{}}{4\atop{}}{6\atop3}{5\atop{}}{4\atop{}}{3\atop{}}{1\atop{}}}$,
$\scriptstyle{{1\atop{}}{2\atop{}}{3\atop1}{2\atop{}}{1\atop{}}{0\atop{}}{0\atop{}}}$,
$\scriptstyle{{1\atop{}}{2\atop{}}{4\atop2}{3\atop{}}{2\atop{}}{1\atop{}}{1\atop{}}}$.\\
\hline
\end{tabular}
\end{table}
The $(1,6,8)$-contributions of the $0$-roots and $1$-roots are $(16,
28,10)$ and $(12,24,8)$, respectively, and the total contribution of
all roots is $(14,26,9)$. It is well known that our nilpotent
element $e$ is special and $e^\vee\in\g^\vee$ has type ${\sf
E_8(a_6)}$ and
$$h^\vee=\textstyle{{0\atop{}}{0\atop{}}{2\atop 0}
{0\atop{}}{0\atop{}}{2\atop{}}{0\atop{}}}.$$

Condition~(C) for $\lambda+\rho=\sum_{i=1}^8a_i\varpi_i$ reads
\begin{eqnarray*}
4a_1+5a_2+7a_3+10a_4+8a_5+6a_6+4a_7+2a_8&=&14\\
6a_1+9a_2+12a_3+18a_4+15a_5+12a_6+8a_7+4a_8&=&26\\
2a_1+3a_2+4a_3+6a_4+5a_5+4a_6+3a_7+2a_8&=&9
\end{eqnarray*}
and the weight ${\textstyle\frac{1}{2}}h^\vee$ does satisfy this
system of linear equations. Unfortunately, it fails to satisfy the rest of Losev's conditions.
In order to amend that we shall replace it by
$$\lambda+\rho:=(s_2s_3s_5s_4)^2s_7
\big({\textstyle\frac{1}{2}}h^\vee\big)=(s_2s_3s_5s_4)^2\big(
\textstyle{{0\atop{}}{0\atop{}}{1\atop 0}
{0\atop{}}{1\atop{}}{(-1)\atop{}}{1\atop{}}}\big)=
\textstyle{{2\atop{}}{(-1)\atop{}}{1\atop (-1)}
{(-1)\atop{}}{3\atop{}}{(-1)\atop{}}{1\atop{}}}
$$
which still satisfies Condition~(C) as
$\langle\varpi_i,\alpha_j^\vee\rangle=0$ for $i\in\{1,6,8\}$ and $j\in\{2,3,4,5,7\}$.

We now show that $\lambda+\rho$ satisfies Condition~(B). Thanks to [\cite{BV},
Proposition~5.10], [\cite{Ja}, Corollar~10.10]
and the truth of the Kazhdan--Lusztig conjecture
it suffices to check that
$(s_2s_3s_5s_4)^2 s_7s_1s_3s_1s_2s_5s_6s_5s_8$ and $s_1s_3s_1s_2s_5s_6s_5s_8$ are in the same left Kazhdan--Lusztig cell of $W$.

Ler $S=\{s_i\,|\,\,1\le i\le 8\}$.
For $w \in W$ we set $\mathcal{L}(w):=\{ s\in S\,|\,\, \ell(sw)<\ell(w) \}$ and given $s,t\in S$ with $(st)^3=1$ we define
$$\mathcal{D}_L(s,t):=\{w\in\ W\,|\,\,\mathcal{L}(w)\cap\{s,t\}\ \mbox{has exactly one element}\}.$$
By [\cite{KL}, \S~4], if $w\in\mathcal{D}_L(s,t)$ then
exactly one of $sw,tw$ lies in $\mathcal{D}_L(s,t)$. This uniquely determined element is denoted by
${^*}w$. The map
$w\mapsto {^*}w$ is an involution on $\mathcal{D}_L(s,t)$.
Given $x,y\in W$ we write $x\,-\,y$ if either $\deg P_{x,y}$ of $\deg P_{y,x}$ reaches its upper bound
$(|l(y)-l(x)|-1)/2$; see  [\cite{Chen}, 2.1] for more detail. Here $P_{u,v}$ is  the Kazhdan--Lusztig polynomial associated with
$u,v\in W$ and $|\cdot|$ stands for the absolute value of a real number. It follows from [\cite{KL}, 2.3.e] that $x\,-\,sx$ for all $s\in S$ and $x\in W$.
For $x,y\in W$ we write $x\leqslant_L y$ it there exists a sequence of elements $x=x_0,x_1,\ldots, x_n=y$ in $W$ such that $x_{i-1}\,-\,x_i$ and $\mathcal{L}(x_{i-1})\not\subseteq\mathcal{L}(x)$ for all $1\le i\le n$. If
$x\leqslant_L y$ and $y\leqslant_L x$, we write $x\sim_Ly$.

Set $w_1=(s_2s_3s_5s_4)^2 s_7$ and $w_2=s_1s_3s_1s_2s_5s_6s_5s_8$.
The proof of the following proposition was kindly communicated to the author by Yu Chen. His arguments rely on the following fact: if $x\,-\,y$ for some $x,y\in\mathcal{D}_L(s,t)$ where $s,t\in S$ and $(st)^3=1$, then  ${^*}x\,-{^*}y$. That fact is a consequence of [\cite{KL}, Theorem~4.2] (and definitions).
 \begin{prop}\label{chen} {\rm(Yu Chen.)} We have that
$w_1w_2\sim_L w_2$.
\end{prop}
\begin{proof}
Set $x_1:=(s_2s_3s_5s_4)^2$ and $x_2:=s_1s_3s_1s_2s_5s_6s_5$. Then $w_1=x_1s_7$ and $w_2=x_2s_8$.

We claim that
$s_4s_2s_3s_5s_4x_2 \sim_L s_2s_4s_2s_3s_5s_4x_2.$
To prove this we analyse Table~I. The first and the last columns are easy to fill in directly, whilst he first line in the middle column is true by our preliminary remarks.
Then the second line in the middle column is true because $s_4s_2s_3s_5s_4x_2$,  $s_2s_4s_2s_3s_5s_4x_2$ are in $\mathcal{D}_L(s_4,s_5)$ and  $s_2s_3s_5s_4x_2={^*}(s_4s_2s_3s_5s_4x_2)$ and  $s_2s_4s_2s_3s_5s_4x_2={^*}(s_2s_4s_3s_5s_4x_2)$.
Then the third line in the middle column is true because $s_2s_3s_5s_4x_2$,  $s_2s_4s_3s_5s_4x_2$ are in $\mathcal{D}_L(s_1,s_3)$ and
$s_2s_5s_4x_2={^*}(s_2s_3s_5s_4x_2)$,
$s_1s_2s_4s_2s_3s_5s_4x_2={^*}(s_2s_4s_3s_5s_4x_2)$.
Finally, the last line in the middle column is true because $s_2s_5s_4x_2$, $s_1s_2s_4s_2s_3s_5s_4x_2$
are in $\mathcal{D}_L(s_5,s_6)$ and
$s_2s_4x_2={^*}(s_2s_5s_4x_2)$, $s_6s_1s_2s_4s_2s_3s_5s_4x_2={^*}(s_1s_2s_4s_2s_3s_5s_4x_2)$.

\smallskip

\begin{center}
\begin{tabular}{p{2.5 truecm} | p{4truecm} @{--\!\!--} p{4truecm} | p{2.5truecm}}
\noalign {\hrule height1.2pt}
\hfil $\mathcal{L}(x)$ & \hfil $x$ & \hfil $y$ & \hfil $\mathcal{L}(y)$ \\ \noalign {\hrule height1.2pt}
\hfil $\{ s_4 \}$ & \hfil $s_4s_2s_3s_5s_4x_2$ & \hfil $s_2s_4s_2s_3s_5s_4x_2$ & \hfil $\{ s_2,s_4 \}$ \\ \hline
\hfil $\{ s_2,s_3,s_5 \}$ & \hfil $s_2s_3s_5s_4x_2$ & \hfil $s_2s_4s_3s_5s_4x_2$ & \hfil $\{ s_2,s_3,s_5 \}$ \\ \hline
\hfil $\{ s_1,s_2,s_5 \}$ & \hfil $s_2s_5s_4x_2$ & \hfil $s_1s_2s_4s_3s_5s_4x_2$ & \hfil $\{ s_1,s_2,s_5 \}$ \\ \hline
\hfil $\{ s_1,s_2,s_4,s_6 \}$ & \hfil $s_2s_4x_2$ & \hfil $s_6s_1s_2s_4s_3s_5s_4x_2$ & \hfil $\{ s_1,s_2,s_6 \}$ \\ \noalign {\hrule height1.2pt}
\end{tabular}
\end{center}
\centerline{\small Table I}

\smallskip

\noindent
Combining the information given in the first and the last columns with our preliminary remarks we now obtain that
\begin{eqnarray*}
s_2s_4s_2s_3s_5s_4x_2& \leqslant_L & s_4s_2s_3s_5s_4x_2
\leqslant_L s_2s_3s_5s_4x_2\leqslant_L s_2s_5s_4x_2
\leqslant_L s_2s_4x_2\\
&\leqslant_L& s_6s_1s_2s_4s_3s_5s_4x_2
\leqslant_L
s_1s_2s_4s_3s_5s_4x_2\leqslant_L s_2s_4s_3s_5s_4x_2\\
&\leqslant_L&
s_2s_4s_2s_3s_5s_4x_2.
\end{eqnarray*}
Next we claim that
$s_2s_4s_2s_3s_5s_4x_2 \sim_L s_2s_3s_4s_2s_3s_5s_4x_2$. In order to see this we analyse Thable~II. It is fairly straightforward to determine $\mathcal{L}(x)$ and $\mathcal{L}(y)$ in all cases and fill in the first and the last columns of Table~II.
The first line of the middle column is true by our introductory remarks. To check that the other lines in the middle column are correct we consider the subsets
$\mathcal{D}_L(s_4,s_5)$, $\mathcal{D}_L(s_2,s_4)$,
$\mathcal{D}_L(s_1,s_3)$, $\mathcal{D}_L(s_3,s_4)$
and $\mathcal{D}_L(s_4,s_5)$ (in that order) and each time pass from $x\,-\,y$ to $^{*}x\,-{^*}y$.

\begin{center}
\begin{tabular}{p{2.5 truecm} | p{4truecm} @{--\!\!--} p{4truecm} | p{2.5truecm}}
\noalign {\hrule height1.2pt}
\hfil $\mathcal{L}(x)$ & \hfil $x$ & \hfil $y$ & \hfil $\mathcal{L}(y)$ \\ \noalign {\hrule height1.2pt}
\hfil $\{ s_2,s_4 \}$ & \hfil $s_2s_4s_2s_3s_5s_4x_2$ & \hfil $s_2s_3s_4s_2s_3s_5s_4x_2$ & \hfil $\{ s_2,s_3,s_4 \}$  \\ \hline
\hfil $\{ s_2,s_3,s_5 \}$ & \hfil $s_2s_4s_3s_5s_4x_2$ & \hfil $s_5s_2s_3s_4s_2s_3s_5s_4x_2$ & \hfil $\{ s_2,s_3,s_5 \}$ \\ \hline
\hfil $\{ s_3,s_4,s_5 \}$ & \hfil $s_4s_3s_5s_4x_2$ & \hfil $s_5s_3s_4s_2s_3s_5s_4x_2$ & \hfil $\{ s_3,s_4,s_5 \}$ \\ \hline
\hfil $\{ s_1,s_4,s_5 \}$ & \hfil $s_1s_4s_3s_5s_4x_2$ & \hfil $s_1s_5s_3s_4s_2s_3s_5s_4x_2$ & \hfil $\{ s_1,s_4,s_5 \}$ \\ \hline
\hfil $\{ s_1,s_3,s_5 \}$ & \hfil $s_1s_3s_5s_4x_2$ & \hfil $s_3s_1s_5s_3s_4s_2s_3s_5s_4x_2$ & \hfil $\{ s_1,s_3,s_5 \}$ \\ \hline
\hfil $\{ s_1,s_3,s_4,s_6 \}$ & \hfil $s_1s_3s_4x_2$ & \hfil $s_3s_1s_3s_4s_2s_3s_5s_4x_2$ & \hfil $\{ s_1,s_3,s_4 \}$ \\
\noalign {\hrule height1.2pt}
\end{tabular}
\end{center}
\centerline{\small Table II}
As a result, we obtain that
\begin{eqnarray*}
\!s_2s_3s_4s_2s_3s_5s_4x_2\!\!&\leqslant_L &\!\! s_2s_4s_2s_3s_5s_4x_2
\leqslant_L s_2s_4s_3s_5s_4x_2\leqslant_L s_4s_3s_5s_4x_2
\leqslant_L s_1s_4s_3s_5s_4x_2\\
\!\!&\leqslant_L&\!\! s_1s_3s_5s_4x_2
\leqslant_L
s_1s_3s_4x_2\leqslant_L\! s_3s_1s_3s_4s_2s_3s_5s_4x_2\\
\!\!&\leqslant_L&\!\!
s_3s_1s_5s_3s_4s_2s_3s_5s_4x_2\leqslant_L\!
s_1s_5s_3s_4s_2s_3s_5s_4x_2\leqslant_L\! s_5s_3s_4s_2s_3s_5s_4x_2\\
\!\!&\leqslant &\!\!s_5s_2s_3s_4s_2s_3s_5s_4x_2\leqslant_L\!
s_2s_3s_4s_2s_3s_5s_4x_2.
\end{eqnarray*}

Next we wish to show that $x_1x_2\sim_L x_2$. To this end we first determine $\mathcal{L}(w)$ for all $w$ listed in the first column of Table~III and then observe that $x\sim_L y$ for all $x,y$ listed
in the first six and the last two lines of that table.

\smallskip

\begin{center}
\begin{tabular}{p{6truecm} | p{6truecm}}
\noalign {\hrule height1.0pt}
\hfil $w$ \hskip1.0truecm \ & \hfil $\mathcal{L}(w)$ \\ \noalign {\hrule height1.0pt}
\hfil $x_2$  & \hfil $\{ s_1,s_2,s_3,s_5,s_6 \}$ \\ \hline
\hfil $s_4x_2$  & \hfil $\{ s_1,s_4,s_6 \}$ \\ \hline
\hfil $s_5s_4x_2$  & \hfil $\{ s_1,s_4,s_5 \}$ \\ \hline
\hfil $s_2s_5s_4x_2$  & \hfil $\{ s_1,s_2,s_5 \}$ \\ \hline
\hfil $s_2s_3s_5s_4x_2$  & \hfil $\{ s_2,s_3,s_5 \}$ \\ \hline
\hfil $s_4s_2s_3s_5s_4x_2$  & \hfil $\{ s_4 \}$ \\ \hline
\hfil $s_2s_4s_2s_3s_5s_4x_2$  & \hfil $\{ s_2,s_4 \}$ \\ \hline
\hfil $s_2s_3s_4s_2s_3s_5s_4x_2$  & \hfil $\{ s_2,s_3,s_4 \}$ \\ \hline
\hfil $s_2s_3s_5s_4s_2s_3s_5s_4x_2$  & \hfil $\{ s_2,s_3,s_5 \}$ \\
\noalign {\hrule height1.0pt}
\end{tabular}
\end{center}
\centerline{\small Table III}

\smallskip

\noindent
 As a consequence, $x_1x_2\sim_L
s_2s_3s_4s_2s_5s_4x_2$ and $x_2\sim_L s_4s_2s_3s_5s_4x_2$. In view of the two claims above (which we have already proved) this yields
$$x_1x_2\sim_L s_2s_3s_4s_2s_5s_4x_2\sim_Ls_2s_4s_2s_5s_4x_2
\sim_L s_4s_2s_5s_4x_2\sim_L x_2.$$ But then
$w_1w_2=x_1s_7x_2s_8 \sim_L x_1x_2s_8  \sim_L x_2s_8= w_2$ as wanted.
\end{proof}

As we explained earlier, it follows from Proposition~\ref{chen} that $I(\lambda)=I(\frac{1}{2}h^\vee-\rho)$.
In particular, Condition~(B) holds for $\lambda+\rho$.

Since in the present case $e$ is no longer of standard Levi type,
verifying Conditions~(A) and (D) becomes more complicated. We
let $\lambda'$ denote the restriction of $\lambda$ to $\h:=\l'\cap
\t$ which coincides with the span of the semisimple root
elements $h'_1:=h_{\alpha_3}$, $h'_2:=h_{\alpha_4}$,
$h'_3:=h_{\alpha_5}$ and $h'_4:=h_{\alpha_2}$. We write $\alpha_i'$ for the corresponding simple roots in $\h^*$ and let
$\{\varpi'_i\,|\,\, 1\le i\le 4\}\subset \h^*$ be the associated system of
fundamental weights, so that $\varpi'_i(h_j)=\langle\varpi_i',(\alpha_j')^\vee\rangle=\delta_{ij}$
for all $1\le i,j\le 4$. Then $\lambda'+\rho'=\varpi_2'-\varpi'_1-\varpi_3-
\varpi'_4$ where
$\rho'=\sum_{i=1}^4\varpi'_i$.
Obviously, the root system of $\l'$ with respect to $\h'$ identifies with $\Phi':=\Phi\cap\big(\bigoplus_{i=2}^5\Z\alpha_i
\big)$.

We denote by $L_0(\lambda')$ be the
irreducible $\l'$-module of highest weight $\lambda'$ and write
$I_0(\lambda')$ for the primitive ideal ${\rm
Ann}_{U(\l')}\,L_0(\lambda')$ of $U(\l')$. As $\langle \lambda+\rho,\alpha_7^\vee\rangle\not\in\Z^{>0}$, in order to verify
Conditions~(A) and (D) in the present case it suffices to show that
$$I_0(\lambda')={\rm
Ann}_{U(\l')}\big(Q_{\l',\,e_0}\otimes_{U(\l',\,e_0)}{\mathbb
C}_\eta\big)$$ for some one-dimensional representation $\eta$ of
$U(\l',e_0)$ (here $Q_{\l',\,e_0}$ and $U(\l',e_0)$ stand for the
generalised Gelfand--Graev $\l'$-module and the finite $W$-algebra
associated with the pair $(\l',e_0)$, respectively).

Let $\p'$ be the standard parabolic subalgebra of $\l'$ whose
standard Levi subalgebra $\l'_0$ is spanned by $\h$ and
$e_{\pm\alpha_4}$ and let $\n'$ be the nilradical of $\p'$. By
construction, $\p'$ is an optimal parabolic subalgebra for $e_0$ and
$e_0$ is a Richardson element of $\p'$ (one should keep in mind that
the subregular nilpotent orbit is distinguished in a Lie algebra of
type ${\sf D}_4$). Let $\chi_0$ be the linear function on $\l'$ given by $\chi_0(x)=(e_0,x)$ where $(\,\cdot\,\,\,\cdot\,)$ is the Killing for of $\l'$ and write
 $\m_{-\chi}$ for the subspace of $U(\l')$ spanned by all $x+\chi_0(x)$ with $x\in\m$

Since
$\lambda'(h_2')=0$ and
$\langle\lambda+\rho,\beta^\vee\rangle\not\in\Z^{>0}$ for any root
$\beta\in (\Phi'\cap \Phi^+)\setminus \{\alpha_4\}$, a well-known criterion for irreducibility of
generalised Verma modules  implies that
\begin{equation}\label{hu}
L_0(\lambda')\,\cong\, U(\l')\otimes_{U(\p')}{\mathbb C}_{\lambda'}
\end{equation}
as $\l'$-modules; see [\cite{Hu1}, Theorem~9.12] for example (here we regard $\lambda'$ as a linear
function on $\p'$ that vanishes on $[\p',\p']$).
Let
$L_0(\lambda')^*$ be the $\l'$-module dual to $L_0(\lambda')$. In
the present case, the nilpotent subalgebra $\m$ involved in the
definition of $Q_{\l',\,e_0}$ is spanned by the root vectors
$e_{-\beta}$ with $\beta\in(\Phi'\cap\Phi^+)\setminus\{\alpha_4\}$.

It is immediate from
(\ref{hu}) that the subspace
$V_0$ of $L_0(\lambda)^*$ consisting of all linear functions vanishing on $\m_{-\chi}L_0(\lambda)$ is one-dimensional.
That subspace carries a
natural $U(\l',e_0)$-module structure and Skryabin's theorem
[\cite{Sk}] implies that the $\l'$-submodule of $L_0(\lambda')^*$
generated by $V_0$ is irreducible and isomorphic to
$Q_{\l',\,e_0}\otimes_{U(\l',\,e_0)}V_0$.

Let $\scriptstyle{\top}$ be the canonical anti-automorphism of
$U(\l')$. Since the annihilator of $L(\lambda')^*$ in $U(\l')$
coincides with $I_0(\lambda')^\top$, repeating verbatim the argument
from [\cite{P11}, Remark~4.3] one observes that
$$I_0(\lambda')^\top\,=\,{\rm
Ann}_{U(\l')}\big(Q_{\l',\,e_0}\otimes_{U(\l',\,e_0)}V_0\big).$$  Since $-1\in W({\sf D_4})$ it must be that
$I_0(\lambda')=\,I_0(\lambda')^{\scriptstyle\top}$ (see the end of Subsection~\ref{3.17} for more detail about this equality). This enables us to conclude that Conditions~(A) and (D)
hold for $\lambda+\rho$.

As a result, $\lambda+\rho$ satisfies all
four Losev's conditions.  Since
in the present case we also have that $\g_e=[\g_e,\g_e]$ (by
[\cite{dG}]), we conclude that $I(\lambda+\rho)=I(\frac{1}{2}h^\vee-\rho)$ is the unique
multiplicity-free primitive ideal in ${\mathcal X}_{\O(e)}$.
\subsection{Type $({\sf E_8, A_3+A_2+A_1})$}\label{3.13}
Here our pinning is
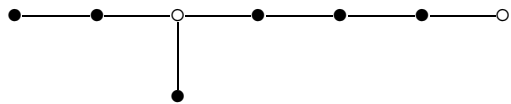
\[\xymatrix{*{\bullet}\ar@{-}[r]& *{\bullet} \ar@{-}[r] & *{\circ} \ar@{-}[r]\ar@{-}[d] & *{\bullet} \ar@{-}[r]
 & *{\bullet} \ar@{-}[r]&*{\bullet}\ar@{-}[r]&*{\circ}\\ && *{\bullet}&&&}\]
and we choose
$$\tau=\textstyle{{2\atop{}}{2\atop{}}{(-6)\atop 2}
{2\atop{}}{2\atop{}}{2\atop{}}{(-3)\atop{}}}$$ as an optimal
cocharacter. Since $\dim\,\g_e=66$, the total number of positive
$0$-roots and $1$-roots is $(66-8)/2=29$. The roots are given below.

\begin{table}[htb]
\label{data2}
\begin{tabular}{|c|c|}
\hline$0$-roots  & $1$-roots
\\ \hline
$\scriptstyle{{1\atop{}}{1\atop{}}{1\atop0}{1\atop{}}{0\atop{}}{0\atop{}}{0\atop{}}}$,
$\scriptstyle{{1\atop{}}{1\atop{}}{1\atop1}{0\atop{}}{0\atop{}}{0\atop{}}{0\atop{}}}$,
$\scriptstyle{{0\atop{}}{1\atop{}}{1\atop1}{1\atop{}}{0\atop{}}{0\atop{}}{0\atop{}}}$,
$\scriptstyle{{0\atop{}}{0\atop{}}{1\atop1}{1\atop{}}{1\atop{}}{0\atop{}}{0\atop{}}}$,\,
&$\scriptstyle{{1\atop{}}{1\atop{}}{1\atop0}{1\atop{}}{1\atop{}}{1\atop{}}{1\atop{}}}$,
$\scriptstyle{{0\atop{}}{1\atop{}}{1\atop1}{1\atop{}}{1\atop{}}{1\atop{}}{1\atop{}}}$,
$\scriptstyle{{1\atop{}}{2\atop{}}{2\atop1}{2\atop{}}{1\atop{}}{1\atop{}}{1\atop{}}}$,
$\scriptstyle{{1\atop{}}{1\atop{}}{2\atop1}{2\atop{}}{2\atop{}}{1\atop{}}{1\atop{}}}$,\\
$\scriptstyle{{0\atop{}}{0\atop{}}{1\atop0}{1\atop{}}{1\atop{}}{1\atop{}}{0\atop{}}}$,
$\scriptstyle{{1\atop{}}{1\atop{}}{2\atop1}{1\atop{}}{1\atop{}}{1\atop{}}{0\atop{}}}$,
$\scriptstyle{{0\atop{}}{1\atop{}}{2\atop1}{2\atop{}}{1\atop{}}{1\atop{}}{0\atop{}}}$,
$\scriptstyle{{1\atop{}}{2\atop{}}{3\atop2}{2\atop{}}{1\atop{}}{1\atop{}}{0\atop{}}}$,
&$\scriptstyle{{0\atop{}}{1\atop{}}{2\atop1}{2\atop{}}{2\atop{}}{2\atop{}}{1\atop{}}}$,
$\scriptstyle{{1\atop{}}{2\atop{}}{3\atop2}{2\atop{}}{2\atop{}}{2\atop{}}{1\atop{}}}$,
$\scriptstyle{{1\atop{}}{2\atop{}}{3\atop1}{3\atop{}}{2\atop{}}{2\atop{}}{1\atop{}}}$,
$\scriptstyle{{1\atop{}}{2\atop{}}{3\atop2}{3\atop{}}{2\atop{}}{1\atop{}}{1\atop{}}}$,\\
$\scriptstyle{{1\atop{}}{2\atop{}}{3\atop1}{2\atop{}}{2\atop{}}{1\atop{}}{0\atop{}}}$,
$\scriptstyle{{1\atop{}}{3\atop{}}{4\atop2}{3\atop{}}{2\atop{}}{1\atop{}}{0\atop{}}}$,
$\scriptstyle{{2\atop{}}{4\atop{}}{6\atop3}{5\atop{}}{4\atop{}}{3\atop{}}{2\atop{}}}$,
$\scriptstyle{{0\atop{}}{1\atop{}}{1\atop0}{1\atop{}}{1\atop{}}{0\atop{}}{0\atop{}}}$,
&$\scriptstyle{{1\atop{}}{2\atop{}}{4\atop2}{4\atop{}}{3\atop{}}{2\atop{}}{1\atop{}}}$,
$\scriptstyle{{1\atop{}}{3\atop{}}{4\atop2}{3\atop{}}{3\atop{}}{2\atop{}}{1\atop{}}}$,
$\scriptstyle{{2\atop{}}{3\atop{}}{5\atop3}{4\atop{}}{3\atop{}}{2\atop{}}{1\atop{}}}$,
$\scriptstyle{{2\atop{}}{4\atop{}}{5\atop2}{4\atop{}}{3\atop{}}{2\atop{}}{1\atop{}}}$,\\
$\scriptstyle{{1\atop{}}{2\atop{}}{2\atop1}{1\atop{}}{1\atop{}}{0\atop{}}{0\atop{}}}$,
$\scriptstyle{{1\atop{}}{1\atop{}}{2\atop1}{2\atop{}}{1\atop{}}{0\atop{}}{0\atop{}}}$.
&$\scriptstyle{{2\atop{}}{4\atop{}}{6\atop3}{5\atop{}}{4\atop{}}{2\atop{}}{1\atop{}}}$,
$\scriptstyle{{2\atop{}}{3\atop{}}{4\atop2}{3\atop{}}{2\atop{}}{2\atop{}}{1\atop{}}}$,
$\scriptstyle{{0\atop{}}{0\atop{}}{0\atop0}{0\atop{}}{1\atop{}}{1\atop{}}{1\atop{}}}$.
\\
\hline
\end{tabular}
\end{table}
The $(4,8)$-contributions of the $0$-roots and $1$-roots are
$(30,2)$ and $(45,15)$, respectively, and the total contribution of
all roots is $\big(\frac{75}{2},\frac{17}{2}\big)$.

Since the orbit $\O(e)$ is non-special, Proposition~5.3.2 from
[\cite{Lo2}] suggests that one should seek a proper regular Lie
subalgebra $\g(\lambda)$ of $\g\cong\g^\vee$ and a special nilpotent
element $X\in\g(\lambda)$ such that $\dim
\g(\lambda)_X=\dim\g_e=66$. A priori we have two possibilities:
either $\g(\lambda)$ has type ${\sf A_7+A_1}$ and $X=0$ or
$\g(\lambda)$ has type ${\sf D}_8$ and ${\bf p}(X)=(2^6,1^4)$.
Indeed, in the first case $\dim\g(\lambda)=63+3=66$ whilst in the
second case $\dim\g(\lambda)_X=\frac{1}{2}(100+36-4)=66$ by
(\ref{so}) (one should keep in mind here that ${\bf r}(X)=(10,6)$
and $n_{\rm odd}(X)=4$).

The present case turned out to be
counter-intuitive. Although the
first possibility somehow looks more appealing, a
careful analysis shows that any weight $\lambda+\rho$ having
$\Phi_\lambda$ of type ${\sf A_7+A_1}$ and verifying Conditions~(B)
and (C) always violates Condition~(A) which then becomes too restrictive (we
omit the details). Thus we must ignore the more appealing option and
seek $\lambda+\rho=\sum_{i=1}^8a_i\varpi_i$ such that all $a_i$ are
half-integers. Setting $\widetilde{a}_i=2a_i$ for $1\le i\le 8$ we
can rewrite Condition~(C) as follows:
\begin{eqnarray*}
10\widetilde{a}_1+15\widetilde{a}_2+20\widetilde{a}_3+30\widetilde{a}_4+
24\widetilde{a}_5+18\widetilde{a}_6+12\widetilde{a}_7+6\widetilde{a}_8&=&75\\
2\widetilde{a}_1+3\widetilde{a}_2+4\widetilde{a}_3+6\widetilde{a}_4+
5\widetilde{a}_5+4\widetilde{a}_6+3\widetilde{a}_7+2\widetilde{a}_8&=&17.
\end{eqnarray*}
Since it turned out to be impossible to find a suitable $\lambda+\rho$ by
a lucky guess, we shall once again invoke the method of
Subsection~\ref{3.11}. In particular, we shall recycle the root
subsystem $\Phi_0$ of type ${\sf D_8}$ introduced there. Recall that
$\{\varpi'_i\,|\,\,1\le i\le 8\}$ is the system of fundamental
weights associated with $\Pi_0=\{\beta_i\,|\,\,1\le i\le 8\}$ and we
have that $\varpi_4=\varpi_1'+\varpi_3'+\varpi_5'+\varpi_7'$ and
$\varpi'_8=\varpi_8$.

Setting $\lambda+\rho=\sum_{i=1}^8b_i\varpi_i'$ with $b_i\in\Z$ and
arguing as in Subsection~\ref{3.11} we rewrite the above system of
linear equations as follows:
\begin{eqnarray*}
{\textstyle\frac{7}{2}}b_1+6b_2+{\textstyle\frac{17}{2}}b_3+10b_4+
{\textstyle \frac{23}{2}}b_5+12b_6+{\textstyle\frac{13}{2}}b_7+6b_8&=&{\textstyle\frac{75}{2}}\\
{\textstyle \frac{1}{2}}b_1+b_2+{\textstyle \frac{3}{2}}b_3+2b_4+
{\textstyle \frac{5}{2}}b_5+3b_6+{\textstyle
\frac{3}{2}}b_7+2b_8&=&{\textstyle \frac{17}{2}}.
\end{eqnarray*}
A fairly decent solution to this system is given by setting
$b_1=b_3=b_5=b_8=2$, $b_2=b_4=b_6=-1$ and $b_7=1$,  which leads to
the weight
$$\lambda+\rho=\,2\varpi_1'-\varpi_2'+2\varpi_3'-
\varpi_4'+2\varpi_5'-\varpi'_6+\varpi_7'+2\varpi_8'.$$

As in Subsection~\ref{3.11} we let $\g_0$ be the regular Lie
subalgebra of $\g$ with root system $\Phi_0$ and write
$I_0(\lambda)$ for the annihilator in $U(\g_0)$ of the irreducible
$\g_0$-module of highest weight $\lambda$. In order to determine the
associated variety ${\rm VA}(I_0(\lambda))$ we shall once again
invoke the algorithm described in [\cite{BG}]
(we refer to Subsecton~\ref{3.11} for the unexplained notation used below).

In view of [\cite{Bo}, Planche~IV] we have that $$\lambda+\rho=
{\textstyle\frac{9}{2}}\varepsilon_1+{\textstyle\frac{5}{2}}\varepsilon_2+
{\textstyle\frac{7}{2}}\varepsilon_3+
{\textstyle\frac{3}{2}}\varepsilon_4+{\textstyle\frac{5}{2}}\varepsilon_5
+{\textstyle\frac{1}{2}}\varepsilon_6+{\textstyle\frac{3}{2}}\varepsilon_7+
{\textstyle\frac{1}{2}}\varepsilon_8.$$ Applying the
Robinson--Shensted algorithm on the sequence of numbers
$$\big({\textstyle\frac{9}{2}},\,{\textstyle\frac{5}{2}},\,{\textstyle\frac{7}{2}},\,
{\textstyle\frac{3}{2}},\,{\textstyle\frac{5}{2}},\,{\textstyle\frac{1}{2}},\,
{\textstyle\frac{3}{2}},\,{\textstyle\frac{1}{2}},\,{\textstyle
-\frac{1}{2}},{\textstyle -\frac{3}{2}},{\textstyle
-\frac{1}{2}},{\textstyle-\frac{5}{2}},{\textstyle
-\frac{3}{2}},{\textstyle-\frac{7}{2}},{\textstyle-\frac{5}{2}},
{\textstyle-\frac{9}{2}}\big)
$$ we arrive at the partition  ${\bf
q}=(2^6,1^4)$ associated with the tableau ${\rm RS}(\lambda+\rho)$
(we leave the details to the interested reader). Then $$(r_i)=(1,2,3,4,6,7,8,9,10,11)$$
which yields $(2s_i)=(2,4,6,8,10)$ and $(2t_i+1)=(1,3,7,9,11)$.

Concatenating the lists $(s_i)$ and $(t_i)$ we obtain the list
$(0,1,1,2,3,3,4,4,5,5)$. This implies that $(s'_i)=(0,1,2,3,4,5)$
and $(t'_i)=(1,2,3,4,5)$. Concatenating the lists $(2s_i'+1)$ and
$(2t_i')$ we get the list $(r_i')=(1,2,3,4,6,7,8,9,10,11)$. Since
$q_i'=r'_i-i+1$ we get $$(q'_i)=(1,1,1,1,2,2,2,2,2,2).$$ This enables
us to conclude that the open orbit of of ${\rm VA}(I_0(\lambda))$
contains a nilpotent element $X$ with ${\bf p}(X)=(2^6,1^4)$.

After expressing the $\varpi_i'$ in terms of
$\{\varpi_1,\ldots,\varpi_8\}$ we obtain that
$$\lambda+\rho={-\textstyle\frac{1}{2}}\varpi_1-{\textstyle\frac{1}{2}}\varpi_2
-{\textstyle\frac{1}{2}}\varpi_3+{\textstyle\frac{5}{2}}\varpi_4-
{\textstyle\frac{1}{2}}\varpi_5-\varpi_6+{\textstyle\frac{1}{2}}
\varpi_7+{\textstyle\frac{3}{2}}\varpi_8.$$ Of course, at this point
one can check directly that this weight satisfies all our
requirements, but we had to find it first! Since
$\Phi_0\subseteq\Phi_\lambda\subsetneq\Phi$, the maximality of
$\Phi_0$ yields $\Phi_\lambda=\Phi_0$. But then $\g(\lambda)=\g_0$
and [\cite{Lo2}, Proposition~5.3.2] in conjunction with the earlier
remarks implies that $\dim\,{\rm
VA}(I\lambda))=\dim\overline{\O(e)}$.

Since $\langle\lambda+\rho,\gamma^\vee\rangle\not\in\Z^{>0}$ for any
$\gamma\in\Phi^+$ which can be expressed as a linear combination of
$\alpha_i$ with $i\in \{1,2,3,5,6,7\}$, our choice of pinning
implies that $\lambda+\rho$ satisfies Condition~(A). Therefore,
${\rm VA}(I(\lambda))=\overline{\O(e)}$. Since $e$ has standard Levi
type, we conclude that $\lambda+\rho$ satisfies all Losev's
conditions. As in the present case $\g_e=[\g_e,\g_e]$ by [\cite{dG}]
we now deduce that
$I({-\textstyle\frac{3}{2}}\varpi_1-{\textstyle\frac{3}{2}}\varpi_2
-{\textstyle\frac{3}{2}}\varpi_3+{\textstyle\frac{3}{2}}\varpi_4-
{\textstyle\frac{3}{2}}\varpi_5-2\varpi_6-{\textstyle\frac{1}{2}}
\varpi_7+{\textstyle\frac{1}{2}}\varpi_8)$ is the only
multiplicity-free primitive ideal in ${\mathcal X}_{\O(e)}$.
\subsection{Type $({\sf E_8, 2A_3})$}\label{3.14}
In this case our pinning for $e$ is
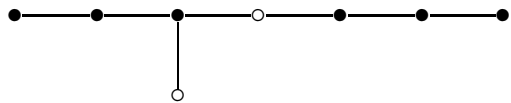
\[\xymatrix{*{\bullet}\ar@{-}[r]& *{\bullet} \ar@{-}[r] & *{\bullet} \ar@{-}[r]\ar@{-}[d] & *{\circ} \ar@{-}[r]
 & *{\bullet} \ar@{-}[r]&*{\bullet}\ar@{-}[r]&*{\bullet}\\ && *{\circ}&&&}\]
and we take
$$\tau=\textstyle{{2\atop{}}{2\atop{}}{2\atop (-3)}
{(-6)\atop{}}{2\atop{}}{2\atop{}}{2\atop{}}}$$ as an optimal
cocharacter; see [\cite{LT}]. Since $\dim\,\g_e=60$, the total
number of positive $0$-roots and $1$-roots is $(60-8)/2=26$. The
roots are given below.
\begin{table}[htb]
\label{data2}
\begin{tabular}{|c|c|}
\hline$0$-roots  & $1$-roots
\\ \hline
$\scriptstyle{{1\atop{}}{1\atop{}}{1\atop0}{1\atop{}}{0\atop{}}{0\atop{}}{0\atop{}}}$,
$\scriptstyle{{0\atop{}}{1\atop{}}{1\atop0}{1\atop{}}{1\atop{}}{0\atop{}}{0\atop{}}}$,
$\scriptstyle{{0\atop{}}{0\atop{}}{1\atop0}{1\atop{}}{1\atop{}}{1\atop{}}{0\atop{}}}$,
$\scriptstyle{{0\atop{}}{0\atop{}}{0\atop0}{1\atop{}}{1\atop{}}{1\atop{}}{1\atop{}}}$,\,
&$\scriptstyle{{0\atop{}}{1\atop{}}{1\atop1}{0\atop{}}{0\atop{}}{0\atop{}}{0\atop{}}}$,
$\scriptstyle{{0\atop{}}{1\atop{}}{1\atop1}{1\atop{}}{1\atop{}}{1\atop{}}{1\atop{}}}$,
$\scriptstyle{{1\atop{}}{1\atop{}}{1\atop1}{1\atop{}}{1\atop{}}{1\atop{}}{0\atop{}}}$,
$\scriptstyle{{1\atop{}}{1\atop{}}{2\atop1}{1\atop{}}{1\atop{}}{0\atop{}}{0\atop{}}}$,\\
$\scriptstyle{{1\atop{}}{2\atop{}}{3\atop2}{2\atop{}}{2\atop{}}{1\atop{}}{0\atop{}}}$,
$\scriptstyle{{1\atop{}}{2\atop{}}{3\atop2}{2\atop{}}{1\atop{}}{1\atop{}}{1\atop{}}}$,
$\scriptstyle{{1\atop{}}{2\atop{}}{4\atop2}{3\atop{}}{2\atop{}}{2\atop{}}{1\atop{}}}$,
$\scriptstyle{{1\atop{}}{2\atop{}}{3\atop2}{3\atop{}}{3\atop{}}{2\atop{}}{1\atop{}}}$,
&$\scriptstyle{{0\atop{}}{1\atop{}}{2\atop1}{1\atop{}}{1\atop{}}{1\atop{}}{0\atop{}}}$,
$\scriptstyle{{1\atop{}}{2\atop{}}{2\atop1}{1\atop{}}{0\atop{}}{0\atop{}}{0\atop{}}}$,
$\scriptstyle{{1\atop{}}{1\atop{}}{2\atop1}{2\atop{}}{2\atop{}}{1\atop{}}{1\atop{}}}$,
$\scriptstyle{{1\atop{}}{2\atop{}}{2\atop1}{2\atop{}}{2\atop{}}{1\atop{}}{0\atop{}}}$,\\
$\scriptstyle{{1\atop{}}{3\atop{}}{4\atop2}{3\atop{}}{2\atop{}}{1\atop{}}{1\atop{}}}$,
$\scriptstyle{{1\atop{}}{3\atop{}}{5\atop2}{4\atop{}}{3\atop{}}{2\atop{}}{1\atop{}}}$,
$\scriptstyle{{2\atop{}}{3\atop{}}{4\atop2}{3\atop{}}{2\atop{}}{1\atop{}}{0\atop{}}}$,
$\scriptstyle{{2\atop{}}{3\atop{}}{4\atop2}{4\atop{}}{3\atop{}}{2\atop{}}{1\atop{}}}$.
&$\scriptstyle{{1\atop{}}{2\atop{}}{2\atop1}{2\atop{}}{1\atop{}}{1\atop{}}{1\atop{}}}$,
$\scriptstyle{{0\atop{}}{1\atop{}}{2\atop1}{2\atop{}}{2\atop{}}{2\atop{}}{1\atop{}}}$,
$\scriptstyle{{1\atop{}}{2\atop{}}{3\atop1}{2\atop{}}{1\atop{}}{1\atop{}}{0\atop{}}}$,
$\scriptstyle{{1\atop{}}{2\atop{}}{3\atop1}{3\atop{}}{2\atop{}}{1\atop{}}{1\atop{}}}$,\\
&$\scriptstyle{{2\atop{}}{4\atop{}}{5\atop3}{4\atop{}}{3\atop{}}{2\atop{}}{1\atop{}}}$,
$\scriptstyle{{2\atop{}}{4\atop{}}{6\atop3}{5\atop{}}{4\atop{}}{3\atop{}}{1\atop{}}}$.
\\
\hline
\end{tabular}
\end{table}

The $(2,5)$-contributions of the $0$-roots and $1$-roots are
$(16,28)$ and $(18,27)$, respectively, and the total contribution of
all roots is $\big(\frac{34}{2},\frac{55}{2}\big)$. In the present
case the orbit $\O(e)$ is non-special and comparing the size of
$\O(e)$ with the size of a root subsystem of type ${\sf D_5+A_3}$ in
$\Phi$ indicates that one should seek
$\lambda+\rho=\sum_{i=1}^8a_i\varpi_i$ such that $a_i\in
\frac{1}{4}\Z$ for all $1\le i\le 8$. Setting
$\widetilde{a}_i:=4a_i$ we can rewrite Condition~(C) for
$\lambda+\rho$ as follows:
\begin{eqnarray*}
5\widetilde{a}_1+8\widetilde{a}_2+10\widetilde{a}_3+15\widetilde{a}_4+
12\widetilde{a}_5+9\widetilde{a}_6+6\widetilde{a}_7+3\widetilde{a}_8&=&68\\
8\widetilde{a}_1+12\widetilde{a}_2+16\widetilde{a}_3+24\widetilde{a}_4+
20\widetilde{a}_5+15\widetilde{a}_6+10\widetilde{a}_7+5\widetilde{a}_8&=&110.
\end{eqnarray*}
A perfect solution to this system of linear
equations is given by $\widetilde{a}_i=1$ for all $1\le i\le 8$.
This leads to the weight $\lambda+\rho=\frac{1}{4}\rho$. Due to our
choice of pinning we have that
$\langle\lambda+\rho,\gamma^\vee\rangle\not\in\Z^{>0}$ for any
$\gamma\in\Phi^+$ which can be expressed as a linear combination of
$\alpha_i$ with $i\in \{1,3,4,6,7,8\}$. In other words,
$\lambda+\rho$ satisfies Condition~(A). Since Condition~(D) is
vacuous in the present case it remains to check that $\lambda+\rho$
satisfies Condition~(B).

First wet note that $\Phi_{\lambda}$ contains the positive roots
$$
\xymatrix{{\beta_1}\ar@{-}[r]& {\beta_2} \ar@{-}[r] & {\beta_3}
\ar@{-}[r]\ar@{-}[d] & {\beta_4}
 & {\beta_6} \ar@{-}[r]&{\beta_7}\ar@{-}[r] &{\beta_8}\\ && {\beta_5}&&&}
$$
where
\begin{eqnarray*}
\beta_1\,=\,{\scriptstyle{0\atop{}}{0\atop{}}{0\atop0}{1\atop{}}{1\atop{}}{1\atop{}}{1\atop{}}},\
\
\beta_2\,=\,{\scriptstyle{1\atop{}}{1\atop{}}{1\atop1}{0\atop{}}{0\atop{}}{0\atop{}}{0\atop{}}},\
\
\beta_3\,=\,{\scriptstyle{0\atop{}}{0\atop{}}{1\atop0}{1\atop{}}{1\atop{}}{1\atop{}}{0\atop{}}},\
\
\beta_4\,=\,{\scriptstyle{0\atop{}}{1\atop{}}{1\atop1}{1\atop{}}{0\atop{}}{0\atop{}}{0\atop{}}},\\
\
\beta_5\,=\,{\scriptstyle{1\atop{}}{2\atop{}}{2\atop1}{1\atop{}}{0\atop{}}{0\atop{}}{0\atop{}}},\
\
\beta_6\,=\,{\scriptstyle{1\atop{}}{1\atop{}}{1\atop0}{1\atop{}}{0\atop{}}{0\atop{}}{0\atop{}}},\
\
\beta_7\,=\,{\scriptstyle{0\atop{}}{0\atop{}}{1\atop1}{1\atop{}}{1\atop{}}{0\atop{}}{0\atop{}}},\
\
\beta_8\,=\,{\scriptstyle{1\atop{}}{2\atop{}}{3\atop1}{2\atop{}}{1\atop{}}{1\atop{}}{1\atop{}}},
\end{eqnarray*}
and hence contains a root subsystem $\Phi_0$ of type ${\sf
D_5+A_3}$. On the other hand, the group ${\rm Aut}\,\g$ contains an
automorphism $\sigma$ of order $4$ such that $\sigma(t)=t$ for all
$t\in\Lie(T)$ and $\Phi_\lambda$ is the root system of the Lie
subalgebra $\g^\sigma$ with respect to $\Lie(T)$. If $\Phi_0$ is a
proper root subsystem of $\Phi_\lambda$, then using the Borel--de
Siebenthal algorithm it is straightforward to see that
$\widetilde{\Phi}_0$ must have type ${\sf D_8}$. But this would
imply that $\sigma^2=1$, a contradiction. We thus deduce that
$\Phi_0=\Phi_\lambda$ and the roots $\{\beta_i\,|\,\,1\le i\le 8\}$
form the basis of $\Phi_{\lambda}$ contained in $\Phi^+$. Since
$\lambda+\rho$ is strongly dominant on $\Phi^+\cap\Phi_\lambda$ and
$$|\Phi^+|-|\Phi_\lambda^+|=120-20-6=94=\textstyle{\frac{1}{2}}(248-60)=\frac{1}{2}\dim\,\O(e),$$
applying [\cite{Jo1}, Corollary~3.5] gives $\dim\,{\rm
VA}(I(\lambda))=\dim\,\O(e)$. But $e\in{\rm VA}(I(\lambda))$ because
we already know that $\lambda+\rho$ satisfies Condition~(A). Hence
${\rm VA}(I(\lambda))=\overline{\O(e)}$ i.e. Condition~(B) holds for
$\lambda+\rho$. Since in the present case $\g_e=[\g_e,\g_e]$ by
[\cite{dG}], we combine [\cite{Lo2}, 5.3] and [\cite{PT},
Proposition~11] to conclude $I(-\frac{3}{4}\rho)$ is the only
multiplicity-free primitive ideal in $\mathcal{X}_{\O(e)}$.
\subsection{Type $({\sf E_8, A_4+A_3})$}\label{3.15} From
this case onwards we are in a partially uncharted territory: the existence of
one-dimensional representations for  the remaining three rigid
orbits in type ${\sf E_8}$ was left open in [\cite{GRU}] and in
Ubly's subsequent work [\cite{U}]. Apparently, expressing commutators in $U(\g,e)$ in terms
of a PBW basis and storing related data on a computer required too much memory for the remaining orbits and the computer would never stop running. In our approach, however, the present case is rather straightforward especially when compared with the two rigid orbits of dimension $202$ (see Sebsections~\ref{3.16} and \ref{3.17}).

First we note that $e$ has a unique
pinning, namely
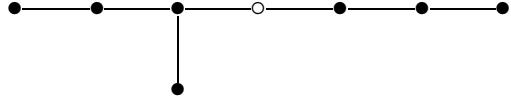
\[\xymatrix{*{\bullet}\ar@{-}[r]& *{\bullet} \ar@{-}[r] & *{\bullet} \ar@{-}[r]\ar@{-}[d] & *{\circ} \ar@{-}[r]
 & *{\bullet} \ar@{-}[r]&*{\bullet}\ar@{-}[r]&*{\bullet}\\ && *{\bullet}&&&}\]
and we can take
$$\tau=\textstyle{{2\atop{}}{2\atop{}}{2\atop 2}
{(-9)\atop{}}{2\atop{}}{2\atop{}}{2\atop{}}}$$ as an optimal
cocharacter; see [\cite{LT}]. Since $\dim\,\g_e=48$, the total number of positive
$0$-roots and $1$-roots is $(48-8)/2=20$. The roots are listed below.
\begin{table}[htb]
\label{data2}
\begin{tabular}{|c|c|}
\hline$0$-roots  & $1$-roots
\\ \hline
$\scriptstyle{{0\atop{}}{1\atop{}}{2\atop1}{2\atop{}}{2\atop{}}{2\atop{}}{1\atop{}}}$,
$\scriptstyle{{1\atop{}}{1\atop{}}{2\atop1}{2\atop{}}{2\atop{}}{1\atop{}}{1\atop{}}}$,
$\scriptstyle{{1\atop{}}{2\atop{}}{2\atop1}{2\atop{}}{1\atop{}}{1\atop{}}{1\atop{}}}$,
$\scriptstyle{{1\atop{}}{2\atop{}}{2\atop1}{2\atop{}}{2\atop{}}{1\atop{}}{0\atop{}}}$,\,
&$\scriptstyle{{0\atop{}}{0\atop{}}{1\atop1}{1\atop{}}{1\atop{}}{1\atop{}}{1\atop{}}}$,
$\scriptstyle{{0\atop{}}{1\atop{}}{1\atop0}{1\atop{}}{1\atop{}}{1\atop{}}{1\atop{}}}$,
$\scriptstyle{{0\atop{}}{1\atop{}}{1\atop1}{1\atop{}}{1\atop{}}{1\atop{}}{0\atop{}}}$,
$\scriptstyle{{1\atop{}}{1\atop{}}{1\atop0}{1\atop{}}{1\atop{}}{1\atop{}}{0\atop{}}}$,\\
$\scriptstyle{{1\atop{}}{2\atop{}}{3\atop1}{2\atop{}}{1\atop{}}{1\atop{}}{0\atop{}}}$,
$\scriptstyle{{1\atop{}}{2\atop{}}{3\atop2}{2\atop{}}{1\atop{}}{0\atop{}}{0\atop{}}}$,
$\scriptstyle{{1\atop{}}{3\atop{}}{5\atop3}{4\atop{}}{3\atop{}}{2\atop{}}{1\atop{}}}$,
$\scriptstyle{{2\atop{}}{3\atop{}}{5\atop2}{4\atop{}}{3\atop{}}{2\atop{}}{1\atop{}}}$.
&$\scriptstyle{{1\atop{}}{1\atop{}}{1\atop1}{1\atop{}}{1\atop{}}{0\atop{}}{0\atop{}}}$,
$\scriptstyle{{0\atop{}}{1\atop{}}{2\atop1}{1\atop{}}{1\atop{}}{0\atop{}}{0\atop{}}}$,
$\scriptstyle{{1\atop{}}{1\atop{}}{2\atop1}{1\atop{}}{0\atop{}}{0\atop{}}{0\atop{}}}$,
$\scriptstyle{{1\atop{}}{2\atop{}}{3\atop2}{3\atop{}}{3\atop{}}{2\atop{}}{1\atop{}}}$,\\
&$\scriptstyle{{1\atop{}}{2\atop{}}{4\atop2}{3\atop{}}{2\atop{}}{2\atop{}}{1\atop{}}}$,
$\scriptstyle{{2\atop{}}{3\atop{}}{4\atop2}{3\atop{}}{2\atop{}}{1\atop{}}{0\atop{}}}$,
$\scriptstyle{{1\atop{}}{3\atop{}}{4\atop2}{3\atop{}}{2\atop{}}{1\atop{}}{1\atop{}}}$,
$\scriptstyle{{2\atop{}}{4\atop{}}{6\atop3}{5\atop{}}{4\atop{}}{3\atop{}}{1\atop{}}}$.
\\
\hline
\end{tabular}
\end{table}

The $5$-contributions of the $0$-roots and $1$-roots are $20$ and
$24$, respectively, and the total contribution of all roots is $22$.
In the present case the orbit $\O(e)$ is non-special and comparing
$\dim \O(e)$ with the size of a root subsystem of type ${\sf
A_4+A_4}$ in $\Phi$ strongly suggests that one should seek
$\lambda+\rho=\sum_{i=1}^8a_i\varpi_i$ such that $a_i\in
\frac{1}{5}\Z$ for all $1\le i\le 8$. Setting
$\widetilde{a}_i:=5a_i$ we can rewrite Condition (C) for
$\lambda+\rho$ as follows:
$$
8\widetilde{a}_1+12\widetilde{a}_2+16\widetilde{a}_3+24\widetilde{a}_4+
20\widetilde{a}_5+15\widetilde{a}_6+10\widetilde{a}_7+5\widetilde{a}_8=110.
$$
A perfect solution to this linear equation is given by setting
$\widetilde{a}_i=1$ for all $1\le i\le 8$, which leads to
$\lambda+\rho=\frac{1}{5}\rho$. Since
$\langle\lambda+\rho,\gamma^\vee\rangle\not\in\Z^{>0}$ for any
$\gamma\in\Phi^+$ which can be expressed as a linear combination of
$\alpha_i$ with $i\in \{1,2,3,4,6,7,8\}$, this weight does satisfy
Condition~(A). Therefore, $e\in{\rm VA}(I(\lambda))$.

Next we observe that the integral root system of our weight
$\lambda+\rho=\textstyle{\frac{1}{5}}\rho$ contains the positive
roots
$$
\xymatrix{{\beta_1}\ar@{-}[r]& {\beta_2} \ar@{-}[r] & {\beta_3}
\ar@{-}[r] & {\beta_4}
 & {\beta_5}\ar@{-}[r] &{\beta_6}\ar@{-}[r] &{\beta_7}\ar@{-}[r]& {\beta_8}}
$$
where
\begin{eqnarray*}
\beta_1\,=\,{\scriptstyle{0\atop{}}{1\atop{}}{1\atop1}{1\atop{}}{1\atop{}}{0\atop{}}{0\atop{}}},\
\
\beta_2\,=\,{\scriptstyle{0\atop{}}{0\atop{}}{1\atop0}{1\atop{}}{1\atop{}}{1\atop{}}{1\atop{}}},\
\
\beta_3\,=\,{\scriptstyle{1\atop{}}{1\atop{}}{1\atop1}{1\atop{}}{0\atop{}}{0\atop{}}{0\atop{}}},\
\
\beta_4\,=\,{\scriptstyle{0\atop{}}{1\atop{}}{1\atop0}{1\atop{}}{1\atop{}}{1\atop{}}{0\atop{}}},\\
\beta_5\,=\,{\scriptstyle{0\atop{}}{1\atop{}}{2\atop1}{1\atop{}}{0\atop{}}{0\atop{}}{0\atop{}}},\
\
\beta_6\,=\,{\scriptstyle{1\atop{}}{1\atop{}}{1\atop0}{1\atop{}}{1\atop{}}{0\atop{}}{0\atop{}}},\
\
\beta_7\,=\,{\scriptstyle{0\atop{}}{0\atop{}}{1\atop1}{1\atop{}}{1\atop{}}{1\atop{}}{0\atop{}}},\
\
\beta_8\,=\,{\scriptstyle{1\atop{}}{2\atop{}}{2\atop1}{1\atop{}}{1\atop{}}{1\atop{}}{1\atop{}}},
\end{eqnarray*}
and hence has type ${\sf A_4+A_4}$ (by the maximality of that root
subsystem). The roots $\{\beta_i\,|\,\,1\le i\le 8\}$ form the basis
of $\Phi_{\lambda}$ contained in $\Phi^+$ and $\lambda+\rho$ is
strongly dominant on $\Phi^+\cap\Phi_\lambda$. Since
$$|\Phi^+|-|\Phi_\lambda^+|=120-10-10=100=\textstyle{\frac{1}{2}}(248-48)=\frac{1}{2}\dim\,\O(e),$$
applying [\cite{Jo1}, Corollary~3.5] we get $\dim\,{\rm
VA}(I(\lambda))=\dim\,\O(e)$. Since $e\in{\rm VA}(I(\lambda))$ we
have that ${\rm VA}(I(\lambda))=\overline{\O(e)}$, i.e.
Condition~(B) holds for $\lambda+\rho$.  Since $\g_e=[\g_e,\g_e]$ by
[\cite{dG}], we argue as before to conclude that
$I(-\frac{4}{5}\rho)$ is the only multiplicity-free primitive ideal
in $\mathcal{X}_{\O(e)}$.
\subsection{Type $({\sf E_8, A_5+A_1})$}\label{3.16} Here $\g_e=\mathbb{C}e\oplus[\g_e,\g_e]$ by [\cite{dG}], so we may expect
$U(\g,e)$ to afford at least $2$ one-dimensional
representation. One such representation was found by Losev in [\cite{Lo2}] and in this subsection we are going to construct another one.

We choose the following pinning for $e$:
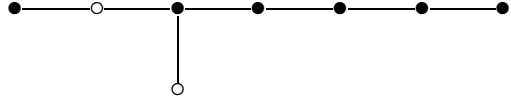
\[\xymatrix{*{\bullet}\ar@{-}[r]& *{\circ} \ar@{-}[r] & *{\bullet} \ar@{-}[r]\ar@{-}[d] & *{\bullet} \ar@{-}[r]
 & *{\bullet} \ar@{-}[r]&*{\bullet}\ar@{-}[r]&*{\bullet}\\ && *{\circ}&&&}\]
and take
$$\tau=\textstyle{{2\atop{}}{(-6)\atop{}}{2\atop (-5)}
{2\atop{}}{2\atop{}}{2\atop{}}{2\atop{}}}$$ as an optimal
cocharacter. Since $\dim\,\g_e=46$, the total number of positive
$0$-roots and $1$-roots is $(46-8)/2=19$. These roots are given in the table below.
\begin{table}[htb]
\label{data2}
\begin{tabular}{|c|c|}
\hline$0$-roots  & $1$-roots
\\ \hline
$\scriptstyle{1\atop{}}{1\atop{}}{1\atop0}{1\atop{}}{0\atop{}}{0\atop{}}{0\atop{{}}}$,
$\scriptstyle{{0\atop{}}{1\atop{}}{1\atop0}{1\atop{}}{1\atop{}}{0\atop{}}{0\atop{}}}$,
$\scriptstyle{{1\atop{}}{2\atop{}}{3\atop2}{2\atop{}}{2\atop{}}{2\atop{}}{1\atop{}}}$,
$\scriptstyle{{1\atop{}}{2\atop{}}{3\atop2}{3\atop{}}{2\atop{}}{1\atop{}}{1\atop{}}}$,\,
&$\scriptstyle{{0\atop{}}{0\atop{}}{1\atop1}{1\atop{}}{1\atop{}}{0\atop{}}{0\atop{}}}$,
$\scriptstyle{{1\atop{}}{1\atop{}}{1\atop1}{1\atop{}}{1\atop{}}{1\atop{}}{1\atop{}}}$,
$\scriptstyle{{1\atop{}}{1\atop{}}{2\atop1}{1\atop{}}{1\atop{}}{1\atop{}}{0\atop{}}}$,
$\scriptstyle{{0\atop{}}{1\atop{}}{2\atop1}{1\atop{}}{1\atop{}}{1\atop{}}{1\atop{}}}$,\\
$\scriptstyle{{1\atop{}}{2\atop{}}{4\atop2}{3\atop{}}{2\atop{}}{1\atop{}}{0\atop{}}}$,
$\scriptstyle{{1\atop{}}{3\atop{}}{4\atop2}{3\atop{}}{3\atop{}}{2\atop{}}{1\atop{}}}$,
$\scriptstyle{{2\atop{}}{3\atop{}}{4\atop2}{3\atop{}}{2\atop{}}{2\atop{}}{1\atop{}}}$,
$\scriptstyle{{2\atop{}}{4\atop{}}{5\atop2}{4\atop{}}{3\atop{}}{2\atop{}}{1\atop{}}}$.
&$\scriptstyle{{0\atop{}}{1\atop{}}{2\atop1}{2\atop{}}{1\atop{}}{1\atop{}}{0\atop{}}}$,
$\scriptstyle{{1\atop{}}{1\atop{}}{2\atop1}{2\atop{}}{1\atop{}}{0\atop{}}{0\atop{}}}$,
$\scriptstyle{{1\atop{}}{2\atop{}}{2\atop1}{2\atop{}}{2\atop{}}{1\atop{}}{1\atop{}}}$,
$\scriptstyle{{1\atop{}}{2\atop{}}{3\atop1}{2\atop{}}{2\atop{}}{1\atop{}}{0\atop{}}}$,\\
&$\scriptstyle{{1\atop{}}{2\atop{}}{3\atop1}{2\atop{}}{1\atop{}}{1\atop{}}{1\atop{}}}$,
$\scriptstyle{{2\atop{}}{3\atop{}}{5\atop3}{4\atop{}}{3\atop{}}{2\atop{}}{1\atop{}}}$,
$\scriptstyle{{2\atop{}}{4\atop{}}{6\atop3}{5\atop{}}{4\atop{}}{2\atop{}}{1\atop{}}}$.
\\
\hline
\end{tabular}
\end{table}

The $(2,3)$-contributions of the $0$-roots and $1$-roots are $(12,18)$ and
$(15,18)$, respectively, and the total contribution of all roots is $\big(\frac{27}{2}, 18\big)$.
Note that $|\Phi^+|-19=101=\dim\O(e))/2.$
Since the orbit $\O(e)$ is non-special and $\Phi$
contains a root subsystem of type ${\sf
A_5+A_2+A_1}$ whose set of positive roots has size $(30+6+2)/2=19$, we should seek
$\lambda+\rho=\sum_{i=1}^8a_i\varpi_i$ such that $a_i\in
\frac{1}{6}\Z$ for all $1\le i\le 8$. Setting
$\widetilde{a}_i:=6a_i$ we can rewrite Condition (C) for
$\lambda+\rho$ as follows:
\begin{eqnarray*}
5\widetilde{a}_1+8\widetilde{a}_2+10\widetilde{a}_3+
15\widetilde{a}_4+
12\widetilde{a}_5+9\widetilde{a}_6+6\widetilde{a}_7+
3\widetilde{a}_8&=&81\\
7\widetilde{a}_1+10\widetilde{a}_2+14\widetilde{a}_3+
20\widetilde{a}_4+
16\widetilde{a}_5+12\widetilde{a}_6+8\widetilde{a}_7+
4\widetilde{a}_8&=&108.
\end{eqnarray*}
A very nice solution to this system of linear equations is given by setting $\widetilde{a}_1=\widetilde{a}_2=2$ and
$\widetilde{a}_i=1$ for all $3\le i\le 8$, which leads to
$\lambda+\rho=\frac{1}{6}(\rho+\varpi_1+\varpi_2)$. Since
$\langle\lambda+\rho,\gamma^\vee\rangle\not\in\Z^{>0}$ for any
$\gamma\in\Phi^+$ which can be expressed as a linear combination of
$\alpha_i$ with $i\in \{1,4,5,6,7,8\}$, this weight does satisfy
Condition~(A). Therefore, $e\in{\rm VA}(I(\lambda))$.

Next we observe that the integral root system
$\lambda+\rho$ contains the linearly independent positive
roots
$$
\xymatrix{{\beta_1}\ar@{-}[r]& {\beta_2} \ar@{-}[r] & {\beta_3}
\ar@{-}[r] & {\beta_4}\ar@{-}[r]
 & {\beta_5}&{\beta_6}\ar@{-}[r] &{\beta_7}& {\beta_8}}
$$
where
\begin{eqnarray*}
\beta_1\,=\,{\scriptstyle{1\atop{}}{1\atop{}}{1\atop1}{0\atop{}}{0\atop{}}{0\atop{}}{0\atop{}}},\
\
\beta_2\,=\,{\scriptstyle{0\atop{}}{1\atop{}}{1\atop0}{1\atop{}}{1\atop{}}{1\atop{}}{1\atop{}}},\
\
\beta_3\,=\,{\scriptstyle{0\atop{}}{0\atop{}}{1\atop1}{1\atop{}}{1\atop{}}{1\atop{}}{0\atop{}}},\
\
\beta_4\,=\,{\scriptstyle{1\atop{}}{1\atop{}}{1\atop0}{1\atop{}}{1\atop{}}{0\atop{}}{0\atop{}}},\\
\beta_5\,=\,{\scriptstyle{0\atop{}}{1\atop{}}{2\atop1}{1\atop{}}{0\atop{}}{0\atop{}}{0\atop{}}},\
\
\beta_6\,=\,{\scriptstyle{0\atop{}}{1\atop{}}{1\atop1}{1\atop{}}{1\atop{}}{0\atop{}}{0\atop{}}},\
\
\beta_7\,=\,{\scriptstyle{1\atop{}}{1\atop{}}{2\atop1}{2\atop{}}{1\atop{}}{1\atop{}}{1\atop{}}},\
\
\beta_8\,=\,{\scriptstyle{1\atop{}}{2\atop{}}{2\atop1}{2\atop{}}{1\atop{}}{1\atop{}}{0\atop{}}},
\end{eqnarray*}
and hence contains a root subsystem of type ${\sf A_5+A_2+A_1}$. Note that $\langle\lambda+\rho,\beta^\vee\rangle>0$ for all $\beta\in\Phi^+$. Keeping this in mind it is straightforward to see that $\{\beta_i\,|\,1\le i\le 6\}$ coincides with the set of all
$\beta\in\Phi^+$ for which $\langle\lambda+\rho,\beta^\vee\rangle=1\}$. Since $\langle\lambda+\rho,\beta_i^\vee\rangle=2$ for $i=7,8$, it follows that
the set $\{\beta_i\,|\,1\le i\le 8\}$ consists of  {\it indecomposable} roots of $\Phi_\lambda^+$
and hence forms a basis of $\Phi_\lambda$. In particular, $\Phi_\lambda$ has type ${\sf A_5+A_2+A_1}$.

Since $\lambda+\rho$ is
strongly dominant on $\Phi_\lambda^+$ and
$|\Phi^+|-|\Phi_\lambda^+|=120-19=101=\frac{1}{2}\dim\,\O(e),$
applying [\cite{Jo1}, Corollary~3.5] yields $\dim\,{\rm
VA}(I(\lambda))=\dim\,\O(e)$. Since $e\in{\rm VA}(I(\lambda))$ we
have the equality ${\rm VA}(I(\lambda))=\overline{\O(e)}$. Thus
Condition~(B) holds for $\lambda+\rho$ and we conclude that
$I(\frac{1}{6}\varpi_1+\frac{1}{6}\varpi_2-\frac{5}{6}\rho)$ is a 
multiplicity-free primitive ideal
in $\mathcal{X}_{\O(e)}$ (one should keep in mind here that in the present case the component group $\Gamma=G_e/G_e^\circ$ is trivial). It will become clear later that the corresponding representation of $U(\g,e)$ coincides with the one constructed in [\cite{Lo2}].

Our next goal is to produce a
highest weight, $\lambda'$, leading to
another one-dimensional representation of $U(\g,e)$. Since it turned out to be impossible to find $\lambda'$ by a lucky guess, we are going to adopt a more scientific approach used in Subsection~\ref{3.11}.
Namely, we impose that $\Phi_{\lambda'}=\Phi_\lambda$ and let $\varpi'_1,\ldots,\varpi'_8\in P(\Phi)_{\mathbb Q}$ be such that
$\langle\varpi_i,\beta_j^\vee\rangle=\delta_{ij}$
for all $1\le i,j\le 8$.
It is well known (and easily seen) that
\begin{eqnarray*}
\varpi'_1&=&\textstyle{\frac{1}{6}}(5\beta_1+4\beta_2+3\beta_3+2\beta_4+\beta_5),\qquad
\ \varpi'_5\ =\ \textstyle{\frac{1}{6}}(\beta_1+2\beta_2+3\beta_3+4\beta_4+5\beta_5),
\\
\varpi'_2&=&\textstyle{\frac{1}{6}}(4\beta_1+8\beta_2+6\beta_3+4\beta_4+2\beta_5),\qquad
\varpi'_6\ =\ \textstyle{\frac{1}{3}}(2\beta_6+\beta_7),
\\
\varpi'_3&=&\textstyle{\frac{1}{6}}(3\beta_1+6\beta_2+9\beta_3+6\beta_4+3\beta_5),\qquad
\varpi'_7\ =\ \textstyle{\frac{1}{3}}(\beta_6+2\beta_7),\\
\varpi'_4&=&\textstyle{\frac{1}{6}}(2\beta_1+4\beta_2+6\beta_3+8\beta_4+4\beta_5), \qquad
\varpi'_8\ =\ \textstyle{\frac{1}{2}}\beta_8.
\end{eqnarray*}
Since $\Phi_\lambda=\Phi_{\lambda'}$ and $\lambda'+\rho=\sum_{i=1}^8a'_i\varpi_i$ satisfies Losev's conditions, it can be expressed as
$\lambda'+\rho=\sum_{i=1}^8b_i\varpi'$ for some $b_i\in\Z^{>0}$. When we evaluate $\langle\lambda'+\rho,\alpha_i^\vee\rangle$ for $1\le i\le 8$ by using the explicit form of $\beta_i$'s
and the above expressions for  $\varpi'_i$'s we obtain the following formulae:
\begin{eqnarray*}
a_1'&=&\textstyle{\frac{1}{3}}(b_1-b_2+b_4-b_5-b_6+b_7),\\
a_2'&=&\textstyle{\frac{1}{3}}(b_1-b_2-2b_4-b_5+2b_6+b_7),\\
a_3'&=&\textstyle{\frac{1}{6}}(b_1+2b_2-3b_3-2b_4-b_5+2b_6-2b_7+3b_8),\\
a_4'&=&\textstyle{\frac{1}{6}}(b_1+2b_2+3b_3+4b_4+
5b_5-4b_6-2b_7-3b_8),\\
a_5'&=&\textstyle{\frac{1}{6}}(-5b_1-4b_2-3b_3-2b_4-
b_5+2b_6+4b_7+3b_8),\\
a_6'&=&\textstyle{\frac{1}{6}}(b_1+2b_2+3b_3+4b_4-
b_5+2b_6-2b_7-3b_8),\\
a_7'&=&\textstyle{\frac{1}{6}}(b_1+2b_2+3b_3-2b_4-
b_5-4b_6-2b_7+3b_8),\\
a_8'&=&\textstyle{\frac{1}{6}}(b_1+2b_2-3b_3-2b_4-
b_5+2b_6+4b_7-3b_8).
\end{eqnarray*}
Rewriting Condition~(C) in terms of $b_i$'s we then obtain
\begin{eqnarray*}
9b_1+12b_2+15b_3+12b_4+9b_5+6b_6+6b_7+3b_8&=&81\\
12b_1+18b_2+18b_3+18b_4+12b_5+6b_6+6b_7+6b_8&=&108.
\end{eqnarray*}
One surprising thing about this system of linear equations is that all its coefficients are positive integers and all its solutions $(b_1,\ldots, b_8)$ with
$b_i\in\Z^{>0}$ are easy to determine.
Setting $b_i=1$ for $i\in\{1,2,3,4,5,7\}$ and $b_6=b_8=2$
we recover our weight $\lambda+\rho$, whilst setting
$b_i=1$ for $1\le i\le 6$ and $b_7=b_8=2$ gives us the same representation of $U(\g,e)$.
However, we have yet another promising option here, namely,
$b_1=2$ and $b_i=1$ for $2\le i\le 8$
(we can also set $b_i=1$ for $i\ne 5$ and $b_5=2$, but this would lead to the same representation of $U(\g,e)$). Then $\lambda'+\rho=2\varpi'_1+\sum_{i=2}^8\varpi'_i$. Using our explicit expressions for $\varpi'_i$ and $\beta_i$ we can now determine $\langle\lambda'+\rho,\alpha_i^\vee\rangle$ for all $1\le i\le 8$ and hence rewrite $\lambda'+\rho$ as a linear combination of the $\varpi_i$'s. After some computations we obtain that
$$
\lambda'+\rho\,=\,\textstyle{\frac{1}{3}}\varpi_1+
\textstyle{\frac{1}{3}}\varpi_2+\textstyle{\frac{1}{6}}\varpi_3+\textstyle{\frac{7}{6}}\varpi_4
-\textstyle{\frac{11}{6}}\varpi_5+\textstyle{\frac{7}{6}}\varpi_6+\textstyle{\frac{1}{6}}\varpi_7+
\textstyle{\frac{1}{6}}\varpi_8.
$$
Since $\lambda+\rho$ and $\lambda'+\rho$ have the same image in $\frac{1}{6}P(\Phi)/P(\Phi)$, it is straightforward to see that Condition~(A) holds for $\lambda'+\rho$ and $\Phi_\lambda=\Phi_{\lambda'}$
(a priori we only have the inclusion $\Phi_\lambda\subseteq \Phi_{\lambda'}$). But then Condition~(C) holds for $\lambda'+\rho$ as well. Since Condition~(D) is vacuous in the present case,
we deduce that $I(\lambda')$ is another multiplicity free primitive ideal in $\mathcal{X}_{\O(e)}$.

Finally, the primitive ideals $I(\lambda)$ and $I(\lambda')$ are distinct. Indeed, otherwise
$w(\lambda+\rho)=\lambda'+\rho$ for some $w\in W$. But since $\Phi_\lambda=\Phi_{\lambda'}$ this would imply that $w\in N_W(\Phi_\lambda)$ forcing $w^{-1}(\beta_8)=\pm\beta_8$. However, this contradicts
the fact that $1=\langle\lambda'+\rho,\beta_8^\vee\rangle\neq\pm2$. Therefore, $I(\lambda)\ne I(\lambda')$. We shall see later that these two are the only multiplicity free primitive ideals in $\mathcal{X}_{\O(e)}$.
\subsection{Type $({\sf E_8, D_5(a_1)+A_2})$}\label{3.17} This case is hard and no completely prime primitive ideals in $\mathcal{X}_{\O(e)}$ have been recorded in the literature until now. The orbit $\O(e)$ is non-special and we again have that $\g_e=\mathbb{C}e\oplus[\g_e,\g_e]$ by [\cite{dG}].
The pinning for $e$ is unique, namely,
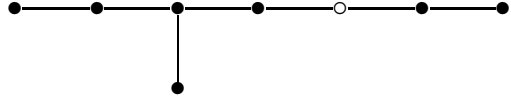
\[\xymatrix{*{\bullet}\ar@{-}[r]& *{\bullet} \ar@{-}[r] & *{\bullet} \ar@{-}[r]\ar@{-}[d] & *{\bullet} \ar@{-}[r]
 & *{\circ} \ar@{-}[r]&*{\bullet}\ar@{-}[r]&*{\bullet}\\ && *{\bullet}&&&}\]
and we may (and will) assume that
$e=e_0+e_{\alpha_7}+e_{\alpha_8}$ where
$$e_0=e_{\alpha_1}+e_{\alpha_2}+e_{\alpha_3}+
e_{\alpha_5}+e_{\alpha_2+\alpha_4}+
e_{\alpha_4+\alpha_5}.$$ Then we can take
$$\tau=\textstyle{{2\atop{}}{2\atop{}}{0\atop 2}
{2\atop{}}{(-9)\atop{}}{2\atop{}}{2\atop{}}}$$ as an optimal
cocharacter for $e$; see [\cite{LT}]. Since $\dim\,\g_e=46$, the total number of positive
$0$-roots and $1$-roots is $(46-8)/2=19$. These roots are given in the table below.
\begin{table}[htb]
\label{data2}
\begin{tabular}{|c|c|}
\hline$0$-roots  & $1$-roots
\\ \hline
$\scriptstyle{0\atop{}}{0\atop{}}{1\atop0}{0\atop{}}{0\atop{}}{0\atop{}}{0\atop{{}}}$,
$\scriptstyle{{1\atop{}}{2\atop{}}{2\atop1}{2\atop{}}{2\atop{}}{2\atop{}}{1\atop{}}}$,
$\scriptstyle{{1\atop{}}{2\atop{}}{3\atop1}{2\atop{}}{2\atop{}}{2\atop{}}{1\atop{}}}$,
$\scriptstyle{{1\atop{}}{2\atop{}}{3\atop2}{2\atop{}}{2\atop{}}{1\atop{}}{1\atop{}}}$,\,
&$\scriptstyle{{1\atop{}}{1\atop{}}{1\atop0}{1\atop{}}{1\atop{}}{1\atop{}}{1\atop{}}}$,
$\scriptstyle{{1\atop{}}{1\atop{}}{1\atop1}{1\atop{}}{1\atop{}}{1\atop{}}{0\atop{}}}$,
$\scriptstyle{{0\atop{}}{1\atop{}}{1\atop1}{1\atop{}}{1\atop{}}{1\atop{}}{1\atop{}}}$,
$\scriptstyle{{0\atop{}}{1\atop{}}{2\atop1}{1\atop{}}{1\atop{}}{1\atop{}}{1\atop{}}}$,\\
$\scriptstyle{{1\atop{}}{2\atop{}}{3\atop2}{3\atop{}}{2\atop{}}{1\atop{}}{0\atop{}}}$,
$\scriptstyle{{1\atop{}}{2\atop{}}{3\atop1}{3\atop{}}{2\atop{}}{1\atop{}}{1\atop{}}}$,
$\scriptstyle{{1\atop{}}{2\atop{}}{4\atop2}{3\atop{}}{2\atop{}}{1\atop{}}{0\atop{}}}$,
$\scriptstyle{{2\atop{}}{4\atop{}}{6\atop3}{5\atop{}}{4\atop{}}{3\atop{}}{1\atop{}}}$.
&$\scriptstyle{{1\atop{}}{1\atop{}}{2\atop1}{1\atop{}}{1\atop{}}{1\atop{}}{0\atop{}}}$,
$\scriptstyle{{1\atop{}}{2\atop{}}{2\atop1}{1\atop{}}{1\atop{}}{0\atop{}}{0\atop{}}}$,
$\scriptstyle{{1\atop{}}{1\atop{}}{2\atop1}{2\atop{}}{1\atop{}}{0\atop{}}{0\atop{}}}$,
$\scriptstyle{{0\atop{}}{1\atop{}}{2\atop1}{2\atop{}}{1\atop{}}{1\atop{}}{0\atop{}}}$,\\
&$\scriptstyle{{1\atop{}}{3\atop{}}{5\atop3}{4\atop{}}{3\atop{}}{2\atop{}}{1\atop{}}}$,
$\scriptstyle{{2\atop{}}{3\atop{}}{4\atop2}{4\atop{}}{3\atop{}}{2\atop{}}{1\atop{}}}$,
$\scriptstyle{{2\atop{}}{3\atop{}}{5\atop2}{4\atop{}}{3\atop{}}{2\atop{}}{1\atop{}}}$.
\\
\hline
\end{tabular}
\end{table}
The $6$-contributions of the $0$-roots and $1$-roots are $16$ and
$17$, respectively, and the total contribution of all roots is $\frac{33}{2}$. The  minimal nilpotent orbit in a Lie algebra of type ${\sf D_5}$ is special and has dimension $14$. Since
$$120-(20+6)+(7+0)=\textstyle{\frac{1}{2}}\dim\O(e),$$ Proposition~5.3.2 from [\cite{Lo2}] suggests that we should look for a highest weight $\lambda$ with $\Phi_\lambda$ of type ${\sf D_5+A_3}$. The semisimple Lie algebra $\g(\lambda)=\g_1(\lambda)\oplus\g_2(\lambda)$ then has two minimal nonzero nilpotent orbits and we should be interested in the one contained in the simple ideal $\g_1(\lambda)$ of type ${\sf D_5}$ (we let $\g_2(\lambda)$ stand for the simple ideal of type ${\sf A_3}$).
In particular, this means that we should seek
$\lambda+\rho=\sum_{i=1}^8a_i\varpi_i$ such that $a_i\in
\frac{1}{4}\Z$ for all $1\le i\le 8$. Setting
$\widetilde{a}_i:=4a_i$ we can rewrite Condition~(C) for
$\lambda+\rho$ as follows:
$$
6\widetilde{a}_1+9\widetilde{a}_2+12\widetilde{a}_3+
18\widetilde{a}_4+
15\widetilde{a}_5+12\widetilde{a}_6+8\widetilde{a}_7+
4\widetilde{a}_8=66.
$$
This linear equation alone does not tell us very much, but it is supplemented by extremely strong restrictions coming from Conditions~(A), (B) and (D).
Firstly, it must be that $\langle\lambda+\rho,\beta^\vee\rangle \not\in \Z^{>0}$ for all $\beta\in\{\alpha_7,\alpha_8,\alpha_7+\alpha_8\}$.
Secondly, if we denote by $\l$ the standard Levi subalgebra of $\g$ generated by $\t$ and root vectors
$e_{\pm \alpha_i}$ with $1\le i\le 5$, then
the annihilator in $U(\l)$ of the irreducible highest weight $\l$-module
$L_\l(\lambda)$  should have the form
${\rm Ann}_{U(\l)}\big(Q_{\l,e_0}\otimes_{U(\l,e_0)}
\mathbb{C}_{\eta}\big)$ for some one-dimensional representation
$\eta$ of the finite $W$-algebra $U(\l,e_0)$ (here we regard
$e_0$ as a subregular nilpotent element of $\l$).
Thirdly, the primitive ideal $I_0(\lambda)$ of $U(\g(\lambda))$ must be generated by  $\g_2(\lambda)$ and a primitive ideal of $U(\g_1(\lambda))$ whose associated variety coincides with the Zariski closure of the minimal nilpotent orbit
of $\g_1(\lambda)$.

After some tedious numerical experiments (omitted) and a lucky guess (included) one comes up with a plausible candidate, namely,
$$\lambda+\rho=\varpi_4+\varpi_6-\textstyle{\frac{1}{4}}(\varpi_1+
\varpi_2+\varpi_3+\varpi_5+\varpi_7+\varpi_8),$$
which is obtained by setting
$\widetilde{a}_i=-1$ for $i\in\{1,2,3,5,7,8\}$ and
$\widetilde{a}_i=4$ for $i=4,6$.
It is straightforward to see that $\Phi_{\lambda}$ contains the positive roots
$$
\xymatrix{{\beta_1}\ar@{-}[r]& {\beta_2} \ar@{-}[r] & {\beta_3}
\ar@{-}[r]\ar@{-}[d] & {\beta_4}
 & {\beta_6} \ar@{-}[r]&{\beta_7}\ar@{-}[r] &{\beta_8}\\ && {\beta_5}&&&}
$$
where
\begin{eqnarray*}
\beta_1\,=\,{\scriptstyle{0\atop{}}{1\atop{}}{1\atop1}{1\atop{}}{1\atop{}}{1\atop{}}{0\atop{}}},\
\
\beta_2\,=\,{\scriptstyle{0\atop{}}{0\atop{}}{1\atop0}{0\atop{}}{0\atop{}}{0\atop{}}{0\atop{}}},\
\
\beta_3\,=\,{\scriptstyle{1\atop{}}{1\atop{}}{1\atop1}{1\atop{}}{0\atop{}}{0\atop{}}{0\atop{}}},\
\
\beta_4\,=\,{\scriptstyle{0\atop{}}{1\atop{}}{1\atop0}{1\atop{}}{1\atop{}}{1\atop{}}{1\atop{}}},\\
\
\beta_5\,=\,{\scriptstyle{0\atop{}}{0\atop{}}{0\atop0}{0\atop{}}{1\atop{}}{0\atop{}}{0\atop{}}},\
\
\beta_6\,=\,{\scriptstyle{0\atop{}}{0\atop{}}{1\atop1}{1\atop{}}{1\atop{}}{1\atop{}}{1\atop{}}},\
\
\beta_7\,=\,{\scriptstyle{1\atop{}}{1\atop{}}{1\atop0}{1\atop{}}{1\atop{}}{1\atop{}}{0\atop{}}},\
\
\beta_8\,=\,{\scriptstyle{0\atop{}}{1\atop{}}{2\atop1}{2\atop{}}{1\atop{}}{0\atop{}}{0\atop{}}}.
\end{eqnarray*}
Moreover, repeating verbatim the argument used in Subsection~\ref{3.14} one find out that
the roots $\{\beta_i\,|\,\,1\le i\le 8\}$
form the basis of $\Phi_{\lambda}$ contained in $\Phi^+$. In particular, $\Phi_\lambda$ has type ${\sf D_5+A_3}$.

In order to simplify our further exposition we postpone
verifying Losev's conditions for $\lambda+\rho$ and
concentrate on finding another weight, $\lambda'+\rho=\sum_{i=1}^8a'_i\varpi_i$, leading to
a one-dimensional representation of $U(\g,e)$. We proceed as in the previous subsection, that is,
make the assumption $\Phi_{\lambda'}=\Phi_\lambda$ and let $\varpi'_1,\ldots,\varpi'_8\in P(\Phi)_{\mathbb Q}$ be such that
$\langle\varpi_i,\beta_j^\vee\rangle=\delta_{ij}$
for all $1\le i,j\le 8$.
Using Bourbaki' tables [\cite{Bo}] we then get
\begin{eqnarray*}
\varpi'_1&=&\textstyle{\frac{1}{2}}(2\beta_1+2\beta_2+2\beta_3+\beta_4+\beta_5),\qquad
\ \ \,\varpi'_5\ =\ \textstyle{\frac{1}{4}}(2\beta_1+4\beta_2+6\beta_3+3\beta_4+5\beta_5),
\\
\varpi'_2&=&\beta_1+2\beta_2+2\beta_3+\beta_4+\beta_5,\qquad\qquad\ \
\varpi'_6\ =\ \textstyle{\frac{1}{4}}(3\beta_6+2\beta_6+\beta_7),
\\
\varpi'_3&=&\textstyle{\frac{1}{2}}(2\beta_1+4\beta_2+6\beta_3+3\beta_4+3\beta_5),\qquad
\varpi'_7\ =\ \textstyle{\frac{1}{4}}(2\beta_6+4\beta_6+2\beta_7),\\
\varpi'_4&=&\textstyle{\frac{1}{4}}(2\beta_1+4\beta_2+6\beta_3+5\beta_4+3\beta_5), \qquad
\varpi'_8\ =\ \textstyle{\frac{1}{4}}(\beta_6+2\beta_7+3\beta_8).
\end{eqnarray*}
Since $\Phi_\lambda=\Phi_{\lambda'}$, it must be that
$\lambda'+\rho=\sum_{i=1}^8b_i\varpi'_i$ for some $b_i\in\Z$. Using the explicit form of $\beta_i$'s
and the above expressions for  $\varpi'_i$'s one finds out that
\begin{eqnarray*}
a_1'&=&\textstyle{\frac{1}{4}}(-2b_1+2b_3-b_4+b_5+b_6+2b_7-b_8),\\
a_2'&=&\textstyle{\frac{1}{4}}(2b_1+2b_3-b_4+b_5+b_6-2b_7-b_8),\\
a_3'&=&\textstyle{\frac{1}{4}}(2b_1+2b_3+3b_4+b_5-3b_6-2b_7-b_8),\\
a_4'&=&b_2,\\
a_5'&=&\textstyle{\frac{1}{4}}(-2b_1-4b_2-2b_3-b_4-
3b_5+b_6+2b_7+3b_8),\\
a_6'&=&b_5,\\
a_7'&=&\textstyle{\frac{1}{4}}(2b_1-2b_3-b_4-
3b_5+b_6+2b_7-b_8)\\
a_8'&=&\textstyle{\frac{1}{4}}(-2b_1+2b_3+3b_4+b_5+
b_6-2b_7-b_8).
\end{eqnarray*}
A straightforward computation then shows that Condition~(C) can be rewritten in terms of $b_i$'s as follows:
$$
8b_1+12b_2+16b_3+10b_4+10b_5+6b_6+8b_7+6b_8=66.
$$
Conditions~(A), (B) and (D) impose severe restrictions on the solutions $(b_1,\ldots, b_8)\in\Z^8$. In particular, we must have
$b_i\in\Z^{>0}$ for $i\in\{6,7,8\}$
and the associated variety of the irreducible $\g_1(\lambda)$-module of highest weight $\sum_{i=1}^5(b_i-1)\varpi'_i$ should coincide with the closure of the minimal nilpotent orbit in $\g_1(\lambda)$.

Setting $b_i=1$ for $i\in\{1,2,4,5,6,7\}$, $b_3=0$, and $b_8=2$
we recover our weight $\lambda+\rho$, and it can be shown that setting
$b_i=1$ for $i\in\{1,2,3,5,7,8\}$, $b_3=0$, and $b_6=2$ leads to the same representation of $U(\g,e)$.
Quite surprisingly, careful inspection reveals that we have yet another promising option here, namely,
$b_3=-1$, $b_2=b_5=2$, and $b_i=1$ for $i\in\{1,4,6,7,8\}$
(we could also set $b_3=-1$, $b_2=b_4=2$ and $b_i=1$ for $i\in\{1,5,6,7,8\}$, but this would not lead to new
one-dimensional representations of $U(\g,e)$). Using the above expressions for $\varpi'_i$ and $\beta_i$ we  quickly determine $\langle\lambda'+\rho,\alpha_i^\vee\rangle$ for all $1\le i\le 8$ and obtain that
$$
\lambda'+\rho\,=\,-\textstyle{\frac{1}{4}}\varpi_1-
\textstyle{\frac{1}{4}}\varpi_2-\textstyle{\frac{1}{4}}\varpi_3+2\varpi_4
-\textstyle{\frac{9}{4}}\varpi_5+2\varpi_6-\textstyle{\frac{1}{4}}\varpi_7-
\textstyle{\frac{1}{4}}\varpi_8.
$$
Note that $\Phi_\lambda=\Phi_{\lambda'}$ because $\lambda'+\rho$ and $\lambda+\rho$ have the same image in $\frac{1}{4}P(\Phi)/P(\Phi)$.

Set $\g_0:=\g_1(\lambda)$, a simple Lie algebra with root system $\Phi_0=\Phi\cap\big(\bigoplus_{i=1}^5\Z\beta_i\big)$ of type ${\sf D_5}$, and identify $\{\beta_i\,|\,\,1\le i\le 5\}$ and $\{\varpi'_i\,|\,\,1\le i\le 5\}$ with a set of simple roots in $\Phi_0$  and the corresponding set of fundamental weights in $P(\Phi_0)$. Set $\rho_0:=\sum_{i=1}^5\varpi_i'$, $\lambda_0:=-\varpi'_3$ and $\lambda_0':=\varpi_2'-2\varpi_3'+\varpi'_5$.
Then $\lambda_0$ and $\lambda_0'$ coincide with the restrictions of $\lambda-\rho_0$ and $\lambda'-\rho_0$ to $\t\cap\g_0$. Let $I(\lambda_0)$ and $I(\lambda_0')$ be the annihilators in $U(\g_0)$ of the highest weight $\g_0$-modules $L(\lambda_0)$ and $L(\lambda_0')$.

In view of our earlier remarks in this subsection we have to check that
${\rm VA}(I(\lambda_0))={\rm VA}(I(\lambda_0'))$ coincides with the closure of the minimal nilpotent orbit in $\g_0$. Thanks to [\cite{BV}, Proposition~5.10] this is clear for
${\rm VA}(I(\lambda_0))$, because there exists an
$\frak{sl}_2$-triple $(e_1^\vee,h_1^\vee,f_1^\vee)$ in the Langlands dual Lie algebra $\g_0^\vee$ such that $e_1^\vee$ is subregular nilpotent and $\lambda_0+\rho_0=\frac{1}{2}h_1^\vee$.

Since $$\lambda'_0+\rho_0=\varpi_1'+2\varpi'_2
-\varpi_3'+\varpi'_4+2\varpi'_5=
s_{\beta_3}(\varpi_1'+\varpi_2'+\varpi_3'+\varpi_5'),$$
we have that $\lambda_0'=s_{\beta_3}s_{\beta_4}\centerdot\mu'$ where $\mu'=-\varpi'_4$ (of course, here $w\centerdot \mu'$ stands for $w(\mu'+\rho_0)-\rho_0$).
Note that $\langle\mu'+\rho_0,\beta_i^\vee\rangle\in \Z^{\geqslant 0}$ for all $1\le i\le 5$ and $\Pi_{\mu'}^0=\{\beta_4\}\subset
\tau(s_{\beta_3}s_{\beta_4})$.
So [\cite{Ja},
Corollar~10.10(a)] yields
$$d\big(U(\g_0)/I(\mu')\big)=
d\big(U(\g_0)/I(s_{\beta_3}s_{\beta_4}\centerdot\mu')\big)=
d\big(U(\g_0)/I(s_{\beta_3}s_{\beta_4}\centerdot\nu')\big)$$ for any
regular dominant $\nu'\in P(\Phi_0)$. Since all roots in $\Phi_0$ have
the same length, [\cite{Ja}, Corollar~10.10(c)] entails that
$$d\big(U(\g_0)/I(s_{\beta_3}s_{\beta_4}\centerdot\nu')
\big)=
d\big(U(\g_0)/I(s_{\beta_3}\centerdot\nu')\big)=
d\big(U(\g_0)/I(s_{\beta_i}\centerdot\nu')\big)$$
for all $1\le i\le 5$.
Applying [\cite{Lo2}, Proposition~5.3.2], for example, we now deduce that
that
$d\big(U(\g_0)/I(\lambda_0')\big)=
\dim{\O}_0(w_0's_{\beta_i})$ where $w_0'$ is
the longest element of the Weyl group $W(\Phi_0)$ and $\O_0(w_0's_\beta)$ is
the special nilpotent orbit in $\g_0$ attached to the double
cell of $W(\Phi_0)$ containing $w_0's_{\beta_i}$. The letter is nothing else but the
minimal nilpotent orbit of $\g_0$.
As a result of these deliberations we conclude that $$\dim {\rm
VA}(I(\lambda))=\dim {\rm VA}(I(\lambda'))=248-(40+12+8)+14+0=202=\dim \O(e).$$
Note that the primitive ideals $I(\lambda)$ and $I(\lambda')$ have distinct central characters
because $\Phi_\lambda=\Phi_{\lambda'}$ but $\lambda_0+\rho_0$ and $\lambda_0'+\rho_0$ are not conjugate under the automorphism group of the root system $\Phi_0$.

We still face the difficult task of verifying Conditions~(A) and (D) for $\lambda+\rho$ and $\lambda'+\rho$.
Put $\l':=[\l,\l]$, a simple Lie algebra of type ${\sf D_5}$, and
let $\lambda_1$ and $\lambda_1'$ denote the restrictions of $\lambda$ and $\lambda'$ to $\h:=\l'\cap
\t$ which, of course, coincides with the span of the semisimple root
elements $h'_1:=h_{\alpha_1}$, $h'_2:=h_{\alpha_3}$,
$h'_3:=h_{\alpha_4}$ and $h'_4:=h_{\alpha_4}$ and $h_5:=h_{\alpha_2}$. Here we adjust our numbering to match that of [\cite{Bo}, Planche~IV] and accordingly we set $\alpha'_2:=\alpha_3$,
$\alpha'_3:=\alpha_4$, $\alpha'_5:=\alpha_2$, and $\alpha'_i:=\alpha_i$ for $i=1,4$.
The root system of $\l'$ with respect to $\h$ identifies with $\Phi_1:=\Phi\cap\big(\bigoplus_{i=1}^5\Z\alpha_i\big)$.
Let
$\{\varpi''_i\,|\,\, 1\le i\le 5\}$ be the corresponding system of
fundamental weights in $\h^*$, so that $\varpi'_i(h_j)=\delta_{ij}$
for all $1\le i,j\le 5$. Then
$\lambda_1+\rho_1=\varpi''_3-\frac{1}{4}(\varpi''_1+\varpi''_2+\varpi''_4+\varpi''_5)$
and
$\lambda_1'+\rho_1=2\varpi''_3-\frac{1}{4}(\varpi''_1+\varpi''_2+\varpi_5'')-\frac{9}{4}\varpi''_4$
where
$\rho_1=\sum_{i=1}^5\varpi''_i$. It is straightforward to see that $\lambda_1$ and $\lambda_1'$ have the same integral root system which we call $\Psi_1$. The basis of simple roots of $\Psi_1$ contained in $\Phi_1^+:=\Phi^+\cap \Phi_1$ equals
$\{\alpha_3',\gamma'\}$, where $\gamma'=\alpha_1'+\alpha_2'+\alpha_3'+\alpha_4'+
\alpha_5'$, and we have that
\begin{eqnarray*}
\langle\lambda_1+\rho_1,(\alpha_3')^\vee\rangle&=&1,\quad\langle\lambda_1+\rho_1,(\gamma')^\vee\rangle\,=\,0,\\
\langle\lambda'_1+\rho_1,(\alpha_3')^\vee\rangle&=&2,\quad\langle\lambda'_1+\rho_1,(\gamma')^\vee\rangle\,=\,-1.
\end{eqnarray*}
So $\lambda_1=s_{\gamma'}(\mu_1+\rho_1)-\rho_1$ and
$\lambda_1'=s_{\gamma'}(\mu_1'+\rho_1)-\rho_1$
where $\mu_1+\rho_1=\lambda_1+\rho_1$ and $$\mu_1'+\rho_1=\lambda_1'+\rho_1+\gamma'\,=\,\textstyle{\frac{3}{4}}\varpi_1''-
\textstyle{\frac{1}{4}}\varpi_2''+\varpi_3''-
\textstyle{\frac{5}{4}}\varpi_4''+\textstyle{\frac{3}{4}}\varpi_5''.$$ Both $\mu_1+\rho_1$ and $\mu'_1+\rho_1$ take values in $\Z^{\geqslant 0}$ on
$\{(\alpha_3')^\vee,(\gamma')^\vee\}$  and $\Psi_1$ has type ${\sf A_2}$.

Let $L_0(\nu)$ be the
irreducible $\l'$-module of highest weight $\nu\in\h^*$ and write
$I_0(\nu)$ for the primitive ideal ${\rm
Ann}_{U(\l')}\,L_0(\nu)$ of $U(\l')$. Since $|\Psi_1\cap\Phi_1^+|=3$ and the minimal nilpotent orbit in a Lie algebra of type ${\sf A_2}$ has dimension $8-4=4$, combining [\cite{Ja}, Corollar~10.10(a)] and [\cite{Lo2},
Proposition~5.3.2] with the preceding remark we obtain that
$$\dim{\rm VA}(I_0(\lambda_1))=\dim{\rm VA}(I_0(\lambda_1'))\,=\,|\Phi_1^+|-3+2=19.$$
From this it is immediate that $I_0(\lambda_1)$ and $I_0(\lambda_1')$ have the same associated variety, namely, the Zariski closure of the subregular nilpotent orbit $\O_0\subset\l'$.

In order to verify
Condition~(D) we need to show that
$$I_0(\lambda_1)={\rm
Ann}_{U(\l')}\big(Q_{\l',\,e_0}\otimes_{U(\l',\,e_0)}{\mathbb
C}_\eta\big)\quad\mbox{and}\quad
I_0(\lambda_1')={\rm
Ann}_{U(\l')}\big(Q_{\l',\,e_0}\otimes_{U(\l',\,e_0)}{\mathbb C}_{\eta'}\big)
$$ for some one-dimensional representations $\eta$
and $\eta'$
of the finite $W$-algebra
$U(\l',e_0)$ where we now regard $e_0$ as a subregular nilpotent element of $\l'$.

Let $\p'$ be the standard parabolic subalgebra of $\l'$ whose
standard Levi subalgebra is spanned by $\h$ and
$e_{\pm\alpha_3'}$ and write $\n'$ for the nilradical of $\p'$.
By
construction, $\p'$ is an optimal parabolic subalgebra for $e_0$ and, in particular,
$e_0$ is a Richardson element of $\p'$ (here we use the well known fact that the subregular nilpotent orbit in a Lie algebra of
type ${\sf D}_5$ is distinguished).

Given $\nu\in\h^*$ with $\nu(h_3')\in\Z^{\geqslant 0}$ we denote by $M_{\p'}(\nu)$ the generalised
Verma module of highest weight $\nu$ for $\l'$ and write $I_{\p'}(\nu)$ for the annihilator of $M_{\p'}(\nu)$ in $U(\l')$. Since $\lambda_1(h_3')=\lambda_1'(h_3')=0$ both $M_{\p'}(\mu_1)$
and $M_{\p'}(\mu_1')$ are generalised Verma modules of scalar type in the terminology of [\cite{Hu}].
As the centre of $U(\l')$ acts on $M_{\p'}(\nu)$ by scalar operators, it follows from Conze's Embedding Theorem and the Dixmier--M{\oe}glin equivalence that the annihilator in $U(\l')$ of any generalised Verma module of scalar type is a primitive ideal of $U(\l')$. By a result of Borho--Brylinski, the associated variety of any such primitive ideal coincides with the Zariski closure of the Richardson orbit that intersects densely with $\n'$; see [\cite{Ja}, 17.17(3)] for example.
In our case this is the subregular orbit  $\O_0$.

Next we observe that the only root $\beta\in\Psi_1$ for which $\langle\lambda_1,\beta^\vee\rangle\in\Z^{>0}$ equals $\alpha_3'+\gamma'$ and we have that $s_\beta(\gamma')=-\alpha_3'$. In this situation Jantzen's criterion for irreducibility of generalised Verma modules implies that the
$\l'$-module $M_{\p'}(\mu_1)$ is simple; see [\cite{Hu}, Corollary~9.13(c)], for example.
Unfortunately, this is no longer true for the $\l'$-module $M_{\p'}(\mu_1')$ and a more subtle approach is required here.

We write $M(\nu)$ for the ordinary Verma $\l'$-module of highest weight $\nu\in\h^*$.
If $\nu(h_3')\in\Z^{\geqslant 0}$ then we denote by $\varphi$ the canonical $\l'$-module epimorphism $M(\nu)\twoheadrightarrow M_{\p'}(\nu)$ which is unique up to a nonzero scalar multiple.
Since $\lambda_1'$ is strongly linked with $\mu_1'$ relative to the integral Weyl group
$W(\Psi_1)$, the Bernstein--Gelfand--Gelfand theorem says that $M(\lambda_1')$ embeds into $M(\mu_1')$.
Recall that $\lambda_1'(h_3')=2$.
Because $\lambda_1'=s_{\gamma'}\centerdot\mu'_1$ and
the root $\gamma'$ is indecomposable in $\Psi_1\cap\Phi_1^+$,
it follows from [\cite{Hu}, Theorem~9.9(c)] that
$\varphi$ sends
the image of $M(\lambda_1')$ in $M(\mu_1')$ onto a {\it nonzero} submodule of $M_{\p'}(\mu_1')$. As a result, $L_0(\lambda_1')$ is a composition factor of $M_{\p'}(\mu_1')$. This, in turn, yields $I_{\p'}(\mu_1')\subseteq I_0(\lambda_1')$. As both ideals are primitive and ${\rm VA}(I_0(\lambda'_1))={\rm VA}(I_{\p'}(\mu'_1))$ by our earlier remarks, applying [\cite{BK}, Corollar~3.6] entails that $I_0(\lambda_1)=I_{\p'}(\mu_1)$ and $I_0(\lambda_1')=I_{\p'}(\mu_1')$.

Let $w_0''$ be the longest element of the Weyl group $W(\Phi_1)$ and put $\mu_2:=-w_0''(\mu_1)$ and $\mu_2':=-w_0'(\mu'_1)$. Since $-w_0''(\alpha_3')=\alpha_3'$ and
the orbit $\O_0$ is stable under all automorphisms of $\l'$, the above discussion also shows that the ideals $I_{\p'}(\mu_2)$ and $I_{\p'}(\mu_2')$ are primitive with
${\rm VA}(I_{\p'}(\mu_2))={\rm VA}(I_{\p'}(\mu_2'))=\overline{\O_0}$. Furthermore, we have that
$I_0(-w_0''(\lambda_1))=I_{\p'}(\mu_2)
)$ and $I_0(-w_0''(\lambda_1'))=I_{\p'}(\mu_2')$.

It is immediate from Duflo's theorem [\cite{Du1}] that ${^t\!}I=I$ for any primitive ideal $I$ of $U(\l')$,
where $u\mapsto {^t\!}u$ is the anti-involution
of $U(\l')$ acting identically on $\h$ and sending
$e_\beta$ to $e_{-\beta}$ for any $\beta\in\Phi_1^+$ (here we assume that $\{e_\beta\,|\,\,\beta\in\Phi_1\}$ is a Chevalley system of $\l'$).
Let $\scriptstyle{\top}$ denote the canonical
anti-involution of $U(\l')$ which acts on $\l'$ as $-{\rm Id}$. As $^t\!(x^{\scriptstyle{\top}})=
(^t\!x)^{\scriptstyle{\top}}$ for all $x\in\l'$, composing these two anti-involutions one obtains a Cartan involution of $\l'$; we call it $c$. It is well known that $c=\sigma\circ s$ where $s$ is an {\it inner} automorphism of $\g$ and $\sigma$ is the Dynkin automorphism of $\g$ acting on $\h$ as $-w_0''$.
From this it follows that $$I^{\scriptstyle{\top}}=\,
(^t\!I)^{\scriptstyle{\top}}=I^c=\,(I^s)^\sigma=
I^\sigma$$
for any primitive ideal $I$ of $U(\l')$. As an immediate consequence, we deduce that  $I_{\p'}(\mu_2)^{\scriptstyle{\top}}=
I_{\p'}(-w_0''(\mu_1))^\sigma=I_{\p'}(\mu_1)=I_0(\lambda_1)$ and likewise
$I_{\p'}(\mu_2')^{\scriptstyle{\top}}
=I_0(\lambda_2')$.

Now let $\nu\in\{\mu_2,\mu_2'\}$ and denote by $\chi_0$  the linear function on $\l'$ given by
$\chi_0(x)=(e_0,x)$ for all $x\in\l'$, where $(\,\cdot\,,\cdot\,)$ is the Killing form of $\l'$.
In
the present case, the nilpotent subalgebra $\m$ involved in the
definition of of $Q_{\l',\,e_0}$ and $U(\l',e_0)$ is spanned by the root vectors
$e_{-\beta}$ with $\beta\in\Phi^+_1\setminus\{\alpha_3'\}$. By our choice of $e_0$, the linear function $\chi_0$ vanishes on
$[\m,\m]$. Given $a\in\mathbb{C}$ we put $\m_{a\chi_0}=\{x-a\chi_0(x)\,|\,\,x\in\m\}$, a subspace of $U(\l')$. Since $M_\nu:=M_{\p'}(\nu)$ is a free $U(\m)$-module of rank $1$,
the subspace $\m_{-\chi_0} M_\nu$
has codimension $1$ in $M_\nu$.
But then the subspace
$${\rm Wh}_{\chi_0}\big(M^*_\nu):=\{f\in M_\nu^*\,|\,\,x\cdot f=\chi_0(x)f\ \,\mbox{for all}\ \, x\in\m\}$$
of the full dual $\l'$-module $M^*_\nu$ is one-dimensional
and carries a natural module structure over the finite $W$-algebra
$U(\l',e_0)\,\cong\, (U(\l')/U(\l')\m_{\chi_0})^{\ad\,\m}.$
By Skryabin's theorem [\cite{Sk}]
this implies that the $\l'$-submodule of  $M^*_\nu$
generated by ${\rm Wh}_{\chi_0}\big(M^*_\nu)$ is irreducible and isomorphic to
$Q_{\l',\,e_0}\otimes_{U(\l',\,e_0)}{\rm Wh}_{\chi_0}\big(M^*_\nu\big)$  (recall that $Q_{\l',\,e_0}$ is tautologically a right $U(\l',e_0)$-module because $U(\l',e_0) = ({\rm End}_{\l'}\,Q_{\l',\,e_0})^{\rm op}$). Moreover, repeating verbatim
the argument
from [\cite{P11}, Remark~4.3]  one observes that
$$I_{\p'}(-w_0''(\nu))\,=\,I_{\p'}(\nu)^{\scriptstyle{\top}}=\,{\rm Ann}_{U(\l')}(M^*_\nu)\,=\,{\rm
Ann}_{U(\l')}\big(Q_{\l',\,e_0}\otimes_{U(\l',\,e_0)}{\rm Wh}_{\chi_0}\big(M^*_\nu\big)\big).$$
In view of our earlier remarks in this subsection,
this enables us to conclude that Conditions~(A) and (D)
hold for both $\lambda+\rho$ and $\lambda'+\rho$. As a result, $I(\lambda)$ and $I(\lambda')$
are distinct multiplicity-free primitive ideals in $\mathcal{X}_{\O(e)}$.

\section{\bf The case of
rigid nilpotent orbits in Lie algebras of type ${\sf E_7}$ and ${\sf E_6}$.}
\subsection{Type $({\sf E_7, A_1})$}\label{4.1}
As in the ${\sf E_8}$-case, $e$ is a special nilpotent element and
$e^\vee\in\g^\vee$ is subregular.
So it can be assumed that
$$h^\vee=\textstyle{{2\atop{}}{2\atop{}}{0\atop 2}
{2\atop{}}{2\atop{}}{2\atop{}}}.$$ Keeping in mind Losev's condition~(A)  we choose the following pinning
for $e$:
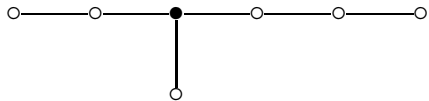
\[\xymatrix{*{\circ}\ar@{-}[r]& *{\circ} \ar@{-}[r] & *{\bullet} \ar@{-}[r]\ar@{-}[d] & *{\circ} \ar@{-}[r]
&*{\circ}\ar@{-}[r]&*{\circ}\\ && *{\circ}&&&}\]
and take
$$\tau=\textstyle{{0\atop{}}{(-1)\atop{}}{2\atop (-1)}
{(-1)\atop{}}{0\atop{}}{0\atop{}}}$$ as an optimal cocharacter. Since
$\dim\,\g_e=99$ the total number of positive $0$-roots and
$1$-roots is $(99-7)/2=46$. The roots are listed in the table
below.
\begin{table}[htb]
\label{data1}
\begin{tabular}{|c|c|}
\hline$0$-roots  & $1$-roots
\\ \hline
$\scriptstyle{{1\atop{}}{0\atop{}}{0\atop0}{0\atop{}}{0\atop{}}{0\atop{}}}$,
$\scriptstyle{{0\atop{}}{0\atop{}}{0\atop0}{0\atop{}}{1\atop{}}{0\atop{}}}$,
$\scriptstyle{{0\atop{}}{0\atop{}}{0\atop0}{0\atop{}}{0\atop{}}{1\atop{}}}$,
$\scriptstyle{{0\atop{}}{0\atop{}}{0\atop0}{0\atop{}}{1\atop{}}{1\atop{}}}$,
$\scriptstyle{{0\atop{}}{1\atop{}}{1\atop1}{0\atop{}}{0\atop{}}{0\atop{}}}$,&
$\scriptstyle{{0\atop{}}{1\atop{}}{1\atop0}{0\atop{}}{0\atop{}}{0\atop{}}}$,
$\scriptstyle{{0\atop{}}{0\atop{}}{1\atop 1}{0\atop{}}{0\atop{}}{0\atop{}}}$,
$\scriptstyle{{0\atop{}}{0\atop{}}{1\atop0}{1\atop{}}{0\atop{}}{0\atop{}}}$,
$\scriptstyle{{1\atop{}}{1\atop{}}{1\atop0}{0\atop{}}{0\atop{}}{0\atop{}}}$,\\
$\scriptstyle{{0\atop{}}{1\atop{}}{1\atop0}{1\atop{}}{0\atop{}}{0\atop{}}}$,
$\scriptstyle{{0\atop{}}{0\atop{}}{1\atop1}{1\atop{}}{0\atop{}}{0\atop{}}}$,
$\scriptstyle{{1\atop{}}{1\atop{}}{1\atop1}{0\atop{}}{0\atop{}}{0\atop{}}}$,
$\scriptstyle{{1\atop{}}{1\atop{}}{1\atop0}{1\atop{}}{0\atop{}}{0\atop{}}}$,
$\scriptstyle{{0\atop{}}{1\atop{}}{1\atop0}{1\atop{}}{1\atop{}}{0\atop{}}}$,&
$\scriptstyle{{0\atop{}}{0\atop{}}{1\atop0} {1\atop{}}{1\atop{}}{0\atop{}}}$,
$\scriptstyle{{0\atop{}}{0\atop{}}{1\atop0}{1\atop{}}{1\atop{}}{1\atop{}}}$,
$\scriptstyle{{0\atop{}}{1\atop{}}{2\atop1}{1\atop{}}{0\atop{}}{0\atop{}}}$,
$\scriptstyle{{1\atop{}}{1\atop{}}{2\atop1}{1\atop{}}{0\atop{}}{0\atop{}}}$,\\
$\scriptstyle{{0\atop{}}{0\atop{}}{1\atop1}{1\atop{}}{1\atop{}}{0\atop{}}}$,
$\scriptstyle{{1\atop{}}{1\atop{}}{1\atop0}{1\atop{}}{1\atop{}}{0\atop{}}}$,
$\scriptstyle{{0\atop{}}{1\atop{}}{1\atop0}{1\atop{}}{1\atop{}}{1\atop{}}}$,
$\scriptstyle{{0\atop{}}{0\atop{}}{1\atop1}{1\atop{}}{1\atop{}}{1\atop{}}}$,
$\scriptstyle{{1\atop{}}{1\atop{}}{1\atop 0}{1\atop{}}{1\atop{}}{1\atop{}}}$,&
$\scriptstyle{{0\atop{}}{1\atop{}}{2\atop1}{1\atop{}}{1\atop{}}{0\atop{}}}$,
$\scriptstyle{{0\atop{}}{1\atop{}}{2\atop1}{1\atop{}}{1\atop{}}{1\atop{}}}$,
$\scriptstyle{{1\atop{}}{1\atop{}}{2\atop1}{1\atop{}}{1\atop{}}{0\atop{}}}$,
$\scriptstyle{{1\atop{}}{1\atop{}}{2\atop1}{1\atop{}}{1\atop{}}{1\atop{}}}$,\\
$\scriptstyle{{1\atop{}}{2\atop{}}{2\atop1}{1\atop{}}{1\atop{}}{0\atop{}}}$,
$\scriptstyle{{1\atop{}}{1\atop{}}{2\atop1}{2\atop{}}{1\atop{}}{0\atop{}}}$,
$\scriptstyle{{1\atop{}}{2\atop{}}{2\atop1}{1\atop{}}{1\atop{}}{1\atop{}}}$,
$\scriptstyle{{1\atop{}}{1\atop{}}{2\atop 1}{2\atop{}}{1\atop{}}{1\atop{}}}$,
$\scriptstyle{{1\atop{}}{1\atop{}}{2\atop1}{2\atop{}}{2\atop{}}{1\atop{}}}$,&
$\scriptstyle{{1\atop{}}{2\atop{}}{3\atop1}{2\atop{}}{1\atop{}}{0\atop{}}}$,
$\scriptstyle{{1\atop{}}{2\atop{}}{3\atop1}{2\atop{}}{1\atop{}}{1\atop{}}}$,
$\scriptstyle{{1\atop{}}{2\atop{}}{3\atop1}{2\atop{}}{2\atop{}}{1\atop{}}}$,
$\scriptstyle{{1\atop{}}{2\atop{}}{4\atop2}{3\atop{}}{2\atop{}}{1\atop{}}}$.\\
$\scriptstyle{{0\atop{}}{1\atop{}}{2\atop1}{2\atop{}}{1\atop{}}{0\atop{}}}$,
$\scriptstyle{{0\atop{}}{1\atop{}}{2\atop1}{2\atop{}}{1\atop{}}{1\atop{}}}$,
$\scriptstyle{{0\atop{}}{1\atop{}}{2\atop 1}{2\atop{}}{2\atop{}}{1\atop{}}}$,
$\scriptstyle{{1\atop{}}{2\atop{}}{3\atop2}{2\atop{}}{1\atop{}}{0\atop{}}}$,
$\scriptstyle{{1\atop{}}{2\atop{}}{3\atop2}{2\atop{}}{1\atop{}}{1\atop{}}}$,&\\
$\scriptstyle{{1\atop{}}{2\atop{}}{3\atop2}{2\atop{}}{2\atop{}}{1\atop{}}}$,
$\scriptstyle{{1\atop{}}{2\atop{}}{3\atop1}{3\atop{}}{2\atop{}}{1\atop{}}}$,
$\scriptstyle{{1\atop{}}{3\atop{}}{4\atop2}{3\atop{}}{2\atop{}}{1\atop{}}}$,
$\scriptstyle{{2\atop{}}{3\atop{}}{4\atop2}{3\atop{}}{2\atop{}}{1\atop{}}}$,
$\scriptstyle{{1\atop{}}{2\atop{}}{2\atop1}{1\atop{}}{1\atop{}}{0\atop{}}}$.&\\
\hline
\end{tabular}
\end{table}

By a routine computation we see that the
$(1,2,3,5,6,7)$-contributions of the $0$-roots and $1$-roots are
$(18,25,34,39,28,15)$ and $(8,12,16,18,12,6)$,
respectively. So the total contribution of all roots is
$(13,\frac{37}{2},25,\frac{57}{2},20,\frac{21}{2})$. In view of [\cite{Bo}, Planche~VI]
Condition~(C) for
$\lambda+\rho$
 reads
\begin{eqnarray*}
2a_1+2a_2+3a_3+4a_4+3a_5+2a_6+a_7&=&13\\
4a_1+7a_2+8a_3+12a_4+9a_5+6a_6+3a_7&=&37\\
3a_1+4a_2+6a_3+8a_4+6a_5+4a_6+2a_7&=&25\\
6a_1+9a_2+12a_3+18a_4+15a_5+10a_6+5a_7&=&57\\
2a_1+3a_2+4a_3+6a_4+5a_5+4a_6+2a_7&=&20\\
2a_1+3a_2+4a_3+6a_4+5a_5+4a_6+3a_7&=&21.
\end{eqnarray*}
Note that
$$\lambda+\rho=\textstyle{{1\atop{}}{1\atop{}}{0\atop 1}
{1\atop{}}{1\atop{}}{1\atop{}}}=\rho-\varpi_4$$ satisfies this system of linear equations. As
$\langle\lambda,\alpha_4^\vee\rangle=0$,
Condition~(A) holds for $\lambda+\rho$. Since $e$ is special and
$\lambda+\rho=\frac{1}{2}h^\vee$ applying [\cite{BV},
Proposition~5.10] we see that this weight also satisfies Condition
(B). Condition (D) is vacuous in the present case as $e$ has
standard Levi type. Since the centraliser of $e$ is a perfect Lie algebra in the present case, applying [\cite{Lo2}, 5.3] we conclude that
$I(\lambda)=I(-\varpi_4)$ is the only multiplicity-free primitive
ideal in $\mathcal{X}_{\O(e)}$. Since $\O(e)=\O_{\rm min}$ in the present case, we thus recover the Joseph ideal by using some tools from the theory of finite $W$-algebras.
\subsection{Type $({\sf E_7, 2A_1})$}\label{4.2}
This is a special orbit and
$e^\vee\in\g^\vee$ has type ${\sf E_7(a_2)}$.
So we may assume that
$h^\vee=\textstyle{{2\atop{}}{2\atop{}}{0\atop 2}
{2\atop{}}{0\atop{}}{2\atop{}}}$. In view of Condition~(A) we choose the following pinning
for $e$:
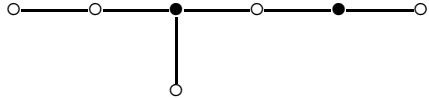
\[\xymatrix{*{\circ}\ar@{-}[r]& *{\circ} \ar@{-}[r] & *{\bullet} \ar@{-}[r]\ar@{-}[d] & *{\circ} \ar@{-}[r]
&*{\bullet}\ar@{-}[r]&*{\circ}\\ && *{\circ}&&&}\]
and we choose
$$\tau=\textstyle{{0\atop{}}{(-1)\atop{}}{2\atop (-1)}
{(-2)\atop{}}{2\atop{}}{(-1)\atop{}}}$$ to be our optimal cocharacter. As
$\dim\,\g_e=81$, the total number of positive $0$-roots and
$1$-roots is $(81-7)/2=37$. These roots are given
below.
\begin{table}[htb]
\label{data1}
\begin{tabular}{|c|c|}
\hline$0$-roots  & $1$-roots
\\ \hline
$\scriptstyle{{1\atop{}}{0\atop{}}{0\atop0}{0\atop{}}{0\atop{}}{0\atop{}}}$,
$\scriptstyle{{0\atop{}}{1\atop{}}{1\atop1}{0\atop{}}{0\atop{}}{0\atop{}}}$,
$\scriptstyle{{0\atop{}}{0\atop{}}{1\atop0}{1\atop{}}{0\atop{}}{0\atop{}}}$,
$\scriptstyle{{0\atop{}}{0\atop{}}{0\atop0}{1\atop{}}{1\atop{}}{0\atop{}}}$,
$\scriptstyle{{1\atop{}}{1\atop{}}{1\atop1}{0\atop{}}{0\atop{}}{0\atop{}}}$,&
$\scriptstyle{{0\atop{}}{1\atop{}}{1\atop0}{0\atop{}}{0\atop{}}{0\atop{}}}$,
$\scriptstyle{{0\atop{}}{0\atop{}}{1\atop 1}{0\atop{}}{0\atop{}}{0\atop{}}}$,
$\scriptstyle{{0\atop{}}{0\atop{}}{0\atop0}{0\atop{}}{1\atop{}}{1\atop{}}}$,
$\scriptstyle{{1\atop{}}{1\atop{}}{1\atop0}{0\atop{}}{0\atop{}}{0\atop{}}}$,\\
$\scriptstyle{{0\atop{}}{1\atop{}}{1\atop1}{1\atop{}}{1\atop{}}{0\atop{}}}$,
$\scriptstyle{{1\atop{}}{1\atop{}}{1\atop1}{1\atop{}}{1\atop{}}{0\atop{}}}$,
$\scriptstyle{{0\atop{}}{1\atop{}}{1\atop0}{1\atop{}}{1\atop{}}{1\atop{}}}$,
$\scriptstyle{{1\atop{}}{1\atop{}}{1\atop0}{1\atop{}}{1\atop{}}{1\atop{}}}$,
$\scriptstyle{{0\atop{}}{0\atop{}}{1\atop1}{1\atop{}}{1\atop{}}{1\atop{}}}$,&
$\scriptstyle{{0\atop{}}{1\atop{}}{1\atop0} {1\atop{}}{1\atop{}}{0\atop{}}}$,
$\scriptstyle{{0\atop{}}{0\atop{}}{1\atop1}{1\atop{}}{1\atop{}}{0\atop{}}}$,
$\scriptstyle{{0\atop{}}{0\atop{}}{1\atop0}{1\atop{}}{1\atop{}}{1\atop{}}}$,
$\scriptstyle{{1\atop{}}{1\atop{}}{1\atop0}{1\atop{}}{1\atop{}}{0\atop{}}}$,\\
$\scriptstyle{{0\atop{}}{1\atop{}}{2\atop1}{1\atop{}}{0\atop{}}{0\atop{}}}$,
$\scriptstyle{{1\atop{}}{1\atop{}}{2\atop1}{1\atop{}}{0\atop{}}{0\atop{}}}$,
$\scriptstyle{{0\atop{}}{1\atop{}}{2\atop1}{2\atop{}}{1\atop{}}{0\atop{}}}$,
$\scriptstyle{{1\atop{}}{1\atop{}}{2\atop1}{2\atop{}}{1\atop{}}{0\atop{}}}$,
$\scriptstyle{{1\atop{}}{2\atop{}}{2\atop1}{1\atop{}}{1\atop{}}{1\atop{}}}$,&
$\scriptstyle{{0\atop{}}{1\atop{}}{2\atop1}{1\atop{}}{1\atop{}}{1\atop{}}}$,
$\scriptstyle{{1\atop{}}{1\atop{}}{2\atop1}{1\atop{}}{1\atop{}}{1\atop{}}}$,
$\scriptstyle{{0\atop{}}{1\atop{}}{2\atop1}{2\atop{}}{2\atop{}}{1\atop{}}}$,
$\scriptstyle{{1\atop{}}{1\atop{}}{2\atop1}{2\atop{}}{2\atop{}}{1\atop{}}}$,\\
$\scriptstyle{{1\atop{}}{2\atop{}}{2\atop1}{2\atop{}}{2\atop{}}{1\atop{}}}$,
$\scriptstyle{{1\atop{}}{2\atop{}}{3\atop2}{2\atop{}}{1\atop{}}{0\atop{}}}$,
$\scriptstyle{{1\atop{}}{2\atop{}}{3\atop1}{2\atop{}}{1\atop{}}{1\atop{}}}$,
$\scriptstyle{{1\atop{}}{2\atop{}}{3\atop 1}{3\atop{}}{2\atop{}}{1\atop{}}}$,
$\scriptstyle{{1\atop{}}{3\atop{}}{4\atop2}{3\atop{}}{2\atop{}}{1\atop{}}}$,&
$\scriptstyle{{1\atop{}}{2\atop{}}{2\atop1}{1\atop{}}{1\atop{}}{0\atop{}}}$,
$\scriptstyle{{1\atop{}}{2\atop{}}{3\atop1}{2\atop{}}{1\atop{}}{0\atop{}}}$,
$\scriptstyle{{1\atop{}}{2\atop{}}{3\atop2}{2\atop{}}{2\atop{}}{1\atop{}}}$,
$\scriptstyle{{1\atop{}}{2\atop{}}{4\atop2}{3\atop{}}{2\atop{}}{1\atop{}}}$.\\
$\scriptstyle{{2\atop{}}{3\atop{}}{4\atop2}{3\atop{}}{2\atop{}}{1\atop{}}}$.
&\\
\hline
\end{tabular}
\end{table}

The
$(1,2,3,5,7)$-contributions of the $0$-roots and $1$-roots are
$(14,19,26,29,9)$ and $(8,12,16,18,8)$,
respectively, and the total contribution of all roots is
$(11,\frac{31}{2},21,\frac{47}{2},\frac{17}{2})$. So Condition~(C) for
$\lambda+\rho$
reads
\begin{eqnarray*}
2a_1+2a_2+3a_3+4a_4+3a_5+2a_6+a_7&=&11\\
4a_1+7a_2+8a_3+12a_4+9a_5+6a_6+3a_7&=&31\\
3a_1+4a_2+6a_3+8a_4+6a_5+4a_6+2a_7&=&21\\
6a_1+9a_2+12a_3+18a_4+15a_5+10a_6+5a_7&=&47\\
2a_1+3a_2+4a_3+6a_4+5a_5+4a_6+3a_7&=&17.
\end{eqnarray*}
A very nice solution to this system of linear equations is given by setting $a_i=0$ for $i=4,6$ and $a_i=1$ for
$i=1,2,3,5,7$.
It follows that $\lambda+\rho=\frac{1}{2}h^\vee$ satisfies Condition (C). As
$\langle\lambda,\alpha_i^\vee\rangle=0$ for $i=4,6$,
Condition~(A) also holds for $\lambda+\rho$. Since $e$ is special and
$\lambda+\rho=\frac{1}{2}h^\vee$, Condition~(B) follows from [\cite{BV},
Proposition~5.10], whilst Condition (D) is again  vacuous as $e$ has
standard Levi type. Since $\g_e=[\g_e,\g_e]$ by [\cite{dG}], we conclude that
$I(\lambda)=I(-\varpi_4-\varpi_6)$ is the only multiplicity-free primitive
ideal in $\mathcal{X}_{\O(e)}$.
\subsection{Type $({\sf E_7, (3A_1})')$}\label{4.3}
Here our pinning for $e$ is as follows:
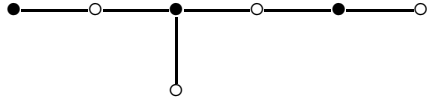
\[\xymatrix{*{\bullet}\ar@{-}[r]& *{\circ} \ar@{-}[r] & *{\bullet} \ar@{-}[r]\ar@{-}[d] & *{\circ} \ar@{-}[r]
&*{\bullet}\ar@{-}[r]&*{\circ}\\ && *{\circ}&&&}\]
and we choose
$$\tau=\textstyle{{2\atop{}}{(-2)\atop{}}{2\atop (-1)}
{(-2)\atop{}}{2\atop{}}{(-1)\atop{}}}$$ to be our optimal cocharacter.
Also, $\dim\,\g_e=69$ and the total number of positive $0$-roots and
$1$-roots is $(69-7)/2=31$. These roots are given in the table below.
\begin{table}[htb]
\label{data1}
\begin{tabular}{|c|c|}
\hline$0$-roots  & $1$-roots
\\ \hline
$\scriptstyle{{1\atop{}}{1\atop{}}{0\atop0}{0\atop{}}{0\atop{}}{0\atop{}}}$,
$\scriptstyle{{0\atop{}}{1\atop{}}{1\atop0}{0\atop{}}{0\atop{}}{0\atop{}}}$,
$\scriptstyle{{0\atop{}}{0\atop{}}{1\atop0}{1\atop{}}{0\atop{}}{0\atop{}}}$,
$\scriptstyle{{0\atop{}}{0\atop{}}{0\atop0}{1\atop{}}{1\atop{}}{0\atop{}}}$,&
$\scriptstyle{{0\atop{}}{0\atop{}}{1\atop1}{0\atop{}}{0\atop{}}{0\atop{}}}$,
$\scriptstyle{{0\atop{}}{0\atop{}}{0\atop 0}{0\atop{}}{1\atop{}}{1\atop{}}}$,
$\scriptstyle{{1\atop{}}{1\atop{}}{1\atop1}{0\atop{}}{0\atop{}}{0\atop{}}}$,
$\scriptstyle{{0\atop{}}{0\atop{}}{1\atop0}{1\atop{}}{1\atop{}}{1\atop{}}}$,\\
$\scriptstyle{{1\atop{}}{1\atop{}}{1\atop0}{1\atop{}}{0\atop{}}{0\atop{}}}$,
$\scriptstyle{{0\atop{}}{1\atop{}}{1\atop0}{1\atop{}}{1\atop{}}{0\atop{}}}$,
$\scriptstyle{{0\atop{}}{0\atop{}}{1\atop1}{1\atop{}}{1\atop{}}{1\atop{}}}$,
$\scriptstyle{{1\atop{}}{1\atop{}}{1\atop1}{1\atop{}}{1\atop{}}{1\atop{}}}$,&
$\scriptstyle{{0\atop{}}{0\atop{}}{1\atop1} {1\atop{}}{1\atop{}}{0\atop{}}}$,
$\scriptstyle{{1\atop{}}{1\atop{}}{1\atop1}{1\atop{}}{1\atop{}}{0\atop{}}}$,
$\scriptstyle{{1\atop{}}{1\atop{}}{1\atop0}{1\atop{}}{1\atop{}}{1\atop{}}}$,
$\scriptstyle{{1\atop{}}{1\atop{}}{2\atop1}{1\atop{}}{0\atop{}}{0\atop{}}}$,\\
$\scriptstyle{{0\atop{}}{1\atop{}}{2\atop1}{1\atop{}}{1\atop{}}{1\atop{}}}$,
$\scriptstyle{{1\atop{}}{2\atop{}}{2\atop1}{1\atop{}}{1\atop{}}{1\atop{}}}$,
$\scriptstyle{{0\atop{}}{1\atop{}}{2\atop1}{2\atop{}}{2\atop{}}{1\atop{}}}$,
$\scriptstyle{{1\atop{}}{2\atop{}}{2\atop1}{2\atop{}}{2\atop{}}{1\atop{}}}$,&
$\scriptstyle{{0\atop{}}{1\atop{}}{2\atop1}{1\atop{}}{1\atop{}}{0\atop{}}}$,
$\scriptstyle{{1\atop{}}{1\atop{}}{2\atop1}{2\atop{}}{1\atop{}}{0\atop{}}}$,
$\scriptstyle{{1\atop{}}{2\atop{}}{2\atop1}{1\atop{}}{1\atop{}}{0\atop{}}}$,
$\scriptstyle{{1\atop{}}{2\atop{}}{3\atop1}{2\atop{}}{1\atop{}}{0\atop{}}}$,\\
$\scriptstyle{{1\atop{}}{2\atop{}}{3\atop1}{3\atop{}}{2\atop{}}{1\atop{}}}$,
$\scriptstyle{{1\atop{}}{2\atop{}}{3\atop1}{2\atop{}}{1\atop{}}{1\atop{}}}$,
$\scriptstyle{{1\atop{}}{2\atop{}}{3\atop2}{2\atop{}}{1\atop{}}{0\atop{}}}$,
$\scriptstyle{{1\atop{}}{1\atop{}}{2\atop 1}{2\atop{}}{1\atop{}}{1\atop{}}}$.&
$\scriptstyle{{1\atop{}}{2\atop{}}{4\atop2}{3\atop{}}{2\atop{}}{1\atop{}}}$,
$\scriptstyle{{2\atop{}}{3\atop{}}{4\atop2}{3\atop{}}{2\atop{}}{1\atop{}}}$,
$\scriptstyle{{1\atop{}}{2\atop{}}{3\atop2}{2\atop{}}{2\atop{}}{1\atop{}}}$.
\\
\hline
\end{tabular}
\end{table}

The orbit $\O(e)$ is non-special and $$\textstyle{\frac{1}{2}}\dim\O(e)=(133-69)/2=32
=63-31=|\Phi^+|-(30+1).$$ Therefore, it seems reasonable to seek a highest weight $\lambda+\rho=\sum_{i=1}^7a_i\varpi_i$ with $\Phi_\lambda$ of type ${\sf D_6+A_1}$. This means that we want all $a_i$ to be half-integers.
The
$(2,3,5,7)$-contributions of the $0$-roots and $1$-roots are
$(11,18,21,9)$ and $(15,17,19,6)$,
respectively, and the total contribution of all roots is
$(13,\frac{35}{2},20,\frac{15}{2})$.

Setting $\widetilde{a}_i:=2a_i$ for $1\le i\le 7$ we can express Condition~(C) for
$\lambda+\rho$ as follows:
\begin{eqnarray*}
4\widetilde{a}_1+7\widetilde{a}_2+8\widetilde{a}_3+
12\widetilde{a}_4+9\widetilde{a}_5+6\widetilde{a}_6+
3\widetilde{a}_7&=&52\\
3\widetilde{a}_1+4\widetilde{a}_2+6\widetilde{a}_3+
8\widetilde{a}_4+6\widetilde{a}_5+4\widetilde{a}_6+
2\widetilde{a}_7&=&35\\
6\widetilde{a}_1+9\widetilde{a}_2+12\widetilde{a}_3+
18\widetilde{a}_4+15\widetilde{a}_5+10\widetilde{a}_6+
5\widetilde{a}_7&=&80\\
2\widetilde{a}_1+3\widetilde{a}_2+4\widetilde{a}_3+
6\widetilde{a}_4+5\widetilde{a}_5+4\widetilde{a}_6+
3\widetilde{a}_7&=&30.
\end{eqnarray*}
A very nice solution to this system of linear equations is obtained by setting $\widetilde{a}_i=1$
for $1\le i\le 6$ and $\widetilde{a}_7=2$, which leads to $$\lambda+\rho=\varpi_7+\textstyle{\frac{1}{2}}
\textstyle{
\sum}_{i=1}^6\varpi_i=\textstyle{\frac{1}{2}}\rho+\textstyle{\frac{1}{2}}\varpi_7.$$
Note that $\Phi_\lambda$ contains the positive roots
$$
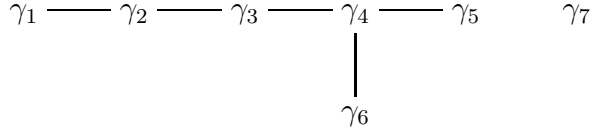
\xymatrix{{\gamma_1}\ar@{-}[r]&{\gamma_2}\ar@{-}[r]& {\gamma_3} \ar@{-}[r] & {\gamma_4}
\ar@{-}[r]\ar@{-}[d] & {\gamma_5}
 & {\gamma_7} \\ &&& {\gamma_6}&&&}
$$
where
\begin{eqnarray*}
&&\gamma_1\,=\,{\scriptstyle{0\atop{}}{0\atop{}}{1\atop0}{1\atop{}}{0\atop{}}{0\atop{}}},\
\
\gamma_2\,=\,{\scriptstyle{1\atop{}}{1\atop{}}{0\atop0}{0\atop{}}{0\atop{}}{0\atop{}}},\
\
\gamma_3\,=\,{\scriptstyle{0\atop{}}{0\atop{}}{1\atop0}{1\atop{}}{0\atop{}}{0\atop{}}},\
\
\gamma_4\,=\,{\scriptstyle{0\atop{}}{0\atop{}}{0\atop0}{1\atop{}}{1\atop{}}{0\atop{}}},\\
\
&&\gamma_5\,=\,{\scriptstyle{0\atop{}}{0\atop{}}{0\atop0}{0\atop{}}{0\atop{}}{1\atop{}}},\
\
\gamma_6\,=\,{\scriptstyle{0\atop{}}{1\atop{}}{1\atop0}{0\atop{}}{0\atop{}}{0\atop{}}},\
\
\gamma_7\,=\,{\scriptstyle{0\atop{}}{1\atop{}}{1\atop1}{1\atop{}}{0\atop{}}{0\atop{}}}.
\end{eqnarray*}
Since it follows from the Borel--de
Siebehthal algorithm that all root subsystems of type
${\sf
D_6+A_1}$ are maximal in $\Phi$, the roots $\gamma_1,\ldots,\gamma_7$ form the basis of $\Phi_\lambda$ contained in $\Phi^+$.  Since
$\langle\lambda+\rho,\gamma_i^\vee\rangle\in\Z^{>0}$ for $1\le i\le 7$
and
$|\Phi^+|-|\Phi_\lambda^+|=\frac{1}{2}\dim\,\O(e)$ by our earlier remark,
applying [\cite{Jo1}, Corollary~3.5] gives $\dim\,{\rm
VA}(I(\lambda))=\dim\,\O(e)$.
Since $\langle\lambda+\rho,\alpha_i^\vee\rangle\not\in\Z^{>0}$ for $i=1,4,6$ by our choice of pinning, we also see that
$\lambda+\rho$ satisfies Condition~(A). As a consequence, $e\in{\rm VA}(I(\lambda))$. But then
${\rm VA}(I(\lambda))=\overline{\O(e)}$ i.e. Condition~(B) holds for
$\lambda+\rho$. Since in the present case Condition~(D) is vacuous and $\g_e=[\g_e,\g_e]$ by
[\cite{dG}], we combine [\cite{Lo2}, 5.3] and [\cite{PT},
Proposition~11] to conclude $I(\frac{1}2{}\varpi_7-\frac{1}{2}\rho)$ is the only
multiplicity-free primitive ideal in $\mathcal{X}_{\O(e)}$.
\subsection{Type $({\sf E_7, 4A_1})$}\label{4.4}
In this case we choose the following pinning for $e$:
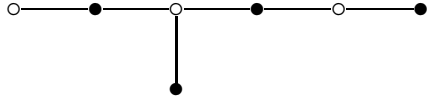
\[\xymatrix{*{\circ}\ar@{-}[r]& *{\bullet} \ar@{-}[r] & *{\circ} \ar@{-}[r]\ar@{-}[d] & *{\bullet} \ar@{-}[r]
&*{\circ}\ar@{-}[r]&*{\bullet}\\ && *{\bullet}&&&}\]
and we take
$$\tau=\textstyle{{(-1)\atop{}}{2\atop{}}{(-3)\atop 2}
{2\atop{}}{(-2)\atop{}}{2\atop{}}}$$ as an optimal cocharacter.
Since $\dim\,\g_e=63$, the total number of positive $0$-roots and
$1$-roots is $(63-7)/2=28$. These roots are given  below.
\begin{table}[htb]
\label{data1}
\begin{tabular}{|c|c|}
\hline$0$-roots  & $1$-roots
\\ \hline
$\scriptstyle{{0\atop{}}{0\atop{}}{0\atop0}{0\atop{}}{1\atop{}}{1\atop{}}}$,
$\scriptstyle{{1\atop{}}{1\atop{}}{1\atop1}{0\atop{}}{0\atop{}}{0\atop{}}}$,
$\scriptstyle{{1\atop{}}{1\atop{}}{1\atop0}{1\atop{}}{0\atop{}}{0\atop{}}}$,
$\scriptstyle{{1\atop{}}{1\atop{}}{1\atop0}{1\atop{}}{1\atop{}}{1\atop{}}}$,&
$\scriptstyle{{1\atop{}}{1\atop{}}{0\atop0}{0\atop{}}{0\atop{}}{0\atop{}}}$,
$\scriptstyle{{0\atop{}}{1\atop{}}{1\atop 0}{1\atop{}}{0\atop{}}{0\atop{}}}$,
$\scriptstyle{{0\atop{}}{1\atop{}}{1\atop0}{1\atop{}}{1\atop{}}{1\atop{}}}$,
$\scriptstyle{{0\atop{}}{0\atop{}}{1\atop1}{1\atop{}}{0\atop{}}{0\atop{}}}$,\\
$\scriptstyle{{0\atop{}}{0\atop{}}{0\atop0}{1\atop{}}{1\atop{}}{0\atop{}}}$,
$\scriptstyle{{1\atop{}}{1\atop{}}{1\atop1}{1\atop{}}{1\atop{}}{0\atop{}}}$,
$\scriptstyle{{0\atop{}}{1\atop{}}{2\atop1}{1\atop{}}{0\atop{}}{0\atop{}}}$,
$\scriptstyle{{0\atop{}}{1\atop{}}{2\atop1}{1\atop{}}{1\atop{}}{1\atop{}}}$,&
$\scriptstyle{{0\atop{}}{0\atop{}}{1\atop1} {1\atop{}}{1\atop{}}{1\atop{}}}$,
$\scriptstyle{{0\atop{}}{1\atop{}}{1\atop1}{0\atop{}}{0\atop{}}{0\atop{}}}$,
$\scriptstyle{{0\atop{}}{1\atop{}}{1\atop1}{1\atop{}}{1\atop{}}{0\atop{}}}$,
$\scriptstyle{{1\atop{}}{2\atop{}}{2\atop1}{1\atop{}}{0\atop{}}{0\atop{}}}$,\\
$\scriptstyle{{0\atop{}}{1\atop{}}{2\atop1}{2\atop{}}{1\atop{}}{0\atop{}}}$,
$\scriptstyle{{0\atop{}}{1\atop{}}{2\atop1}{2\atop{}}{2\atop{}}{1\atop{}}}$,
$\scriptstyle{{1\atop{}}{2\atop{}}{3\atop2}{2\atop{}}{1\atop{}}{0\atop{}}}$,
$\scriptstyle{{1\atop{}}{2\atop{}}{3\atop2}{2\atop{}}{2\atop{}}{1\atop{}}}$,&
$\scriptstyle{{1\atop{}}{2\atop{}}{2\atop1}{1\atop{}}{1\atop{}}{1\atop{}}}$,
$\scriptstyle{{1\atop{}}{2\atop{}}{2\atop1}{2\atop{}}{1\atop{}}{0\atop{}}}$,
$\scriptstyle{{1\atop{}}{2\atop{}}{2\atop1}{2\atop{}}{2\atop{}}{1\atop{}}}$,
$\scriptstyle{{1\atop{}}{3\atop{}}{4\atop2}{3\atop{}}{2\atop{}}{1\atop{}}}$,\\
$\scriptstyle{{2\atop{}}{3\atop{}}{4\atop2}{3\atop{}}{2\atop{}}{1\atop{}}}$,
$\scriptstyle{{1\atop{}}{2\atop{}}{3\atop1}{2\atop{}}{1\atop{}}{1\atop{}}}$,
$\scriptstyle{{1\atop{}}{2\atop{}}{3\atop1}{3\atop{}}{2\atop{}}{1\atop{}}}$.&
$\scriptstyle{{1\atop{}}{1\atop{}}{2\atop1}{2\atop{}}{1\atop{}}{1\atop{}}}$.
\\
\hline
\end{tabular}
\end{table}
The
$(1,4,6)$-contributions of the $0$-roots and $1$-roots are
$(10,28,16)$ and $(7,20,10)$,
respectively, and the total contribution of all roots is
$(\frac{17}{2},24,13)$.

The orbit $\O(e)$ is non-special and $\textstyle{\frac{1}{2}}\dim\O(e)=(133-63)/2=35=
|\Phi^+|-28,$ hence we seek $\lambda+\rho=\sum_{i=1}^7a_i\varpi_i$ for which $\Phi_\lambda$ has type ${\sf A_7}$. To that end we assume that the $a_i$'s are half-integers.
Let $\widetilde{a}_i=2a_i$, whetre $1\le i\le 7$. Then  Condition~(C) for
$\lambda+\rho$ reads
\begin{eqnarray*}
2\widetilde{a}_1+2\widetilde{a}_2+3\widetilde{a}_3+
4\widetilde{a}_4+3\widetilde{a}_5+2\widetilde{a}_6+
\widetilde{a}_7&=&17\\
4\widetilde{a}_1+6\widetilde{a}_2+8\widetilde{a}_3+
12\widetilde{a}_4+9\widetilde{a}_5+6\widetilde{a}_6+
3\widetilde{a}_7&=&48\\
2\widetilde{a}_1+3\widetilde{a}_2+4\widetilde{a}_3+
6\widetilde{a}_4+5\widetilde{a}_5+4\widetilde{a}_6+
2\widetilde{a}_7&=&26.
\end{eqnarray*}
A perfect solution to this system of linear equations is given by setting $\widetilde{a}_i=1$
for all $1\le i\le 7$, leading to $$\lambda+\rho=\textstyle{\frac{1}{2}}
\textstyle{
\sum}_{i=1}^7\varpi_i=\textstyle{\frac{1}{2}}\rho.$$
It is straightforward to see that $\Phi_\lambda$ contains the positive roots
$$
\xymatrix{{\gamma_1}\ar@{-}[r]&{\gamma_2}\ar@{-}[r]& {\gamma_3} \ar@{-}[r] & {\gamma_4}
\ar@{-}[r]&{\gamma_5}\ar@{-}[r]&{\gamma_6}\ar@{-}[r]& {\gamma_7}}
$$
where
\begin{eqnarray*}
&&\gamma_1\,=\,{\scriptstyle{0\atop{}}{1\atop{}}{1\atop0}{0\atop{}}{0\atop{}}{0\atop{}}},\
\
\gamma_2\,=\,{\scriptstyle{0\atop{}}{0\atop{}}{0\atop0}{1\atop{}}{1\atop{}}{0\atop{}}},\
\
\gamma_3\,=\,{\scriptstyle{0\atop{}}{0\atop{}}{1\atop1}{0\atop{}}{0\atop{}}{0\atop{}}},\
\
\gamma_4\,=\,{\scriptstyle{1\atop{}}{1\atop{}}{0\atop0}{0\atop{}}{0\atop{}}{0\atop{}}},\\
\
&&\gamma_5\,=\,{\scriptstyle{0\atop{}}{0\atop{}}{1\atop0}{1\atop{}}{0\atop{}}{0\atop{}}},\
\
\gamma_6\,=\,{\scriptstyle{0\atop{}}{0\atop{}}{0\atop0}{0\atop{}}{1\atop{}}{1\atop{}}},\
\
\gamma_7\,=\,{\scriptstyle{0\atop{}}{1\atop{}}{1\atop1}{1\atop{}}{0\atop{}}{0\atop{}}}.
\end{eqnarray*}
Repeating verbatim the argument used in the previous case we deduce that $\gamma_1,\ldots,\gamma_7$ form the basis of $\Phi_\lambda$ contained in $\Phi^+$. Since $\langle\lambda+\rho,\gamma_i^\vee\rangle\in\Z^{>0}$ for all $i$, it follows that $$\dim\,{\rm
VA}(I(\lambda))=|\Phi|-|\Phi_\lambda|=\dim\,\O(e).$$
Since $\langle\lambda+\rho,\alpha_i^\vee\rangle\not\in\Z^{>0}$ for $i=2,3,5,7$ by our choice of pinning,
$\lambda+\rho$ satisfies Condition~(A) yielding $e\in{\rm VA}(I(\lambda))$. But then
${\rm VA}(I(\lambda))=\overline{\O(e)}$ i.e. Condition~(B) holds for
$\lambda+\rho$. Since in the present case Condition~(D)is again vacuous and $\g_e=[\g_e,\g_e]$ by
[\cite{dG}], we conclude that $I(-\frac{1}{2}\rho)$ is the only
multiplicity-free primitive ideal in $\mathcal{X}_{\O(e)}$.
\subsection{Type $({\sf E_7, A_2+2A_1})$}\label{4.5}
This is a special orbit and
$e^\vee\in\g^\vee$ has type ${\sf E_7(a_4)}$.
So we may assume that
$h^\vee=\textstyle{{2\atop{}}{0\atop{}}{2\atop 0}
{0\atop{}}{0\atop{}}{2\atop{}}}$. In view of Condition~(A) we use the following pinning
for $e$:
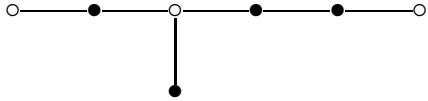
\[\xymatrix{*{\circ}\ar@{-}[r]& *{\bullet} \ar@{-}[r] & *{\circ} \ar@{-}[r]\ar@{-}[d] & *{\bullet} \ar@{-}[r]
&*{\bullet}\ar@{-}[r]&*{\circ}\\ && *{\bullet}&&&}\]
and we choose
$$\tau=\textstyle{{(-1)\atop{}}{2\atop{}}{(-4)\atop 2}
{2\atop{}}{2\atop{}}{(-2)\atop{}}}$$ to be our optimal cocharacter. As
$\dim\,\g_e=51$, the total number of positive $0$-roots and
$1$-roots is $(51-7)/2=22$. These roots are given
in the table below.
\begin{table}[htb]
\label{data1}
\begin{tabular}{|c|c|}
\hline$0$-roots  & $1$-roots
\\ \hline
$\scriptstyle{{0\atop{}}{1\atop{}}{1\atop1}{0\atop{}}{0\atop{}}{0\atop{}}}$,
$\scriptstyle{{0\atop{}}{1\atop{}}{1\atop0}{1\atop{}}{0\atop{}}{0\atop{}}}$,
$\scriptstyle{{0\atop{}}{0\atop{}}{0\atop0}{0\atop{}}{1\atop{}}{1\atop{}}}$,
$\scriptstyle{{0\atop{}}{0\atop{}}{1\atop1}{1\atop{}}{0\atop{}}{0\atop{}}}$,&
$\scriptstyle{{1\atop{}}{1\atop{}}{0\atop0}{0\atop{}}{0\atop{}}{0\atop{}}}$,
$\scriptstyle{{1\atop{}}{1\atop{}}{1\atop 0}{1\atop{}}{0\atop{}}{0\atop{}}}$,
$\scriptstyle{{1\atop{}}{1\atop{}}{1\atop0}{1\atop{}}{1\atop{}}{0\atop{}}}$,
$\scriptstyle{{1\atop{}}{1\atop{}}{1\atop1}{1\atop{}}{1\atop{}}{1\atop{}}}$,\\
$\scriptstyle{{0\atop{}}{0\atop{}}{1\atop1}{1\atop{}}{1\atop{}}{1\atop{}}}$,
$\scriptstyle{{0\atop{}}{1\atop{}}{1\atop0}{1\atop{}}{1\atop{}}{1\atop{}}}$,
$\scriptstyle{{0\atop{}}{0\atop{}}{1\atop0}{1\atop{}}{1\atop{}}{0\atop{}}}$,
$\scriptstyle{{0\atop{}}{1\atop{}}{2\atop1}{1\atop{}}{1\atop{}}{0\atop{}}}$,&
$\scriptstyle{{1\atop{}}{2\atop{}}{2\atop1} {1\atop{}}{1\atop{}}{0\atop{}}}$,
$\scriptstyle{{1\atop{}}{1\atop{}}{2\atop1}{1\atop{}}{1\atop{}}{0\atop{}}}$,
$\scriptstyle{{1\atop{}}{1\atop{}}{2\atop1}{2\atop{}}{2\atop{}}{1\atop{}}}$,
$\scriptstyle{{1\atop{}}{2\atop{}}{2\atop1}{2\atop{}}{1\atop{}}{1\atop{}}}$,\\
$\scriptstyle{{0\atop{}}{1\atop{}}{2\atop1}{2\atop{}}{1\atop{}}{1\atop{}}}$,
$\scriptstyle{{2\atop{}}{3\atop{}}{4\atop2}{3\atop{}}{2\atop{}}{1\atop{}}}$.&
$\scriptstyle{{1\atop{}}{2\atop{}}{3\atop1}{3\atop{}}{2\atop{}}{1\atop{}}}$,
$\scriptstyle{{1\atop{}}{2\atop{}}{3\atop2}{2\atop{}}{1\atop{}}{0\atop{}}}$,
$\scriptstyle{{1\atop{}}{2\atop{}}{3\atop2}{2\atop{}}{2\atop{}}{1\atop{}}}$,
$\scriptstyle{{1\atop{}}{3\atop{}}{4\atop2}{3\atop{}}{2\atop{}}{1\atop{}}}$.
\\
\hline
\end{tabular}
\end{table}
The
$(1,4,7)$-contributions of the $0$-roots and $1$-roots are
$(2,14,5)$ and $(12,24,6)$,
respectively, and the total contribution of all roots is
$(7,19,\frac{11}{2})$.
 So Condition~(C) for
$\lambda+\rho$
reads
\begin{eqnarray*}
2a_1+2a_2+3a_3+4a_4+3a_5+2a_6+a_7&=&7\\
4a_1+7a_2+8a_3+12a_4+9a_5+6a_6+3a_7&=&19\\
2a_1+3a_2+4a_3+6a_4+5a_5+4a_6+3a_7&=&11.
\end{eqnarray*}
A nice solution to this system of linear equations is given by setting $a_i=0$ for $i=2,3,5,6$ and $a_i=1$ for
$i=1,4,7$.
It follows that $\lambda+\rho=\frac{1}{2}h^\vee$ satisfies Losev's condition~(C). As
$\langle\lambda,\alpha_i^\vee\rangle=0$ for $i=3,4,5,6$, our choice of pinning shows that
Condition~(A) also holds for $\lambda+\rho$. Since $e$ is special and
$\lambda+\rho=\frac{1}{2}h^\vee$, Condition~(B) follows from [\cite{BV},
Proposition~5.10]. As Condition (D) is again vacuous and $\g_e=[\g_e,\g_e]$ by [\cite{dG}], we conclude that
$I(\lambda)=I(-\varpi_2-\varpi_3-\varpi_5-\varpi_6)$ is the only multiplicity-free primitive
ideal in $\mathcal{X}_{\O(e)}$.
\subsection{Type $({\sf E_7, 2A_2+A_1})$}\label{4.6}
Here our pinning for $e$ is
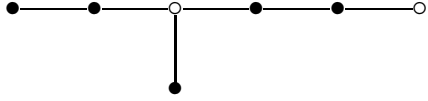
\[\xymatrix{*{\bullet}\ar@{-}[r]& *{\bullet} \ar@{-}[r] & *{\circ} \ar@{-}[r]\ar@{-}[d] & *{\bullet} \ar@{-}[r]
&*{\bullet}\ar@{-}[r]&*{\circ}\\ && *{\bullet}&&&}\]
and we take
$$\tau=\textstyle{{2\atop{}}{2\atop{}}{(-5)\atop 2}
{2\atop{}}{2\atop{}}{(-2)\atop{}}}$$ as an optimal cocharacter.
Since $\dim\,\g_e=43$, the total number of positive $0$-roots and
$1$-roots is $(43-7)/2=18$. These roots are given  below.
\begin{table}[htb]
\label{data1}
\begin{tabular}{|c|c|}
\hline$0$-roots  & $1$-roots
\\ \hline
$\scriptstyle{{0\atop{}}{0\atop{}}{0\atop0}{0\atop{}}{1\atop{}}{1\atop{}}}$,
$\scriptstyle{{1\atop{}}{1\atop{}}{2\atop1}{1\atop{}}{1\atop{}}{0\atop{}}}$,
$\scriptstyle{{1\atop{}}{2\atop{}}{2\atop1}{1\atop{}}{0\atop{}}{0\atop{}}}$,
$\scriptstyle{{1\atop{}}{2\atop{}}{2\atop1}{1\atop{}}{1\atop{}}{1\atop{}}}$,&
$\scriptstyle{{1\atop{}}{1\atop{}}{1\atop1}{0\atop{}}{0\atop{}}{0\atop{}}}$,
$\scriptstyle{{0\atop{}}{1\atop{}}{1\atop 1}{1\atop{}}{0\atop{}}{0\atop{}}}$,
$\scriptstyle{{1\atop{}}{1\atop{}}{1\atop0}{1\atop{}}{0\atop{}}{0\atop{}}}$,
$\scriptstyle{{0\atop{}}{1\atop{}}{1\atop0}{1\atop{}}{1\atop{}}{0\atop{}}}$,\\
$\scriptstyle{{0\atop{}}{1\atop{}}{2\atop1}{2\atop{}}{1\atop{}}{0\atop{}}}$,
$\scriptstyle{{0\atop{}}{1\atop{}}{2\atop1}{2\atop{}}{2\atop{}}{1\atop{}}}$,
$\scriptstyle{{1\atop{}}{3\atop{}}{4\atop2}{3\atop{}}{2\atop{}}{1\atop{}}}$,
$\scriptstyle{{1\atop{}}{1\atop{}}{2\atop1}{2\atop{}}{1\atop{}}{1\atop{}}}$.&
$\scriptstyle{{0\atop{}}{0\atop{}}{1\atop1} {1\atop{}}{1\atop{}}{0\atop{}}}$,
$\scriptstyle{{1\atop{}}{1\atop{}}{1\atop0}{1\atop{}}{1\atop{}}{1\atop{}}}$,
$\scriptstyle{{0\atop{}}{1\atop{}}{1\atop1}{1\atop{}}{1\atop{}}{1\atop{}}}$,
$\scriptstyle{{1\atop{}}{2\atop{}}{3\atop2}{2\atop{}}{1\atop{}}{0\atop{}}}$,\\
&$\scriptstyle{{1\atop{}}{2\atop{}}{3\atop2}{2\atop{}}{2\atop{}}{1\atop{}}}$,
$\scriptstyle{{1\atop{}}{2\atop{}}{3\atop1}{3\atop{}}{2\atop{}}{1\atop{}}}$.
\\
\hline
\end{tabular}
\end{table}
The
$(4,7)$-contributions of the $0$-roots and $1$-roots are
$(16,5)$ and $(16,4)$,
respectively, and the total contribution of all roots is
$(16,\frac{9}{2})$.

The orbit $\O(e)$ is non-special and $\textstyle{\frac{1}{2}}\dim\O(e)=(133-43)/2=45=
|\Phi^+|-(15+3),$ hence we seek $\lambda+\rho=\sum_{i=1}^7a_i\varpi_i$ for which $\Phi_\lambda$ has type ${\sf A_5+A_2}$. So it is reasonable to assume that the $a_i\in\frac{1}{3}\Z$ for all $i$.
Setting $\widetilde{a}_i:=3a_i$ for $1\le i\le 7$ we can express Condition~(C) for
$\lambda+\rho$ as follows:
\begin{eqnarray*}
4\widetilde{a}_1+6\widetilde{a}_2+8\widetilde{a}_3+
12\widetilde{a}_4+9\widetilde{a}_5+6\widetilde{a}_6+
3\widetilde{a}_7&=&48\\
2\widetilde{a}_1+3\widetilde{a}_2+4\widetilde{a}_3+
6\widetilde{a}_4+5\widetilde{a}_5+4\widetilde{a}_6+
3\widetilde{a}_7&=&27.
\end{eqnarray*}
A perfect solution to this system of linear equations is given by setting $\widetilde{a}_i=1$
for all $i$, leading to $$\lambda+\rho=\textstyle{\frac{1}{3}}
\textstyle{
\sum}_{i=1}^7\varpi_i=\textstyle{\frac{1}{3}}\rho.$$
Note that $\Phi_\lambda$ contains the positive roots
$$
\xymatrix{{\gamma_1}\ar@{-}[r]&{\gamma_2}\ar@{-}[r]& {\gamma_3} \ar@{-}[r] & {\gamma_4}
\ar@{-}[r]&{\gamma_5}&
{\gamma_6}\ar@{-}[r]& {\gamma_7}}
$$
where
\begin{eqnarray*}
&&\gamma_1\,=\,{\scriptstyle{0\atop{}}{0\atop{}}{1\atop0}{1\atop{}}{1\atop{}}{0\atop{}}},\
\
\gamma_2\,=\,{\scriptstyle{0\atop{}}{1\atop{}}{1\atop1}{0\atop{}}{0\atop{}}{0\atop{}}},\
\
\gamma_3\,=\,{\scriptstyle{0\atop{}}{0\atop{}}{0\atop0}{1\atop{}}{1\atop{}}{1\atop{}}},\
\
\gamma_4\,=\,{\scriptstyle{1\atop{}}{1\atop{}}{1\atop0}{0\atop{}}{0\atop{}}{0\atop{}}},\\
\
&&\gamma_5\,=\,{\scriptstyle{0\atop{}}{0\atop{}}{1\atop1}{1\atop{}}{0\atop{}}{0\atop{}}},\
\
\gamma_6\,=\,{\scriptstyle{0\atop{}}{1\atop{}}{1\atop0}{1\atop{}}{0\atop{}}{0\atop{}}},\
\
\gamma_7\,=\,{\scriptstyle{1\atop{}}{1\atop{}}{1\atop1}{1\atop{}}{1\atop{}}{0\atop{}}}.
\end{eqnarray*}
Since the subsystems of type ${\sf A_5+A_2}$ are maximal in $\Phi$, the roots $\gamma_1,\ldots,\gamma_7$ form the basis of $\Phi_\lambda$ contained in $\Phi^+$. Since $\langle\lambda+\rho,\gamma_i^\vee\rangle\in\Z^{>0}$ for all $i$, this yields $$\dim\,{\rm
VA}(I(\lambda))=|\Phi|-|\Phi_\lambda|=\dim\,\O(e).$$
Since $\langle\lambda+\rho,\beta^\vee\rangle\not\in\Z^{>0}$ for any root $\beta\in\Phi^+$ which can be expressed as a linear combination $\alpha_i$ with $i\in\{1,2,3,5,6\}$,
we see that
$\lambda+\rho$ satisfies Condition~(A). But then $e\in{\rm VA}(I(\lambda))$, forcing
${\rm VA}(I(\lambda))=\overline{\O(e)}$. So  Condition~(B) holds for
$\lambda+\rho$ as well. Since Condition~(D) is still vacuous in the present case and $\g_e=[\g_e,\g_e]$ by
[\cite{dG}], we conclude that $I(-\frac{2}{3}\rho)$ is the only
multiplicity-free primitive ideal in $\mathcal{X}_{\O(e)}$.
\subsection{Type $({\sf E_7, (A_3+A_1})')$}\label{4.7}
In this case $\g_e=\mathbb{C}e\oplus[\g_e,\g_e]$ by [\cite{dG}] and so we may expect that $U(\g,e)$ affords at least $2$ one-dimensional representations.
In fact, there are exactly two of them, by [\cite{GRU}], and our goal is to find Duflo realisations of the corresponding primitive ideals of $U(\g)$.
We choose the following pinning for $e$:
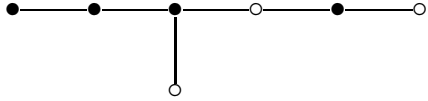
\[\xymatrix{*{\bullet}\ar@{-}[r]& *{\bullet} \ar@{-}[r] & *{\bullet} \ar@{-}[r]\ar@{-}[d] & *{\circ} \ar@{-}[r]
&*{\bullet}\ar@{-}[r]&*{\circ}\\ && *{\circ}&&&}\]
and we take
$$\tau=\textstyle{{2\atop{}}{2\atop{}}{2\atop (-3)}
{(-4)\atop{}}{2\atop{}}{(-1)\atop{}}}$$ as our optimal cocharacter.
Since $\dim\,\g_e=41$, the total number of positive $0$-roots and
$1$-roots is $(41-7)/2=17$. These roots are given  below.
\begin{table}[htb]
\label{data1}
\begin{tabular}{|c|c|}
\hline$0$-roots  & $1$-roots
\\ \hline
$\scriptstyle{{0\atop{}}{1\atop{}}{1\atop0}{1\atop{}}{0\atop{}}{0\atop{}}}$,
$\scriptstyle{{0\atop{}}{0\atop{}}{1\atop0}{1\atop{}}{1\atop{}}{0\atop{}}}$,
$\scriptstyle{{1\atop{}}{1\atop{}}{1\atop1}{1\atop{}}{1\atop{}}{1\atop{}}}$,
$\scriptstyle{{0\atop{}}{1\atop{}}{2\atop1}{1\atop{}}{1\atop{}}{1\atop{}}}$,&
$\scriptstyle{{0\atop{}}{0\atop{}}{0\atop0}{0\atop{}}{1\atop{}}{1\atop{}}}$,
$\scriptstyle{{0\atop{}}{1\atop{}}{1\atop 1}{0\atop{}}{0\atop{}}{0\atop{}}}$,
$\scriptstyle{{0\atop{}}{1\atop{}}{1\atop0}{1\atop{}}{1\atop{}}{1\atop{}}}$,
$\scriptstyle{{1\atop{}}{1\atop{}}{1\atop1}{1\atop{}}{1\atop{}}{0\atop{}}}$,\\
$\scriptstyle{{1\atop{}}{2\atop{}}{2\atop1}{2\atop{}}{1\atop{}}{1\atop{}}}$,
$\scriptstyle{{1\atop{}}{1\atop{}}{2\atop1}{2\atop{}}{2\atop{}}{1\atop{}}}$,
$\scriptstyle{{1\atop{}}{2\atop{}}{3\atop1}{3\atop{}}{2\atop{}}{1\atop{}}}$,
$\scriptstyle{{1\atop{}}{2\atop{}}{3\atop2}{2\atop{}}{1\atop{}}{0\atop{}}}$.&
$\scriptstyle{{0\atop{}}{1\atop{}}{2\atop1} {1\atop{}}{1\atop{}}{0\atop{}}}$,
$\scriptstyle{{1\atop{}}{2\atop{}}{2\atop1}{2\atop{}}{1\atop{}}{0\atop{}}}$,
$\scriptstyle{{1\atop{}}{2\atop{}}{3\atop2}{2\atop{}}{2\atop{}}{1\atop{}}}$,
$\scriptstyle{{1\atop{}}{3\atop{}}{4\atop2}{3\atop{}}{2\atop{}}{1\atop{}}}$,\\
&$\scriptstyle{{1\atop{}}{1\atop{}}{2\atop1}{1\atop{}}{0\atop{}}{0\atop{}}}$.
\\
\hline
\end{tabular}
\end{table}
The
$(2,5,7)$-contributions of the $0$-roots and $1$-roots are
$(7,13,5)$ and $(9,11,4)$,
respectively, and the total contribution of all roots is
$(8,12,\frac{9}{2})$.
The
orbit $\O(e)$ is non-special and analysing all options available in the present case
one eventually comes to the conclusion that
$\lambda+\rho$ must have an integral root system of type ${\sf D_6+A_1}$.
Furthermore, $\lambda+\rho$ must be strongly dominant
on the positive root coming from the ${\sf A_1}$-component of $\Phi_\lambda$. Let $\g_1(\lambda)$ denote the simple ideal of type ${\sf D_6}$ in the Lie algebra $\g(\lambda)$.
In view of [\cite{Lo2}, Proposition~5.3.2] we should
look for a special nilpotent
orbit of dimension
$(133-41)-(133-66-3)=28$ in $\g_1(\lambda)$.

Let $V$ be a $12$-dimensional vector space over $\mathbb C$ and let
$\Psi$ be a non-degenerate symmetric bilinear form on $V$. We
identify $\g_1(\lambda)$ with the stabiliser of $\Psi$  in $\mathfrak{gl}(V)$ and adopt the notation introduced in Subsection~\ref{3.11}.
Let $X$ be a nilpotent element of $\g_1(\lambda)$ with ${\bf p}(X)=(2^4,1^4)$ (in the present case
${\bf p}(X)$ is a partition of $12$). Then
${\bf r}(X)=(8,4)$ and $n_{\rm odd}(X)=4$ implying
$$
\dim \g_1(\lambda)_{X}=\,\textstyle{{\frac{1}{2}}}\big(\big({\textstyle\sum}_{i=1}^k
r_i^2\big)-n_{\rm odd}(X)\big)=\textstyle{\frac{1}{2}}(64+16-4)=38.
$$
So the orbit of $X$
is special (by [\cite{C}, p.~437], for example) of dimension $66-38=28$.

The above discussion also indicates that we should seek
$\lambda+\rho=\sum_{i=1}^7a_i\varpi_i$ for which all $a_i$ are
half-integers. Setting $\widetilde{a}_i=2a_i$ for $1\le i\le 7$ we
can express Condition~(C) as follows:
\begin{eqnarray*}
4\widetilde{a}_1+7\widetilde{a}_2+8\widetilde{a}_3+
12\widetilde{a}_4+9\widetilde{a}_5+6\widetilde{a}_6+
3\widetilde{a}_7&=&32\\
6\widetilde{a}_1+9\widetilde{a}_2+12\widetilde{a}_3+
18\widetilde{a}_4+15\widetilde{a}_5+10\widetilde{a}_6+
5\widetilde{a}_7&=&48,\\
2\widetilde{a}_1+3\widetilde{a}_2+4\widetilde{a}_3+
6\widetilde{a}_4+5\widetilde{a}_5+4\widetilde{a}_6+
3\widetilde{a}_7&=&18.
\end{eqnarray*}
A rather nice solution to this system of linear equations is given by setting
$\widetilde{a_i}=1$ for $i=1,2,4,6,7$ and $\widetilde{a}_i=0$ for $i=,3,5$. Although the corresponding highest weight $\mu+\rho=\frac{1}{2}(\varpi_1+\varpi_2+\varpi_4+\varpi_6+\varpi_7)$ does not satisfy Condition~(A), this can be easily amended by replacing it with
$$\lambda+\rho=s_3s_1s_4(\mu+\rho)=\textstyle{\frac{1}{2}}(\varpi_1+\varpi_4+\varpi_5+\varpi_6+\varpi_7)+
\varpi_2-\varpi_3,$$ which still satisfies Condition~(C). Since $\langle\lambda+\rho,\gamma^\vee\rangle\not\in\Z^{>0}$ for any root $\gamma\in\Phi^+$ that can be expressed as a linear combination of $\alpha_1,\alpha_3,\alpha_4,\alpha_6$, our choice of pinning shows that Condition~(A) also holds for $\lambda+\rho$.

The integral root system of $\lambda+\rho$ contains the positive roots
$$
\xymatrix{{\beta_1}\ar@{-}[r]&{\beta_2}\ar@{-}[r]& {\beta_3} \ar@{-}[r] & {\beta_4}
\ar@{-}[r]\ar@{-}[d] & {\beta_5}
 & {\beta_7} \\ &&& {\beta_6}&&&}
$$
where
\begin{eqnarray*}
&&\beta_1\,=\,{\scriptstyle{0\atop{}}{0\atop{}}{0\atop0}{1\atop{}}{1\atop{}}{0\atop{}}},\
\
\beta_2\,=\,{\scriptstyle{1\atop{}}{1\atop{}}{1\atop0}{0\atop{}}{0\atop{}}{0\atop{}}},\
\
\beta_3\,=\,{\scriptstyle{0\atop{}}{0\atop{}}{0\atop1}{0\atop{}}{0\atop{}}{0\atop{}}},\
\
\beta_4\,=\,{\scriptstyle{0\atop{}}{0\atop{}}{1\atop0}{1\atop{}}{0\atop{}}{0\atop{}}},\\
\
&&\beta_5\,=\,{\scriptstyle{0\atop{}}{0\atop{}}{0\atop0}{0\atop{}}{1\atop{}}{1\atop{}}},\
\
\beta_6\,=\,{\scriptstyle{0\atop{}}{1\atop{}}{0\atop0}{0\atop{}}{0\atop{}}{0\atop{}}},\
\
\beta_7\,=\,{\scriptstyle{0\atop{}}{1\atop{}}{2\atop1}{1\atop{}}{1\atop{}}{0\atop{}}}.
\end{eqnarray*}
From the maximality of the root subsystems of type ${\sf D_6+A_1}$ in $\Phi$ it follows that $\beta_1,\ldots,\beta_7$ form a basis of $\Phi_\lambda$ contained in $\Phi^+$.
Let $\varpi'_1,\ldots,\varpi'_7\in P(\Phi)_{\mathbb Q}$ be the corresponding fundamental weights, so  that $\langle\varpi'_i,\beta_j^\vee\rangle=\delta_{ij}$ for all $1\le i,j\le 7$.

Using the explicit expressions for the $\beta_i$'s it is straightforward to check that
$$
\lambda+\rho\,=\,\varpi_1'+\varpi_3'+\varpi'_4+
\varpi'_5-\varpi'_6+2\varpi'_7\,=\,
s_{\beta_6}s_{\beta_4}s_{\beta_2}(\varpi'_1+\varpi'_3+\varpi'_5+\varpi'_6+
2\varpi'_7).
$$

Set $\nu:=s_{\beta_2}s_{\beta_4}s_{\beta_6}
\centerdot\lambda$ and recall the notation introduced in Subsection~\ref{3.7}. The above shows that
$\nu+\rho\in\Lambda^+$ and $\Pi_\nu^0=\{\beta_2,\beta_4\}=\tau_\Lambda(s_{\beta_6}s_{\beta_4}
s_{\beta_2}),$
Thanks to
[\cite{Ja}, Corollar~10.10] we have that
$$d\big(U(\g)/I(\lambda)\big)=
d\big(U(\g)/I(s_{\beta_6}s_{\beta_4}s_{\beta_2}
\centerdot\nu)\big)=
d\big(U(\g)/I(s_{\beta_6}s_{\beta_4}s_{\beta_2}
\centerdot\eta)\big)$$
for any regular $\eta\in\Lambda^+$, and
\begin{eqnarray*}
d\big(U(\g)/I(s_{\beta_6}s_{\beta_4}s_{\beta_2}
\centerdot\eta)\big)&=&
d\big(U(\g)/I(s_{\beta_2}\underline{s_{\beta_4}
s_{\beta_6}}\centerdot\eta)\big)\,=\,
d\big(U(\g)/I(s_{\beta_2}s_{\beta_4}
\centerdot\eta)\big)\\
&=&d\big(U(\g)/I(s_{\beta_2}s_{\beta_4}
\centerdot\nu)\big)\,=\,
d\big(U(\g)/I(\nu)\big).
\end{eqnarray*}
Since $\g(\lambda)$ is a Lie algebra of type
of type ${\sf D_6+A_1}$, applying [\cite{Lo2},
Proposition~5.3.2] now shows that
$$d\big(U(\g)/I(\lambda)\big)=\,
d\big(U(\g)/I(\nu)\big)=\,d\big(U(\g)/
I(s_{\beta_2}s_{\beta_4}\centerdot\eta)\big)=\,
64+\dim\O_\lambda(w_0s_{\beta_2}s_{\beta_4})$$
where $w_0$ is the longest element of $W_\lambda$ and
$\O_\lambda(w_0s_{\beta_2}s_{\beta_4})$ is the special nilpotent
orbit in $\g(\lambda)$ attached to the double cell of $W_\lambda$
containing $w_0s_{\beta_2}s_{\beta_4}$.

Let $\t'=\t\cap\g_1(\lambda)$ and identify $\beta_1,\ldots,\beta_6$ with a basis of simple roots in the root system of $\g_1(\lambda)$ with respect to $\t'$
(this will also enable us to view $\varpi'_1,\ldots,\varpi'_6$ as the corresponding system of fundamental weights).
Let $\rho'=\sum_{i=1}^6$ $\varpi_i'$ and denote by
$\nu'$ the restriction of $\nu+\rho-\rho'$ to $\t'$. Let $I_1(\nu')$ be the annihilator in $U(\g_1(\lambda))$ of the irreducible $\g_1(\lambda)$-module of highest weight $\nu'$.
Since
$\langle\nu+\rho,\beta_7^\vee\rangle\in\Z^{>0}$, the aove shows that
$
d\big(U(\g)/I(\nu)\big)=64+\dim{\rm VA}(I_1(\nu'))
$ and
$$
\dim\O_\lambda(w_0s_{\beta_2}s_{\beta_4})=\,\dim{\rm VA}(I_1(\nu'))={\rm VA}(I_1(-\varpi'_2-\varpi'_4)).
$$
Let $X^\vee$ be an element from the nilpotent orbit of $\g_1(\lambda)$ corresponding to the orbit of $X$ under the Spaltenstein duality. As $X$ is regular in a Levi subalgebra of $\g_1(\lambda)$ whose semisimple part has type ${\sf A_1+A_1}$, it
follows from [\cite{BV}, Appendix], for example,
that ${\bf p}(X^\vee)=(7,5)$.
Applying the recipe described in [\cite{C}, p.~395] we now observe that the Dynkin label of
the orbit of $X^\vee$ equals
$$\textstyle{{2\atop{}}{0\atop{}}{2\atop{}}
{0\atop2}{2\atop{}}}\,=\,2(\nu'+\rho').$$
But then [\cite{BV}, Proposition~5.10] yields that $X$ lies in the open orbit of ${\rm VA}(I_1(\nu'))$.
As a result, $\dim {\rm VA}(I(\lambda))=64+28=92=\dim\O(e).$ Since Condition~(A) holds for $\lambda+\rho$ by our choice of pinning, we now deduce that ${\rm VA}(I(\lambda))=\overline{\O(e)}$, that is Condition~(B) holds for $\lambda+\rho$. As $e$ has standard Levi type, Condition~(D) holds for $\lambda+\rho$  automatically. So $I(-2\varpi_3-\frac{1}{2}(\varpi_1+\varpi_4+\varpi_5+\varpi_6+\varpi_7))$ is a multiplicity-free primitive ideal in $\mathcal{X}_{\O(e)}$.

As in Subsection~\ref{3.9}, we are going to search for the second  multiplicity-free primitive ideal $I(\lambda')\in\mathcal{X}_{\O(e)}$
by imposing two additional assumptions on $\lambda'+\rho=\sum_{i=1}^7b_i\varpi_i$:
\begin{itemize}
\item[(a)\ ]
$\Phi_{\lambda'}=\Phi_\lambda$;

\smallskip

\item[(b)\ ]
$\langle\lambda'+\rho,\beta_7^\vee\rangle=1$.
\end{itemize}
Condition~(C) together with (b) leads to the following
system of linear equations:
\begin{eqnarray*}
\widetilde{b}_2+\widetilde{b}_3+2\widetilde{b}_4+
\widetilde{b}_5+\widetilde{b}_6&=&2\\
4\widetilde{b}_1+7\widetilde{b}_2+8\widetilde{b}_3+
12\widetilde{b}_4+
9\widetilde{b}_5+6\widetilde{b}_6+3\widetilde{b}_7
&=&32\\
6\widetilde{b}_1+9\widetilde{b}_2+12\widetilde{b}_3+
18\widetilde{b}_4+
15\widetilde{b}_5+10\widetilde{b}_6+5\widetilde{b}_7
&=&48\\
2\widetilde{b}_1+3\widetilde{b}_2+4\widetilde{b}_3+
6\widetilde{b}_4+
5\widetilde{b}_5+4\widetilde{b}_6+3\widetilde{b}_7
&=&18.
\end{eqnarray*}
Here  $\widetilde{b}_i=2b_i$ for all $1\le i\le 7$. After several unsuccessful attempts we realised that setting $\widetilde{b}_1=b_5=3$, $\widetilde{b}_2=-\widetilde{b}_3=4$, $\widetilde{b}_4=-1$ and
$\widetilde{b_i}=1$ for $i\in \{5,6,7\}$ provides us with a promising solution
which
leads to the weight
$$\lambda'+\rho\,=\textstyle{\frac{1}{2}}(3\varpi_1+4\varpi_2-4\varpi_3-\varpi_4+
3\varpi_5+\varpi_6+\varpi_7).$$
By our choice of pinning, Conditions~(A) holds for $\lambda'+\rho$ and computations show that
$$\lambda'+\rho\,=\,2\varpi'_1-\varpi'_2+2\varpi'_3+
\varpi'_4+\varpi'_5-2\varpi'_6+\varpi'_7\,=\,
s_{\beta_6}s_{\beta_4}s_{\beta_2}s_{\beta_3}
s_{\beta_5}\big(\textstyle{\sum}_{i\ne
3,5}\,\varpi'_i\big).$$ Then
$\lambda'=s_{\beta_6}s_{\beta_4}s_{\beta_2}
s_{\beta_3}s_{\beta_5}\centerdot\mu'$
where $\mu'=
s_{\beta_5}s_{\beta_3}s_{\beta_2}s_{\beta_4}s_{\beta_6}(\lambda'+\rho)-\rho\in\Lambda^+$,
and
$$\Pi_{\mu'}^0=\{\beta_3,\beta_5\}\subseteq
\tau_\Lambda(s_{\beta_6}s_{\beta_2}s_{\beta_4}
s_{\beta_3}s_{\beta_5}).$$
Applying [\cite{Ja}, Corollar~10.10] we that
$$d\big(U(\g)/I(\lambda')\big)=d\big(U(\g)/
I(s_{\beta_6}s_{\beta_4}s_{\beta_2}s_{\beta_3}s_{\beta_5}\centerdot\mu')\big)=
d\big(U(\g)/I(s_{\beta_6}s_{\beta_4}
s_{\beta_2}s_{\beta_3}s_{\beta_5}\centerdot\eta)\big)$$
for any regular $\eta\in\Lambda^+$ and
\begin{eqnarray*}
d\big(U(\g)/I(\lambda')\big)&=&
d\big(U(\g)/I(s_{\beta_6}s_{\beta_4}s_{\beta_5}
\underline{s_{\beta_2}s_{\beta_3}}\centerdot\eta)
\big)=\,
d\big(U(\g)/I(s_{\beta_6}s_{\beta_4}s_{\beta_5}
s_{\beta_2}\centerdot\eta)\big)\\
&=&d\big(U(\g)/I(s_{\beta_2}s_{\beta_5}
\underline{s_{\beta_4}s_{\beta_6}}
\centerdot\eta)\big)=\,
d\big(U(\g)/I(s_{\beta_2}s_{\beta_5}
s_{\beta_4}\centerdot\eta)\big)\\
&=&
d\big(U(\g)/I(s_{\beta_4}s_{\beta_5}
s_{\beta_2}\centerdot\eta)\big)=\,
d\big(U(\g)/I(s_{\beta_2}\underline{s_{\beta_4}
s_{\beta_5}}\centerdot\eta)\big)\\
&=&
d\big(U(\g)/I(s_{\beta_2}s_{\beta_4}\centerdot\eta)
\big)=\,d\big(U(\g)/I(\lambda))=\,\dim
{\rm VA}(\O(e)),
\end{eqnarray*}
where the last equality follows from our earlier results in this subsection. Since Condition~(D) is vacuous in the present case, we conclude that $\lambda'+\rho$ satisfies all Losev's conditions and hence
$I(\varpi_2-3\varpi_3+\frac{1}{2}(\varpi_1-3\varpi_4
+\varpi_5-\varpi_6-\varpi_7))$
is another multiplicity-free primitive ideal in $\mathcal{X}_{\O(e)}$.

If $I(\lambda)=I(\lambda')$ then $\lambda'+\rho=w(\lambda+\rho)$ for
some $w\in W$. Since $\Phi_\lambda=\Phi_{\lambda'}$ it must be that
$w^{-1}(\Phi_\lambda)=\Phi_\lambda$. But then $w$ must preserve the ${\sf A_1}$-component of
$\Phi_\lambda$ forcing $w^{-1}(\beta_7)=\pm\beta_7$ The
latter, however, is impossible since $\lambda+\rho$ and
$\lambda'+\rho$ take different positive values at $\beta_7^\vee$. Hence $I(\lambda)$ and $I(\lambda')$ are
the distinct multiplicity-free primitive ideals in $\mathcal{X}_{\O(e)}$.
\subsection{Type $({\sf E_6, A_1})$}\label{4.8}
In this case $e$ is special and $e^\vee\in\g^\vee$ is subregular. So it can be assumed that
$h^\vee=\textstyle{{2\atop{}}{2\atop{}}{0\atop 2}
{2\atop{}}{2\atop{}}}$. Due to Condition~(A) our pinning for $e$ is
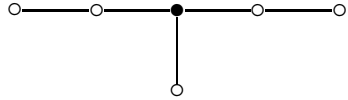
\[\xymatrix{*{\circ}\ar@{-}[r]& *{\circ} \ar@{-}[r] & *{\bullet} \ar@{-}[r]\ar@{-}[d] & *{\circ} \ar@{-}[r]
&*{\circ}\\ && *{\circ}&&}\]
and we choose
$$\tau=\textstyle{{0\atop{}}{(-1)\atop{}}{2\atop (-1)}
{(-1)\atop{}}{0\atop{}}}$$ to be our optimal cocharacter. As
$\dim\,\g_e=56$, the total number of positive $0$-roots and
$1$-roots is $(56-5)/2=25$. These roots are given
below.
\begin{table}[htb]
\label{data1}
\begin{tabular}{|c|c|}
\hline$0$-roots  & $1$-roots
\\ \hline
$\scriptstyle{{1\atop{}}{0\atop{}}{0\atop0}{0\atop{}}{0\atop{}}}$,
$\scriptstyle{{0\atop{}}{0\atop{}}{0\atop0}{0\atop{}}{1\atop{}}}$,
$\scriptstyle{{0\atop{}}{1\atop{}}{1\atop1}{0\atop{}}{0\atop{}}}$,
$\scriptstyle{{0\atop{}}{1\atop{}}{1\atop0}{1\atop{}}{0\atop{}}}$,
$\scriptstyle{{0\atop{}}{0\atop{}}{1\atop1}{1\atop{}}{0\atop{}}}$,&
$\scriptstyle{{0\atop{}}{1\atop{}}{1\atop0}{0\atop{}}{0\atop{}}}$,
$\scriptstyle{{0\atop{}}{0\atop{}}{1\atop 0}{1\atop{}}{0\atop{}}}$,
$\scriptstyle{{0\atop{}}{0\atop{}}{1\atop1}{0\atop{}}{0\atop{}}}$,
$\scriptstyle{{1\atop{}}{1\atop{}}{1\atop0}{0\atop{}}{0\atop{}}}$,\\
$\scriptstyle{{1\atop{}}{1\atop{}}{1\atop1}{0\atop{}}{0\atop{}}}$,
$\scriptstyle{{1\atop{}}{1\atop{}}{1\atop0}{1\atop{}}{0\atop{}}}$,
$\scriptstyle{{0\atop{}}{0\atop{}}{1\atop1}{1\atop{}}{1\atop{}}}$,
$\scriptstyle{{0\atop{}}{1\atop{}}{1\atop0}{1\atop{}}{1\atop{}}}$,
$\scriptstyle{{1\atop{}}{1\atop{}}{1\atop0}{1\atop{}}{1\atop{}}}$,&
$\scriptstyle{{0\atop{}}{0\atop{}}{1\atop0} {1\atop{}}{1\atop{}}}$,
$\scriptstyle{{0\atop{}}{1\atop{}}{2\atop1}{1\atop{}}{0\atop{}}}$,
$\scriptstyle{{1\atop{}}{1\atop{}}{2\atop1}{1\atop{}}{0\atop{}}}$,
$\scriptstyle{{0\atop{}}{1\atop{}}{2\atop1}{1\atop{}}{1\atop{}}}$,\\
$\scriptstyle{{1\atop{}}{2\atop{}}{2\atop1}{1\atop{}}{0\atop{}}}$,
$\scriptstyle{{0\atop{}}{1\atop{}}{2\atop1}{2\atop{}}{1\atop{}}}$,
$\scriptstyle{{1\atop{}}{1\atop{}}{2\atop1}{2\atop{}}{1\atop{}}}$,
$\scriptstyle{{1\atop{}}{2\atop{}}{2\atop1}{1\atop{}}{1\atop{}}}$,
$\scriptstyle{{1\atop{}}{2\atop{}}{3\atop2}{2\atop{}}{1\atop{}}}$.&
$\scriptstyle{{1\atop{}}{1\atop{}}{2\atop1}{1\atop{}}{1\atop{}}}$,
$\scriptstyle{{1\atop{}}{2\atop{}}{3\atop1}{2\atop{}}{1\atop{}}}$.\\
\hline
\end{tabular}
\end{table}

The
$(1,2,3,5,6)$-contributions of the $0$-roots and $1$-roots are
$(8,10,14,14,8)$ and $(4,6,8,8,4)$,
respectively, and the total contribution of all roots is
$(6,8,11,11,6)$. In view of [\cite{Bo}, Planche~V]  Condition~(C) for
$\lambda+\rho=\sum_{i=1}^6a_i\varpi_i$
reads
\begin{eqnarray*}
4a_1+3a_2+5a_3+6a_4+4a_5+2a_6&=&18\\
a_1+2a_2+2a_3+3a_4+2a_5+a_6&=&8\\
5a_1+6a_2+10a_3+12a_4+8a_5+4a_6&=&33\\
4a_1+6a_2+8a_3+12a_4+10a_5+5a_6&=&33\\
2a_1+3a_2+4a_3+6a_4+5a_5+4a_6&=&18.
\end{eqnarray*}
A perfect solution to this system of linear equations is given by setting $a_4=0$ and $a_i=1$ for
$i=1,2,3,5,6$.
It follows that $\lambda+\rho=\frac{1}{2}h^\vee$ satisfies Condition (C). As
$\langle\lambda,\alpha_4^\vee\rangle=0$,
Condition~(A) also holds for $\lambda+\rho$. Since $e$ is special and
$\lambda+\rho=\frac{1}{2}h^\vee$, Condition~(B) follows from [\cite{BV},
Proposition~5.10]. Since Condition (D) is  vacuous in the present case and  $\g_e=[\g_e,\g_e]$ by [\cite{dG}], we conclude that
$I(\lambda)=I(-\varpi_4)$ is the only multiplicity-free primitive
ideal in $\mathcal{X}_{\O(e)}$. Not surprisingly, it coincides with the Joseph ideal of $U(\g)$.
\subsection{Type $({\sf E_6, 3A_1})$}\label{4.9}
Here our pinning for $e$ is
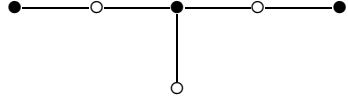
\[\xymatrix{*{\bullet}\ar@{-}[r]& *{\circ} \ar@{-}[r] & *{\bullet} \ar@{-}[r]\ar@{-}[d] & *{\circ} \ar@{-}[r]&*{\bullet}\\ && *{\circ}&&}\]
and we take
$$\tau=\textstyle{{2\atop{}}{(-2)\atop{}}{2\atop (-1)}
{(-2)\atop{}}{2\atop{}}}$$ as an optimal cocharacter.
Since $\dim\,\g_e=38$, the total number of positive $0$-roots and
$1$-roots is $(38-6)/2=16$. These roots are given  below.
\begin{table}[htb]
\label{data1}
\begin{tabular}{|c|c|}
\hline$0$-roots  & $1$-roots
\\ \hline
$\scriptstyle{{1\atop{}}{1\atop{}}{0\atop0}{0\atop{}}{0\atop{}}}$,
$\scriptstyle{{0\atop{}}{1\atop{}}{1\atop0}{0\atop{}}{0\atop{}}}$,
$\scriptstyle{{0\atop{}}{0\atop{}}{1\atop0}{1\atop{}}{0\atop{}}}$,
$\scriptstyle{{0\atop{}}{0\atop{}}{0\atop0}{1\atop{}}{1\atop{}}}$,&
$\scriptstyle{{0\atop{}}{0\atop{}}{1\atop1}{0\atop{}}{0\atop{}}}$,
$\scriptstyle{{1\atop{}}{1\atop{}}{1\atop 1}{0\atop{}}{0\atop{}}}$,
$\scriptstyle{{0\atop{}}{0\atop{}}{1\atop1}{1\atop{}}{1\atop{}}}$,
$\scriptstyle{{1\atop{}}{1\atop{}}{2\atop1}{1\atop{}}{0\atop{}}}$,
$\scriptstyle{{0\atop{}}{1\atop{}}{2\atop1}{1\atop{}}{1\atop{}}}$,\\
$\scriptstyle{{1\atop{}}{1\atop{}}{1\atop0}{1\atop{}}{0\atop{}}}$,
$\scriptstyle{{0\atop{}}{1\atop{}}{1\atop0}{1\atop{}}{1\atop{}}}$,
$\scriptstyle{{1\atop{}}{2\atop{}}{3\atop2}{2\atop{}}{1\atop{}}}$.&
$\scriptstyle{{1\atop{}}{1\atop{}}{1\atop1} {1\atop{}}{1\atop{}}}$,
$\scriptstyle{{1\atop{}}{2\atop{}}{2\atop1}{1\atop{}}{1\atop{}}}$,
$\scriptstyle{{1\atop{}}{1\atop{}}{2\atop1}{2\atop{}}{1\atop{}}}$,
$\scriptstyle{{1\atop{}}{2\atop{}}{3\atop1}{2\atop{}}{1\atop{}}}$.\\
\hline
\end{tabular}
\end{table}
The
$(2,3,5)$-contributions of the $0$-roots and $1$-roots are
$(2,6,6)$ and $(9,9,9)$,
respectively, and the total contribution of all roots is
$(\frac{11}{2},\frac{15}{2},\frac{15}{2})$.

The orbit $\O(e)$ is non-special and $\textstyle{\frac{1}{2}}\dim\O(e)=(78-38)/2=20=
|\Phi^+|-(15+1),$ hence we seek $\lambda+\rho=\sum_{i=1}^6a_i\varpi_i$ for which $\Phi_\lambda$ has type ${\sf A_5+A_1}$. So it is reasonable assume that the $a_i$'s are half-integers.
Setting $\widetilde{a}_i:=2a_i$ for $1\le i\le 6$ we rewrite Condition~(C) for
$\lambda+\rho$ as follows:
\begin{eqnarray*}
\widetilde{a}_1+2\widetilde{a}_2+2\widetilde{a}_3+
3\widetilde{a}_4+2\widetilde{a}_5+\widetilde{a}_6
&=&11\\
5\widetilde{a}_1+6\widetilde{a}_2+10\widetilde{a}_3+
12\widetilde{a}_4+8\widetilde{a}_5+4\widetilde{a}_6
&=&45\\
4\widetilde{a}_1+6\widetilde{a}_2+8\widetilde{a}_3+
12\widetilde{a}_4+10\widetilde{a}_5+5\widetilde{a}_6
&=&45.
\end{eqnarray*}
A perfect solution to this system of linear equations is given by setting $\widetilde{a}_i=1$
for all $1\le i\le 6$, leading to $$\lambda+\rho=\textstyle{\frac{1}{2}}
\textstyle{
\sum}_{i=1}^6\varpi_i=\textstyle{\frac{1}{2}}\rho.$$
It is straightforward to see that $\Phi_\lambda$ contains the positive roots
$$
\xymatrix{{\gamma_1}\ar@{-}[r]&{\gamma_2}\ar@{-}[r]& {\gamma_3} \ar@{-}[r] & {\gamma_4}
\ar@{-}[r]&{\gamma_5}&
{\gamma_6}}
$$
where
\begin{eqnarray*}
&&\gamma_1\,=\,{\scriptstyle{0\atop{}}{1\atop{}}{1\atop0}{0\atop{}}{0\atop{}}},\
\
\gamma_2\,=\,{\scriptstyle{0\atop{}}{0\atop{}}{0\atop0}{1\atop{}}{1\atop{}}},\
\
\gamma_3\,=\,{\scriptstyle{0\atop{}}{0\atop{}}{1\atop1}{0\atop{}}{0\atop{}}},\\
&&\gamma_4\,=\,{\scriptstyle{1\atop{}}{1\atop{}}{0\atop0}{0\atop{}}{0\atop{}}},\
\
\gamma_5\,=\,{\scriptstyle{0\atop{}}{0\atop{}}{1\atop0}{1\atop{}}{0\atop{}}},\
\
\gamma_6\,=\,{\scriptstyle{0\atop{}}{1\atop{}}{1\atop1}{1\atop{}}{0\atop{}}}.
\end{eqnarray*}
Since the subsystems of type ${\sf A_5+A_1}$ are maximal in $\Phi$, the roots $\gamma_1,\ldots,\gamma_6$ form the basis of $\Phi_\lambda$ contained in $\Phi^+$. Since $\langle\lambda+\rho,\gamma_i^\vee\rangle\in\Z^{>0}$ for all $i$, this yields $$\dim\,{\rm
VA}(I(\lambda))=|\Phi|-|\Phi_\lambda|=\dim\,\O(e).$$
Since $\langle\lambda+\rho,\alpha_i^\vee\rangle\not\in\Z^{>0}$ for $i\in\{1,4,6\}$,
Condition~(A) holds for
$\lambda+\rho$, forcing
${\rm VA}(I(\lambda))=\overline{\O(e)}$. Since Condition~(D) is vacuous and $\g_e=[\g_e,\g_e]$ by
[\cite{dG}], we deduce that $I(-\frac{1}{2}\rho)$ is the only
multiplicity-free primitive ideal in $\mathcal{X}_{\O(e)}$.
\subsection{Type $({\sf E_6, 2A_2+A_1})$}\label{4.10}
Here the pinning for $e$ is unique:
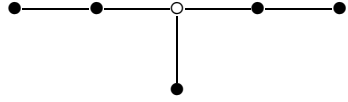
\[\xymatrix{*{\bullet}\ar@{-}[r]& *{\bullet} \ar@{-}[r] & *{\circ} \ar@{-}[r]\ar@{-}[d] & *{\bullet} \ar@{-}[r]&*{\bullet}\\ && *{\bullet}&&}\]
and we choose
$$\tau=\textstyle{{2\atop{}}{2\atop{}}{(-5)\atop 2}
{2\atop{}}{2\atop{}}}$$ as our optimal cocharacter.
As $\dim\,\g_e=24$, the total number of positive $0$-roots and
$1$-roots is $(24-6)/2=9$. These roots are given below.
\begin{table}[htb]
\label{data1}
\begin{tabular}{|c|c|}
\hline$0$-roots  & $1$-roots
\\ \hline
$\scriptstyle{{1\atop{}}{1\atop{}}{2\atop1}{1\atop{}}{1\atop{}}}$,
$\scriptstyle{{0\atop{}}{1\atop{}}{2\atop1}{2\atop{}}{1\atop{}}}$,
$\scriptstyle{{1\atop{}}{2\atop{}}{2\atop1}{1\atop{}}{0\atop{}}}$.
&
$\scriptstyle{{1\atop{}}{1\atop{}}{1\atop1}{0\atop{}}{0\atop{}}}$,
$\scriptstyle{{0\atop{}}{1\atop{}}{1\atop 1}{1\atop{}}{0\atop{}}}$,
$\scriptstyle{{0\atop{}}{0\atop{}}{1\atop1}{1\atop{}}{1\atop{}}}$,
$\scriptstyle{{1\atop{}}{1\atop{}}{1\atop0}{1\atop{}}{0\atop{}}}$,
$\scriptstyle{{0\atop{}}{1\atop{}}{1\atop0}{1\atop{}}{1\atop{}}}$,
$\scriptstyle{{1\atop{}}{2\atop{}}{3\atop2}{2\atop{}}{1\atop{}}}$.
\\
\hline
\end{tabular}
\end{table}
The
$4$-contributions of the $0$-roots and $1$-roots are
$6$ and $8$,
respectively, and the total contribution of all roots is $7$.

Our orbit is non-special and $$\textstyle{\frac{1}{2}}\dim\O(e)=(78-24)/2=27=
|\Phi^+|-(3+3+3).$$ Therefore, it is reasonable to seek $\lambda+\rho=\sum_{i=1}^6a_i\varpi_i$ for which $\Phi_\lambda$ has type ${\sf A_2+A_2+A_2}$.
In the present case Condition~(C) boils down to one linear equation
$$
2a_1+3a_2+4a_3+
6a_4+4a_5+2a_6
=7.
$$
Setting $a_i=\frac{1}{3}$
for all $1\le i\le 6$ provides us with a perfect solution which leads to $\lambda+\rho=\textstyle{\frac{1}{3}}\rho.$
It is straightforward to see that $\Phi_\lambda$ contains the positive roots
$$
\xymatrix{{\gamma_1}\ar@{-}[r]&{\gamma_2}
&
{\gamma_3}\ar@{-}[r]&{\gamma_4}
&
{\gamma_5} \ar@{-}[r] & {\gamma_6}}
$$
where
\begin{eqnarray*}
&&\gamma_1\,=\,{\scriptstyle{1\atop{}}{1\atop{}}{1\atop0}{0\atop{}}{0\atop{}}},\
\
\gamma_2\,=\,{\scriptstyle{0\atop{}}{0\atop{}}{1\atop1}{1\atop{}}{0\atop{}}},\
\
\gamma_3\,=\,{\scriptstyle{0\atop{}}{1\atop{}}{1\atop1}{0\atop{}}{0\atop{}}},\\
&&\gamma_4\,=\,{\scriptstyle{0\atop{}}{0\atop{}}{1\atop0}{1\atop{}}{1\atop{}}},\
\
\gamma_5\,=\,{\scriptstyle{0\atop{}}{1\atop{}}{1\atop0}{1\atop{}}{0\atop{}}},\
\
\gamma_6\,=\,{\scriptstyle{1\atop{}}{1\atop{}}{1\atop1}{1\atop{}}{1\atop{}}}.
\end{eqnarray*}
Since the subsystems of type ${\sf A_2+A_2+A_2}$ are maximal in $\Phi$, the roots $\gamma_1,\ldots,\gamma_6$ form the basis of $\Phi_\lambda$ contained in $\Phi^+$. Since $\langle\lambda+\rho,\gamma_i^\vee\rangle\in\Z^{>0}$ for all $i$, this yields $\dim\,{\rm
VA}(I(\lambda))=|\Phi|-|\Phi_\lambda|=\dim\,\O(e).$
Since $\langle\lambda+\rho,\beta^\vee\rangle\not\in\Z^{>0}$ for any root $\beta\in\Phi^+$ which can be expressed as a linear combination of $\alpha_i$ with $i\in\{1,2,3,5,6\}$,
Condition~(A) holds for
$\lambda+\rho$.  Since Condition~(D) is again vacuous and $\g_e=[\g_e,\g_e]$ by
[\cite{dG}], we argue as before to conclude that $I(-\frac{2}{3}\rho)$ is the only
multiplicity-free primitive ideal in $\mathcal{X}_{\O(e)}$.
\section{\bf The case of
rigid nilpotent orbits in Lie algebras of type ${\sf F_4}$ and ${\sf G_2}$.}
In this section we assume that $\Phi$ has type ${\sf F_4}$ or ${\sf G_2}$.
Since in both cases there are roots of different lengths,
we distinguish between $\Phi$ with $\Phi^\vee$ and regard $\Phi_\lambda$ as a root
subsystem of $\Phi^\vee$.
\pagebreak
\subsection{Type $({\sf F_4, A_1})$}\label{5.1}
Here our pinning for $e$ is
\[\xymatrix{*{\bullet}\ar@{-}[r]& *{\circ} \ar@{=}[r]|*\dir{>} & *{\circ} \ar@{-}[r]&*{\circ}}\]
and we choose $\tau=(2,-1,0,0)$ as an optimal cocharacter. Since $\dim\g_e=36$, the total
number of $0$-roots and $1$-roots is $(36-4)/2=16$. These roots are given below.
\begin{table}[htb]
\label{data1}
\begin{tabular}{|c|c|}
\hline$0$-roots  & $1$-roots
\\ \hline
$\scriptstyle{0010}$,
$\scriptstyle{0001}$,
$\scriptstyle{0011}$,
$\scriptstyle{1220}$,
$\scriptstyle{1221}$,&
$\scriptstyle{1100}$,
$\scriptstyle{1110}$,
$\scriptstyle{1111}$,
$\scriptstyle{1121}$,\\
$\scriptstyle{1222}$,
$\scriptstyle{1231}$,
$\scriptstyle{1232}$,
$\scriptstyle{1242}$.&
$\scriptstyle{1120}$,
$\scriptstyle{1122}$,
$\scriptstyle{2342}$.
\\
\hline
\end{tabular}
\end{table}
The $(2,3,4)$-contributions of $0$-roots and $1$-roots are $(12,18,10)$ and $(9,12,6)$, respectively, and the total
contribution of all roots is $(\frac{21}{2},15,8)$. The orbit $\O(e)$ is non-special and
$\frac{1}{2}\dim\O(e)=(52-36)/2=8=|\Phi^+|-16.$ So it is reasonable to assume that $\Phi_\lambda\subset\Phi^\vee$ has type
${\sf B_4}$ or ${\sf C_4}$.
Since the centre of the standard Levi subalgebra $\l$ of $\g$ with $[\l,\l]$ generated by $e_{\pm \alpha_1}$
is spanned by $3\alpha_1^\vee+6\alpha_2^\vee+4\alpha_3^\vee+3\alpha_4^\vee$,
$4\alpha_1^\vee+8\alpha_2^\vee+6\alpha_3^\vee+3\alpha_4^\vee$ and
$2\alpha_1^\vee+4\alpha_2^\vee+3\alpha_3^\vee+2\alpha_4^\vee$, Condition~(C) for $\lambda+\rho=\sum_{i=1}^4a_i\varpi_i$
reads
\begin{eqnarray*}
3\widetilde{a}_1+6\widetilde{a}_2+4\widetilde{a}_3+3\widetilde{a}_4&=&21\\
4\widetilde{a}_1+8\widetilde{a}_2+6\widetilde{a}_3+3\widetilde{a}_4&=&30\\
2\widetilde{a}_1+4\widetilde{a}_2+3\widetilde{a}_3+2\widetilde{a}_4&=&16
\end{eqnarray*}
where $\widetilde{a}_i=2a_i$. One obvious solution to this system of linear equations is
$(\widetilde{a}_1,\widetilde{a}_2,\widetilde{a}_3,\widetilde{a}_4)=(1,1,2,2)$, which leads to
$$\lambda+\rho=\textstyle{\frac{1}{2}}\varpi_1+\textstyle{\frac{1}{2}}\varpi_2+\varpi_3+\varpi_4.$$
Using [\cite{Bo}, Planche~VIII] it is easy to check that
$\lambda+\rho=4\varepsilon_1+\frac{3}{2}\varepsilon_2+\varepsilon_3+\frac{1}{2}\varepsilon_4$.
Then $\Phi_\lambda$ contains
\[\xymatrix{ {\beta_1^\vee} \ar@{=}[r]| *\dir{<} & {\beta_2^\vee}  \ar@{-}[r]&{\beta_3^\vee} \ar@{-}[r]& {\beta_4^\vee}}\]
where $\beta_1=\varepsilon_2-\varepsilon_4$, $\beta_2=\varepsilon_4$, $\beta_3=\frac{1}{2}(\varepsilon_1-
\varepsilon_2-\varepsilon_3-\varepsilon_4)$ and $\beta_4=\varepsilon_3$. Hence $\Phi_\lambda$ contains a root
subsystem of type ${\sf B_4}$. From the maximality in $\Phi^\vee$ of root subsystems of type ${\sf B_4}$ it follows that
$\beta_1^\vee,\beta_2^\vee,\beta_3^\vee,\beta_4^\vee$ form the basis of $\Phi_\lambda$ contained in $(\Phi^+)^\vee$.
Since
$\langle\lambda+\rho,\alpha_1^\vee\rangle\not\in\Z^{>0}$, $\dim\O(e)=|\Phi|-|\Phi_\lambda|$, and
$\langle\lambda+\rho,\beta_i^\vee\rangle\in\Z^{>0}$ for $1\le i\le 4$, we see that $\lambda+\rho$ satisfies Condition~(A)
and ${\rm VA}(I\lambda))=\overline{\O(e)}$. Since $e$ has standard Levi type, $\lambda+\rho$ satisfies all Losev's conditions.
As $\g_e=[\g_e,\g_e]$ we thus deduce that $I(-\frac{1}{2}(\varpi_1+\varpi_2))$ is the only multiplicity-free
primitive ideal of $\mathcal{X}_{\O(e)}$. We thus recover the well-known Joseph ideal of $U(\g)$.
\subsection{Type $({\sf F_4,
\widetilde{A}_1})$}\label{5.2}
We choose the following pinning for $e$:
\[\xymatrix{*{\circ}\ar@{-}[r]& *{\circ} \ar@{=}[r]|*\dir{>} & *{\bullet} \ar@{-}[r]&*{\circ}}\]
and take $\tau=(0,-2,2,-1)$ as our optimal cocharacter. Since $\dim\g_e=30$, the total number of $0$-roots and $1$-roots is $(30-4)/2=13$. The roots are given below.
\begin{table}[htb]
\label{data1}
\begin{tabular}{|c|c|}
\hline$0$-roots  & $1$-roots
\\ \hline
$\scriptstyle{1000}$,
$\scriptstyle{1110}$,
$\scriptstyle{0110}$,
$\scriptstyle{1220}$,
$\scriptstyle{0122}$,&
$\scriptstyle{0011}$,
$\scriptstyle{0121}$,
$\scriptstyle{1121}$,
$\scriptstyle{1231}$.\\
$\scriptstyle{1122}$,
$\scriptstyle{1232}$,
$\scriptstyle{1342}$,
$\scriptstyle{2342}$.&
\\
\hline
\end{tabular}
\end{table}

The $(1,2,4)$-contributions of $0$-roots and $1$-roots are $(8,14,10)$ and $(2,4,4)$, respectively, and the total
contribution of all roots is $(5,9,7)$.

The orbit $\O(e)$ is special and $e^\vee\in\g^\vee$ has type ${\sf F_4(a_1)}$. Since the Dynkin diagram of $\Phi^\vee$ has the form
$$\xymatrix{*{\circ}\ar@{-}[r]& *{\circ} \ar@{=}[r]|*\dir{<} & *{\circ} \ar@{-}[r]&*{\circ}}$$
we may assume that $\frac{1}{2}h^\vee=1011$.
The centre of the standard Levi subalgebra $\l$ of $\g$ with $[\l,\l]$ generated by $e_{\pm \alpha_3}$
is spanned by $2\alpha_1^\vee+3\alpha_2^\vee+2\alpha_3^\vee+\alpha_4^\vee$,
$3\alpha_1^\vee+6\alpha_2^\vee+4\alpha_3^\vee+2\alpha_4^\vee$ and
$2\alpha_1^\vee+4\alpha_2^\vee+3\alpha_3^\vee+
2\alpha_4^\vee$. Therefore, Condition~(C) for $\lambda+\rho=\sum_{i=1}^4a_i\varpi_i$
reads
\begin{eqnarray*}
2a_1+3a_2+2a_3+a_4&=&5\\
3a_1+6a_2+4a_3+2a_4&=&9\\
2a_1+4a_2+3a_3+2a_4&=&7
\end{eqnarray*}
where $a_i\in\Z$. Clearly,
$(a_1,a_2,a_3,a_4)=(1,1,-1,2)$ is a nice solution to this system of linear equations, and it leads to
$$\lambda+\rho=\varpi_1+\varpi_2-2\varpi_3+\varpi_4=
s_3s_2(\textstyle{\frac{1}{2}}h^\vee).$$
Set $\mu:=\frac{1}{2}h^\vee-\rho$. Then $\mu\in P^+(\Phi)$ and $\Pi_\mu^0=\{\alpha_2\}=\tau(s_3s_2)$, so that
$$d\big(U(\g)/I(\lambda)\big)=
d\big(U(\g)/I(s_3s_2
\centerdot\mu)\big)=\,
d\big(U(\g)/I(s_3s_2
\centerdot\nu)\big)=\,d\big(U(\g)/I(s_2s_3
\centerdot\nu)\big)$$
for any regular $\nu\in P^+(\Phi)$; see [\cite{Ja}, Corollar~10.10]. Since $\alpha_2,\alpha_3\in\Pi$ span a root subsystem of type ${\sf B_2}$ in $\Phi$  we now apply
[\cite{Ja}, 10A.2] to deduce that
$$
d\big(U(\g)/I(s_2s_3
\centerdot\nu)\big)=\,
d\big(U(\g)/I(s_2
\centerdot\nu)\big)
=\,d\big(U(\g)/I(s_2\centerdot\mu)\big)=\,
d\big(U(\g)/I(\mu)\big).
$$
In view of [\cite{BV}, Proposition~5.10], this yields $\dim {\rm VA}(I(\lambda))=\dim{\rm VA}(I(\frac{1}{2}h^\vee-\rho))=\dim\O(e)$.
But then $I(\lambda) = I(\frac{1}{2}h^\vee)$ by the maximality of $I(\frac{1}{2}h^\vee-\rho)$. In particular, ${\rm VA}(I(\lambda))=\overline{\O(e)}$ implying that Condition~(B) holds for $\lambda+\rho$. Note that
Condition~(A) holds for $\lambda+\rho$ by our choice of pinning and Condition~(D) is vacuous in the present case. Since $\g_e=[\g_e,\g_e]$ by [\cite{dG}], we conclude that $I(\lambda)=I(\frac{1}{2}h^\vee-\rho)=I(-\varpi_2)$ is
the only multiplicity-free primitive ideal in $\mathcal{X}_{\O(e)}$.
\subsection{Type $({\sf F_4,
A_1+\widetilde{A}_1})$}\label{5.3} This orbit is special and $e^\vee\in\g^\vee$ has type ${\sf F_4(a_2)}$, so that $\frac{1}{2}h^\vee=1010$. Keeping this in mind, we choose our pinning for $e$ as follows:
\[\xymatrix{*{\circ}\ar@{-}[r]& *{\bullet} \ar@{=}[r]|*\dir{>} & *{\circ} \ar@{-}[r]&*{\bullet}}\]
and take $\tau=(-1,2,-2,2)$ as an optimal cocharacter
Since $\dim\g_e=24$, the total number of $0$-roots and $1$-roots is $(24-4)/2=10$ and roots are displayed below.
\begin{table}[htb]
\label{data1}
\begin{tabular}{|c|c|}
\hline$0$-roots  & $1$-roots
\\ \hline
$\scriptstyle{0110}$,
$\scriptstyle{0011}$,
$\scriptstyle{0121}$,
$\scriptstyle{2342}$.&
$\scriptstyle{1100}$,
$\scriptstyle{1111}$,
$\scriptstyle{1122}$,
$\scriptstyle{1221}$,
$\scriptstyle{1232}$,
$\scriptstyle{1342}$.
\\
\hline
\end{tabular}
\end{table}

The $(1,3)$-contributions of $0$-roots and $1$-roots are $(2,8)$ and $(6,12)$, respectively, and the total
contribution of all roots is $(4,10)$.
The centre of the standard Levi subalgebra $\l$ of $\g$ with $[\l,\l]$ generated by $e_{\pm \alpha_2}$
and $e_{\pm\alpha_4}$
is spanned by $2\alpha_1^\vee+3\alpha_2^\vee+2\alpha_3^\vee+
\alpha_4^\vee$ and
$4\alpha_1^\vee+8\alpha_2^\vee+6\alpha_3^\vee+
3\alpha_4^\vee$. So Condition~(C) for $\lambda+\rho=\sum_{i=1}^4a_i\varpi_i$
reads
\begin{eqnarray*}
2a_1+3a_2+2a_3+a_4&=&4\\
4a_1+8a_2+6a_3+3a_4&=&10
\end{eqnarray*}
where $a_i\in\Z$. One obvious solution to this system of linear equations is given by $(a_1,a_2,a_3,a_4)=(1,0,1,0)$, which leads to
$\lambda+\rho=\frac{1}{2}h^\vee$. This weight satisfies Condition~(A) by our choice of pinning, whilst Condition~(D) is again vacuous.
Combining [\cite{BV}, Proposition~5.10] with
[\cite{dG}] we conclude that $I(\lambda)=I(-\varpi_2-\varpi_4)$ is the only multiplicity-free primitive ideal in $\mathcal{X}_{\O(e)}$.
\subsection{Type $({\sf F_4,
A_2+\widetilde{A}_1})$}\label{5.4}
Here our pinning for $e$ is
\[\xymatrix{*{\bullet}\ar@{-}[r]& *{\bullet} \ar@{=}[r]|*\dir{>} & *{\circ} \ar@{-}[r]&*{\bullet}}\]
and our optimal cocharacter is $\tau=(2,2,-3,2)$. Since $\dim\g_e=18$, the total
number of $0$-roots and $1$-roots is $(18-4)/2=7$. These roots are listed below.
\begin{table}[htb]
\label{data1}
\begin{tabular}{|c|c|}
\hline$0$-roots  & $1$-roots
\\ \hline
$\scriptstyle{1121}$,
$\scriptstyle{0122}$,
$\scriptstyle{1220}$,
$\scriptstyle{1342}$.
&
$\scriptstyle{1110}$,
$\scriptstyle{0111}$,
$\scriptstyle{1232}$.
\\
\hline
\end{tabular}
\end{table}
The $3$-contributions of $0$-roots and $1$-roots are $10$ and $5$, respectively, and the total
contribution of al roots is $\frac{15}{2}$. The orbit $\O(e)$ is non-special and
$\frac{1}{2}\dim\O(e)=(52-18)/2=17=|\Phi^+|-7.$ So it is reasonable to assume that $\Phi_\lambda\subset\Phi^\vee$ has type
${\sf A_3+\widetilde{A}_1}$ (a priori it may happen
that $\Phi_\lambda$ has type ${\sf D_4}$, but a closer look reveals that this cannot occur).
Since the centre of the standard Levi subalgebra $\l$ of $\g$ with $[\l,\l]$ of type ${\sf A_2+\widetilde{A}_1}$
is spanned by $4\alpha_1^\vee+8\alpha_2^\vee+6\alpha_3^\vee+
3\alpha_4^\vee$ , Condition~(C) for $\lambda+\rho=\sum_{i=1}^4a_i\varpi_i$
reads
\begin{eqnarray*}
4\widetilde{a}_1+8\widetilde{a}_2+6\widetilde{a}_3+
3\widetilde{a}_4\,=\,30
\end{eqnarray*}
where $\widetilde{a}_i=4a_i$ for all $1\le i\le 4$. One obvious solution to this linear equation is
$(\widetilde{a}_1,\widetilde{a}_2,\widetilde{a}_3,\widetilde{a}_4)=(1,1,2,2)$, which leads to
$$\lambda+\rho=\textstyle{\frac{1}{4}}\varpi_1+\textstyle{\frac{1}{4}}\varpi_2+\textstyle{\frac{1}{2}}\varpi_3+
\textstyle{\frac{1}{2}}\varpi_4.$$
Our discussion in Subsection~\ref{5.1} implies that
$\lambda+\rho=2\varepsilon_1+\frac{3}{4}\varepsilon_2+\frac{1}{2}\varepsilon_3+\frac{1}{4}\varepsilon_4$.
Then $\Phi_\lambda$ contains
\[\xymatrix{{\gamma_1^\vee} \ar@{-}[r] & {\gamma_2^\vee}  \ar@{-}[r]&{\gamma_3^\vee} & {\gamma_4^\vee}}\]
where $\gamma_1=\frac{1}{2}(\varepsilon_1-
\varepsilon_2-\varepsilon_3+\varepsilon_4)$, $\gamma_2 = \varepsilon_3$, $\gamma_3=\frac{1}{2}(\varepsilon_1+
\varepsilon_2-\varepsilon_3-\varepsilon_4)$
and $\gamma_4=\varepsilon_2+\varepsilon_4$.
Furthermore,
using [\cite{Bo}, Planche~VIII], it is straightforward to check that
$$\Phi_\lambda^\vee\cap\Phi^+\,=\,
\{\varepsilon_1,\,\varepsilon_3,\,\varepsilon_2+\varepsilon_4,\,\textstyle{\frac{1}{2}}
(\varepsilon_1+\varepsilon_2\pm\varepsilon_3-
\varepsilon_4),\,\textstyle{\frac{1}{2}}
(\varepsilon_1-\varepsilon_2\pm\varepsilon_3+
\varepsilon_4)\}.
$$
From this description it is immediate that
that
$\gamma_1^\vee,\gamma_2^\vee,\gamma_3^\vee,\gamma_4^\vee$ form the basis of $\Phi_\lambda$ contained in $(\Phi^+)^\vee$.
Condition~(A) holds for $\lambda+\rho$ because
$\langle\lambda+\rho,\beta^\vee\rangle\not\in\Z^{>0}$ for all positive roots contained in the root subsystem of $\Phi$ spanned by $\alpha_1,\alpha_2,\alpha_4$. Since $\dim\O(e)=|\Phi|-|\Phi_\lambda|$ and
$\langle\lambda+\rho,\gamma_i^\vee\rangle\in\Z^{>0}$ for $1\le i\le 4$, it follows that
${\rm VA}(I\lambda))=\overline{\O(e)}$. Since $e$ has standard Levi type,  $\lambda+\rho$ satisfies all Losev's conditions.
As $\g_e=[\g_e,\g_e]$ we thus deduce that $I(-\frac{3}{4}\varpi_1-\frac{3}{4}\varpi_2-\frac{1}{2}\varpi_3-\frac{1}{2}\varpi_4)$ is the only multiplicity-free
primitive ideal of $\mathcal{X}_{\O(e)}$.
\subsection{Type $({\sf F_4,
\widetilde{A}_2+A_1})$}\label{5.5}
In this case $\g_e=\mathbb{C}e\oplus[\g_e,\g_e]$ and according to [\cite{GRU}] the algebra $U(\g,e)$ affords $2$ one-dimensional representations. Our goal
in this subsection is to find  Duflo realisations of the corresponding primitive ideals of $U(\g)$. The pinning for $e$ is unique:
\[\xymatrix{*{\bullet}\ar@{-}[r]& *{\circ} \ar@{=}[r]|*\dir{>} & *{\bullet} \ar@{-}[r]&*{\bullet}}\]
and we chose $\tau=(2,-5,2,2)$ as out optimal cocharacter. Since $\dim\g_e=16$, the total
number of $0$-roots and $1$-roots is $(16-4)/2=6$. These roots are listed below.
\begin{table}[htb]
\label{data1}
\begin{tabular}{|c|c|}
\hline$0$-roots  & $1$-roots
\\ \hline
$\scriptstyle{1222}$,
$\scriptstyle{1231}$.&
$\scriptstyle{1120}$,
$\scriptstyle{0121}$,
$\scriptstyle{1111}$,
$\scriptstyle{2342}$.
\\
\hline
\end{tabular}
\end{table}
The $2$-contributions of $0$-roots and $1$-roots are $4$ and $6$, respectively, and the total
contribution of al roots is $5$. The orbit $\O(e)$ is non-special and
$\frac{1}{2}\dim\O(e)=(52-16)/2=18=|\Phi^+|-6.$ So it is reasonable to assume that $\Phi_\lambda$ has type
${\sf A_2+\widetilde{A}_2}$.
Since the centre of the standard Levi subalgebra $\l$ of $\g$ with $[\l,\l]$ of type ${\sf A_1+\widetilde{A}_2}$
is spanned by $3\alpha_1^\vee+6\alpha_2^\vee+4\alpha_3^\vee+
2\alpha_4^\vee$ , Condition~(C) for $\lambda+\rho=\sum_{i=1}^4a_i\varpi_i$
reads
\begin{eqnarray*}
3\widetilde{a}_1+6\widetilde{a}_2+4\widetilde{a}_3+
2\widetilde{a}_4\,=\,15
\end{eqnarray*}
where $\widetilde{a}_i=3a_i$ for all $1\le i\le 4$. One obvious solution to this linear equation is
$(\widetilde{a}_1,\widetilde{a}_2,\widetilde{a}_3,\widetilde{a}_4)=(1,1,1,1)$, which leads to
$$\lambda+\rho=\textstyle{\frac{1}{3}}\varpi_1+\textstyle{\frac{1}{3}}\varpi_2+\textstyle{\frac{1}{3}}\varpi_3+
\textstyle{\frac{1}{3}}\varpi_4=\textstyle{\frac{1}{3}}\rho.$$
Using [\cite{Bo}, Planche~VIII] one observes that $\lambda+\rho=\frac{11}{6}\varepsilon_1+\frac{5}{6}\varepsilon_2+\frac{1}{2}\varepsilon_3+\frac{1}{6}\varepsilon_4$.
Then $\Phi_\lambda$ contains
\[\xymatrix{{\gamma_1^\vee} \ar@{-}[r] & {\gamma_2^\vee}  & {\gamma_3^\vee}  \ar@{-}[r]&{\gamma_4^\vee}}\]
where $\gamma_1=\varepsilon_1-\varepsilon_2$, $\gamma_2=\varepsilon_2+\varepsilon_4$,
$\gamma_3=\varepsilon_3$ and $\gamma_4=\frac{1}{2}(\varepsilon_1+
\varepsilon_2-\varepsilon_3-\varepsilon_4)$. Furthermore, direct verification based on 
the above expression for $\lambda+\rho$ shows that 
$$\Phi_\lambda\cap(\Phi^+)^\vee\,=\,\{\gamma_1^\vee,\gamma_2^\vee,\gamma_3^\vee,\gamma_4^\vee\}.$$  
Since
$\langle\lambda+\rho,\beta^\vee\rangle\not\in\Z^{>0}$ for all positive roots contained in the root subsystem of $\Phi$ spanned by $\alpha_1,\alpha_3,\alpha_4$, Condition~(A) holds for $\lambda+\rho$. Since $\dim\O(e)=|\Phi|-|\Phi_\lambda|$ and
$\langle\lambda+\rho,\gamma_i^\vee\rangle\in\Z^{>0}$ for $1\le i\le 4$, we deduce that
${\rm VA}(I\lambda))=\overline{\O(e)}$. Since $e$ has standard Levi type, it follows that $\lambda+\rho$ satisfies all Losev's conditions and hence
$I(\lambda)=I(-\frac{2}{3}\rho)$ is a multiplicity-free
primitive ideal of $\mathcal{X}_{\O(e)}$.

In order to find  a highest weight $\lambda'+\rho=\sum_{i=1}^4b_i\varpi_i$, leading to the second multiplicity-free primitive ideal in $\mathcal{X}_{\O(e)}$ we again impose that $\Phi_{\lambda'}=\Phi_\lambda$. Using [\cite{Bo}, Planche~VIII] we compute:
\begin{eqnarray*}
\lambda'+\rho&=&
b_1(\varepsilon_1+\varepsilon_2)
+b_2(2\varepsilon_1+\varepsilon_2+\varepsilon_3)+
b_3(\textstyle{\frac{3}{2}}\varepsilon_1+\textstyle{\frac{1}{2}}\varepsilon_2+\textstyle{\frac{1}{2}}\varepsilon_3+
\textstyle{\frac{1}{2}}\varepsilon_4)+b_4\varepsilon_1\\
&=&(b_1+2b_2+\textstyle{\frac{3}{2}}b_3+b_4)\varepsilon_1+(b_1+b_2+\textstyle{\frac{1}{2}}b_3)\varepsilon_2+(b_2+\textstyle{\frac{1}{2}}b_3)\varepsilon_3+\textstyle{\frac{1}{2}}b_3\varepsilon_4.
\end{eqnarray*}
Since
$(\lambda'+\rho\,\vert\,\varepsilon_1-\varepsilon_2)
\in\Z^{>0}$, $(\lambda'+\rho\,\vert\,\varepsilon_2+\varepsilon_4)
\in\Z^{>0}$, $2(\lambda'+\rho\,\vert\,\varepsilon_3)\in\Z^{>0}$ and
$2(\lambda'+\rho\,\vert\,
\varepsilon_1+\varepsilon_2-\varepsilon_3-
\varepsilon_4)\in\Z^{>0}$, it must be that
$b_2+b_3+b_4\in\Z^{>0}$, $b_1+b_2+b_3\in\Z^{>0}$,
$2b_2+b_3\in\Z^{>0}$ and $2b_1+2b_2+b_3+b_4\in\frac{1}{2}\Z^{>0}$.
Recall that $b_i\in\frac{1}{3}\Z$ for all $1\le i\le 4$ and
$3b_1+6b_2+4b+2b_4=5$ due to Condition~(C).
We seek a small positive value of $(\lambda'+\rho\,\vert\,\varepsilon_1-\varepsilon_2)$, but we cannot take $1$
as we also require that $\lambda'+\rho\not\in W\cdot(\lambda+\rho)$. So we impose that $(\lambda'+\rho\,\vert\,\varepsilon_1-\varepsilon_2)=2$.

Setting $(b_1,b_2,b_3,b_4)=(-\frac{1}{3},-\frac{1}{3},\frac{5}{3},\frac{2}{3})$ we obtain a solution which satisfies all the requirements mentioned above and leads to $$\lambda+\rho\,=-\textstyle{\frac{1}{3}}\varpi_1-
\textstyle{\frac{1}{3}}\varpi_2+\textstyle{\frac{5}{3}}\varpi_3+\textstyle{\frac{2}{3}}\varpi_4\,=\,
\textstyle{\frac{13}{6}}\varepsilon_1+\textstyle{\frac{1}{6}}\varepsilon_2
+\textstyle{\frac{1}{2}}\varepsilon_2
+\textstyle{\frac{1}{2}}\varepsilon_3
+\textstyle{\frac{5}{6}}\varepsilon_4.
$$
It is straightforward to check that
$\langle\lambda'+\rho,\gamma_i^\vee\rangle\in\Z^{>0}$ for $1\le i\le 4$ and
$\langle\lambda'+\rho,\beta^\vee\rangle\not\in\Z^{>0}$ for all $\beta\in\Phi^+$ that can be expressed as a linear combination of  $\alpha_1,\alpha_3,\alpha_4$.
As $\lambda'+\rho$ and $-\lambda-\rho$ have the same image in $P(\Phi)/\frac{1}{3}P(\Phi)$, we also have that $\Phi_\lambda=\Phi_{\lambda'}$.
Arguing as before we now deduce that $\lambda'+\rho$ satisfies all Losev's conditions. As a consequence,
$I(-\frac{4}{3}\varpi_1-\frac{4}{3}\varpi_2+\frac{2}{3}\varpi_3-\frac{1}{3}\varpi_4)$
is another multiplicity-free
primitive ideal of $\mathcal{X}_{\O(e)}$. Using the  above expressions for $\lambda+\rho$ and $\lambda'+\rho$ in the $\varepsilon$-basis, it is easy to check that $(\lambda+\rho\,|\,\lambda+\rho)\ne
(\lambda'+\rho\,|\,\lambda'+\rho)$. Therefore, $I(\lambda)\ne I(\lambda')$.
\subsection{Type $({\sf G_2,
A_1})$}\label{5.6}
In the next two subsections our root system $\Phi$ has type ${\sf G_2}$ and $\Pi=\{\alpha_1,\alpha_2\}$ where $\alpha_1$
is a short root. In the present case, we can chose $e$ and $\tau$ such that $e=e_{\alpha_2}$ and $\tau=(-1,2)$. The orbit $\O(e)$ is non-special and $\dim\g_e=8$. Since $\dim\g-\dim\g_2=|\Phi|-6$, it is reasonable to assume that $\Phi_\lambda$ has type ${\sf A_2}$ and hence consists of all long roots of $\Phi^\vee$.
The number of $0$-roots and $1$-roots is $(8-2)/2=3$ and the roots are $2\alpha_1+\alpha_2$, $\alpha_1+\alpha_2$ and $3\alpha_1+2\alpha_2$. The total $\alpha_1$-contribution of all roots is
$(2+1+3)/2=3$.

In view of the above we may assume that $\lambda+\rho=a_1\varpi_1+a_2\varpi_2$ where $a_1,a_2\in\frac{1}{3}\Z$. Since the centre of the standard Levi subalgebra $\l$ with $[\l,\l]$ generated by $e_{\pm \alpha_2}$ is spanned by $2\alpha_1^\vee+3\alpha_2^\vee$,  Condition~(C) for $\lambda+\rho$ reads
$$2a_1+3a_2=3.$$ Setting $(a_1,a_2)=(1,\frac{1}{3})$ we obtain a nice solution to this linear equation leading to
$$\lambda+\rho=\varpi_1+\textstyle{\frac{1}{3}}\varpi_2.$$
Using [\cite{Bo}, Planche~IX] it is straightforward to check that $\langle\lambda+\rho,\gamma^\vee\rangle\in\Z$ if and only if $\gamma$ is a short root. Then $\Phi_\lambda$ consists of all long roots of $\Phi^\vee$ as desired. Since $\langle
\lambda+\rho,\gamma^\vee\rangle\in\Z^{>0}$ for all
positive short roots $\gamma\in\Phi$ and
$\langle\lambda+\rho,\alpha_2^\vee\rangle\not\in\Z^+$, we argue as in the previous cases to conclude that
$I(\lambda)=I(-\frac{2}{3}\varpi_2)$ is a multiplicity-free primitive ideal of $\mathcal{X}_{\O(e)}$. Such an ideal is unique
because $\g_e=[\g_e,\g_e]$. So  it should not come as a surprise that $I(\lambda)$ coincides with the Joseph ideal of $U(\g)$.
\subsection{Type $({\sf G_2,
\widetilde{A}_1})$}\label{5.7} We have finally come to our last case where $e=e_{\alpha_1}$ (recall that $\alpha_1$ is the simple short root in $\Pi$). Then $\g_e=\mathbb{C}e\oplus[\g_e,\g_e]$ and hence we may expect  $U(\g,e)$ to afford at least $2$ one-dimensional representations. In fact, [\cite{GRU}] shows that there exactly two of them.
The orbit $\O(e)$ is non-special and $\dim\g_e=6$. Since $\dim\O(e)= 14-6=8=|\Phi|-4$, we should seek $\lambda+\rho=a_1\varpi_1+a_2\varpi_2$ such that $\Phi_\lambda$ has type ${\sf A_1+\widetilde{A}_1}$.
For that to hold, $a_1$ and $a_2$ must be half-integers. We may also assume that $\tau=(2,-3)$.
The number of $0$-roots and $1$-roots
is $(6-2)/2=2$, and the roots are $3\alpha_1+2\alpha_2$ and $2\alpha_1+\alpha_2$, respectively. The total $\alpha_2$-contribution of
these roots is $(2+1)/2=\frac{3}{2}$. Since
 the centre of the standard Levi subalgebra $\l$ with $[\l,\l]$ generated by $e_{\pm \alpha_1}$ is spanned by $\alpha_1^\vee+2\alpha_2^\vee$,  Condition~(C) for $\lambda+\rho$ reads
$$a_1+2a_2\,=\,\textstyle{\frac{3}{2}}.$$
Clearly,
$(a_1,a_2)=(\frac{1}{2},\frac{1}{2})$ is a nice solution to this linear equation, and it leads to
$$\lambda+\rho\,=\textstyle{\frac{1}{2}}\rho\,=\,
-\textstyle{\frac{1}{2}}\varepsilon_1-\varepsilon_2+\textstyle{\frac{3}{2}}\varepsilon_3.$$
Using [\cite{Bo}, Planche~IX] it is easy to check that
$\Phi_\lambda^\vee\cap\Phi^+ =\,\{\gamma_1,\gamma_2\}$ where
$\gamma_1=\varepsilon_3-\varepsilon_1$ and $\gamma_2=-2\varepsilon_2+\varepsilon_1+
\varepsilon_3$.
Since $\langle\lambda+\rho,\alpha_1^\vee\rangle\not\in
\Z^{>0}$ and $\langle\lambda+\rho,\gamma_i^\vee\rangle
\in\Z^{>0}$ for  $i=1,2$, we now can argue as before to conclude that $\lambda+\rho$ satisfies all Losev's conditions.

In order to find another highest weight, $\lambda'+\rho=b_1\varpi_1+b_2\varpi_2$, leading to a multiplicity-free primitive ideal of $\mathcal{X}_{\O(e)}$ we shall once again assume that $\Phi_\lambda=\Phi_{\lambda'}$. Then $b_2\in\frac{1}{2}\Z$ and $b_1=\frac{3}{2}-2b_2$ thanks to Condition~(C). Since $\varpi_1=2\alpha_1+\alpha_2=\varepsilon_3-
\varepsilon_2$ and
$\varpi_2=\widetilde{\alpha}=3\alpha_1+2\alpha_2
=2\varepsilon_3-\varepsilon_1-\varepsilon_2$, we obtain
\begin{eqnarray*}
\lambda'+\rho&=&(\textstyle{\frac{3}{2}}-2b_2)(\varepsilon_3-\varepsilon_2)+b_2(2\varepsilon_3
-\varepsilon_1-\varepsilon_2)\\
&=&\textstyle{\frac{3}{2}}\varepsilon_3+(b_2-\textstyle{\frac{3}{2}})\varepsilon_2-b_2\varepsilon_1.
\end{eqnarray*}
As a consequence,
\begin{eqnarray*}
\langle\lambda'+\rho,\gamma_1^\vee\rangle&=&2(\lambda'+\rho\,|\,\varepsilon_3-\varepsilon_1)\,=\,\textstyle{\frac{3}{2}}+b_2\in\Z^{>0}\\
\langle\lambda'+\rho,\gamma_2^\vee\rangle&=&
\textstyle{\frac{2}{6}}(\lambda'+\rho\,|\,-2\varepsilon_2+\varepsilon_1+\varepsilon_3)
\,=\,
\textstyle{\frac{3}{2}}-b_2\in\Z^{>0}.
\end{eqnarray*}
As $b_2\in\frac{1}{2}\Z$, we have only two solutions, namely, $b_2=\pm\frac{1}{2}$. Setting $b_2=\frac{1}{2}$ yields $\lambda+\rho=\frac{1}{2}\rho$ which we have already met, whilst setting $b=-\frac{1}{2}$ gives $\lambda'+\rho=\frac{5}{2}\varpi_1-\frac{1}{2}\varpi_2$ which is what we are looking for. As $\langle\lambda'+\rho,\alpha_1^\vee\rangle\not\in
\Z^{>0}$ and $\langle\lambda'+\rho,\gamma_i^\vee\rangle
\in\Z^{>0}$ for  $i=1,2$, we see that $I(\lambda)=I(\frac{1}{2}\rho)$ and  $I(\lambda')=I(\frac{3}{2}\varpi_1-\frac{3}{2}\varpi_2)$ are  multiplicity free primitive ideals of $\mathcal{X}_{\O(e)}$. Using the above expressions for $\lambda+\rho$ and $\lambda'+\rho$ in the $\varepsilon$-basis it is easy to check that $(\lambda+\rho\,|\lambda+\rho)\ne (\lambda+\rho\,|\lambda'+\rho)$. Hence $I(\lambda)\ne I(\lambda')$.
\begin{rem}
The primitive ideals $I(\lambda)$ and $I(\lambda')$ from Subsection~\ref{5.7} were first discovered by Joseph [\cite{Jo2}] who proved, by using his Goldie rank polynomials, that they are multiplicity free and exhaust the
whole set of completely prime primitive ideals in $\mathcal{X}_{\O(e)}$. He also claimed 
without a proof that
the corresponding graded ideals ${\rm gr}(I(\lambda))$ and
 ${\rm gr}(I(\lambda'))$ of the symmetric algebra $S(\g)$ are prime. This is, of course, a remarkable property possessed only by some very special primitive ideals. However, Joseph's claim was later refuted by Vogan who showed in [\cite{Vo}]
 that the ideal ${\rm gr}(I(\lambda))$ cannot be prime and {\it conjectured} the primeness of 
 ${\rm gr}(I(\lambda'))$;
 see [\cite{Vo}, Conjecture~5.6]. It seems that this conjecture is still open. 
 The complete primeness of 
 $I(\lambda)$ was later reproved by Levasseur and Smith [\cite{LS}] who used a completely different method involving the Joseph ideal in type ${\sf B_3}$.

Vogan also proved in [\cite{Vo}] that the primitive ideals of $U(\g)$ arising as  left (or right) annihilators of Harish-Chandra bimodules associated with irreducible unitary representations of the complex Lie group $G$ are always completely prime. His argument is quite short and relies on a 
ring-theoretic lemma due to Kaplansky. 
To add more spice to this story we mention that only one of the completely prime ideals $I(\lambda)$ and $I(\lambda')$ 
can be attached (in a meaningful way) to an irreducible unitary representation of the complex Lie group of type ${\sf G_2}$. This follows from
Duflo's classification of irreducible unitary representations of the complex Lie groups of rank $2$; see [\cite{Du2}].
\end{rem}

Suppose $e$ is a rigid nilpotent element in a finite dimensional simple Lie algebra $\g$. The above remark brings up the following important and challenging problems:
\begin{itemize}
\item[1.] Describe all primitive ideals $I\in\mathcal{X}_{\O(e)}$ for which ${\rm gr}(I)$ is a prime ideal of $S(\g)$ (such ideals are necessarily multiplicity free). Is it true that the subset of such ideals in $\mathcal{X}_{\O(e)}$
is always non-empty?

\medskip

\item[2.] Describe all multiplicity-free primitive ideals $I\in\mathcal{X}_{\O(e)}$ which arise as left (or right) annihilators of Harish-Chandra bimodules associated with irreducible unitary representations of the complex Lie group $G$.  Is it true that the subset of such ideals in $\mathcal{X}_{\O(e)}$
is always non-empty?

\medskip

\item[3.] Describe the rigid orbits $\O(e)$ for which {\it all}
completely prime ideals in $\mathcal{X}_{\O(e)}$ are multiplicity free. Is it true that all rigid orbits in $\g$ have this property?

\end{itemize}
\begin{rem}
Levasseur and Smith showed in [\cite{LS}] that the orbit from Subsection~\ref{5.7} has a non-normal Zariski closure in $\g$. Comparing the results of this paper with results
of Kraft--Procesi, Broer and Sommers on orbits with normal closures
one observes that for any rigid nilpotent element $e$ in a finite dimensional simple Lie algebra $\g$ of type other than ${\sf E_7}$ and ${\sf E_8}$  the affine variety $\overline{\O(e)}\subset\g$ is normal if and only if the algebra $U(\g,e)$ affords a unique one-dimensional representation. This statement is likely to hold for all simple Lie algebras, but verifying it for types ${\sf E_7}$ and ${\sf E_8}$ 
would require more information on nilpotent orbits with normal Zariski closures than is currently available in the literature. A possible link between 
normality and the number of one-dimensional representations
 of $U(\g,e)$ for $e$ rigid was suggested to the author by Thierry Levasseur.
\end{rem}
\begin{rem}
A well-known conjecture of
McGovern states that any non-maximal completely prime primitive ideal of $U(\g)$ is parabolically induced from a completely prime primitive ideal of $U(\l)$ for some proper Levi subalgebra $\l$ of $\g$; see [\cite{Mc}, Conjecture~8.3].
One can see by inspection that the primitive ideals $I$ constructed in this paper 
have the form $I={\rm Ann}_{U(\g)}\,L(\lambda)$ for some $\lambda\in\Lambda^+$. Therefore, all of them are maximal in $U(\g)$. 
In conjunction with [\cite{PT}, Theorem~5] and some results on the number of one-dimensional $U(\g,e)$-modules proved in [\cite{P2}] this shows that McGovern's conjecture holds true for all
multiplicity-free primitive ideals of $U(\g)$
whose associated varieties are not listed in [\cite{PT}, Table~0].  That list contains just one induced orbit
in types ${\sf E_6}$, ${\sf E_7}$, ${\sf F_4}$ and four induced orbits in type ${\sf E_8}$.
\end{rem}

In the table that follows we collect some  information on rigid orbits in exceptional Lie algebras. One should keep in mind that the expressions for $2\rho_e=\sum_{i=1}^\ell m_i\alpha_i$ in the sixth column of the table are tied up with
our choice of pinning for $e$. Here our convention is that
$$2\rho_e\,=\,\textstyle{{m_1\atop{}}{m_3\atop{}}{m_4\atop m_2}{m_5\atop{}}{\,\cdots\,\,\atop{}}{m_\ell\atop{}}},\quad\ 2\rho_e\,=\,m_1\,m_2\,m_3\,m_4,\quad\  2\rho_e\,=\,m_1\,m_2$$ in types ${\sf E}_\ell$, ${\sf F_4}$ and ${\sf G_2}$, respectively. The seventh column of the table contains the values of $\frac{1}{2}h^\vee-\rho$ for all nonzero {\it special} rigid orbits in exceptional Lie algebras. We note that the Dynkin labels of $\O(e)$ for $e$ rigid and of $\O(e^\vee)\subset\g^\vee$ for $e$ special rigid are listed in [\cite{Lo2}, pp.~4861, 4862]. The dimensions of the centralisers of nilpotent elements $e$ in exceptional Lie algebras $\g$
and the isomorphism classes of the component groups $\Gamma=G_e/G_e^\circ$ can be found in many places; see [\cite{C}, pp.~401--407], for example. Finally, in [\cite{dG}] one can find the values of $\dim\big(\g_e/[\g_e,\g_e]\big)$ for all such $e$.

\vfill

\pagebreak

\begin{table}
{\sc Some data for nonzero rigid orbits in exceptional Lie algebras}

\bigskip

\medskip

\begin{tabular}{|c|c|c|c|c|c|c|c|c|c|}
\hline\hline Type &Dynkin label & $\Gamma$ &$|\Phi^+(0)|$&
$|\Phi^+(1)|$ &$2\rho_e$ &  $\frac{1}{2}h^\vee-\rho$ & $|\mathcal{E}|$  \\ \hline
${\sf E_8}$& ${\sf A_1}$ & $1$ & 63& 28 &$\scriptstyle{{72\atop{}}{142\atop{}}{224\atop 106}{172\atop{}}{132\atop{}}{90\atop{}}{46\atop{}}}$ & $-\varpi_4$ &1 \\
${\sf E_8}$&${\sf 2A_1}$ & $1$ & 42 & 32 
&$\scriptstyle{{60\atop{}}{118\atop{}}{182\atop 88}{142\atop{}}{116\atop{}}{74\atop{}}{38\atop{}}}$& $-\varpi_4-\varpi_6$ &1 \\
${\sf E_8}$&${\sf 3A_1}$ & $1$ &  37 & 27 &$\scriptstyle{{56\atop{}}{103\atop{}}{157\atop 77}{125\atop{}}{100\atop{}}{66\atop{}}{34\atop{}}}$&  &1 \\
${\sf E_8}$&${\sf 4A_1}$ & $1$ & 28& 28 &$\scriptstyle{{46\atop{}}{94\atop{}}{135\atop 71}{113\atop{}}{84\atop{}}{60\atop{}}{29\atop{}}}$&  &1\\
${\sf E_8}$& ${\sf A_2+A_1}$ & $S_2$ & 30 & 22 &$\scriptstyle{{44\atop{}}{86\atop{}}{130\atop 65}{104\atop{}}{83\atop{}}{54\atop{}}{28\atop{}}}$& 
$-\textstyle{\sum}_{i=2,3,5,7}\,\varpi_i$ &1\\
${\sf E_8}$&${\sf A_2+2A_1}$ & $1$ & 23& 24 &$\scriptstyle{{40\atop{}}{80\atop{}}{116\atop 60}{98\atop{}}{76\atop{}}{50\atop{}}{26\atop{}}}$&  $-\textstyle{\sum}_{i=2,3,5,6}\,\varpi_i$&1
\\
${\sf E_8}$&${\sf A_2+3A_1}$ & $1$ & 22 & 21 &$\scriptstyle{{38\atop{}}{73\atop{}}{105\atop 54}{87\atop{}}{66\atop{}}{46\atop{}}{23\atop{}}}$&  &1
\\
${\sf E_8}$&${\sf 2A_2+A_1}$ & $1$ & 21 & 18 &$\scriptstyle{{34\atop{}}{68\atop{}}{98\atop 50}{80\atop{}}{63\atop{}}{42\atop{}}{22\atop{}}}$&  &1\\
${\sf E_8}$&${\sf A_3+A_1}$ & $1$ & 21 & 17 &$\scriptstyle{{32\atop{}}{68\atop{}}{97\atop 49}{79\atop{}}{64\atop{}}{42\atop{}}{22\atop{}}}$&  &2
\\
${\sf E_8}$&${\sf 2A_2+2A_1}$ & $1$ & 16& 20&$\scriptstyle{{33\atop{}}{63\atop{}}{90\atop 46}{76\atop{}}{56\atop{}}{38\atop{}}{21\atop{}}}$ &  &1
\\
${\sf E_8}$&${\sf A_3+2A_1}$ & $1$ & 16 & 18&$\scriptstyle{{30\atop{}}{60\atop{}}{87\atop 44}{72\atop{}}{56\atop{}}{37\atop{}}{19\atop{}}}$ &  &1 \\
${\sf E_8}$&${\sf D_4(a_1)+A_1}$ & $S_3$ & 16 & 16 &$\scriptstyle{{28\atop{}}{56\atop{}}{84\atop 42}{68\atop{}}{52\atop{}}{37\atop{}}{18\atop{}}}$&  $-\textstyle{\sum}_{i\ne 4,7}\,\varpi_i$&1 \\
${\sf E_8}$&${\sf A_3+A_2+A_1}$ & $1$ & 14& 15 &$\scriptstyle{{26\atop{}}{51\atop{}}{75\atop 38}{62\atop{}}{48\atop{}}{33\atop{}}{17\atop{}}}$ &  &1
\\
${\sf E_8}$&${\sf 2A_3}$ & $1$ & 12 & 14 &$\scriptstyle{{23\atop{}}{47\atop{}}{67\atop 34}{55\atop{}}{42\atop{}}{29\atop{}}{14\atop{}}}$&  &1
\\
${\sf E_8}$&${\sf A_4+A_3}$ & $1$ & 8 & 12 &$\scriptstyle{{18\atop{}}{36\atop{}}{54\atop 28}{44\atop{}}{34\atop{}}{23\atop{}}{11\atop{}}}$&  &1 \\
${\sf E_8}$&${\sf A_5+A_1}$ & $1$ & 8 & 11 &$\scriptstyle{{19\atop{}}{36\atop{}}{54\atop 27}{43\atop{}}{33\atop{}}{21\atop{}}{11\atop{}}}$&  &$\ge 2$ \\
${\sf E_8}$&${\sf D_5(a_1)+A_2}$ & $1$ & 8& 11&$\scriptstyle{{18\atop{}}{34\atop{}}{52\atop 26}{42\atop{}}{33\atop{}}{23\atop{}}{11\atop{}}}$ &  &$\ge 2$
\\
${\sf E_7}$&${\sf A_1}$ & $1$ & 30 & 16&$\scriptstyle{{26\atop{}}{50\atop{}}{80\atop 37}{57\atop{}}{40\atop{}}{21\atop{}}}$ & $-\varpi_4$ &1 \\
${\sf E_7}$&${\sf 2A_1}$ & $1$ & 21 & 16 &$\scriptstyle{{22\atop{}}{42\atop{}}{64\atop 31}{47\atop{}}{36\atop{}}{17\atop{}}}$& $-\varpi_4-\varpi_6$ &1 \\
${\sf E_7}$&${\sf (3A_1)'}$ & $1$ & 16 & 15 &$\scriptstyle{{20\atop{}}{35\atop{}}{53\atop 26}{40\atop{}}{30\atop{}}{15\atop{}}}$&  &1 \\
${\sf E_7}$&${\sf 4A_1}$ & $1$ & 15 & 13 &$\scriptstyle{{17\atop{}}{36\atop{}}{48\atop 25}{38\atop{}}{26\atop{}}{14\atop{}}}$&  &1 \\
${\sf E_7}$&${\sf A_2+2A_1}$ & $1$ & 10 & 12 &$\scriptstyle{{14\atop{}}{27\atop{}}{38\atop 19}{30\atop{}}{22\atop{}}{11\atop{}}}$& $-\textstyle{\sum}_{i\ne 1,4,7}\,\varpi_i$ &1 \\
${\sf E_7}$&${\sf 2A_2+A_1}$ & $1$ & 8& 10 &$\scriptstyle{{11\atop{}}{23\atop{}}{32\atop 17}{25\atop{}}{18\atop{}}{9\atop{}}}$&  &1 \\
${\sf E_7}$&${\sf (A_3+A_1)'}$ & $1$ & 8 & 9&$\scriptstyle{{10\atop{}}{22\atop{}}{31\atop 16}{24\atop{}}{18\atop{}}{9\atop{}}}$ &  &2 \\
${\sf E_6}$&${\sf A_1}$ & $1$ & 15 & 10&$\scriptstyle{{12\atop{}}{22\atop{}}{35\atop 16}{22\atop{}}{12\atop{}}}$ & $-\varpi_4$ &1 \\
${\sf E_6}$&${\sf 3A_1}$ & $1$ & 7 & 9 &$\scriptstyle{{9\atop{}}{15\atop{}}{22\atop 11}{15\atop{}}{9\atop{}}}$&  &1 \\
${\sf E_6}$&${\sf 2A_2+A_1}$ & $1$ & 3 & 6&$\scriptstyle{{5\atop{}}{10\atop{}}{14\atop 8}{10\atop{}}{5\atop{}}}$ &  &1 \\
${\sf F_4}$&${\sf A_1}$ & $1$ & 9 & 7& $\scriptscriptstyle{14\,\,21\,\,30\,\,16}$&  &1 \\
${\sf F_4}$&${\sf \widetilde{A}_1}$ & $S_2$ & 9 & 4 & $\scriptscriptstyle{10\,\,18\,\,27\,\,14}$& $-\varpi_2$ &1 \\
${\sf F_4}$&${\sf A_1+\widetilde{A}_1}$ & $1$ & 4 & 6& $\scriptscriptstyle{8\,\,15\,\,20\,\,12}$ & 
$-\varpi_2-\varpi_4$ &1 \\
${\sf F_4}$&${\sf A_2+\widetilde{A}_1}$ & $1$ & 4 & 3 & $\scriptscriptstyle{5\,\,11\,\,15\,\,8}$&  &1 \\
${\sf F_4}$&${\sf \widetilde{A}_2+A_1}$ & $1$ & 2 & 4 & $\scriptscriptstyle{6\,\,10\,\,14\,\,7}$&  &2 \\
${\sf G_2}$&${\sf A_1}$ & $1$ & 2 & 1  & $\scriptscriptstyle{6\,\,4}$&  &1 \\
${\sf G_2}$&${\sf \widetilde{A}_1}$ & $1$ & 1 & 1 & $\scriptscriptstyle{5\,\,3}$&  &2 \\
\hline
\end{tabular}
\end{table}


\begin{thebibliography} {99}
\nextref BV0 \auth W. Barbasch D. Vogan \endauth \paper{Primitive
ideals and  orbital integrals in classical groups} \journal{Math.
Ann.} \Vol{259} \Year{1982} \Pages{153--199}

\nextref BV \auth W. Barbasch D. Vogan \endauth \paper{Unipotent
representations of complex semisimple groups} \journal{Ann. of Math.
(2)} \Vol{121} \nombre{1} \Year{1985} \Pages{41--–110}

\nextref BK \auth W. Borho H. Kraft
\endauth
\paper{\"Uber die Gelfand--Kirillov-Dimension} \journal{Math. Ann.}
\Vol{220}\Year{1976}\Pages{1--24}


\nextref Bo \auth N. Bourbaki
\endauth
\book{Groups et Alg{\`e}bres de Lie} \bookseries{Chapitres IV, V,
VI} \publisher{Hermann, Paris} \Year{1968}

\nextref BG \auth J. Brown S. Goodwin
\endauth
\paper{Finite dimensional irreducible representations of finite
$W$-algebras associated to even multiplicity nilpotent orbits in
classical Lie algebras} \journal{Math. Z.}
\Vol{273}\Year{2013}\Pages{123--160}


\nextref BGK \auth J. Brundan S. Goodwin A. Kleshchev \endauth
\paper{Highest weight theory for finite $W$-algebras} \journal{IMRN}
\nombre{15} \Year{2008} Art. ID rnn051, 53 pp.

\nextref Bry
\auth R. Brylinski
\endauth
\paper{Dixmier algebras for classical complex nilpotent orbits via
Kraft--Procesi models} \bookseries{Progress in Math., Vol.~213}
\publisher{Birkh{\"a}user, Boston} 2003, pp.~49--67.

\nextref C \auth R.W. Carter
\endauth
\book{Finite Groups of Lie Type: Conjugacy Classes and Complex Characters} 
\bookseries{Pure and Appl. Math.} \publisher{Wiley, Chichester, New York etc.} \Year{1985}

\nextref Chen \auth Y. Chen
\endauth
\paper{Left cells in the Weyl group of type $E_8$} \journal{J.
Algebra} \Vol{230}\Year{2000}\Pages{805--830}

\nextref CS \auth Y. Chen J.-y. Shi
\endauth
\paper{Left cells in the Weyl group of type $E_7$} \journal{Comm.
Algebra} \Vol{26} \nombre{11} \Year{1998}\Pages{3837--3852}

\nextref CM \auth D.H. Collingwood W.M. McGovern
\endauth
\book{Nilpotent Orbits in Semisimple Lie Algebras} 
\publisher{Van Nostrand Reinhold, New York} \Year{1993}

\nextref dG \auth W.A.{\,\,}de Graaf
\endauth
\paper{Computations with nilpotent elements in ${\sf SLA}$} available at
{\tt arXiv:1301.1149v1 [math.RA]}.

\nextref Dix \auth J. Dixmier
\endauth
\book{Alg{\`e}bres Enveloppantes} \publisher{Gauthier-Villars,
Paris, Bruxelles, Montr{\'e}al} \Year{1974}

\nextref Du1 \auth M. Duflo
\endauth
\paper{Sur la classification des id{\'e}aux primitifs dans
l'alg{\`e}bre enveloppante d'une alg{\`e}bre de Lie semi-simple}
\journal{Ann. of Math.} \Vol{105}\Year{1977}\Pages{107--130}


\nextref Du2 \auth M. Duflo
\endauth
\paper{Repr{\'e}sentations unitaires irr{\'e}ducibles des groupes semi-simples complexes de rang deux}
\journal{Bull. Soc. Math. France} \Vol{107}\Year{1979}\Pages{55--96}



\nextref GG \auth W.L. Gan V. Ginzburg
\endauth
\paper{Quantization of Slodowy slices} \journal{Int. Math. Res.
Not.} \Vol{5}\Year{2002}\Pages{243--255}

\nextref Ge \auth M. Geck \endauth \paper{PyCox: Computing with
(finite) Coxeter groups and Iwahori--Hecke algebras} \journal{LMS J.
Comput. Math.} \Vol{15} \Year{2012} \Pages{231--256}

\nextref Gi \auth V. Ginzburg
\endauth
\paper{Harish-Chandra bimodules for quantized Slodowy slices}
\journal{Represent. Theory} \Vol{13} \Year{2009}\Pages{236--371}


\nextref{GRU} \auth S. Goodwin G. R{\"o}hrle  G. Ubly
\endauth
\paper{On $1$-dimensional representations of finite $W$-algebras
associated to simple Lie algebras of exceptional type} \journal{LMS
J. Comput. Math.} \Vol{13}\Year{2010}\Pages{357--369}

\nextref Hu \auth J.E. Humphreys
\endauth
\paper{Modular representations of simple Lie algebras}
\journal{Bull. Amer. Math. Soc. (N.S.)}
\Vol{35}\Year{1998}\Pages{105--122}

\nextref Hu1 \auth J.E. Humphreys
\endauth
\book{Representations of Semisimple Lie Algebras in the BGG Category
$\mathcal O$} \bookseries{Grad. Stud. Math.} \Vol{94}
\publisher{Amer. Math. Soc., Providence, RI} \Year{2008}


\nextref Ja \auth J.C. Jantzen
\endauth
\book{Einh{\"u}lende Algebren Halbeinfacher Lie-Algebren}
\bookseries{Ergebnisse der Math.} \Vol{3} \publisher{Springer, New
York, Tokyo, etc.} \Year{1983}

\nextref Ja1 \auth J.C. Jantzen
\endauth
\paper{Nilpotent orbits in representation theory} \In{``Lie Theory.
Lie Algebras and Representations'', J.-P. Anker and B. Orsted, eds.}
\bookseries{Progress in Math.} \Vol{228} \publisher{Birkh\"auser,
Boston} \Year{2004}\Pages{1--211}

\nextref Jo1 \auth A. Joseph
\endauth
\paper{Gelfand--–Kirillov dimension for the annihilators of simple
quotients of Verma modules} \journal{J.
London Math. Soc. (2)} \Vol{18}\nombre{1} \Year{1978} \Pages{50--60}


\nextref JoT \auth A. Joseph
\endauth
\paper{Primitive ideals in enveloping algebras} \In{Proceedings of
the International Congress of Mathematicians, Vol. 1,2 (Warsaw, 1983)} \publisher{PWN, Warsaw}
\Year{1984}\Pages{403--414}

\nextref Jo2 \auth A. Joseph
\endauth
\paper{Kostant's problem and Goldie rank} \In{``Non-commutative Harmonic Analysis and
Lie Groups''}
\bookseries{Lecture Notes in Mathematics} \Vol{880} \publisher{Springer-Verlag, Berlin-Heidelberg-New York} 
\Year{1981}\Pages{249--266}



\nextref{KL} \auth D. Kazhdan G. Lusztig
\endauth
\paper{Representations of Coxeter groups and Hecke algebras} \journal{Invent. Math.} \Vol{53}\Year{1979}\Pages{165--184}

 
\nextref LT \auth R. Lawther D.M. Testerman
\endauth
\paper{Centres of Centralizers of Unipotent Elements in Simple
Algebraic Groups} \journal{Mem. Amer. Math. Soc.}
\Vol{210}\Year{2011} \nombre{980} vi+188 pp.

\nextref{LS} \auth T. Levasseur S.P. Smith
\endauth
\paper{Primitive ideals and nilpotent orbits in type $G_2$} \journal{J. Algebra} \Vol{114}\Year{1988}\Pages{81--105} 


\nextref Lo \auth I.V. Losev
\endauth
\paper{Quantized symplectic actions and $W$-algebras} \journal{J.
Amer. Math. Soc.}
\Vol{23}\Year{2010}\Pages{35--59}


\nextref Lo1 \auth I.V. Losev
\endauth
\paper{Finite dimensional representations of $W$-algebras}
\journal{Duke Math. J.} \Vol{159} \Year{2011} \Pages{99--143}

\nextref Lo2 \auth I.V. Losev
\endauth
\paper{$1$-dimensional representations and parabolic induction for
$W$-algebras} \journal{Adv. Math.} \Vol{226} \Year{2011}
\Pages{4841--4883}

\nextref Lo10 \auth I.V. Losev
\endauth
\paper{On the structure of the category $\O$ for $W$-algebras} \journal{S{\'e}minaires et Congr{\`e}s}
\Vol{24}\Year{2012}\Pages{351--368}


\nextref Mc \auth W.M. McGovern
\endauth
\paper{Completely Prime Maximal Ideals and Quantization} \journal{Mem. Amer. Math. Soc.}
\Vol{108}\Year{1994} \nombre{519} viii+67 pp.


\nextref Mo \auth C. M{\oe}glin
\endauth
\paper{Id{\'e}aux compl{\`e}tement premiers de l'alg{\`e}bre enveloppante de 
$\mathfrak{gl}_n(\mathbb{C})$} \journal{J. Algebra} \Vol{106} \Year{1987}
\Pages{287--366}







\nextref P02 \auth A. Premet
\endauth
\paper{Special transverse slices and their enveloping algebras}
\journal{Adv. Math.} \Vol{170} \Year{2002} \Pages{1--55}

\nextref
P07
\auth A. Premet
\endauth
\paper{Enveloping algebras of Slodowy slices and the Joseph ideal}
\journal{J. Eur. Math. Soc.} \Vol{9} \Year{2007} \Pages{487--543}

\nextref P07' \auth A. Premet
\endauth
\paper{Primitive ideals, non-restricted representations and finite
$W$-algebras} \journal{Mosc. Math. J.} \Vol{7} \Year{2007}
\Pages{743--762}



\nextref P10 \auth A. Premet
\endauth
\paper{Commutative quotients of finite $W$-algebras} \journal {Adv.
Math.} \Vol{225} \Year{2010} \Pages{269--306}

\nextref P11 \auth A. Premet
\endauth
\paper{Enveloping algebras of Slodowy slices and Goldie rank}
\journal {Transform. Groups} \Vol{16}\Year{2011} \Pages{857--888}

\nextref P2 \auth A. Premet
\endauth
\paper{One-dimensional representations of finite $W$-algebras and Humphreys' conjecture}
in preparation.

\nextref PT \auth A. Premet L. Topley
\endauth
\paper{Derived subalgebras of centralisers and finite $W$-algebras} \journal{Compositio Math.}, to appear; preprint 
available at {\tt arXiv:1301.4653v2 [math.RT]}.


\nextref Sk \auth S. Skryabin
\endauth
\paper{A category equivalence} Appendix to [\cite{P02}].



\nextref U \auth G. Ubly
\endauth
\book{A Computational Approach to $1$-dimensional Representations of
Finite $W$-algebras Associated with Simple Lie Algebras of
Exceptional Type} PhD thesis, University of Southampton, ePrints,
Soton, 2010, 177 pages; http:/\!/eprints.soton.ac.uk

\nextref Vo \auth D. Vogan 
\endauth
\paper{The orbit method and primitive ideals for semisimple Lie algebras} 
\In{``Lie Algebras and Related Topics''(Windsor, Ont. 1984)}
\bookseries{CMS Conf. Proc.} \Vol{5} \publisher{AMS, Providence, RI}
\Year{1986}\Pages{281--316}

\nextref Y \auth O. Yakimova
\endauth
\paper{On the derived algebra of a centraliser} \journal {Bull. Sci.
math.} \Vol{134}\Year{2010} \Pages{579--587}
\end{thebibliography}
\end{document}